\documentclass[12pt]{amsart}

\usepackage{amssymb}
\usepackage{latexsym}
\usepackage{amsmath}
\usepackage{amsthm}
\usepackage{mathrsfs}
\usepackage[colorlinks=true,bookmarksdepth=3]{hyperref}
\usepackage[all]{xy}
\usepackage{enumitem}
\usepackage[usenames,dvipsnames]{xcolor}
\usepackage[marginparwidth=1.8cm,margin=2.5cm]{geometry}
\usepackage{seqsplit}
\usepackage{tikz-cd}
\usepackage{tikz}
\usepackage{chngcntr}

\usetikzlibrary{calc}
\usetikzlibrary{decorations.markings}

\newcommand*\mycirc[1]{%
  \begin{tikzpicture}
      \node[draw,circle,inner sep=.5pt] {#1};
   \end{tikzpicture}}

\mathchardef\mhyphen="2D

\def\AA{\mathbf A}

\def\Cone{\operatorname{Cone}}
\def\Mon{\mathbf{Mon}}

\def\RPC{\mathbf{RPC}}
\def\Spec{\operatorname{Spec}}

\def\ZZ{\mathbf Z}

\def\trop{\operatorname{trop}}
\def\et{\mathrm{\acute{e}t}}
\def\overnorm#1{\overline{#1}\vphantom{#1}}

\def\TroPic{\operatorname{Tro\,Pic}}
\def\bTroPic{\operatorname{\mathbf{Tro\,Pic}}}
\def\TroJac{\operatorname{Tro\,Jac}}
\def\LogPic{\operatorname{Log\,Pic}}
\def\bLogPic{\operatorname{\mathbf{Log\,Pic}}}
\def\Pic{\operatorname{Pic}}
\def\bPic{\operatorname{\mathbf{Pic}}}

\def\Gm{\mathbf{G}_m}
\def\BGm{\mathrm B\Gm}
\def\Ga{\mathbf{G}_a}
\def\BGa{\mathrm B\Ga}
\def\tropGm{\mathbf{G}_m^{\mathrm{trop}}}
\def\logGm{\mathbf{G}_{\log}}
\def\ologGm{\overnorm{\mathbf{G}}_{\log}}

\def\LogSch{\mathbf{LogSch}}

\def\ShpVal{\mathbf{Cone}^\circ}
\def\tropfont{\mathscr}
\def\NS{\operatorname{NS}}
\def\PL{\mathsf{PL}}
\def\L{\mathsf{L}}
\def\Hom{\operatorname{Hom}}
\def\Sym{\operatorname{Sym}}

\newcounter{margmark}
\newtheorem{maintheorem}{Theorem}
\newtheorem{maincorollary}[maintheorem]{Corollary}
\newtheorem{theorem}{Theorem}
\numberwithin{theorem}{subsection}
\newtheorem{lemma}[theorem]{Lemma}
\newtheorem{corollary}[theorem]{Corollary}
\newtheorem{proposition}[theorem]{Proposition}

\newtheorem{sublemma}[equation]{Lemma}

\theoremstyle{definition}
\newtheorem{definition}[theorem]{Definition}

\theoremstyle{remark}
\newtheorem{example}[theorem]{Example}
\newtheorem{remark}[theorem]{Remark}
\newtheorem{subremark}[equation]{Remark}

\begin{document}

\title{The logarithmic Picard group and its tropicalization}
\author{Samouil Molcho}
\email{samouil.molcho@math.ethz.ch}
\address{ETH Z\"urich, R\"amistrasse 101, CH-8092, Z\"urich, Switzerland}

\author{Jonathan Wise}
\email{jonathan.wise@colorado.edu}
\address{University of Colorado, Campus Box 395, Boulder, CO 80309-0395 USA}


\date{\today}
\maketitle

\begin{abstract}
We construct the logarithmic and tropical Picard groups of a family of logarithmic curves and realize the latter as the quotient of the former by the algebraic Jacobian.  We show that the logarithmic Jacobian is a proper family of logarithmic abelian varieties over the moduli space of Deligne--Mumford stable curves, but does not possess an underlying algebraic stack.  However, the logarithmic Picard group does have logarithmic modifications that are representable by logarithmic schemes, all of which are obtained by pullback from subdivisions of the tropical Picard group.
\end{abstract}

\setcounter{tocdepth}{1}
\tableofcontents

\numberwithin{theorem}{subsection}
\numberwithin{equation}{theorem}
\section{Introduction}

Our concern in this paper is the extension of the universal Picard group to the boundary of the Deligne--Mumford moduli space of stable curves.  Over the interior, the Picard group of a smooth, proper, connected curve is well-known to be an extension of the integers by a smooth, proper, connected, commutative group scheme, the Jacobian.  These properties do not persist over the boundary, and natural variants sacrifice one or another of them to obtain others.

The Deligne--Mumford compactification of the moduli space of curves admits curves with nodal singularities.  As long as the dual graph of the curve is a tree, the Picard group remains an extension of a discrete, free abelian group --- the group of multidegrees --- by an abelian variety, but it becomes nonseparated in families because the multidegrees do.  One can focus here on the component of multidegree zero, which is an abelian variety and is well-behaved in families.

Should a curve degenerate so that its dual graph contains nontrivial loops, the multidegree~$0$ component of the Picard group remains separated, but fails to be universally closed.  The construction of compactifications of this group is the subject of a vast literature~\cite{ishida1978compactifications,dsouza1979compactification,oda1979compactifications,altman1980compactifying1,altman1980compactifying2,kajiwara1993logarithmic,caporaso1994a,MR1308406,jarvis2000compactification,MR1828599,caporaso2008compactified,caporaso2008neron,melo2011compactified,chiodo2015n}, of which the above references are only a sample.  We must direct the reader to the references for a history of the subject.

While we do not attempt to summarize all of the different approaches to compactifying the Picard group, we emphasize that all operate in the category of schemes, and none produces a proper group scheme.  Indeed, it is not possible to produce a proper group scheme, for the multidegree~$0$ component of the Picard group of a maximally degenerate curve is a torus, and there is no way of completing a torus to a proper group \emph{scheme}.

On the other hand, K.~Kato observed that the multiplicative group does have compactifications --- with group structure --- in the category of \emph{logarithmic} schemes~\cite[Section~2.1]{Kato-LogDeg}.  This gives reason to hope that the Picard group might also find a natural compactification in the category of logarithmic schemes, as Kato himself anticipated.  Kato proposed a definition for, and then calculated, the logarithmic Picard group of the Tate curve~\cite[Section~2.2.4]{Kato-LogDeg}.  Illusie advanced the natural generalization of Kato's calculation as a definition for the Picard group of an arbitrary logarithmic scheme~\cite[Section~3.3]{illusie1994logarithmic}. In the analytic category, Kajiwara, Kato and Nakayama constructed the logarithmic Picard group using Hodge-theoretic methods~\cite{kajiwara_kato_nakayama_2008}.  Significantly, they discovered the need to restrict attention to a subfunctor of the one defined by Illusie in order to get the logarithmic Picard group, and logarithmic abelian varieties in general, to vary well geometrically over logarithmic base schemes.  In the present work, we work entirely in the algebraic category -- but the condition of Kajiwara, Kato and Nakayama, which appears here under the heading of \emph{bounded monodromy}, first introduced in Section~\ref{sec:monodromy}, will play an essential role throughout.

\subsection*{The provenance of logarithmic geometry}

This section is intended to motivate the presence of logarithmic geometry in the compactification of the Picard group.  Consider a family of logarithmic curves $X$ over a $1$-parameter base $S$ with generic point $\eta$ and a line bundle $L_\eta$ on the general fiber of $X$.  Let $i : s \to S$ denote the inclusion of the closed point and also write $i : X_s \to X$ for the inclusion of the closed fiber; write $j : \eta \to S$ and $j : X_\eta \to X$ for the inclusion of the generic point and the generic fiber.    

Let $\tropfont S$ denote the ringed space $(s, i^{-1} j_\ast \mathcal O_\eta)$ and let $\tropfont X$ denote the ringed space $(X_s, i^{-1} j_\ast \mathcal O_{X_\eta})$.  Then $\tropfont L = i^{-1} j_\ast L_\eta$ is a line bundle on $\tropfont X$.

We can describe $\tropfont L$ by giving local trivializations and transition functions in $\Gm$.  However, these cannot necessarily be restricted to $X_s$ because a unit of $i^{-1} j_\ast \mathcal O_{X_\eta}^\ast$ may have zeroes or poles along components of the special fiber.

If the dual graph of $X_s$ is a tree then it is possible to modify the local trivializations to ensure that the transition functions have no zeroes or poles, but in general such a modification may not exist.

The degeneration of transition functions suggests we might compactify the Picard group by allowing `line bundles' whose transition functions are sometimes allowed to vanish or have poles.  If transition functions are thus permitted not to take values in a group then the objects assembled from them will no longer have a group structure.  However, this leads naturally to the consideration of rank~$1$, torsion-free sheaves.

Logarithmic geometry takes a different approach to the same idea.  Instead of keeping track of only the zeroes and poles of the transition functions, we instead keep track of their orders of vanishing and leading coefficients.  Together, order of vanishing and leading coefficient have the structure of a group and therefore the objects glued with transition functions in this group can be organized into a group as well.

The way this is actually done is to take the image of a transition function $f \in i^{-1} j_\ast \mathcal O_{X_\eta}^\ast$, not in $\mathcal O_{X_s} \cup \{ \infty \}$, but instead in $M_{X_s}^{\rm gp}$:
\begin{equation*}
M_{X_s}^{\rm gp} = i^{-1} j_\ast \mathcal O_{X_\eta}^\ast / \ker (i^{-1} \mathcal O_X^\ast \to \mathcal O_{X_s}^\ast)
\end{equation*}
That is, we obtain a natural limit $M_{X_s}^{\rm gp}$-torsor $P$ for $L_\eta$, whose isomorphism class lies in $H^1(X_s, M_{X_s}^{\rm gp})$.

Taking transition functions in $M_{X_s}$ has an added benefit, even when the dual graph of the special fiber is a tree.  Indeed, if $L_\eta$ extends to $L$ with limit $L_s$, one can always produce another limit $L(D)_s$ by twisting $L_s$ by a component $D$ of the special fiber.  But the effect of twisting by $D$ on $\tropfont L$ is to modify the local trivializations of $\tropfont L$ by units of $\mathcal O_{X_\eta}$.  This changes the local trivializations of $P$ by elements of $M_{X_s}^{\rm gp}$, but that only affects a cocycle representative by a coboundary.  In other words, the class of $P$ in $H^1(X_s, M_{X_s}^{\rm gp})$ is independent of twisting by components of $X_s$.

\subsection*{Logarithmic line bundles} 

It is sensible to take $M_X^{\rm gp}$-torsors as a candidate for a compactification of $\mathcal{O}_X^*$ in general, even when the base $S$ of the family is an arbitrary logarithmic scheme. In this paper, we use this observation to define logarithmic line bundles on a family of logarithmic curves $X \rightarrow S$ as torsors under the logarithmic multiplicative group, as Kato and Illusie proposed, that satisfy the additional bounded monodromy condition. This definition produces a stack $\bLogPic(X/S)$ --- the logarithmic Picard stack --- with respect to the strict \'etale topology on $S$, and an associated sheaf $\LogPic(X/S)$ --- the logarithmic Picard group --- via rigidification.  In Section~\ref{sec:unint} we explain why the bounded monodromy condition is necessary if infinitesimal deformation of logarithmic line bundles is to have the expected relationship to formal families of logarithmic line bundles. Thus the bounded monodromy condition can be considered to be the first subtlety of the theory. The second subtlety is that even with this condition, the logarithmic Picard stack and the logarithmic Picard group are not representable by an algebraic stack or algebraic space with a logarithmic structure respectively. The reason is essentially that the logarithmic multiplicative group is \emph{itself} not representable (see Section~\ref{sec:logGm}).  Nevertheless, $\bLogPic(X/S)$ and $\LogPic(X/S)$ do have all the formal properties of an algebraic stack and an algebraic space, albeit only in the logarithmic category. Specifically, $\LogPic(X/S)$ has a \emph{logarithmically} \'etale cover by a logarithmic scheme, and we prove that it is a smooth algebraic group object in the category of logarithmic schemes, with proper components:

\begin{maintheorem} \label{thm:A}
Let $X$ be a proper, vertical logarithmic curve over $S$.  The logarithmic Picard group $\LogPic(X/S)$ has a logarithmically smooth cover by a logarithmic scheme, is logarithmically smooth with proper components, is a commutative group object, has finite diagonal, and contains $\Pic^{[0]}(X/S)$ as a subgroup.
\end{maintheorem}
\begin{proof}
See Corollary~\ref{cor:sheaf-algebraic} for the existence of a logarithmically smooth cover, Theorem~\ref{thm:smooth} for the logarithmic smoothness, Corollary~\ref{cor:logpic-proper} for the properness, and Theorem~\ref{thm:diag} for the finiteness of the diagonal.  The group structure and inclusion of $\Pic^{[0]}(X/S)$ are immediate from the construction in Definition~\ref{def:logpic}.
\end{proof}

\begin{maincorollary} \label{cor:C}
The logarithmic Jacobian is a logarithmic abelian variety, in the sense of Kajiwara, Kato, and Nakayama~\cite{kajiwara_kato_nakayama_2008a,kajiwara_kato_nakayama_2008b}.
\end{maincorollary}
\begin{proof}
See Theorem~\ref{thm:log-ab-var}.
\end{proof}

Our results for the logarithmic Picard \emph{stack}, which remembers automorphisms, are similar, but a bit more technical:

\begin{maintheorem} \label{thm:D}
Let $X$ be a proper, vertical logarithmic curve over $S$.  The logarithmic Picard stack $\bLogPic(X/S)$ has a logarithmically smooth cover by a logarithmic scheme and its diagonal is representable by logarithmic spaces (sheaves with logarithmically smooth covers by logarithmic schemes).  The logarithmic Picard stack is logarithmically smooth and proper, is a commutative group stack, and receives a canonical homomorphism from the algebraic stack $\bPic^{[0]}(X/S)$.
\end{maintheorem}
\begin{proof}
See Theorem~\ref{thm:algebraic} for the existence of a logarithmically smooth cover and Corollary~\ref{cor:diag} for the claim about the diagonal.  The logarithmic smoothness is proved in Theorem~\ref{thm:smooth} and the properness is Corollary~\ref{cor:logpic-proper}.  The group structure and the map from $\Pic^{[0]}(X/S)$ are come directly from Definition~\ref{def:logpic}.  
\end{proof}

The difference between Theorems~\ref{thm:A} and~\ref{thm:D} and Olsson's result~\cite[Theorem~4.4]{olsson2004semistable} is that Olsson works with a fixed logarithmic structure on the base while we allow the logarithmic structure to vary.  This is necessary for the logarithmic Picard group to be proper.  Our method of proof also differs from Olsson's: we do not rely on the Artin--Schlessinger representability criteria (for which there is not yet an analogue in logarithmic geometry) and instead construct logarithmically smooth covers directly.

\subsection*{Connection with tropical geometry} 

Our analysis of the logarithmic Picard group, and our construction of the covers invoked in Theorems \ref{thm:A} and \ref{thm:D}, are direct:  we do not rely on general representability criteria, nor the theory of log 1-motives. Instead, our main tool is the intimate connection between algebraic, logarithmic and tropical geometry. This connection with tropical geometry is in our view a significant advantage, and perhaps the central point of this paper. From a strictly algebraic perspective $\bLogPic$ and $\LogPic$ may be mysterious objects, lying outside of the province of algebraic geometry.  However, the logarithmic perspective affords them a modular description and a \emph{tropicalization} --- which has a modular description of its own --- that precisely controls and explains our transgression beyond the boundaries of algebraic geometry.

We are not the first to oberve a connection between the logarithmic Picard group and tropical geometry: indeed, Foster, Ranganathan, Talpo, and Ulirsch observed that the geometry of the logarithmic Picard group is intimately tied up with the geometry of the tropical Picard group~\cite{foster2016logarithmic}, and the connection can also be seen in Kajiwara's work~\cite{kajiwara1993logarithmic}, albeit without explicit mention of tropical geometry. Our main contribution here is perhaps to extend connection to a family of logarithmic curves over an arbitrary base.

A tropical curve is simply a metric graph.  Baker and Norine introduced the tropical Jacobian as a quotient of tropical divisors by linear equivalence~\cite{baker_norine_2007}.  At first, the tropical Jacobian of a fixed graph (not yet metrized) was only a finite set, but subdivision of the graph suggests the presence of a finer geometric structure.  This was explained by Gathmann and Kerber~\cite{gathmann_kerber_2008}, who extended Baker and Norine's results to metric graphs, and Amini and Caporaso added a vertex weighting~\cite{AMINI20131}.  Mikhalkin and Zharkov defined tropical line bundles as torsors under a suitably defined sheaf of linear functions on a tropical curve~\cite[Definition~4.5]{mikhalkin2008tropical}.  They gave a separate definition of the tropical Jacobian as a quotient of a vector space by a lattice~\cite[Section~6.1]{mikhalkin2008tropical}, and proved an analogue of the Abel--Jacobi theorem, showing that the tropical Jacobian parameterizes tropical line bundles of degree~$0$.  We will recover this result in Corollary~\ref{cor:tropic}.

In order to relate the tropical Picard group and tropical Jacobian to their logarithmic analogues, we require a formalism by which tropical data may vary over a logarithmic base scheme.  This formalism is supplied by Cavalieri, Chan, Ulirsch, and the second author~\cite[Section~5]{cavalieri2017moduli}, who allow an arbitrary partially ordered abelian group to stand in for the real numbers in the definition of a tropical curve as a metric graph.  Logarithmic schemes come equipped with sheaves of partially ordered abelian groups and one can therefore speak of tropical curves over logarithmic base schemes.  We summarize these ideas in Sections~\ref{sec:trop-mod}--\ref{sec:trop-curves}. In a nutshell, given a logarithmic curve $X \rightarrow S$, we obtain a family of tropical curves $\tropfont X$ over $S$, whose fiber over a geometric point $s \in S$ is the dual graph of $X_s$, metrized by the characteristic monoid $\overnorm{M}_{S,s}^{\rm gp}$ of the logarithmic structure at $s$. We call the family $\tropfont X$ the tropicalization of $X/S$. The family $\tropfont X$, although an object over an arbitrary logarithmic scheme $S$, is essentially a combinatorial object: the logarithmic scheme $S$ has a stratification on which the characteristic monoid $\overnorm{M}_S$ is constant, and the fibers of $\tropfont X$ are constant on each stratum. The tropicalization $\tropfont X$ can thus be thought of as a combinatorial shadow of $X/S$, which remembers the combinatorics of the irreducible components of each fiber, and how the nodes of each fiber deform as one moves along strata of $S$. 

However, the connection of $X/S$ with $\tropfont X$ only comes to life after we begin doing some geometry on $\tropfont X$. Indeed, two of the most crucial constructions of the paper occur in Section~\ref{sec:tropicjac}, where we define a topology, and sheaves $\PL$ and $\L$ of piecewise linear and linear functions respectively on $\tropfont X$, over an arbitrary base $S$. This allows us, so to speak, to do some honest tropical \emph{geometry} on $\tropfont X$. 

Our sheaf $\PL$ is different from the sheaves of piecewise linear functions that are usually encountered in tropical geometry in that our piecewise linear functions are allowed to take values in a group $\overnorm{M}^{\rm gp}$ of arbitrary finite rank instead of the integers or real numbers. The groups $\overnorm{M}^{\rm gp}$ that appear vary over points of $S$, but are essentially the groups of sections of $\overnorm{M}_S^{\rm gp}$ over appropriately small neighborhoods around each point of $S$. This is the formalism that allows us to capture the fact that $X$ varies over a logarithmic base scheme $S$ instead of being the total space of a $1$-parameter degeneration. The sheaf of linear functions is built from $\PL$ by imposing the analogue of the \emph{balancing condition} that is ubiquitous in tropical geometry. 

We use the sheaf $\L$ to define the tropical Picard group $\TroPic(\tropfont X/S)$, and Picard stack $\bTroPic(\tropfont X/S)$ over logarithmic bases: We define them as the sheaf or stack of \emph{tropical line bundles}, which are the bounded monodromy torsors under $\L$. The sheaf $\TroPic(\tropfont X/S)$ and stack $\bTroPic(\tropfont X/S)$ are combinatorial objects, which in practice are simple to compute. For example, one still has a formula for the tropical Jacobian analogous to the formula of \cite{mikhalkin2008tropical}, which, over a point $s \in S$ takes the form 
\numberwithin{equation}{section}
\begin{equation} \label{eqn:130}
\TroJac(\tropfont X_s/s) = \Hom \bigl(H_1(\tropfont X_s, \overnorm{M}_{S,s}^\textup{gp})\bigr)^{\dagger}/H_1(\tropfont X_s)
\end{equation}
Here the $\dagger$ symbol indicates the bounded monodromy condition, which has a simple description in terms of the above formula: every loop $\gamma$ in the homology of the tropical curve $\tropfont X_s$ has a length $\ell(\gamma)$ valued in the monoid $\overnorm{M}_{S,s}$, and a homomorphism in $\Hom (H_1(\tropfont X_s, \overnorm{M}_{S,s}^\textup{gp}))$ has bounded monodromy if it sends every loop to an element of $\overnorm{M}_{S,s}^{\rm gp}$ that is bounded by some multiple of the length of the loop. This has the effect that if one generizes from a point $s$ to a point $t$, smoothing some of the nodes of the curve $X_s$, and therefore contracting some edges in $\tropfont X_s$, the homomorphism descends to be well defined on the homology of $\tropfont X_t$.  We refer the reader to Section~\ref{sec:tropicjac} for a thorough explanation of this phenomenon and the rest of the terms appearing in the definition of $\TroJac(\tropfont X/S)$.  We note also that when we are working with a one parameter degeneration of a curve, the group $\overnorm{M}_{S,s}^{\rm gp}$ reduces to $\mathbf{R}$, the bounded monodromy condition is automatic, and our formula recovers the \cite{mikhalkin2008tropical} formula. 

The connection with the theory of tropical divisors is also simple: the sheaves of linear and piecewise linear functions fit into an exact sequence 
\begin{equation} \label{eqn:129}
0 \to \L \to \PL \to \mathsf V \to 0
\end{equation}
with $\mathsf V$ the sheaf of tropical divisors. Thus, each tropical divisor $D$ determines a tropical line bundle $\L(D)$, which is the $\L$-torsor describing the obstruction of lifting $D$ to a piecewise linear function whose bend locus is $D$. It is shown in Section \ref{sec:monodromy} that, if the base $S$ is a valuation ring of arbitrary rank,  $\TroJac(\tropfont X/S)$ is precisely the group of tropical divisors on all semistable models (that is, subdivisions) of $\tropfont X$, up to piecewise linear functions. This gives another interpretation of the bounded monodromy condition in the valuative case, as those torsors that can be represented by a divisor on a semistable model; but for general $S$, the group $\TroJac(\tropfont X/S)$ maybe larger, with additional torsors that correspond to divisors on semistable models of $\tropfont X$ over logarithmic modifications of $S$ as well.

The presentation~\eqref{eqn:130}, and the exact sequence~\eqref{eqn:129}, describe two different means of producing tropical line bundles: from local systems and from tropical divisors.  The relationship between these is encoded in diagram~\eqref{eqn:34}.  The referee pointed out to us that this gives a compelling third method of producing tropical line bundles, from a labelling of the \emph{edges} of the tropical curve by integers (see Section~\ref{sec:degree}).

The connection of the tropical picture with logarithmic geometry is obtained through the process of \emph{tropicalization}. We do not attempt to explain this in the introduction, but we mention that the germ of the idea is elementary, utilizing the following formula, which is valid for every geometric point $s \in S$: 
\begin{align*}
H^i(X_s,\overnorm{M}_{X,s}) = H^i(\tropfont X_s,\PL_{\tropfont X_s})
\end{align*}
Thus, in a sense, $\tropfont X$ together with $\PL$ capture all information of $X/S$ that is reflected in $\overnorm{M}_X$.  

The connection between $M_X^{\rm gp}$ and the sheaf of linear functions, $\L$, and the connection between logarithmic line bundles and tropical line bundles, are more subtle.  In Section~\ref{sec:trop-pic}, we observe that the sheaf $\mathsf V$ is the tropicalization of the N\'eron--Severi group:  while the N\'eron--Severi group is only a presheaf on $X$, it descends to a sheaf on the tropicalization $\tropfont X$.  We obtain a tropicalization map $\bLogPic(X/S) \to \bTroPic(\tropfont X/S)$ from a morphism of complexes of presheaves $M_X^{\rm gp} \to [ \overnorm M_X^{\rm gp} \to \NS ]$ that is derived from the fundamental exact sequence of logarithmic geometry:
\begin{equation} \label{eqn:126}
0 \rightarrow \mathcal{O}_{X}^* \rightarrow M_{X}^{\rm gp} \rightarrow \overnorm{M}_{X}^{\rm gp} \rightarrow 0
\end{equation}

\begin{maintheorem} \label{thm:B}
Let $X$ be a proper, vertical logarithmic curve over $S$ and let $\tropfont X$ be its tropicalization.  There is an exact sequence:
\begin{equation*}
0 \to \Pic^{[0]}(X/S) \to \LogPic(X/S) \to \TroPic(\tropfont X/S) \to 0
\end{equation*}
There is also an exact sequence of group stacks:
\begin{equation*}
0 \to \bPic^{[0]}(X/S) \to \bLogPic(X/S) \to \bTroPic(\tropfont X/S) \to 0
\end{equation*}
\end{maintheorem}

Here the symbol $[0]$ denotes the multi-degree 0 part of $\Pic$. This is an instance where the connection between algebraic, logarithmic, and tropical geometry becomes exceptionally transparent. The tropicalization morphism $\LogPic(X/S) \to \TroPic(\tropfont X/S)$ allows us to to understand $\LogPic(X/S)$ in terms of a simple combinatorial object and a classical algebro-geometric object, the semi-abelian scheme $\Pic^{[0]}(X/S)$. The tropicalization morphism also allows us to identify the combinatorial data associated with $\TroPic(\tropfont X/S)$ necessary to construct proper, schematic compactifications of the Picard group, following Kajiwara, Kato, and Nakayama~\cite{kajiwara1993logarithmic,kajiwara_kato_nakayama_2015}, in Section~\ref{sec:prop-mod}.  

\begin{maintheorem} \label{thm:E}
	Let $X$ be a proper, vertical logarithmic curve over $S$ with tropicalization $\tropfont X$.  Polyhedral subdivisions of $\TroJac(\tropfont X/S)$ correspond to toroidal compactifications of $\Pic^{[0]}(X/S)$.
\end{maintheorem}

The exact sequences of Theorem~\ref{thm:B} are also needed in our demonstrations of Theorems~\ref{thm:A} and~\ref{thm:D}, particularly in the demonstration of the boundedness of $\bLogPic$ and its diagonal.  The tropical boundedness statements, proved in Sections~\ref{sec:trop-qcpt} and \ref{sec:trojac-bdd-diag}, are surely the most technical parts of the paper, and were the most difficult parts for us.  We rely on what might be called an `arithmetic $\epsilon$-$\delta$' formalism, in which $\epsilon$ and $\delta$ take values in a monoid; one cannot simply `choose $\delta > 0$' but must choose it from the monoid of available positive elements of the monoid.

\subsection*{Invariance properties and construction of the cover}

Polyhedral subdivisions of $\TroJac(\tropfont X/S)$ yield toroidal compactifications of $\Pic^{[0]}(X/S)$, which can in turn be interpreted as \emph{logarithmic modifications} of $\LogPic(X/S)$. Logarithmic modifications, together with logarithmic root stacks are purely combinatorial operations yielding proper monomorphisms of logarithmic schemes. Together with \'etale maps, the two operations generate the (full) \emph{logarithmic \'etale topology} of a logarithmic scheme. Logarithmic modifications, and similarily roots, form an inverse system, where $f_2: X_2 \to X$ can be considered finer than $f_1: X_1 \to X$ if $f_2$ factors as $g \circ f_1$ for a log modification $g:X_2 \to X_1$. Thus, Theorem~\ref{thm:E} can be seen as a heuristic `formula':
\begin{equation*}
\LogPic^0(X/S) = \varinjlim  \{ \text{toroidal compactifications of }\Pic^{[0]}(X/S) \}
\end{equation*}

Of course, this colimit does not exist as a scheme, but logarithmically it expresses $\LogPic^0(X/S)$ as the colimit of all its logarithmic modifications --- the minimal toroidal compactification of $\Pic^{[0]}(X/S)$, so to speak.

	One of the remarkable properties of $\LogPic(X/S)$ is that it is often invariant under both logarithmic modifications and root constructions: suppose that $S$ is logarithmically flat and that $f: T \rightarrow S$ is a logarithmic modification or a root $S$, and $Y$ is a logarithmic modification or root of the pullback of $X$ on $T$, for which $Y \rightarrow T$ is a logarithmic curve. Then $f^*\LogPic(X/S) = \LogPic(Y/T)$. In particular $\LogPic(X/S)$ forms a sheaf for the (small) full logarithmic \'etale topology on $S$: see Corollary~\ref{cor:log-etale-desc}. This invariance, together with the fundamental exact sequence, allows us to relate $\LogPic(X/S)$ with \emph{all} Picard groups $\Pic(Y/T)$ of all semistable models $Y$ of $X$ over logarithmic modifications and roots $T$ of $S$, by combining the natural map $\Pic(Y/T) \to \LogPic(Y/T)$ with the isomorphism to $\LogPic(X/S)$. In fact, it is this collection of spaces $\Pic(Y/T)$ that provides the cover in Theorem~\ref{thm:A}. As the kernel of the map $\Pic(Y/T) \to \LogPic(Y/T)$ is precisely the group of piecewise linear functions on the tropicalization of $Y/T$, we can write another heuristic `formula': 
\begin{equation*}
\LogPic(X/S) = \varinjlim \Pic(Y)/\PL(\tropfont Y)
\end{equation*}  
with $\PL(\tropfont Y)$ denoting the piecewise linear functions on the tropicalization of $Y$.  The map $\PL(\tropfont Y) \to \Pic(Y)$ comes from the fundamental exact sequence,~\eqref{eqn:126}.

	With the benefit of hindsight, this formula could have been used as a definition of $\LogPic(X/S)$ (at least over a logarithmically flat base): as line bundles on semistable models of $X/S$, up to the equivalence relation generated by pulling back to a further semistable model and the action of piecewise linear functions.  In our point of view, this presentation of $\LogPic(X/S)$ is an extrinsic presentation, whereas the definition we have chosen, in terms of torsors, is intrinsic. 

This intrinsic/extrinsic interplay is now seen in multiple places in logarithmic geometry. For example, it is observed in logarithmic Gromov-Witten theory, where an `intrinsic' definition of logarithmic stable maps is given by \cite{Chen,AC,GS}, whereas an extrinsic definition is given in the work of \cite{Li1,Li2,Kim,Dhruv}. In the stable map setting, the different definitions produce different spaces; yet, their Gromov-Witten invariants coincide.  This as an incarnation of the principle that logarithmic geometry captures the geometry of the interior of a space, and not of the specific logarithmic compactification chosen. Remarkably, for $\LogPic$, both intrisic and extrinsic approach yield the same space in many cases (rather than the same invariants of the space). This property has proved to be very useful in the study of $\LogPic$; it is for example key in the construction of a principal polarization, or in the study of N\'eron models via $\LogPic$. In recent years, various central problems in logarithmic geometry have been studied, using either an intrinsic or extrinsic approach.  For example, Chow theory for logarithmic schemes (by Barrott extrinsically~\cite{Barrott}, or Herr intrinsically~\cite{Herr}) or Donaldson-Thomas theory (\cite{MR}, extrinsically). We expect that understanding the connection between the dual approaches in any given problem would prove to be very fruitful.

\subsection*{Future work}

The Jacobian (and even the Picard stack) is equipped with a canonical principal polarization.  We are mute about the logarithmic analogue in this paper, but we will construct it in a subsequent one.

Our results are limited to relative dimension~$1$ because we do not have the means yet to study families of tropical varieties of higher dimension over logarithmic bases.  We also do not yet understand the higher dimensional analogue of the bounded monodromy condition.

Neither have we addressed any algebraicity properties of the tropical Picard group in a systematic way.  It follows from our results that the tropical Picard group has a logarithmically \'etale cover by a Kato fan, but it is less clear how one should characterize its diagonal (we prove only that it is quasicompact here), or whether one should demand further properties of a purely tropical cover. 

In Section~\ref{sec:prop-mod}, we indicate how the tropical Picard group can be used to construct proper schematic models of the logarithmic Picard group over a local base. Recent work of Abreu and Pacini describes polyhedral subdivisions of $\TroPic(\tropfont X/S)$ when $\tropfont X$ is the universal curve over the moduli space of $1$-pointed tropical curves (and, for certain degrees, over the moduli space of unpointed tropical curves)~\cite{abreu-pacini}.  They show that the corresponding compactification of the Picard group coincides with Esteves's compactification~\cite{MR1828599}.  We are pursuing a global construction of more general toroidal compactifications over the moduli space of stable curves in collaboration with Melo, Ulirsch, and Viviani.

The tropicalization method used in Section~\ref{sec:trop-pic} appears to generalize well to higher dimensional logarithmic varieties.  We hope to make further use of this construction in the future.

\subsection*{Conventions}

Let $X$ be a curve over $S$.  We use the term `Picard group' to refer to the sheaf on $S$ of isomorphism classes of line bundles on $X$, up to isomorphism and denote it $\Pic(X/S)$.  The stack of $\Gm$-torsors on $X$ is denoted in boldface:  $\bPic(X/S)$.  We use a superscript to denote a restriction on degree, and we refer to $\Pic^0(X/S)$ as the Jacobian of $X$.  We apply similar terminology when $X$ is a logarithmic curve or tropical curve over a logarithmic base $S$.

Throughout, we consider a logarithmic curve $X$ over $S$.  We regularly use $\pi : X \to S$ to denote the projection.

\subsection*{Acknowledgements}
\label{sec:ack}

This work benefitted from the suggestions, corrections, and objections of Dan Abramovich, Sebastian Casalaina-Martin, William D.\ Gillam, David Holmes, Dhruv Ranganathan, Martin Ulirsch, and the anonymous referee.  The example in Section~\ref{sec:ex-gen2} was worked out with the help of Margarida Melo, Martin Ulirsch, and Filippo Viviani at the workshop on Foundations of Tropical Schemes at the American Institute of Mathematics.  We thank them all heartily.

We are also grateful to the participants of the 2019 Intercity Geometry Seminar, organized by David Holmes, Chris Lazda, Adrien Sauvaget, and Arne Smeets for their feedback, from which this paper has benefitted considerably.

A substantial part of this paper was written, and some of its results proven, during the Workshop on Tropical Varieties in Higher Dimensions and a subsequent research visit at the Intitut Mittag--Leffler in April, 2018.  We gratefully acknowledge the Institute's hospitality during this period.

S.M.\ was supported by ERC-2017-AdG-786580-MACI.  This project has received funding from the European Research Council (ERC) under the European Union Horizon 2020 research and innovation program (grant agreement No.\ 786580).

J.W.\ was supported by an NSA Young Investigator's Grant, Award Number H98230-16-1-0329, a Simons Collaboration Grant, Award \#636210, and a Simons Fellowship, Award \#822534.

\section{Monoids, logarithmic structures, and tropical geometry}
\label{sec:background}

\subsection{Monoids}
\label{sec:monoids}
	
	In this paper, all monoids will be commutative, unital, integral, and saturated, although some results in this section are valid without those assumptions.  The monoid operation will be written additively, unless indicated otherwise.  Homomorphisms of monoids are assumed to preserve the unit.

\numberwithin{theorem}{subsubsection}
\subsubsection{Partially ordered groups}
\label{sec:pogroups}

\begin{definition} \label{def:sharp}
	A homomorphism of monoids $f : N \to M$ is called \emph{sharp} if each invertible element of $M$ has a unique preimage under $f$.  A monoid $M$ is called \emph{sharp} if the unique homomorphism $0 \to M$ is sharp.
\end{definition}

\begin{remark} 
Our definition is different from the one given in \cite[4.1.1]{Ogus}; it is equivalent to the \emph{logarithmic} homomorphisms of op.\ cit.
\end{remark}

  We write $M^\ast$ for the subgroup of invertible elements of $M$ and $\overnorm M$ for the quotient $M/M^\ast$, which we call the \emph{sharpening} of $M$.  Even when they do not arise as sharpenings of other monoids, we often notate sharp monoids with a bar above them.

\begin{remark} \label{rem:sharp}
	A homomorphism $f : N \to M$ of sharp monoids is sharp if and only if $f^{-1} \{ 0 \} = \{ 0 \}$.  Note that $f^{\rm gp}$ need not necessarily be injective.

	In this situation, sharp homomorphisms are analogous to local homomorphisms of local rings, and some authors prefer to call sharp homomorphisms between sharp monoids \emph{local}.  We will favor sharp in order not to create a conflict with connections to topology to be explored in \cite{GW}.  Some indications about those connections are given in Section~\ref{sec:trojac-bdd-diag}.
\end{remark}

	Every monoid $M$ is contained in a smallest associated group $M^{\rm gp}$, and $M$ determines a partial semiorder on $M^{\rm gp}$ in which $M$ is the subset of elements that are $\geq 0$.  If $M$ is sharp then the semiorder is a partial order.  As $M$ can be recovered from the induced partial order on $M^{\rm gp}$, we are free to think of monoids as partially (semi)ordered groups, and we frequently shall.

\subsubsection{Valuative monoids}

\begin{definition}
A \emph{valuative monoid} is an integral\footnote{Despite our convention that all monoids are saturated, we allow valuative monoids not to be saturated a priori, since they are so posteriori.} monoid $M$ such that, for all $x \in M^{\rm gp}$, either $x \in M$ or $-x \in M$.
\end{definition}

If $M$ is an integral monoid, and $x,y \in M^{\rm gp}$, we say that $x \leq y$ if $y - x \in M$.  We say that $x$ and $y$ are \emph{comparable} if $x \leq y$ or $y \leq x$.

\begin{lemma}
All valuative monoids are saturated.
\end{lemma}
\begin{proof}
Suppose that $M$ is valuative, $x \in M^{\rm gp}$, and $nx \in M$.  If $x \not\in M$ then $-x \in M$.  But as $nx \in M$ this means $-x$ is a unit of $M$, which means that $x \in M$.
\end{proof}

\begin{corollary}
All sharp valuative monoids are torsion free.
\end{corollary}
\begin{proof}
If $nx=0$ for some $x \in M^{\rm gp}$, then $x \in M$ since $0 \in M$ and $M$ is saturated.  But $M$ is sharp so $x$ must be $0$.
\end{proof}

\begin{example}
The non-negative elements of $\mathbf Z$ and of $\mathbf R$ are valuative monoids.  More generally, elements that are $\geq 0$ in the lexiographic order on $\mathbf R^n$ form a valuative monoid, as are the $\geq 0$ elements in any subgroup.  More generally still, if $\Omega$ is a totally ordered set then formula sums of well-ordered subsets of $\Omega$, with real coefficients, form a totally ordered abelian group.  The elements $\geq 0$ in this group are a valuative monoid, and a theorem of Hahn asserts that all valuative monoids arise as the elements $\geq 0$ in a subgroup of such a group~\cite[\S2]{Hahn}.
\end{example}

\begin{remark}
Finitely generated monoids arise as the monoids of functions on rational polyhedral cones that take integral values on the integral lattice.  Monoids that are not finitely generated can nevertheless be approximated by an ascending union of finitely generated monoids.  The ascending union corresponds dually to a descending intersection of rational polyhedral cones.  This gives a way to visualize valuative monoids inside of a real vector space as infinitesimal thickenings of rays in the dual vector space (see Figure~\ref{fig:2}).  This perspective will be important when we study prorepresentability in Section~\ref{sec:trojac-subdiv}.  

The reader who is so inclined may verify that an extension of a finitely generated monoid to a valuative monoid corresponds to a ray in its dual cone, together with a flag of infinitesimal extensions of that ray. 
\end{remark}

\begin{lemma} \label{lem:shpvalinj}
Suppose that $f : M \rightarrow N$ is a sharp homomorphism of monoids and $M$ is valuative.  Then $f$ is injective.
\end{lemma}
\begin{proof}
Suppose that $x \in M^{\rm gp}$ and $f(x) = 0$.  Either $x \in M$ or $-x \in M$.  We assume the former without loss of generality.  But $0 \in N$ has a unique preimage in $M$ by sharpness, so $x = 0$ and $f^{\rm gp}$ is injective.
\end{proof}

\begin{remark}
This property is similar to one enjoyed by fields in commutative algebra.  Valuative monoids will play a role in tropical geometry analogous to that played by fields in algebraic geometry.
\end{remark}

\begin{lemma} \label{lem:val-gp}
Suppose that $f : N \rightarrow M$ is a sharp homomorphism of valuative monoids.  Then $f$ is an isomorphism if and only if it induces an isomorphism on associated groups.
\end{lemma}
\begin{proof}
By Lemma~\ref{lem:shpvalinj}, we know $f$ is injective, so we replace $N$ by its image and assume $f$ is the inclusion of a submonoid with the same associated group.  If $\alpha \in M$ then either $\alpha$ or $-\alpha$ is in $N$.  In the former case we are done, and in the latter, $\alpha$ is an invertible element of $M$, so $\alpha \in N$ since the inclusion is sharp.
\end{proof}

\begin{definition}
A homomorphism of monoids $\varrho : N \rightarrow M$ is called \emph{relatively valuative} or an \emph{infinitesimal extension} if, whenever $\alpha \in N^{\rm gp}$ and $\varrho(\alpha) \in M$ either $\alpha \in N$ or $-\alpha \in N$.
\end{definition}

\begin{lemma}
If $\varrho : N \rightarrow M$ is relatively valuative and $M$ is valuative then $N$ is valuative.
\end{lemma}
\begin{proof}
Suppose that $\alpha \in N^{\rm gp}$.  Either $\varrho(\alpha) \in M$ or $-\varrho(\alpha) \in M$.  In either case, either $\alpha$ or $-\alpha$ is in $N$, by definition.
\end{proof}

\begin{lemma} \label{lem:extend}
Any partial order on an abelian group can be extended to a total order.
\end{lemma}
\begin{proof}
By Zorn's lemma, every partial order on an abelian group has a maximal extension.  Assume therefore that $\overnorm M^{\rm gp}$ is a maximal partially ordered abelian group and let $\overnorm M \subset \overnorm M^{\rm gp}$ be the submonoid of elements $\geq 0$.  Let $x$ be an element of $\overnorm M^{\rm gp}$ that is not in $\overnorm M$.  Then $\overnorm M[x]^{\rm sat}$ is the monoid of elements $\geq 0$ in a semiorder on $\overnorm M^{\rm gp}$ strictly extending the one corresponding to $\overnorm M$.  This semiorder cannot be a partial order because $\overnorm M$ was maximal, so $\overnorm M[x]^{\rm sat}$ cannot be sharp.  Therefore there are some $y,z \in \overnorm M$ and some positive integer $n$ and $m$ such that $(y + nx) + (z + mx) = 0$.  That is $y + z = -(n+m)x$.  As $\overnorm M$ is saturated (by its maximality), this implies that $-x \in \overnorm M$, which shows that every $x \in \overnorm M^{\rm gp}$ is either $\geq 0$ or $\leq 0$.
\end{proof}

\begin{figure}
\begin{tikzpicture}
\begin{scope}
\clip (-4,2.5) -- (4,-2.5) -- (5,-2.5) -- (5,0) -- (5,5) -- (-5,5) -- cycle;
\shade[opacity=.5,inner color=blue, outer color=white] (0,0) circle (5);
\end{scope} 
\begin{scope}
\clip (0,0) rectangle (5,5);
\shade[opacity=.75,inner color=black, outer color=white] (0,0) circle (5);
\end{scope}
\draw[thick,->] (0,0) -- (0,5);
\draw[thick,->] (0,0) -- (5,0);
\draw[thick,->, color=blue] (0,0) -- (-4,2.5);
\draw[thick,->, color=blue, dashed] (0,0) -- (4,-2.5);
\end{tikzpicture}
\caption{The monoid $\mathbf R_{\geq 0}^2$ and an extension to a valuative monoid in blue.}
\label{fig:1}
\end{figure}

\begin{figure}
\begin{tikzpicture}
\draw[thick,->] (0,0) -- (0,5);
\draw[thick,->] (0,0) -- (5,0);
\begin{scope}
\clip (0,0) -- (2.5,4) -- (2.3,4) -- cycle;
\shade[shading=radial,inner color=blue, outer color=white] (0,0) circle (5); 
\end{scope}
\draw[thick,->, color=blue] (0,0) -- (2.8125,4.5);
\end{tikzpicture}
\caption{The dual of Figure~\ref{fig:1}, with the respect to the standard Euclidean pairing.  Notice that the ray is thickened slightly on one side.}
\label{fig:2}
\end{figure}

\begin{example}
Suppose that $R$ is a valuation ring.  The valuation group of $R$ is a totally ordered abelian group, $\overnorm V^{\rm gp}$, and the valuations of nonzero elements of $R$ are the submonoid of positive elements, $\overnorm V \subset \overnorm V^{\rm gp}$.  The ideals, and in particular the prime ideals, of $R$ are totally ordered by inclusion.  Therefore the spectrum of $R$ is totally ordered by the specialization relation.  

Now suppose $\overnorm M$ is a finitely generated monoid.  Let $k$ be a field and let $X = \Spec k[\overnorm M]$ be the associated affine toric variety.  An extension of $\overnorm M$ to a valuative monoid $\overnorm V$ corresponds to an extension of $k[\overnorm M]$ to a valuation ring, $R$.  The dual map $\Spec R \to X$ gives a chain of specializations between generic points of torus-invariant strata in $X$.

Conversely, one may imagine a complete flag of torus invariant subspaces $X = X_0 \supset X_1 \supset \cdots$.  This correspond to a sequence of localization homomorphisms $\overnorm M = \overnorm M_0 \to \overnorm M_1 \to \cdots$ such that the kernel of $\overnorm M_\ell^{\rm gp} \to \overnorm M_{\ell+1}^{\rm gp}$ is isomorphic to $\mathbf Z$.  In fact, the isomorphism to $\mathbf Z$ is canonical, because $\overnorm M_\ell \to \overnorm M_{\ell+1}$ is a localization homomorphism, so the kernel contains an element of $\overnorm M_\ell$; we choose the isomorphism to $\mathbf Z$ so this element correpsonds to a positive element of $\mathbf Z$.

Let $p_\ell : \overnorm M \to \overnorm M_\ell$ be the projection.  One obtains a valuative monoid extending $\overnorm M$ by including $\alpha \in \overnorm M^{\rm gp}$ in $\overnorm V$ if $\alpha = 0$ or $p_\ell(\alpha)$ is a nonzero element of $\overnorm M_\ell$ for some $\ell$.  Indeed, suppose that $\alpha \in \overnorm M^{\rm gp}$ is nonzero and select the largest $\ell$ such that $p_{\ell+1}(\alpha) = 0$.  Then $p_\ell(\alpha)$ is a nonzero element of $\ker(\overnorm M_\ell^{\rm gp} \to \overnorm M_{\ell+1}^{\rm gp}) = \mathbf Z$.  If $p_\ell(\alpha) > 0$ in $\mathbf Z$ then $\alpha \in \overnorm V$ and if $p_\ell(\alpha) < 0$ in $\mathbf Z$ then $-\alpha \in \overnorm V$.

Not all valuative extensions of $\overnorm M^{\rm gp}$ arise this way, although one does get all of the ones where $|R| = \dim X + 1$.  For a complete list, one must add limits of affine charts of toric modifications.  These correspond to rays of irrational slope in the toric fan, and infinitesimal extensions thereof, in the manner illustrated in Figures~\ref{fig:1} and~\ref{fig:2} (note that the boundary of the blue region would not contain any lattice points in this case).
\end{example}

\subsubsection{Bounded elements of monoids}
\label{sec:bounded}

\begin{definition} \label{def:bounded}
Suppose that $\alpha$ and $\delta$ are elements of a partially ordered abelian group, with $\delta \geq 0$.  We will say that $\alpha$ is \emph{bounded} by $\delta$ if there are integers $m$ and $n$ such that $m \delta \leq \alpha \leq n \delta$.  We write $\alpha \prec \delta$ to indicate that $\alpha$ is bounded by $\delta$.

We say that $\alpha$ is \emph{dominated} by $\delta$, and write $\alpha \ll \delta$, if $n \alpha \leq \delta$ for all integers $n$.
\end{definition}

\begin{lemma} \label{lem:root-bdd}
Let $M$ be a (saturated) monoid, let $\delta \in M$, and let $\alpha \in M^{\rm gp}$.  Then $\alpha \prec \delta$ in $M$ if and only if $\alpha \prec \delta$ in $\mathbf QM$.
\end{lemma}
\begin{proof}
If $m \delta \leq \alpha \leq n \delta$ in $\mathbf QM$ then there is a positive integer $k$ such that $k (\alpha - m \delta)$ and $k(n \delta - \alpha)$ are both in $M$.  But $M$ is saturated, so this implies $m \delta \leq \alpha \leq n \delta$, as required.
\end{proof}

\begin{lemma} \label{lem:bounded}
Let $M$ be a monoid. Suppose $\delta \in M$.  The elements of $M^{\rm gp}$ that are bounded by $\delta$ is precisely $M[-\delta]^\ast$.
\end{lemma}
\begin{proof}
If $k \delta \leq \alpha \leq \ell \delta$ then $0 \leq \alpha \leq 0$ in the sharpening $\overline{ M[-\delta] }$ of $M[-\delta]$ and therefore $\alpha \in M[-\delta]^\ast$.  Conversely, if $\alpha \in M^{\rm gp}$ is a unit of $M[-\delta]$ then there is some $\beta \in M$ such that $\alpha + \beta \in \mathbf Z \delta$ --- in other words, $\alpha \leq \ell \delta$ for some integer $\ell$.  Applying the same reasoning to $-\alpha$ supplies an integer $k$ such that $-\alpha \leq k \delta \in M$.  Therefore $-k \delta \leq \alpha \leq \ell \delta$, as required.
\end{proof}

\begin{definition} \label{def:archimedean}
An \emph{archimedean group} is a totally ordered abelian group $M^{\rm gp}$ such that if $x, y \in M^{\rm gp}$ with $x > 0$ then $y$ is bounded by $x$.
\end{definition}

\begin{remark} 
A totally ordered abelian group $M^{\rm gp}$ is archimedean if and only $M$ it has no $\prec$-closed submonoids other than $0$ and $M^{\rm gp}$.
\end{remark}

The following theorem is due to H\"older~\cite{Holder}, but is also a special case of the theorem of Hahn~\cite{Hahn}.

\begin{theorem}[H\"older] \label{thm:archimedean}
Every archimedean group can be embedded by an order preserving homomorphism into the real numbers.  The homomorphism is unique up to scaling.
\end{theorem}
\begin{proof}
This is trivial for the zero group, so assume $M$ is a nonzero archimedean group.  Choose a nonzero element $x$ of $M$.  It will be equivalent to show that there is a unique order-preserving homomorphism $M \to \mathbf R$ sending $x$ to $1$.

For any $y \in M$, let $S$ be the set of rational numbers $p/q$ such that $px \leq qy$ in $M$.  Let $T$ be the set of rationals $p/q$ such that $px \geq qy$.  Then $S$ and $T$ are a Dedekind cut of $\mathbf Q$, hence define a unique real number $f(y)$.  This proves the uniqueness part.

All that remains is to show that $f$ is a homomorphism.  This amounts to the assertion that if $px \leq qy$ and $p'x \leq q'y'$ then $(pq' + p'q)x \leq qq'(y + y')$, which is an immediate verification.
\end{proof}

\begin{lemma} \label{lem:prec-total}
If $x$ and $y$ are positive elements of a totally ordered abelian group then $x \prec y$ or $y \ll x$.
\end{lemma}
\begin{proof}
Suppose that $y$ does not bound $x$.  As $x \geq 0$, this means there is no integer such that $x \leq ny$.  But the group is totally ordered, so we must therefore have $x \geq ny$ for all $n$.  That is $x \gg y$.
\end{proof}

\begin{proposition} \label{prop:arch-filt}
Let $M$ be a valuative monoid.  The collection of subsets $N$ of $M$ closed under $\prec$ are submonoids and are totally ordered by inclusion.  The graded pieces of this filtration are archimedean.
\end{proposition}
\begin{proof}
Lemma~\ref{lem:bounded} implies that these subsets are submonoids.  Suppose that $N$ and $P$ are $\prec$-closed subgroups and there is some $x \in N$ that is not contained in $P$.  If $y \in P$ then either $y \prec x$ or $x \prec y$ by Lemma~\ref{lem:prec-total}, but $P$ is $\prec$-closed so it must be the former.  Thus $P \prec x$ so $P \subset N$ since $N$ is $\prec$-closed.

Now suppose that $N \subset P$ and there are no intermediate $\prec$-closed submonoids.  The image of $P$ in $P^{\rm gp} / N^{\rm gp}$ therefore has no $\prec$-closed submonoids other than $0$ and itself, so it is archimedean.
\end{proof}

\subsection{Logarithmic structures}
\label{sec:log-str}

We review some of the basics of logarithmic geometry.  The canonical reference is Kato's original paper~\cite{kato1989logarithmic}.

\subsubsection{Systems of invertible sheaves}
\label{sec:inv-sh}

We recall a perspective on logarithmic structures favored by Borne and Vistoli~\cite[Definition~3.1]{MR2964607}.

\begin{definition} \label{def:log-str}
Let $X$ be a scheme.  A \emph{logarithmic structure} on $X$ is an (integral, saturated) \'etale sheaf of monoids $M_X$ on $X$ and a sharp homomorphism $\varepsilon : M_X \to \mathcal O_X$, the target given its multiplicative monoid structure.  The quotient $M_X / \varepsilon^{-1} \mathcal O_X^\ast$ is known as the \emph{characteristic monoid} of $M_X$ and is denoted $\overnorm M_X$.

A morphism of logarithmic structures $M_X \to N_X$ is a homomorphism of monoids commuting with the homomorphisms $\varepsilon$.
\end{definition}


Let $X$ be a logarithmic scheme.  For each local section $\alpha$ of $\overnorm M_X^{\rm gp}$, we denote the fiber of $M_X$ over $\overnorm M_X$ by $\mathcal O_X^\ast(-\alpha)$.  This is a $\mathcal O_X^\ast$-torsor because $M_X^{\rm gp}$ is an $\mathcal O_X^\ast$-torsor over $\overnorm M_X^{\rm gp}$.  We write $\mathcal O_X(-\alpha)$ for the associated invertible sheaf, obtained by contracting $\mathcal O_X^\ast(-\alpha)$ with $\mathcal O_X$ using the action of $\mathcal O_X^\ast$.

We can think of the assignment $\alpha \mapsto \mathcal O_X(-\alpha)$ as a map $\overnorm M_X^{\rm gp} \to \BGm$.  We have canonical isomorphisms $\mathcal O_X(\alpha + \beta) \simeq \mathcal O_X(\alpha) \otimes \mathcal O_X(\beta)$ making the morphism $\overnorm M_X^{\rm gp} \to \BGm$ into a homomorphism of group stacks.

Moreover, if $\alpha \in \overnorm M_X$ then the restriction of $\varepsilon$ gives a $\mathcal O_X^\ast$-equivariant map $\mathcal O_X^\ast(-\alpha) \to \mathcal O_X$, hence a morphism of invertible sheaves $\mathcal O_X(-\alpha) \to \mathcal O_X$.  If $\beta \geq \alpha$ then $\alpha - \beta \leq 0$ and we obtain $\mathcal O_X(\alpha - \beta) \to \mathcal O_X$; twisting by $\mathcal O_X(\beta)$, we get $\mathcal O_X(\alpha) \to \mathcal O_X(\beta)$.

If we regard $\overnorm M_X^{\rm gp}$ as a sheaf of categories over $X$, with a unique morphism $\alpha \to \beta$ whenever $\alpha \leq \beta$, then the logarithmic structure induces a monoidal functor $\overnorm M_X^{\rm gp} \to \tropfont L_X$ where $\tropfont L_X$ is the stack of invertible sheaves on $X$.  It is clearly possible to recover the original logarithmic structure on $X$ from this monoidal functor, so we often think of logarithmic strutures in these terms.

\subsubsection{Coherent logarithmic structures}
\label{sec:coh-log-str}

Let $X$ be a scheme with a logarithmic structure $M_X$.  If $\overnorm N$ is a(n integral, saturated) monoid and $e : \overnorm N \to \Gamma(X, \overnorm M_X)$ is a homomorphism, there is an initial logarithmic structure $M'_X$ and morphism $M'_X \to M_X$ such that $e$ factors through $\Gamma(X, \overnorm M'_X) \to \Gamma(X, \overnorm M_X)$.  If $M'_X \to M_X$ is an isomorphism then $\overnorm N$ and $e$ are called a chart of $\overnorm M_X$.

\begin{definition} \label{def:log-sch}
A logarithmic scheme is a scheme equipped with a logarithmic structure that has \'etale-local charts by integral, saturated monoids.  It is said to be \emph{locally of finite type} if the underlying scheme is locally of finite type and the charts can be chosen to come from finitely generated monoids.
\end{definition}

A logarithmic scheme that is locally of finite type comes equipped with a stratification, defined as follows.  Assume that $M_X$ has a global chart by a finitely generated monoid $\overnorm N$.  For each of the finitely many generators $\alpha$ of $\overnorm N$, the image of the homomorphism $\mathcal O_X(-\alpha) \to \mathcal O_X$ is an ideal, which determines a closed subset of $X$.  All combinations of intersections and complements of these closed subsets stratify $X$.

To patch this construction into a global one, we must argue that the stratification defined above does not depend on the choice of chart.  To see this, it is sufficient to work locally, and therefore to assume $X$ is the spectrum of a henselian local ring with closed point $x$.  Then the strata correspond to the ideals of the characteristic monoid $\overnorm M_{X,x}$, and are therefore independent of the choice of chart.

On each stratum, the characteristic monoid of $X$ is locally constant.

\begin{definition} \label{def:atomic}
A logarithmic scheme $X$ of finite type is called \emph{atomic} if $\Gamma(X, \overnorm M_X) \to \overnorm M_{X,x}$ is a bijection for all geometric points of the closed stratum and the closed stratum is connected.
\end{definition}

\begin{example}
An affine toric variety with its canonical logarithmic structure is an atomic neighborhood for its unique closed torus orbit.
\end{example}

\begin{lemma}
The closed stratum of an atomic logarithmic scheme $X$ is connected and $\overnorm M_X$ is constant on it.
\end{lemma}
\begin{proof}
Assume that $X$ is an atomic logarithmic scheme.  It is immediate that $\overnorm M_X$ is constant on the closed stratum, for we have a global isomorphism to a constant sheaf there, by definition.
\end{proof}

\begin{proposition} \label{prop:atomic}
Suppose that $X$ is a logarithmic scheme of finite type.  Then $X$ has an \'etale cover by atomic logarithmic schemes.
\end{proposition}
\begin{proof}
For each point geometric $x$ of $X$, choose an \'etale neighborhood $U$ of $x$ such that $\Gamma(U, \overnorm M_X) \to \overnorm M_{X,x}$ admits a section.  This is possible because $\overnorm M_{X,x}$ is finitely generated (because of the existence of charts by finitely generated monoids), hence finitely presented by R\'edei's theorem~\cite[Proposition~9.2]{grillet2017semigroups}.  As $\Gamma(U, \overnorm M_X)$ is finitely generated, $\Gamma(U, \overnorm M_X^{\rm gp})$ is a finitely generated abelian group, and therefore the kernel of~\eqref{eqn:42}
\begin{equation} \label{eqn:42}
\Gamma(U, \overnorm M_X^{\rm gp}) \to \overnorm M_{X,x}^{\rm gp}
\end{equation}
is a finitely generated abelian group.  By shrinking $U$, we can therefore ensure it is an isomorphism.  Finally, we delete any closed strata of $U$ other than the one containing $x$.
\end{proof}

\subsubsection{Finite type and finite presentation}
\label{sec:ftfp}

Because we admit logarithmic structures whose underlying monoids are not locally finitely generated, we must adapt the definitions of finite type and finite presentation.

\begin{definition} \label{def:loc-fin}
A morphism of logarithmic schemes $f : X \to Y$ is said to be \emph{locally of finite type} if, locally in $X$ and $Y$, it is possible to construct $X$ relative to $Y$ by adjoining finitely many elements to $\mathcal O_Y$ and $M_Y$, imposing some relations, and then passing to the associated saturated logarithmic structure.  It is said to be \emph{locally of finite presentation} if the relations can also be taken to be finite in number.

We say $X$ is of \emph{finite type} over $Y$ if, in addition to being locally of finite type over $Y$, it is quasicompact over $Y$.  For \emph{finite presentation}, we require local finite presentation, quasicompactness, and quasiseparatedness.
\end{definition}

\begin{lemma}
\begin{enumerate}
\item A logarithmic scheme of finite type over a noetherian scheme (with trivial logarithmic structure) is of finite presentation.
\item A logarithmic scheme of finite type over a logarithmic scheme of finite type is itself of finite type.
\end{enumerate}
\end{lemma}

\begin{remark}
Because we insist on saturated monoids, some unexpected phenomena can occur when working over bases that are not finitely generated.  For example, let $Y$ be a punctual logarithmic scheme whose characteristic monoid is the submonoid of $\mathbf R^2_{\geq 0}$ consisting of all $(a,b)$ such that $a + b \in \mathbf Z$.  Let $X$ be the logarithmic scheme obtained from $Y$ by adjoining $(1,-1)$ to the characteristic monoid.  This can be effected by adjoining an element $\gamma$ to $M_Y$ and imposing the relation $\beta \gamma = \alpha$, where $\alpha$ and $\beta$ are elements of $M_Y$ whose images in $\overnorm M_Y$ are $(1,0)$ and $(0,1)$, respectively.  This requires adjoining $\varepsilon(\gamma)$ to $\mathcal O_Y$.  In the category of not-necessarily-saturated logarithmic schemes, this would suffice to construct $X$ with underlying scheme $\mathbf A^1 \times Y$.

The monoid $\overnorm M_Y[(1,-1)]$ is not saturated, and the saturation $\overnorm M_X$ involves the adjunction of infinitely many additional elements.  Each of these elements requires an image in $\mathcal O_X$, and therefore neither the characteristic monoid nor the underlying scheme of $X$ --- when working with saturated logarithmic schemes --- is finitely generated over $Y$.  However, as a \emph{saturated logarithmic scheme}, $X$ is finitely generated over $Y$ and therefore deserves to be characterized as of finite type.

For further evidence that $X$ should be considered of finite type over $Y$, consider that $Y$ admits a morphism to $Y'$ that is an isomorphism on underlying schemes and such that $\overnorm M_{Y'} = \mathbf N^2 \subset \mathbf R^2$.  Then $X$ is the base change of $X'$, which is representable by a logarithmic structure $\mathbf A^1 \times Y$ (the construction from the first paragraph produces a saturated logarithmic structure when executed over $Y'$).  The morphism $X' \to Y'$ must certainly be considered of finite type.  If finite type is to be a property stable under base change to logarithmic schemes whose logarithmic structures are not necessarily locally finitely generated then we must admit that $X \to Y$ be of finite type as well. 

The following characterization of morphisms locally of finite presentation (Lemma~\ref{lem:lfp-morphism}) gives further justification for our choice of definition (cf., the characterization of morphisms locally of finite presentation in \cite[IV.8.14.2]{EGA}).  It says, effectively, that to specify a morphism from an aribtrary logarithmic scheme $S$ into a logarithmic scheme $X$ of finite presentation requires only finitely many of the data used to construct $S$.
\end{remark}

\numberwithin{equation}{theorem}
\begin{lemma}
\label{lem:lfp-morphism}
A morphism of logarithmic schemes $f : X \to Y$ is locally of finite presentation if and only if, for every cofiltered system of affine logarithmic schemes $S_i$ over $Y$, the map~\eqref{eqn:113} is a bijection:
\begin{equation} \label{eqn:113}
\varinjlim \Hom_Y(S_i, X) \to \Hom_S(\varprojlim S_i, X)
\end{equation}
\end{lemma}
\begin{proof}
First we prove that local finite presentation guarantees that~\eqref{eqn:113} is a bijection.  We demonstrate only the surjectivity, with the injectivity being similar.  Let $S = \varprojlim S_i$ and let $f : S \to X$ be a $Y$-morphism.  Choose covers of $X$ and $Y$ by $U_k$ and $V_k$ such that $U_k$ can be presented with finitely many data and finitely many relations relative to $V_k$. Since $S$ is affine, it is quasicompact, so finitely many of the $U_k$ suffice to cover the image of $S$.  Let $\{ W_k \}$ be a cover $S$ by open affines such that $W_k \subset f^{-1} U_k$ (repeat some of the $U_k$ if $f^{-1} U_k$ is not affine).

The open sets $W_k$ are pulled back from open sets $W_{ik} \subset S_i$ for $i$ sufficiently large, and $S_i = \bigcup W_{ik}$ for $i$ potentially larger.  Since $U_k$ can be presented with finitely many data and finitely many relations, the $V_k$-map $W_k \to U_k$ descends to $W_{ik} \to U_k$ for $i$ sufficiently large.  The maps $W_{ik}$ and $W_{i\ell}$ may not agree on their common domain of definition, but we can cover it with finitely many affines (since $S_i$ is quasiseparated) and therefore arrange for agreement when $i$ is sufficiently large.  This descends $f$ to $S_i$.

Now we consider the converse.  That is, we assume~\eqref{eqn:113} is a bijection for all cofiltered affine systems $\{ S_i \}$ and deduce that $X$ can be defined by finitely many data and finitely many relations relaive to $Y$.  This assertion is local in $X$ and $Y$, so we may assume that $X$ and $Y$ are affine and have global charts.  We argue first that $\mathcal O_X$ and $M_X$ are generated, up to saturation, relative to $\mathcal O_Y$ and $M_Y$ by finitely many elements.  Indeed, we can write the pair $(\mathcal O_X, M_X)$ as a union of finitely generated sub-logarithmic structures $(\mathcal O_{S_i}, M_{S_i})$.  These correspond to maps $S_i \to Y$ and their limit is $X \to Y$.  By~\eqref{eqn:113}, $S_i \to Y$ lifts to $X$ for all sufficiently large $i$ and therefore $(\mathcal O_X, M_X) = (\mathcal O_{S_i}, M_{S_i})$ for all sufficiently large $i$.  This proves that $X$ is locally of finite type over $Y$.

Now we check that $X$ can be defined, relative to $Y$, with only finitely many relations.  Let $(\mathcal O_{S_0}, M_{S_0})$ be freely generated over $(\mathcal O_Y, M_Y)$ by a choice of finitely many generators for $(\mathcal O_X, M_X)$.  Every finite subset of the relations among those generators determines a quotient $(\mathcal O_{S_i}, M_{S_i})$ and a map $S_i \to Y$.  For all sufficiently large~$i$, we get a lift to $X$ by~\eqref{eqn:113}, which means that $S_i = X$ for all sufficiently large $i$.  This completes the proof.
\end{proof}

\subsubsection{Universal surjectivity}
\label{sec:univ-surj}

\begin{definition}
Let $S$ be a logarithmic scheme.  By a \emph{valuative geometric point} of $S$ we will mean a point of $S$ valued in a logarithmic scheme whose underlying scheme is the spectrum of an algebraically closed field and whose characteristic monoid is valuative.
\end{definition}

\begin{proposition}[Gillam]
A morphism of logarithmic schemes $f : X \to Y$ is universally surjective if and only if every valuative geometric point of $Y$ can be lifted to a valuative geometric point of $X$, possibly after enlargement of the residue field and logarithmic structure.  If $f$ is of finite type, no enlargement of the residue field is necessary.
\end{proposition}
\begin{proof}
Suppose first that $f$ is universally surjective.  Let $T \to Y$ be a valuative geometric point.  Replacing $Y$ by $T$ and $X$ by $X \mathop\times_Y T$, we may assume $Y$ is the spectrum of an algebraically closed field $k$ with a valuative logarithmic structure.  Since $f : X \to Y$ is surjective, $X$ is nonempty.  Therefore $X$ has a $K$-point for some extension $K$ of $k$.  Replacing $k$ by $K$, we may assume $X$ has a $k$-point, and then replacing $X$ by that point, we may assume $X$ and $Y$ have the same underlying scheme, $\Spec k$.  Finally, we use Lemma~\ref{lem:extend} to embed $\overnorm M_X$ in a valuative monoid and conclude .

Now suppose that $f : X \to Y$ is surjective on valuative geometric points.  Then this is also true universally, so it is sufficient to prove that $f$ is surjective and therefore to assume $Y$ is the spectrum of an algebraically closed field.  But any monoid can be embedded in a valuative monoid by Lemma~\ref{lem:extend}, so after embedding $\overnorm M_Y$ in a valuative monoid $\overnorm N$ we can construct a morphism $Y' \to Y$, with $\overnorm M_{Y'} = \overnorm N$ valuative, that is an isomorphism on the underlying schemes.  Then $X' = X \mathop\times_Y Y'$ surjects onto $Y'$ by assumption.  As $Y' \to Y$ is surjective, this implies that $X \to Y$ is surjective, as required.
\end{proof}

\subsubsection{Valuative criteria}

\begin{lemma}
Let $S$ be the spectrum of a valuation ring with generic point $\eta$ and assume that $M_\eta$ is a logarithmic structure on $\eta$.  Then there is a maximal logarithmic structure $M$ on $S$ extending $M_\eta$ such that $M^{\rm gp} = M_\eta^{\rm gp}$.  The map $\varrho : M \rightarrow M_\eta$ is relatively valuative.
\end{lemma}
\begin{proof}
Let $\varepsilon : M_\eta \rightarrow \mathcal O_\eta$ be the structure morphism of $M_\eta$.  Define $M = \varepsilon^{-1} \mathcal O_S$. 

Note that $\varepsilon$ restricts to a bijection on $\varepsilon^{-1} \mathcal O_\eta^\ast$, so it also restricts to a bijection on $\varepsilon^{-1} \mathcal O_S^\ast$.  Therefore $\varepsilon : M \rightarrow \mathcal O_S$ is a logarithmic structure.  In fact, it is the direct image logarithmic structure defined more generally by Kato \cite[(1.4)]{kato1989logarithmic}.

The maximality of $M$ is the universal property of the direct image logarithmic structure, which we verify explicitly.  If $M'$ also extends $M_\eta$ then we have a commutative diagram
\begin{equation*} \xymatrix{
M' \ar[r] \ar[d] & M'_\eta \ar[d] \\
\mathcal O_S \ar[r] & \mathcal O_\eta
} \end{equation*}
from which we obtain $M' \rightarrow M$ by the universal property of the fiber product.

We argue $M \rightarrow M_\eta$ is relatively valuative.  Suppose $\alpha \in M^{\rm gp}$ and $\varrho(\alpha) \in M_\eta$.  As $\mathcal O_S$ is a valuation ring, either $\varepsilon(\varrho(\alpha)) \in \mathcal O_S$ or $\varepsilon(\varrho(\alpha)) \in \mathcal O_\eta^\ast$ and $\varepsilon(\varrho(\alpha))^{-1} \in \mathcal O_S$.  In the first case $\alpha \in \varepsilon^{-1} \mathcal O_S$ so $\alpha \in M$ and in the latter case, $\varrho(-\alpha) \in M$ and $-\alpha \in \varepsilon^{-1} \mathcal O_S$ so $-\alpha \in M$.
\end{proof}

\begin{theorem} \label{thm:val-prop}
The morphism of schemes underlying a morphism of logarithmic schemes $X \to Y$ satisfies the valuative criterion for properness if and only if it has the unique right lifting property with respect to inclusions $S \subset \overnorm S$ where $\overnorm S$ is the spectrum of a valuation ring, $S$ is its generic point, $S$ has a valuative logarithmic structure $M_S$, and the logarithmic structure of $\overnorm S$ is the maximal extension of $M_S$.
\end{theorem}
\begin{proof} 
Let $S = \Spec K$ and $\overnorm S = \Spec R$, and let $j : S \to \overnorm S$ be the inclusion.  Let $M_K$ be a logarithmic structure on $S$ and let $M_R$ be the maximal logarithmic structure extending $M_K$ to $R$.  Let $M'_K$ be a valuative logarithmic structure on $K$ extending $M_K$ and contained in $M_K^{\rm gp}$ (whose existence is guaranteed by Lemma~\ref{lem:extend}), and let $M'_R$ be its maximal extension to $R$, which is valuative.  We consider a lifting problem~\eqref{eqn:32} with $M_S$ pulled back from $X$:
\begin{equation} \label{eqn:32} \vcenter{ \xymatrix{
S' \ar[r] \ar[d] & S \ar[r] \ar[d] & X \ar[d] \\
\overnorm S' \ar[r] \ar@{-->}[urr]^(.33)f & \overnorm S \ar[r]^(.55)h \ar@{-->}[ur]_(.6)g & Y
}} \end{equation}
Note that the valuative criterion for properness for the underlying schemes of $X$ over $Y$ is equivalent to the existence of a unique arrow lifting the square on the right and that the assertion of the theorem is therefore that lifts of the square on the right are in bijection with lifts of the outer rectangle.  Let us assume $f$ has been specified and show that there is a unique choice of $g$.

We draw the maps of monoids and rings implied by~\eqref{eqn:32}:
\begin{equation*} \xymatrix{
K \ar@{}[dr]|{\scriptstyle\mycirc{A}} & \ar@{}[dr]|{\scriptstyle\mycirc{B}}\ar[l] j_\ast M'_K & j_\ast M_K \ar[l] & f^\ast M_X \ar[l] \ar[dll] \ar@{-->}[dl] \\
R \ar[u] & \ar[l] M'_R \ar[u] & M_R \ar[u] \ar[l] & h^\ast M_Y \ar[u] \ar[l]
} \end{equation*}
By definition of the maximal extension of a logarithmic structure, the rectangles $A$ and $A \cup B$ are cartesian.  Therefore $B$ is cartesian and we get a unique dashed arrow by the universal property of fiber product.
\end{proof}

\subsubsection{Logarithmic modifications and root stacks}
\label{sec:log-mod}

\begin{definition} \label{def:log-mod}
Let $X$ be a logarithmic scheme.  A \emph{logarithmic modification} is a morphism $Y \to X$ that is, locally in $X$, the base change of a toric modification (proper, birational, toric morphism) of toric varieties.

More generally, we say that a morphism of presheaves $G \to F$ or fibered categories on the category of logarithmic schemes is a logarithmic modification if, for every logarithmic scheme $X$ and morphism $X \to F$, the base change $X \mathop\times_F G \to X$ is a logarithmic modification.
\end{definition}

Let $X$ be a logarithmic scheme and let $\gamma$ and $\delta$ be two sections of $\overnorm M_X^{\rm gp}$.  We say that $\gamma$ and $\delta$ are \emph{locally comparable} on $X$ if, for each geometric point $x$ of $X$, we have $\gamma \leq \delta$ or $\delta \leq \gamma$ at $x$ (in other words, $\delta - \gamma \in \overnorm M_{X,x}$ or $\gamma - \delta \in \overnorm M_{X,x}$).

Given $X$, $\gamma$, and $\delta$, as above, but not necessarily locally comparable, the property of local comparability defines a subfunctor of the one represented by $X$.  That is, we can make the following definition:
\begin{equation*}
Y(W) = \{ f : W \to X \: \big| \: \text{$f^\ast \gamma$ and $f^\ast \delta$ are locally comparable} \}
\end{equation*}
Then $Y$ is representable by a logarithmic modification of $X$.  Indeed, locally in $X$, we can find a morphism $X \to \mathbf A^2$, with the target given its standard logarithmic structure, such that $\gamma$ and $\delta$ are pulled back from the canonical generators of the characteristic monoid of $\mathbf A^2$.  Then $Y$ is the pullback of the blowup of $\mathbf A^2$ at the origin.

\begin{definition} \label{def:root-stack}
Let $X$ be a logarithmic scheme and let $\overnorm M_X^{\rm gp} \subset \overnorm N^{\rm gp} \subset \mathbf Q \overnorm M_X^{\rm gp}$ be a locally finitely generated extension of $\overnorm M_X^{\rm gp}$ (a \emph{Kummer extension}).  Define $Y$ to be the following subfunctor of the one represented by $X$:
\begin{equation*}
Y(U) = \{ f : W \to X \: \big| \: \text{$f^\ast \overnorm M_X^{\rm gp} \to \overnorm M_W^{\rm gp}$ factors through $f^\ast \overnorm N^{\rm gp}$} \}
\end{equation*}
An algebraic stack with a logarithmic structure that represents $Y$ is called the \emph{root stack} of $X$ along $\overnorm N^{\rm gp}$.
\end{definition}

We can give a concrete description of the root stack representing $Y$ as follows. Working locally, we may assume that $X$ has a global chart $X \rightarrow \Spec \ZZ[P]$ for a sharp, integral, saturated monoid $P$, with $\overnorm{M}_X(X)^{\rm gp} = P^{\rm gp}$. The Kummer extension $\overnorm N^{\rm gp}$ is then determined by a finitely generated extension of lattices $P^{\rm gp} \to Q^{\rm gp}$ with $Q^{\rm gp}/P^{\rm gp}$ finite. The homomorphism $P^{\rm gp} \to Q^{\rm gp}$ gives rise to a homomorphism of tori $\Spec \ZZ[Q^{\rm gp}] \to \Spec \ZZ[P^{\rm gp}]$ with finite kernel $K$. If we think of $P$ as the intersection of a cone $C$ in $P^{\rm gp} \otimes \mathbf R$ with the lattice $P^{\rm gp}$, then $P^{\rm gp} \to Q^{\rm gp}$ determines a monoid $Q: = C \cap Q^{\rm gp}$, and the root stack representing $Y$ is explicitly given by the quotient stack
\begin{equation*}
[X \times_{\Spec \ZZ[P]} \Spec \ZZ[Q]/K]
\end{equation*} 
with its logrithmic structure descended from $X \times_{\Spec \ZZ[P]} \Spec \ZZ[Q]$. 

\begin{remark}
We note that in the category of logarithmic schemes, the structure maps of both logarithmic modifications and root stacks of $X$ are \emph{monomorphisms}, despite the fact that the map from the underlying scheme or stack of the modification or root stack to the underlying scheme of $X$ is far from a monomorphism.  
\end{remark}

\subsubsection{The logarithmic multiplicative group}
\label{sec:logGm}

\begin{definition}
Define functors $\LogSch^{\rm op} \rightarrow \mathbf{Sets}$ by the following formulas:
\begin{align*}
\logGm(S) & = \Gamma(S, M_S^{\rm gp}) \\
\ologGm(S) & = \Gamma(S, \overnorm M_S^{\rm gp})
\end{align*}
We call the first of these the \emph{logarithmic multiplicative group}.
\end{definition}

\begin{proposition} \label{prop:nonrep}
Neither $\logGm$ nor $\ologGm$ is representable by an algebraic stack with a logarithmic structure.
\end{proposition}
\begin{proof}
We will treat $\logGm$.  The argument is essentially the same with $\ologGm$.

Suppose that there is an algebraic stack $X$ with a logarithmic structure representing $\logGm$.  Let $S_0$ be the spectrum of a field $k$, equipped with a logarithmic structure $k^\ast \times (\mathbf N e_1 + \mathbf N e_2)$.  The element $e_2$ gives a map $f : S_0 \rightarrow X$, hence $f^\ast \overnorm M_X \rightarrow \overnorm M_{S_0}$. 

Now, for each $t \in \mathbf Z$, let $S_t$ have the same underlying scheme as $S_0$, with the logarithmic structure $k^\ast \times (\mathbf N e_1 + \mathbf N (e_2 + t e_1))$.  Then $M_{S_t}^{\rm gp} = M_{S_0}^{\rm gp}$ for all $t$, so the map $S_0 \rightarrow X$ factors through $S_0 \subset S_t$ for all $t \geq 0$.  Therefore the map $f^\ast \overnorm M_X \rightarrow \overnorm M_{S_0}$ factors through $\overnorm M_{S_t}$ for all $t \geq 0$.  Thus $\overnorm M_X \rightarrow \overnorm M_{S_0}$ factors throught $\bigcap_{t} \overnorm M_{S_t} = \mathbf N e_1$.  But the element $e_2 \in \Gamma(S_0, M_{S_0}^{\rm gp})$ is clearly not induced from an element of $\mathbf Z e_1$.
\end{proof}

\begin{lemma} \label{lem:logGm-P1}
Let $\mathbf P$ be the subfunctor of $\logGm$ whose $S$ points consist of those $\alpha \in \logGm(S)$ that are locally (in $S$) comparable to $0$.
\begin{enumerate}
\item $\mathbf P$ is isomorphic to $\mathbf P^1$ with its toric logarithmic structure.
\item $\mathbf P$ is a logarithmic modification of $\logGm$.
\end{enumerate}
\end{lemma}
\begin{proof}

Note that the logarithmic structure $M_{\mathbf A^2}$ has two tautological sections, $\alpha$ and $\beta$, coming from the two projections to $\mathbf A^1$.  The difference of these sections determines a map $\mathbf A^2 \to \Gm^{\log}$.  The open subset $\mathbf A^2 - \{ 0 \}$ may be presented as the union of the loci where $\alpha \geq 0$ and where $\beta \geq 0$, which coincide, respectively, with the loci where $\alpha - \beta \leq 0$ or $\alpha - \beta \geq 0$.  We note that adjusting $\alpha$ and $\beta$ simultaneously by the same unit leaves $\alpha - \beta$ unchanged, so that we have constructed a map $\mathbf P^1 \to \mathbf P$.

To see that this is an isomorphism, consider the open subfunctors of $\mathbf P$ where $\alpha \geq 0$ and where $\alpha \leq 0$.  These are each isomorphic to $\mathbf A^1$ and their preimages under $\mathbf P^1 \to \mathbf P$ are the two standard charts of $\mathbf P^1$.

Finally, we verify that $\mathbf P \to \logGm$ is a logarithmic modification.  We need to show that if $Z$ is a logarithmic scheme and $Z \to \Gm^{\log}$ is any morphism then $Z \mathop\times_{\logGm} \mathbf P \to Z$ is a logarithmic modification.  This is a local assertion in $Z$, and section in $M_Z^{\rm gp}$ is locally pulled back from a logarithmic map to an affine toric variety, so we can assume $Z$ is an affine toric variety with cone $\sigma$.  

Let $\overnorm\alpha$ be the image of $\alpha$ in $\overnorm M_Z^{\rm gp}$.  We can regard sections of $\overnorm M_Z^{\rm gp}$ as linear functions with integer slope on the ambient vector space of $\sigma$.  Then $Z \mathop\times_{\Gm^{\log}} \mathbf P$ is representable by the subdivision of $\sigma$ along the hyperplane where $\overnorm\alpha$ vanishes.
\end{proof}

\begin{corollary} \label{cor:log-covs}
Both $\logGm$ and $\ologGm$ have logarithmically smooth covers by logarithmic schemes.
\end{corollary}
\begin{proof}
We have just seen that $\logGm$ has a logarithmically \'etale cover by $\mathbf P^1$, and therefore $\ologGm = \logGm / \Gm$ has a logarithmically \'etale cover by $[ \mathbf P^1 / \Gm ]$.
\end{proof}

\begin{proposition} \label{prop:ologGm-zero}
The inclusion of the origin in $\ologGm$ is representable by affine logarithmic schemes of finite type.
\end{proposition}
\begin{proof} 
Suppose that $S$ is a logarithmic scheme and $S \to \ologGm$ is a morphism corresponding to a section $\alpha \in \Gamma(S, \overnorm M_S^{\rm gp})$.  Let $N$ be the submonoid of $M_S^{\rm gp}$ generated by $M_S$ and $\mathcal O_S(\alpha)$.  The assertion is local (in the Zariski topology, say) of $S$, so we may choose a local trivialization $\tilde\alpha$ of $\alpha$.  Then the pullback of the origin in $\ologGm$ to $S$ is represented by $\Spec \mathcal O_S[t, t^{-1}] / I$ where $t$ and $t^{-1}$ are indeterminates representing the images of $\tilde\alpha$ and $\tilde\alpha^{-1}$ and $I$ is the ideal of relations necessary to permit a monoid homomorphism $N \to \mathcal O_S[t,t^{-1}]$ that restricts to $\varepsilon : M_S \to \mathcal O_S$ on $M_S$ and sends $\tilde\alpha$ to  $t$.  This gives the universal (not necessarily saturated) logarithmic scheme over $S$ on which $\alpha$ restricts to $0$, and saturation completes the proof.
\end{proof}

\begin{example}
Consider the map $\AA^2 \rightarrow \ologGm$ given by twice the difference $2e_2-2e_1$ of the two generators $e_1,e_2 \in (\overline{M}_{\AA^2}(\AA^2))^{\rm gp}$. A map $T \rightarrow \AA^2 \times_{\ologGm} 0$ corresponds to two sections $x,y \in M_T(T)$ such that $y^2/x^2$ is a unit. Thus, the fiber product is representable by the closed subscheme $Z \subset \AA^2 \times \Gm$ defined by the ideal $x^2-ty^2$, with logarithmic structure induced from $\AA^2$ but with the relation that $x^2=ty^2$ required to hold in the logarithmic structure as well. Passing to the saturation gives the logarithmic scheme representing $\AA^2 \times_{\ologGm} 0$ in the category of saturated logarithmic schemes.   
\end{example}

\subsection{Tropical geometry}
\label{sec:trop}

\subsubsection{Tropical moduli problems}
\label{sec:trop-mod}

We summarize \cite{cavalieri2017moduli}.  For the purposes of this paper, a tropical moduli problem is a covariant functor on, or category covariantly fibered in groupoids over, the category of integral, saturated, sharp, commutative monoids.  Such a moduli problem extends automatically, in a canonical fashion, to one defined on all integral, saturated, sharp, commutative monoidal spaces, and even all such monoidal topoi.  In particular, it extends to logarithmic schemes, by regarding logarithmic schemes as monoidal topoi by way of the characteristic monoid.

There are two ways to produce this extension of the moduli problem.  The first, and perhaps simpler, of the two is to extend a moduli problem $F$ on commutative monoids to one defined on monoidal spaces (or topoi) by setting $F(S) = F(\Gamma(S, \overnorm M_S))$ and then sheafifying (or stackifying) the result.

An equivalent construction, when working over logarithmic schemes with coherent logarithmic structures, is to define $F(S)$ to be the set of systems of data $\xi_s \in F(\overnorm M_{S,s})$, one for each geometric point $s$ of $S$, such that $\xi_t \mapsto \xi_s$ under the morphism $F(\overnorm M_{S,t}) \to F(\overnorm M_{S,s})$ associated to a geometric specialization $s \leadsto t$.  This has the effect of building stackification into the definition, but either construction is adequate for our needs.

In practice, when formulating a tropical moduli problem, the difficult part seems to tend to lie in describing the functoriality with respect to monoid homomorphisms.  More specifically, any homomorphism of commutative monoids can be factored into a localization homomorphism followed by a sharp homomorphism.  Functoriality with respect to sharp homomorphisms is straightforward, but localizations tend to involve changes of topology that are more difficult to control.  For the tropical Picard group and tropical Jacobian, the notion that makes this work is called \emph{bounded monodromy}, and is first discussed in Section~\ref{sec:monodromy}.

The principal concern of \cite{cavalieri2017moduli} was the question of algebraicity of tropical moduli problems, meaning possession of a well-behaved cover by rational polyhedral cones.  None of the moduli problems we consider here is algebraic in this sense, although they often do have logarithmically smooth covers by logarithmic schemes.  This suggests the subject of algebraicity should be revisited with a more inclusive perspective.  To do so will require a less chaotic topology than the one introduced in \cite{cavalieri2017moduli}, such as the one that appears implicitly in Section~\ref{sec:troptop}, and a bit more explicitly in Section~\ref{sec:trojac-bdd-diag} of this paper.

\subsubsection{Tropical topology}
\label{sec:conetop}

We introduce a tropical topology that does not appear in~\cite{cavalieri2017moduli}.  This material will be needed in Section~\ref{sec:trojac-bdd-diag} and nowhere else, so we develop only the few facts we will need there.  A thorough treatment will be taken up elsewhere.

\begin{definition}
Let $\overnorm M$ be a sharp (integral, saturated) monoid.  A \emph{sharp valuation} of $\overnorm M$ is an isomorphism class of surjective homomorphisms from $\overnorm M^{\rm gp}$ to totally ordered abelian groups that preserve the \emph{strict} order of $\overnorm M^{\rm gp}$.  Here, an isomorphism between $v : \overnorm M^{\rm gp} \to \overnorm V^{\rm gp}$ and $w : \overnorm M^{\rm gp} \to \overnorm W^{\rm gp}$ is an isomorphism $f : \overnorm V^{\rm gp} \to \overnorm W^{\rm gp}$ such that $f v = w$.
\end{definition}

\begin{remark}
Equivalently, a sharp valuation of $\overnorm M$ is an isomorphism class of sharp homomorphisms $\overnorm M \to V$ where $\overnorm V$ is a (sharp) valuative monoid and $\overnorm M^{\rm gp} \to \overnorm V^{\rm gp}$ is surjective.  
\end{remark}

\begin{proposition}
\label{prop: ConeMquasicompact}
Let $\overnorm M$ be a sharp (integral, saturated) monoid and let $\ShpVal(\overnorm M)$ be its set of sharp valuations.  Give $\ShpVal(\overnorm M)$ the coarsest topology in which a subset defined by a finite set of strict inequalities among elements of $\overnorm M^{\rm gp}$ is open.  Then $\ShpVal(\overnorm M)$ is quasicompact.
\end{proposition}
\begin{proof}
Consider a descending sequence of closed subsets $\ShpVal(\overnorm M) = Z_0 \supset Z_1 \supset Z_2 \supset \cdots$, with $Z_i$ defined relative to $Z_{i-1}$ by an inequality $\alpha_i \geq 0$, with $\alpha_i \in \overnorm M^{\rm gp}$.  Then $Z_i$ is represented by the submonoid $\overnorm M[\alpha_1, \ldots, \alpha_i] \subset \overnorm M^{\rm gp}$ in the sense that a valuation of $\overnorm M$ (sharp or not) with valuation monoid $\overnorm V$ lies in $Z_i$ if and only if the homomorphism $\overnorm M \to \overnorm V$ factors through $\overnorm M[\alpha_1, \ldots, \alpha_i]$.  By Lemma~\ref{lem:extend}, the condition that $\bigcap Z_i = \varnothing$ means that $\overnorm M[\alpha_1, \alpha_2, \ldots]$ contains the inverse $-\beta$ of some element $\beta \in \overnorm M$.  But then $-\beta$ is a finite combination of the $\alpha_i$ and elements of $\overnorm M$ and lies therefore in $\overnorm M[\alpha_1, \ldots, \alpha_i]$ for some $i$.  We conclude that $Z_i = \varnothing$.
\end{proof}

\begin{remark}
The basic open subsets of $\ShpVal(\overnorm M)$ are the subsets representable as $\ShpVal(\overnorm N)$ where $\overnorm N \subset \overnorm M^{\rm gp}$ is a finitely generated extension.
\end{remark}

\begin{remark}
Suppose that $\delta \in \mathbf Q \overnorm M$.  Then there is some positive integer $n$ such that $n \delta \in \overnorm M$ and the inequality $n \delta > 0$ determines an open subset of $\Cone^\circ(\overnorm M)$.  Since valuative monoids are always saturated, this open subset does not depend on the choice of $n$.  We can therefore construct open subsets of $\ShpVal(\overnorm M)$ from inequalities in $\mathbf Q \overnorm M$.
\end{remark}

\subsubsection{Tropical curves}
\label{sec:trop-curves}

The main example in \cite{cavalieri2017moduli} is the moduli space of tropical curves.  We recall the main definition here, with small modifications, one of which is significant:  first, we have no use for marked points here (which appear as unbounded legs in the graphs of tropical curves), so we omit them below; second, we allow unrooted edges that are not attached at any vertex.  This second modification is essential for the definition of the topology in Section~\ref{sec:troptop}.

\begin{definition} \label{def:trop-curve}
Let $\overnorm M$ be a commutative monoid.  A \emph{tropical curve} metrized by $\overnorm M$ is a tuple $\tropfont X = (G, r, i, \ell)$ where 
\begin{enumerate}
\item $G$ is a set, 
\item $r : G \to G$ is a partially defined idempotent function, 
\item $i : G \to G$ is an involution, and
\item $\ell : G \to \overnorm M$ is a function
\end{enumerate}
such that 
\begin{enumerate}[resume*]
\item $\ell(i(x)) = \ell(x)$ for all $x$, and 
\item $r(x) = x$ if and only if $i(x) = x$ if and only if $\ell(x) = 0$.
\end{enumerate}
We often abuse notation and write $x \in \tropfont X$ to mean that $x \in G$.

If $x \in \tropfont X$ then $\ell(x)$ is called its \emph{length}.  The elements of $\tropfont X$ of length~$0$ are called \emph{vertices}.  The remaining elements are called \emph{flags} or \emph{oriented edges}.  An unordered pair of flags exchanged by $i$ is called an \emph{edge}.  We call $\tropfont X$ \emph{compact} if $r$ is defined on all of $G$.
\end{definition}

We imagine the set $G$ as the disjoint union of a set of vertices and a set of flags (a vertex with an incident edge).  The function $r$ sends a flag to the vertex at which it is attached and restricts to the identity on the vertices, which are characterized by this property.  We think of an element of $G$ on which $r$ is not defined as an oriented edge that is not rooted at any vertex.

\begin{figure}

\def\centerarc[#1](#2)(#3:#4:#5)
    { \draw[#1] ($(#2)+({#5*cos(#3)},{#5*sin(#3)})$) arc (#3:#4:#5); }

\begin{minipage}{.45\textwidth}
\centering
\scalebox{1.5}{%
\begin{tikzpicture}
\node (V) at (0,0) {};
\node (W) at (-1,1) {};
\node (X) at (-1,-1) {};

\draw[fill] (V) circle (.5ex);
\draw (W) circle (.5ex);
\draw[fill] (X) circle (.5ex);
\draw[thick,->] ([shift={(1,1)}]V) -- ([shift={(1,1)}]W);
\draw[thick,->] ([shift={(-1,-1)}]W) -- ([shift={(-1,-1)}]V); 
\draw[thick,->] ([shift={(-1,1)}]V) -- ([shift={(-1,1)}]X); 
\draw[thick,->] ([shift={(1,-1)}]X) -- ([shift={(1,-1)}]V);

\centerarc[thick,->](.6,0)(165:-165:.6)
\centerarc[thick,->](.6,0)(-165:165:.5)
\end{tikzpicture}}
\end{minipage} 
\begin{minipage}{.45\textwidth}
\centering
\scalebox{1.5}{%
\begin{tikzpicture}

\node (V) at (0,1) {};
\node (W) at (0,0) {};

\draw (V) circle (.5ex);
\draw (W) circle (.5ex);

\draw[thick,->] (V) to[out=255,in=105] (W);
\draw[thick,->] (W) to[out=75,in=285] (V);

\end{tikzpicture}}

\end{minipage}

\caption{Graphical representations of tropical curves.  Filled circles are vertices while open circles are endpoints of edges with absent vertices.}
\label{fig:trop-curve}
\end{figure}

\begin{remark}
It is customary to include a weighting by non-negative integers on the vertices in the definition of a tropical curve, standing for the genus of a component of a stable curve.  Such a weighting could be added to Definition~\ref{def:trop-curve} with no significant change to the rest of the paper.  As the weighting has no effect on the definition of the tropical Picard group, we have omitted it to keep the notation as light as possible.

The work of Amini and Caporaso on the Riemann--Roch theorem for tropical curves with vertex weights~\cite{AMINI20131} suggests that a vertex with positive weight~$g$ can be imagined as a vertex of weight~$0$ with~$g$ phantom loops attached, all of length~$0$.  They prove Riemann--Roch by endowing these loops with positive length~$\epsilon$ and then allowing~$\epsilon$ to shrink to zero.  The most naive application of the same approach would yield a different tropical Picard group than the one we consider, and would not have the same relationship to the logarithmic Picard group.
\end{remark}

If $f : \overnorm M \to \overnorm N$ is a homomorphism of commutative monoids, and $\tropfont X$ is a tropical curve metrized by $\overnorm M$, and $r$ is defined on every flag $x$ of $\tropfont X$ such that $f(\ell(x)) = 0$, then $f$ induces an \emph{edge contraction} of $\tropfont X$, as follows.  Let $\tropfont Y$ be the quotient of $\tropfont X$ in which a flag $x$ is identified with $r(x)$ if $f(\ell(x)) = 0$.  Note that if $f(\ell(x)) = 0$, this identification also identifies $r(x) \sim r(i(x))$ since $i^2(r(x)) = x$.  Then $\ell$ descends to a well-defined function on $\tropfont Y$, valued in $\overnorm N$ and makes $\tropfont Y$ into a tropical curve.

Following the procedure outlined in Section~\ref{sec:trop-mod}, we can now think of tropical curves as a tropical moduli problem:  for any sharp monoid $\overnorm P$, we define $\mathcal M^{\rm trop}(\overnorm P)$ to be the groupoid of tropical curves metrized by $\overnorm P$.  Note, however, that Definition~\ref{def:trop-curve} is slightly different from the one considered in \cite{cavalieri2017moduli}.

\begin{definition} \label{def:trop-curve-fam}
Let $S$ be a logarithmic scheme.  A \emph{tropical curve} over $S$ is the choice of a tropical curve $\tropfont X_s$ for each geometric point $s$ of $S$ and an edge contraction $\tropfont X_s \to \tropfont X_t$ for each geometric specialization $t \leadsto s$ such that the edges of $\tropfont X_s$ contracted in $\tropfont X_t$ are precisely the ones whose lengths lie in the kernel of $\overnorm M_s^{\rm gp} \to \overnorm M_t^{\rm gp}$.
\end{definition}

\begin{definition} \label{def:linear}
Let $\tropfont X$ be tropical curve metrized by a monoid $\overnorm M$ with vertex set $V$.  We define $\PL(\tropfont X)$ to be the set of functions $\lambda = (\alpha, \mu) : G_{\tropfont X} \to \overnorm M^{\rm gp} \times \mathbf Z$ satisfying the following conditions:
\begin{enumerate}
\item if $x$ is a vertex then $\mu(x) = 0$,
\item we have $\alpha(r(x)) = \alpha(x)$ for all $x$ on which $r$ is defined, and
\item we have $\alpha(i(x)) = \alpha(x) + \mu(x) \ell(x)$.
\end{enumerate}
Note that the third condition implies that $\mu(x) = - \mu(i(x))$ since 
\begin{equation*}
\alpha(x) = \alpha(i^2(x)) = \alpha(i(x))) + \mu(i(x)) \ell(x)
\end{equation*}
and $\ell(x)$ is nonzero.  We define $\L(\tropfont X)$ to be the subset of $\PL(\tropfont X)$ where the following additional condition is satisfied:
\begin{enumerate}[resume*]
\item (\emph{balancing}) for each vertex $x$ of $\tropfont X$, we have $\sum_{r(y) = x} \mu(y) = 0$.
\end{enumerate}
Elements of $\PL(\tropfont X)$ are called \emph{piecewise linear functions} on $\tropfont X$ and elements of $\L(\tropfont X)$ are called \emph{linear functions}.
\end{definition}

\begin{remark}
The terms \emph{balanced} and \emph{harmonic} are often employed synonymously with linear.
\end{remark}

If $\tropfont X$ is a tropical curve metrized by $\overnorm M$ and $\overnorm M \to \overnorm N$ is a monoid homomorphism inducing a tropical curve $\tropfont Y$ metrized by $\overnorm N$ then there are natural homomorphisms $\PL(\tropfont X) \to \PL(\tropfont Y)$ and $\L(\tropfont X) \to \L(\tropfont Y)$.  Thus tropical curves equipped with piecewise linear functions are a tropical moduli problem.  See Proposition~\ref{prop:gen-P} for further details.

\subsubsection{Subdivision of tropical curves}
\label{sec:subdivision}

\begin{definition} \label{def:subdivision}
Let $\tropfont Y$ be a tropical curve metrized by a commutative monoid $\overnorm M$.  Let $y$ be a $2$-valent vertex of $\tropfont Y$.  We construct a new tropical curve $\tropfont X$ by removing $y$ from $\tropfont Y$ along with the two flags $e$ and $f$ incident to $y$ and defining
\begin{gather*}
i_{\tropfont X}(i_{\tropfont Y}(e))  = i_{\tropfont Y}(f) , \qquad i_{\tropfont X}(i_{\tropfont Y}(f)) = i_{\tropfont Y}(e) \\
\ell_{\tropfont X}(i_{\tropfont Y}(e)) = \ell_{\tropfont X}(i_{\tropfont Y}(f)) = \ell_{\tropfont Y}(e) + \ell_{\tropfont Y}(f) .
\end{gather*}
	We call $\tropfont Y$ a \emph{basic subdivision} of $\tropfont X$ at the edge $\{ i_{\tropfont Y}(e), i_{\tropfont Y}(f) \}$.  If $\tropfont X$ is obtained from $\tropfont Y$ by a sequence of basic subdivisions, we call $\tropfont Y$ a \emph{subdivision} of $\tropfont X$.
\end{definition}

If $\tropfont Y$ is a subdivision of $\tropfont X$ then $G_{\tropfont Y}$ contains a copy of $G_{\tropfont X}$.  An isomorphism of subdivisions is an isomorphism of tropical curves that respects this copy of the underlying set.

\begin{lemma}
If $\tropfont X'$ is a subdivision of a tropical curve $\tropfont X$ metrized by $\overnorm M$, and $\overnorm M \to \overnorm N$ is a localization homomorphism, then the edge contraction $\tropfont Y'$ of $\tropfont X'$ is naturally a subdivision of the edge contraction $\tropfont Y$ of $\tropfont X$.
\end{lemma}
\begin{proof}
It is sufficient to assume that $\tropfont X'$ is a basic subdivision of $\tropfont X$ at an edge $e$ into edges $e'$ and $e''$.  The main point is that if $\ell(e)$ maps to $0$ in $\overnorm N$ then $\ell(e')$ and $\ell(e'')$ do as well, since $0 \leq \ell(e') \leq \ell(e)$ and $0 \leq \ell(e'') \leq \ell(e)$, which implies that $e'$ and $e''$ are both contracted if $e$ is.
\end{proof}

\subsection{Logarithmic curves}
\label{sec:log-curves}

\subsubsection{Logarithmic structure}
\label{sec:local-str}

\begin{definition}
Let $S$ be a logarithmic scheme.  A \emph{logarithmic curve} over $S$ is an integral, saturated, logarithmically smooth morphism $\pi : X \to S$ of relative dimension~$1$.
\end{definition}

\begin{theorem}[F.\ Kato] \label{thm:FKato}
Let $X$ be a logarithmic curve over $S$.  Then the underlying scheme of $X$ is a flat family of nodal curve over $S$ and, for a geometric point $x$ of $X$ lying above the geometric point $s$ of $S$, one of the following applies:
	\begin{enumerate}[ref=(\arabic{*})]
\item $x$ is a smooth point of its fiber in the underlying schematic curve of $X$, and $\overnorm M_{S,s} \to \overnorm M_{X,x}$ is an isomorphism.
\item $x$ is a marked point, and there is an isomorphism $\overnorm M_{S,s} + \mathbf N \alpha \to \overnorm M_{X,x}$ where $\mathcal O_X(-\alpha)$ is the ideal of the marking.
\item \label{item:node} $x$ is a node of its fiber and there is an isomorphism $\overnorm M_{S,s} + \mathbf N \alpha + \mathbf N \beta / (\alpha + \beta = \delta) \to \overnorm M_{X,x}$, with $\delta \in \overnorm M_{S,s}$.  The invertible sheaf $\mathcal O_S(-\delta)$ is the pullback of the ideal sheaf of the boundary divisor corresponding to the node $x$ from the moduli space of curves, and $\mathcal O_X(-\alpha)$ and $\mathcal O_X(-\beta)$ are the pullbacks ideal sheaves of the two branches of the universal curve at $x$.
\end{enumerate}
	If $X$ is \emph{vertical} over $S$ then the second possibility does not occur.
\end{theorem}

We write $\mathcal M_{g,n}^{\log}$ for the moduli space of logarithmic curves of genus~$g$ with $n$ (ordered) marked points.


The following theorem characterizes the logarithmic structures of logarithmic curves over valuative bases.  For a proof, see~\cite[Proposition~3.6.4]{aj}.

\begin{theorem} \label{thm:curve-val}
Let $S$ be the spectrum of a valuation ring with generic point $\eta$ and let $X$ be a family of nodal curves over $S$.  Assume that $X_\eta$ and $\eta$ have been given logarithmic structures $M_{X_\eta}$ and $M_\eta$ making $X_\eta$ into a logarithmic curve over $\eta$, with $M_\eta$ valuative.  Let $M_X$ and $M_S$ be the maximal extensions, respectively, of $M_{X_\eta}$ and $M_\eta$ to $X$ and to $S$.  Then $X$ is a logarithmic curve over $S$.
\end{theorem}

\subsubsection{Tropicalizing logarithmic curves}
\label{sec:tropicalizing-curves}

Theorem~\ref{thm:FKato} allows us to construct a family of tropical curves over $S$ from a family $X$ of logarithmic curves over $S$.  For each geometric point $s$ of $S$, let $\tropfont X_s$ be the dual graph of $X_s$, metrized by $\overnorm M_{S,s}$ with $\ell(e) = \delta$ when $e$ is the edge associated to the node $x$ in the notation of Theorem~\ref{thm:FKato}~\ref{item:node}.  

If $s \leadsto t$ is a geometric specialization, then $\tropfont X_s$ is obtained from $\tropfont X_t$ by contracting the edges of $\tropfont X_t$ that correspond to nodes of $X_t$ smoothed in $X_s$.  Therefore the association $X_s \mapsto \tropfont X_s$ commutes with the geometric generization maps and defines a morphism $\mathcal M_{g,n}^{\log} \to \mathcal M^{\trop}$ from the moduli space of logarithmic curves to the moduli space of tropical curves.  See \cite[Section~5]{cavalieri2017moduli} for further details.

The essence of the following lemma comes from Gross and Siebert~\cite[Section~1.4]{GS}.  It allows us to relate the characteristic monoid of a logarithmic curve to piecewise linear functions on the tropicalization.

\begin{lemma} \label{lem:GS}
Let $\overnorm M$ be a commutative monoid.  Then
\begin{equation*}
\overnorm M + \mathbf N \alpha + \mathbf N \beta / (\alpha + \beta = \delta) \xrightarrow{\sim} \{ (a,b) \in \overnorm M \times \overnorm M \: \big| \: a - b \in \mathbf Z \delta \}
\end{equation*}
where $\alpha \mapsto (0, \delta)$, $\beta \mapsto (\delta, 0)$, and $\gamma \mapsto (\gamma, \gamma)$ for all $\gamma \in \overnorm M$.
\end{lemma}
\begin{proof}
	The map is well-defined by the universal property of the pushout.  The following formula gives the inverse:
	\begin{equation*}
		(a,b) \mapsto \begin{cases} a + \frac{b-a}{\delta} \alpha & b \geq a \\ b + \frac{a-b}{\delta} \beta & a \geq b \end{cases}
	\end{equation*}
\end{proof}

\begin{corollary} \label{cor:GS}
Let $S$ be a the spectrum of an algebraically closed field and let $X$ be a logarithmic curve over $S$ with tropicalization $\tropfont X$.  Then $\Gamma(X, \overnorm M_X^{\rm gp})$ and $\Gamma(\tropfont X, \PL)$ are naturally identified.
\end{corollary}
\begin{proof}
Lemma~\ref{lem:GS} identifies the stalk of $\overnorm M_X^{\rm gp}$ at a node of $X$ with the linear functions of integer slope on the corresponding edge of $\tropfont X$.  Generizing to one branch or the other of the node corresponds to evaluating the function at one endpoint or the other of the edge.  Therefore a global section of $\overnorm M_X^{\rm gp}$ amounts to a function on $\tropfont X$ taking values in $\overnorm M_S^{\rm gp}$ that is linear along the edges with integer slopes.
\end{proof}

We give a more local version of this corollary, using the tropical topology from Section~\ref{sec:troptop}.

Let $S$ be a logarithmic scheme whose underlying scheme is the spectrum of an algebraically closed field, let $X$ be a logarithmic curve over $S$, and let $\tropfont X$ be its tropicalization.  Suppose that $p : \tropfont U \to \tropfont X$ is a tropical local isomorphism.  Each $v$ vertex of $\tropfont X$ corresponds to a component $X_v$ of the normalization of $X$ and each edge $v$ of $\tropfont X$ corresponds to a node $X_v$ of $X$.  Let $U = \varinjlim_{u \in U} X_{p(u)}$.  Effectively, $U$ is the union of components of the normalization of $X$ indexed by the vertices of $\tropfont U$, joined along nodes indexed by the edges of $\tropfont U$, together with some disjoint nodes corresponding to unattached edges of $\tropfont U$.  

There is a canonical projection $U \to X$ that is \'etale except at the points corresponding to $0$- and $1$-sided edges.  We give $U$ the logarithmic structure pulled back from $X$.

\begin{remark}
This construction extends to families with locally constant dual graph, but no further.  Should $\tropfont U$ be a covering space of $\tropfont X$ then $U$ will be \'etale over $X$ and therefore this construction extends infinitesimally, but not necessarily any further than that.  If $\tropfont U$ is in addition \emph{finite} over $\tropfont X$ then the construction can be extended to an aribtrary base.
\end{remark}

The construction described above gives a functor $t^{-1}$ from the category of local isomorphisms $\tropfont U \to \tropfont X$ to the category of finite strict $X$-schemes.  We refer to this as an \emph{anticontinuous morphism} from $X$ to $\tropfont X$, but we make no attempt to develop a general theory of anticontinuous maps here.


\begin{lemma} \label{lem:log-pwl}
We have $t_\ast \overnorm M^{\rm gp}_X = \tropfont P_{\tropfont X}$.  That is, for any open subset $\tropfont U$ of $\tropfont X$, we have $\Gamma(\tropfont U, \tropfont P) = \Gamma(t^{-1} \tropfont U, \overnorm M_X^{\rm gp})$.
\end{lemma}

\subsubsection{Subdivision of logarithmic curves}
\label{sec:subdiv-log}

Let $X$ be a logarithmic curve over $S$ and let $\tropfont X$ be its tropicalization.  Suppose that $\tropfont Y \to \tropfont X$ is a subdivision.  We construct an associated logarithmic modification $Y \to X$ such that the tropicalization of $Y$ is $\tropfont X$.

We may make this construction \'etale-locally on $S$, provided we do so in a manner compatible with further localization.  Every subdivision of tropical curves is locally an iterate of basic subdivisions, so we may assume that $\tropfont Y$ is a basic subdivision of $\tropfont X$.  We now describe $Y \to X$ locally in $X$.

Suppose that $e$ is the edge of $\tropfont X$ subdivided in $\tropfont Y$, and that $Z$ is the corresponding node of $X$.  Note that $Z$ is a closed subset of $X$, not necessarily a point unless $S$ is a point.  Over the complement of $Z$, we take the map $Y \to X$ to be an isomorphism.  It remains to describe $Y$ on an \'etale neighborhood of $Z$.

We may work \'etale-locally in $X$, again provided that our construction is compatible with further \'etale localization.  We can therefore work in an \'etale neighborhood $U$ of a geometric point $x \in Z$ and an \'etale neighborhood $T$ of its image in $S$, and we can assume that 
\begin{enumerate}
\item $\overnorm M_{X,x} = \overnorm M_{S,s} + \mathbf N \alpha + \mathbf N \beta / (\alpha + \beta = \delta)$ for some $\delta \in \overnorm M_{S,s}$,
\item $\alpha$ and $\beta$ come from global sections of $\overnorm M_X$ over $U$, and
\item $\delta$ comes from a global section of $\overnorm M_S$ over $T$.  
\end{enumerate}

Now, recall we may think of $\alpha$ and $\beta$ as barycentric coordinates on the edge $e$ of $\tropfont X$ that was subdivided in $\tropfont Y$.  Suppose that this edge was subdivided at the point where $\alpha = \gamma$ (and therefore $\beta = \delta - \gamma$) for some $\gamma \in \Gamma(T, \overnorm M_S)$.  We ask $V$ to represent the subfunctor of the functor represented by $U$ where $\alpha$ and $\gamma$ are locally comparable.  Then $V$ is a logarithmic modification of $U$.

To see that the construction is compatible with further localization, the main point is that the only ambiguity in the above construction is the choice of labelling of the generators of $\overnorm M_{X,x}$ as $\alpha$ and $\beta$.  This choice is in bijection with the choice of orientation of the edge $e$.  Reversing the labelling also reverses the orientation, and we impose the comparability of $\beta$ with $\delta - \gamma$.  But $\alpha = \delta - \beta$, so $\alpha$ is comparable to $\gamma$ if and only if $\beta$ is comparable to $\delta - \gamma$ and the resulting logarithmic modification is the same.  These local modifications therefore patch together to give a logarithmic modification $Y \to X$.  

\begin{remark}
It is possible to understand $Y \to X$ as the pullback of $\tropfont Y \to \tropfont X$ along the tropicalization map $t : X \to \tropfont X$.  This point of view will be developed in~\cite{GW}.
\end{remark}

\section{The tropical Picard group and the tropical Jacobian}
\label{sec:tropicjac}

\numberwithin{theorem}{subsection}
\subsection{The topology of a tropical curve}
\label{sec:troptop}

\begin{definition}
Let $\tropfont X$ be a tropical curve and let $x$ be a vertex of $\tropfont X$.  The \emph{star} of $x$ is the set of all $y \in \tropfont X$ such that $r(y) = x$.
\end{definition}

\begin{definition}
Let $\tropfont Y$ and $\tropfont X$ be tropical curves metrized by the same monoid $\overnorm M$.  A function $f : \tropfont Y \to \tropfont X$ is called a \emph{local isomorphism} if it commutes with all of the functions $r$, $\ell$, and $i$ and it restricts to a bijection on the star of each vertex.  

A local isomorphism is called an \emph{open embedding} if it is also injective.  The image of an open embedding of tropical curves is called an \emph{open subcurve}.
\end{definition}

In Figure~\ref{fig:trop-curve}, there are $6$ distinct local isomorphisms from the curve on the right to the curve on the left, assuming that all edges have the same length.

\begin{lemma}
An open subcurve of $\tropfont X$ is a subset of $\tropfont X$ that is stable under $i$ and $r^{-1}$.
\end{lemma}
\begin{proof}
This is immediate.
\end{proof}

\begin{remark}
A tropical curve with real edge lengths has an evident realization as a metric space.  The open subcurves of $\tropfont X$ are the subcurves whose realizations are open subsets of the realization of $\tropfont X$.  Since the tropical topology depends only on the underlying graph of $\tropfont X$, and not on its metric, this remark characterizes the tropical topology of all tropical curves.
\end{remark}

\begin{figure}

\def\centerarc[#1](#2)(#3:#4:#5)
    { \draw[#1] ($(#2)+({#5*cos(#3)},{#5*sin(#3)})$) arc (#3:#4:#5); }

\begin{minipage}{.45\textwidth}
\centering
\scalebox{1.5}{%
\begin{tikzpicture}
\node (V) at (0,0) {};
\node (W) at (1,1) {};
\node (X) at (1,-1) {};

\draw[fill] (V) circle (.5ex);
\draw (W) circle (.5ex);
\draw (X) circle (.5ex);
\draw[thick,->] ([shift={(-1,1)}]V) -- (W);
\draw[thick,->] ([shift={(1,-1)}]W) -- (V); 
\draw[thick,->] ([shift={(1,1)}]V) -- (X); 
\draw[thick,->] ([shift={(-1,-1)}]X) -- (V);

\end{tikzpicture}}
\end{minipage} 
\begin{minipage}{.45\textwidth}
\centering
\scalebox{1.5}{%
\begin{tikzpicture}

\node (V) at (0,0) {};

\draw[fill] (V) circle (.5ex);
\centerarc[thick,->](.6,0)(165:-165:.6)
\centerarc[thick,->](.6,0)(-165:165:.5)

\end{tikzpicture}}
\end{minipage}

\caption{The curve on the left is locally isomorphic to the curve on the right.}
\label{fig:3}
\end{figure}

\begin{example}
Let $\tropfont X$ be a tropical curve with one vertex, $x$, and one edge $\{ e, i(e) \}$, of length $\delta$, connecting that vertex to itself.  Let $\tropfont Y$ by a tropical curve with one vertex, $y$, and two edges $\{ f, i(f) \}$ and $\{ g, i(g) \}$, both of length $\delta$, with $r(e) = r(f) = y$ and with $r(i(e))$ and $r(i(f))$ both undefined.  See Figure~\ref{fig:3} for a picture.  There is a local isomorphism $\tropfont Y \to \tropfont X$ sending $y$ to $x$, sending $f$ to $e$, and sending $g$ to $i(e)$.  This local isomorphism does not restrict to open embeddings on any open cover of $\tropfont Y$.
\end{example}

\begin{lemma} \label{lem:min-cov}
	Any logarithmic curve $\tropfont X$ has a minimal cover by a local isomorphism $\tropfont Y \to \tropfont X$.  That is, for any cover $\tropfont Z \to \tropfont X$, there is a (not necessarily unique) factorization $\tropfont Y \to \tropfont Z$ of the projection from $\tropfont Y$ to $\tropfont X$.
\end{lemma}
\begin{proof}
Let $\tropfont X$ be a tropical curve and let $\tropfont Y_0$ be the disjoint union of the stars of the vertices.  Construct $\tropfont Y$ by adjoining a new flag $i(x)$ for each non-vertex flag $x$ of $\tropfont Y_0$.
\end{proof}

\begin{definition} \label{def:cover}
A collection of local isomorphisms $p_i : \tropfont U_i \to \tropfont X$ of a tropical curve $\tropfont X$ is called a \emph{cover} if $\tropfont X = \bigcup p_i(\tropfont U_i)$.  We call this the \emph{tropical topology} of $\tropfont X$.
\end{definition}

%

Let $\tropfont Y$ be a subdivision of $\tropfont X$.  We construct an associated morphism of sites $\rho : \tropfont Y \to \tropfont X$.  Let $\tau : \tropfont U \to \tropfont X$ be a local isomorphism.  For each edge $e$ of $\tropfont U$, the restriction of $\tau$ to $e$ is a bijection.  Form $\rho^{-1} \tropfont U$ by subdividing $e$ in precisely the same way $\tau(e)$ is subdivided in $\tropfont Y$.  Then we have an evident local isomorphism $\rho^{-1} \tropfont U \to \tropfont Y$.

\begin{proposition}
The construction outlined above determines a morphism of sites $\rho$ from that $\tropfont Y$ to that of $\tropfont X$.
\end{proposition}
\begin{proof}
One must verify that the construction respects covers and fiber products of local isomorphisms.  Both are immediate.
\end{proof}

Suppose that $\tropfont X$ is a tropical curve over a logarithmic scheme $S$.  This construction makes it possible to organize the sites of the fibers of $\tropfont X$ over $S$ into a fibered site \cite[7.2.1]{sga4-VI} over $\et(S)^{\rm op}$, the \emph{opposite} of the \'etale site of $S$.

\subsection{The sheaves of linear and piecewise linear functions}
\label{sec:linear}


If $\tropfont U \to \tropfont X$ is a local isomorphism then we have maps $\PL(\tropfont X) \to \PL(\tropfont U)$ and $\L(\tropfont X) \to \L(\tropfont U)$ by restriction.  This makes $\PL$ and $\L$ into presheaves on the category of tropical curves with local isomorphisms to $\tropfont X$.

\begin{proposition}
The presheaves $\L$ and $\PL$ are sheaves in the tropical topology.
\end{proposition}
\begin{proof}
Since piecewise linear functions are functions defined on the underlying set of a tropical curve, and tropical covers are set-theoretic covers, it is immediate that $\PL$ forms a sheaf.  The subpresheaf $\L$ is defined by the balancing condition at each vertex of the underlying graph, which depends only on the star of that vertex.  By definition, a tropical cover induces a bijection on the star of each vertex, and therefore the balancing condition is visible locally in a tropical cover.
\end{proof}

\begin{proposition} \label{prop:trop-subdiv}
Let $\rho : \tropfont Y \to \tropfont X$ be a subdivision of tropical curves.  Then $\L_{\tropfont X} \to \mathrm R \rho_\ast \L_{\tropfont Y}$ is an isomorphism.
\end{proposition}
\begin{proof}
By induction, we can also assume that $\tropfont Y$ is a basic subdivision of $\tropfont X$.  The assertion is local on $\tropfont X$ so we can assume that $\tropfont X$ is a bare edge with no vertices.  In that case, $\tropfont Y$ is a vertex with two edges.  Thus, $\tropfont Y$ has no nontrivial covers, so $\mathrm R^p \rho_\ast \L_{\tropfont Y} = 0$ for all $p > 0$ and the isomorphism $\L_{\tropfont X} \simeq \rho_\ast \L_{\tropfont Y}$ is a straightforward calculation.
\end{proof}

Suppose that $\tropfont X$ is a tropical curve over $S$.  On each stratum $Z$ of $S$, the tropical curve $\tropfont X$ is locally constant, so the cohomology $H^\ast(\tropfont X_Z, \L)$ can be represented by a complex of locally constant abelian groups.  If $s \leadsto t$ is a geometric generization of $S$ there is a map $\L(\tropfont X_t) \to \L(\tropfont X_s)$, but there is no guarantee of a generization map $H^1(\tropfont X_t, \L) \to H^1(\tropfont X_s, \L)$ if $s$ and $t$ are in different strata.  To get the generization map, we will need to impose the bounded monodromy condition in Section~\ref{sec:monodromy}.

\begin{proposition} \label{prop:trojac-tropen}
Let $\tropfont X$ be a logarithmic curve metrized by $\overnorm M$, let $\overnorm M \to \overnorm N$ be a homomorphism such that $\overnorm M^{\rm gp} \to \overnorm N^{\rm gp}$ is an isomorphism, and let $\tropfont Y$ be the induced tropical curve metrized by $\overnorm N$.  Then $\mathrm B\L_{\tropfont X} \to \mathrm B\L_{\tropfont Y}$ is an isomorphism of stacks on $\tropfont X$.
\end{proposition}
\begin{proof}
Since $\overnorm M^{\rm gp} = \overnorm N^{\rm gp}$, the sheaves $\L_{\tropfont X}$ and $\L_{\tropfont Y}$ are the same when we identify the underlying graphs of $\tropfont X$ and $\tropfont Y$.
\end{proof}

\subsection{The intersection pairing on a tropical curve}
\label{sec:intersection}

The following definition seems to be due originally to Grothendieck~\cite[eq.\ (12.4.5)]{sga7-IX}:

\begin{definition} \label{def:intersection}
Let $\tropfont X$ be a tropical curve metrized by a monoid $\overnorm M$, and let $\ell(e) \in \overnorm M$ denote the length of an edge $e$ of $\tropfont X$.  If $e$ and $f$ are oriented edges of $\tropfont X$, we define
\begin{equation*}
e . f = \begin{cases} 
\ell(e) & f = e \\
-\ell(e) & f = e' \\
0 & \text{else} \end{cases}
\end{equation*}
and extend by linearity to an \emph{intersection pairing} on the free abelian group generated by the oriented edges of $\tropfont X$.  By restriction it also gives a pairing on the first homology of $\tropfont X$.
\end{definition}

\begin{lemma} \label{lem:pair-local}
Suppose that $\tropfont X$ is a tropical curve metrized by a monoid $\overnorm M$ and $u : \overnorm M \to \overnorm N$ is a homomorphism inducing an edge contraction $\tropfont Y$ of $\tropfont X$.  Then the intersection pairing is compatible with $u$, in the sense that diagram~\eqref{eqn:39} commutes:
\begin{equation} \label{eqn:39} \vcenter{ \xymatrix{
H_1(\tropfont X) \times H_1(\tropfont Y) \ar@<-3.65em>[d] \subset \mathbf Z^{E(\tropfont X)} \times \mathbf Z^{E(\tropfont X)} \ar[r] \ar@<4em>[d] & \overnorm M^{\rm gp} \ar[d] \\
H_1(\tropfont Y) \times H_1(\tropfont Y) \subset \mathbf Z^{E(\tropfont Y)} \times \mathbf Z^{E(\tropfont Y)} \ar[r] & \overnorm N^{\rm gp}
}} \end{equation}
\end{lemma}
\begin{proof}
The proof is immediate.
\end{proof}

\subsection{The tropical degree}
\label{sec:trop-deg}

Let $\mathsf V$ denote the quotient $\PL / \L$.  Then $\mathsf V(\tropfont U)$ is the free abelian group generated by the vertices of $\tropfont U$.

Let $\tropfont X$ be a tropical curve metrized by $\overnorm M$.  There is an embedding of the constant sheaf $\overnorm M^{\rm gp}$ inside $\L_{\tropfont X}$ as the constant functions.  We write $\mathsf H$ for the quotient of $\L$ by $\overnorm M^{\rm gp}$ and $\mathsf E$ for the quotient of $\PL$ by $\overnorm M^{\rm gp}$.  This yields a commutative diagram~\eqref{eqn:34} with exact rows and columns:
\numberwithin{equation}{subsection}
\begin{equation}\label{eqn:34} \vcenter{ \xymatrix{
& & 0 \ar[d] & 0\ar[d]  \\
0 \ar[r] & \overnorm M^{\rm gp} \ar[r] \ar@{=}[d] & \L \ar[r] \ar[d] & \mathsf H \ar[r] \ar[d] & 0 \\
0 \ar[r] & \overnorm M^{\rm gp} \ar[r] & \PL \ar[r] \ar[d] & \mathsf E \ar[r] \ar[d] & 0 \\
& & \mathsf V \ar@{=}[r] \ar[d] & \mathsf V \ar[d] \\
& & 0 & 0
}} \end{equation}
We note that $\mathsf E$ is the sheaf freely generated by the edges and that $\mathsf E \to \mathsf V$ is the coboundary map in homology.  Therefore $\mathsf H$ is the sheaf whose value on $\tropfont U$ is the first Borel--Moore homology of the topological realization of $\tropfont U$.  Note that because $\tropfont X$ can have $1$-sided or even $0$-sided edges, the Borel--Moore homology is not locally trivial.

\setcounter{theorem}{\value{equation}}
\begin{remark}
Diagram~\eqref{eqn:34} gives several ways of producing a tropical line bundle.  First, a class in $H^1(\tropfont X, \overnorm M^{\rm gp})$ (that is, a local system with transition functions in $\overnorm M^{\rm gp}$) gives a tropical line bundle by extending the structure group.  Second, an integer linear combination of vertices of $\tropfont X$ (a section of $\mathsf V$) yields a tropical line bundle by the coboundary map $H^0(\tropfont X, \mathsf V) \to H^1(\tropfont X, \L)$.

These are related by a third construction.  Beginning with a section of $\mathsf E$ --- that is, with an integer linear combination of \emph{edges} of $\tropfont X$ --- we get a class in $H^1(\tropfont X, \L)$ by the diagonal of the following commutative square:
\begin{equation*} \xymatrix{
H^0(\tropfont X, \mathsf E ) \ar[r] \ar[d] & H^1( \tropfont X, \overnorm M^{\rm gp}) \ar[d] \\
H^0(\tropfont X, \mathsf V) \ar[r] & H^1( \tropfont X, \L ) 
} \end{equation*}

The referee pointed out to us that the class in $H^1(\tropfont X, \overnorm M^{\rm gp}) = \Hom(H_1(\tropfont X), \overnorm M^{\rm gp})$ is precisely the obstruction to lifting a section of $\mathsf E$ to a piecewise linear function.  Thus edge weightings with trivial monodromy lift to piecewise linear functions and those with nontrivial monodromy produce tropical line bundles.

These edge weightings will be used to decorate the example in Figure~\ref{fig:genus-2}.
\end{remark}

\numberwithin{equation}{theorem}
\begin{lemma} \label{lem:flasque}
Let $\tropfont X$ be a tropical curve.  Then $H^p(\tropfont X, \mathsf E) = 0$ for all $p > 0$ and $H^p(\tropfont X, \mathsf V) = 0$ for all $p > 0$.
\end{lemma}
\begin{proof}
Note first that $\mathsf V$ is the pushforward along the closed embedding of the vertices of $\tropfont X$ of the constant sheaf $\mathbf Z$.  Therefore, writing $V(\tropfont X)$ for the set of vertices of $\tropfont X$, we have $H^p(\tropfont X, \mathsf V) = H^p(V(\tropfont X), \mathbf Z) = 0$ for all $p > 0$.

Next, note that $\mathsf E$ is the direct sum of sheaves $\mathsf E_i$ supported on each of the edges of $\tropfont X$.  Then $\mathsf E_i$ is the pushforward along the closed embedding of either an interval or a circle.  We can therefore assume that $\tropfont X$ is either an interval or a circle.

If $\tropfont X$ is an interval then its topology is generated by open subsets and $\mathsf E$ is flasque, hence has no higher cohomology.  If $\tropfont X$ is a circle then its universal cover $\tropfont Y$ has no self-loops, so $\mathsf E$ is flasque on $\tropfont Y$.  Therefore $H^p(\tropfont X, \mathsf E)$ can be identified with the group cohomology $H^p(\mathbf Z, \mathsf E(\tropfont Y))$.  The group cohomology of $\mathbf Z$ vanishes for $p > 1$ and for $p = 1$ it coincides with the coinvariants of $\mathsf E(\tropfont Y)$.  We identify $\mathsf E(\tropfont Y)$ with $\prod_{n=-\infty}^\infty \mathbf Z$ with $\mathbf Z$ acting by shift.  The coinvariants are therefore zero and the lemma is proved.
\end{proof}


\begin{corollary} \label{cor:Q}
If $\tropfont X$ is a compact tropical curve then $H^0(\tropfont X, \mathsf H) = H_1(\tropfont X)$ and $H^1(\tropfont X, \mathsf H) = H_0(\tropfont X)$.
\end{corollary}
\begin{proof}
This is immediate, as $\mathsf E(\tropfont X)$ is the free abelian group generated by the edges of $\tropfont X$ and $\mathsf V(\tropfont X)$ is the free abelian group generated by the vertices, and the map between them is the boundary map.
\end{proof}

\begin{lemma} \label{lem:harmonic}
If $\tropfont X$ is a compact tropical curve then $H^0(\tropfont X, \overnorm M^{\rm gp}) \to H^0(\tropfont X, \L)$ is an isomorphism.
\end{lemma}
\begin{proof}
We want to show that on a compact tropical curve, every globally defined linear function is locally constant.  Replacing $\overnorm M$ with a valuative submonoid of $\overnorm M^{\rm gp}$ that contains $\overnorm M$ does not change $\overnorm M^{\rm gp}$ or $\L$.  We may therefore assume that $\overnorm M$ is valuative.  Let $\tropfont Z$ be a maximal connected subgraph where $f$ takes its minimum value.  Then if $e$ is a flag of $\tropfont X$ exiting $\tropfont Z$, the slope of $f$ along $e$ must be positive.  But by the balancing condition, $\sum_e \mu(e) = 0$, when the sum is taken over all edges exiting $\tropfont Z$.  The only way a sum of positive numbers can be zero is if it is empty, so we conclude that $\tropfont Z$ is a connected component of $\tropfont X$ and that $f$ is locally constant.
\end{proof}

Using Corollary~\ref{cor:Q} and Lemma~\ref{lem:harmonic}, we write down the long exact sequence in cohomology associated to the short exact sequence in the first two rows of~\eqref{eqn:34}:
\numberwithin{equation}{subsection}
\setcounter{equation}{\value{theorem}}
\begin{gather} \label{eqn:35}
0 \to H_1(\tropfont X) \to H^1(\tropfont X, \overnorm M^{\rm gp}) \to H^1(\tropfont X, \L) \xrightarrow{\deg} H_0(\tropfont X) \to 0 \\
\mathsf E(\tropfont X) \to H^1(\tropfont X, \overnorm M^{\rm gp}) \to H^1(\tropfont X, \PL) \to 0 \notag
\end{gather}
The homomorphism $H^1(\tropfont X, \L) \to H_0(\tropfont X)$ is called the \emph{degree}.  We can also identify $H^1(\tropfont X, \overnorm M^{\rm gp}) = \Hom(H_1(\tropfont X), \overnorm M^{\rm gp})$.

\setcounter{theorem}{\value{equation}}
\numberwithin{equation}{theorem}
\begin{lemma} \label{lem:pairing}
The homomorphisms
\begin{align} \label{eqn:84}
H_1(\tropfont X) = H^0(\tropfont X, \mathsf H) \to H^1(\tropfont X, \overnorm M^{\rm gp}) & = \Hom(H_1(\tropfont X), \overnorm M^{\rm gp}) \\
\mathsf E(\tropfont X) = H^0(\tropfont X, \mathsf E) \to H^1(\tropfont X, \overnorm M^{\rm gp}) & = \Hom(H_1(\tropfont X), \overnorm M^{\rm gp}) \notag
\end{align}
in the exact sequences~\eqref{eqn:35} are the intersection pairing on $\tropfont X$.
\end{lemma}
\begin{proof}
The first homomorphism is induced from the second by restriction to $H_1(\tropfont X) \subset \mathsf E(\tropfont X)$, so it suffices to consider the second.  Suppose that $\alpha \in H^0(\tropfont X, \mathsf E)$.  We can regard $\alpha$ as an integer-valued function on the edges of $\tropfont X$.   The class of its coboundary in $H^1(\tropfont X, \overnorm M^{\rm gp})$ is the $\overnorm M^{\rm gp}$-torsor on $\tropfont X$ of piecewise linear functions having slopes $\alpha$ along the edges.

Such a torsor is classified by its failure to be representable by a well-defined, piecewise linear function, in the form of its monodromy around the loops of $\tropfont X$.  In other words, we may make the following identification:
\begin{equation} \label{eqn:47}
H^1(\tropfont X, \overnorm M^{\rm gp}) = \Hom(H_1(\tropfont X), \overnorm M^{\rm gp})
\end{equation}
Given $\alpha \in H^0(\tropfont X, \mathsf E)$ and a $\gamma \in H_1(\tropfont X)$, represented as a sum of oriented edges of $\tropfont X$, the monodromy of $\alpha$ around $\gamma$ is
\begin{equation*}
\sum_{e \in \gamma} \alpha(e) 
\end{equation*}
which is exactly the same as $\alpha . \gamma$.
\end{proof}

We summarize our results in the following corollary, which may be viewed as a tropical Abel theorem:

\begin{corollary} \label{cor:tropic}
Let $\tropfont X$ be a compact tropical curve metrized by $\overnorm M$.  Then there are exact sequences~\eqref{eqn:48}, where $\partial$ is the intersection pairing:
\begin{gather} \label{eqn:48}
0 \to H_1(\tropfont X) \xrightarrow{\partial} \Hom(H_1(\tropfont X), \overnorm M^{\rm gp}) \to H^1(\tropfont X, \L) \xrightarrow{\deg} H_0(\tropfont X) \to 0 \\
\mathsf E(\tropfont X) \xrightarrow{\partial} \Hom(H_1(\tropfont X), \overnorm M^{\rm gp}) \to H^1(\tropfont X, \PL) \to 0 \notag
\end{gather}
\end{corollary}

\subsection{Monodromy}
\label{sec:monodromy}

Let $\tropfont X$ be a tropical curve metrized by $\overnorm M$.  Let $Q$ be a $\PL$-torsor on $\tropfont X$.  By Corollary~\ref{cor:tropic}, there is an $\alpha \in \Hom(H_1(\tropfont X), \overnorm M^{\rm gp})$ inducing $Q$, uniquely determined by $Q$ up to addition of $\partial(e)$, for edges $e$ of $\tropfont X$.  We refer to $\alpha$ as a \emph{monodromy representative} of~$Q$.

\begin{proposition} \label{prop:bounded}
Let $\tropfont X$ be a compact, connected tropical curve metrized by a \emph{valuative} monoid $\overnorm M$.  The following conditions are equivalent of $Q \in H^1(\tropfont X, \PL)$:
\begin{enumerate}
	\item \label{it:subdiv} There exists a subdivision $\tropfont Y$ of $\tropfont X$ such that the restriction of $Q$ to $\tropfont Y$ is trivial.
	\item \label{it:monodromy} For any monodromy representative $\alpha$ of $Q$ and any $\gamma \in H_1(\tropfont X)$, the monodromy of $\alpha$ around $\gamma$ is bounded by the length of $\gamma$ (in the sense of Definition~\ref{def:bounded}).
\end{enumerate}
\end{proposition}

Before we begin the proof, we note that to verify the monodromy condition, it is sufficient to consider a single monodromy representative:

\begin{lemma} \label{lem:bdd-ind}
Suppose that $\alpha$ and $\beta$ are monodromy representatives of the same $\PL$-torsor on a tropical curve $\tropfont X$, and let $\gamma \in H_1(\tropfont X)$.  The monodromy of $\alpha$ around $\gamma$ is bounded by the length of $\gamma$ if and only if the monodromy of $\beta$ around $\gamma$ is bounded by the length of $\gamma$.
\end{lemma}
\begin{proof}
Since $\alpha$ and $\beta$ differ by a linear combination of $\partial(e)$, for $e$ among the edges of $\tropfont X$, it is sufficient by Lemma~\ref{lem:bounded} to show that $\partial(e)$ has bounded monodromy around each $\gamma \in H_1(\tropfont X)$.  But the monodromy of $\partial(e)$ around $\gamma$ is $e.\gamma$.  If $e$ is not contained in $\gamma$ then $e.\gamma = 0$, which is obviously bounded by $\ell(\gamma)$.  If $e$ is contained in $\gamma$ then $e.\gamma = \pm \ell(e)$, and $\ell(e)$ is bounded by $\ell(\gamma)$ because $e$ is contained in $\gamma$.
\end{proof}

\begin{lemma} \label{lem:subdiv-bounded}
Suppose that $\tau : \tropfont Y \to \tropfont X$ is a subdivision of tropical curves.   Let $\alpha$ be an element of $H^1(\tropfont X, \overnorm M^{\rm gp})$.  The monodromy of $\alpha$ around the loops of $\tropfont X$ is bounded by their lengths if and only if same holds of the monodromy of $\tau^\ast \alpha$ around the loops of $\tropfont Y$.
\end{lemma}
\begin{proof}
The lengths of the loops of $\tropfont Y$ is the same as the length of the loops in $\tropfont X$ and the monodromy around them is the same as the monodromy around the loops in $\tropfont X$.
\end{proof}

\begin{proof}[Proof of Proposition~\ref{prop:bounded}]
Suppose first that $Q$ can be trivialized on a subdivision $\tau : \tropfont Y \to \tropfont X$.  Let $\mu : H_1(\tropfont X) \to \overnorm M^{\rm gp}$ be a monodromy representative of $Q$.  Then $\mu$ lies in the image of $\partial : \mathsf E(\tropfont Y) \to \Hom(H_1(\tropfont Y), \overnorm M^{\rm gp}) = \Hom(H_1(\tropfont X), \overnorm M^{\rm gp})$, so by Lemma~\ref{lem:bdd-ind}, its monodromy around the loops of $\tropfont Y$ is certainly bounded by the lengths of the loops.  But by Lemma~\ref{lem:subdiv-bounded}, this implies that $\mu$ has the same property.

Now assume that the monodromy of $\mu$ around the loops of $\tropfont X$ is bounded by their lengths.  We construct a subdivision $\tau : \tropfont Y \to \tropfont X$ such that $\tau^\ast \mu$ is in the image of $\partial : \mathsf E(\tropfont Y) \to \Hom(H_1(\tropfont Y), \overnorm M_S^{\rm gp})$.

The proof will be by induction on the rank of the image of the monodromy homomorphism~\eqref{eqn:49}:
\begin{equation} \label{eqn:49}
\mu : H_1(\tropfont X) \rightarrow \overnorm M^{\rm gp} 
\end{equation}
Our strategy will be to subdivide $\tropfont X$ so that $\mathsf E(\tropfont Y)$ enlarges and adjust $\mu$ by the addition of elements in its image so that the rank of the image of $\mu$ decreases.  We therefore permit ourselves to adjust $\mu$ as necessary by elements of the image of $\partial$.

By Proposition~\ref{prop:arch-filt}, there is a filtration
\begin{equation*}
0 = V_0 \subset V_1 \subset \cdots \subset  \overnorm M_S^{\rm gp}
\end{equation*}
of ordered subgroups such that each $V_n / V_{n-1}$ may be embedded in $\mathbf R$, preserving the ordering.  Let $n$ be the largest index such that the image of $\mu$ is contained in $V_n$.  If $n = 0$ we are done.  Otherwise, choose an embedding of $V_n / V_{n-1}$ in $\mathbf R$.

This induces a metric on $\tropfont X$ with lengths in $\mathbf R$.  We write $\overnorm{\tropfont X}$ for the tropical curve obtained by collapsing those edges in $\tropfont X$ whose lengths in $\mathbf R$ are zero.  Note that there is a well-defined monodromy function
\begin{equation*}
\overnorm\mu : H_1(\overnorm{\tropfont X}) \rightarrow \mathbf R
\end{equation*}
precisely because the monodromy around $\gamma \in H_1(\overnorm{\tropfont X})$ is bounded by the length of $\gamma$.  Indeed, $\gamma \in H_1(\tropfont X)$ has length $\delta$, and the image of $\delta$ in $\mathbf R$ is zero then the boundedness of the mondromy around $\gamma$ implies that $\overnorm\mu(\gamma) = 0$ as well.

Choose a spanning tree of $\overnorm{\tropfont X}$ and let $E$ be the set of edges of $\overnorm{\tropfont X}$ not in the spanning tree.  Each of these edges corresponds uniquely to an edge of $\tropfont X$, so we will also think of $E$ as a set of edges of $\tropfont X$.

For each $e \in E$, let $\gamma_e$ be the corresponding basis element of $H_1(\overnorm{\tropfont X})$.  Let $\delta_e$ be the length of $\gamma_e$ and let $\mu_e = \overnorm\mu(\gamma_e)$ be the monodromy around it, both valued in $\mathbf R$.  Since $\delta_e \neq 0$, there is some integer $k$ such that $k \delta_e \leq \mu_e \leq (k+1)\delta_e$.  We replace $\mu$ by $\mu - k \partial(e)$ so that we may assume that $0 \leq \mu_e < \delta_e$.  Note that doing so does not change $\mu_f$ for any $f \neq e$.  Let $\tau : \tropfont Y \to \tropfont X$ be a subdivision of $\tropfont X$ that divides the edge $e$ into edges $e'$ and $e''$ of lengths $\mu_e$ and $\delta_e - \mu_e$, respectively.  Then $\tau^\ast \mu - \partial(e')$ has no monodromy around $\gamma_e$, and the monodromy around all $f \neq e$ in $E$ remains unchanged.

Repeating this procedure for all $e$ in $E$, we arrive at a representative for the monodromy of $Q$ such that the image of $\mu$ in $V_n / V_{n-1}$ is zero.  Now we repeat the process with $n$ replaced by $n-1$ until we have replaced $\mu$ by $0$.
\end{proof}

\begin{corollary} \label{cor:bdd-mono}
Let $\tropfont X$ be a compact, connected tropical curve metrized by a monoid $\overnorm M$.  Then $Q \in H^1(\tropfont X, \PL)$ satisfies Condition~\ref{it:monodromy} of Proposition~\ref{prop:bounded} if and only if satisfies Condition~\ref{it:subdiv} over every valuative extension of $\overnorm M$.  
\end{corollary}
\begin{proof}
Let $\alpha$ be a monodromy representative of $Q$.  The corollary reduces to the assertion that $\alpha \prec \delta$ in $\overnorm M$ if and only $\alpha \prec \delta$ in every valuative extension of $\overnorm M$.  For each pair $n, m$, let $U_{n,m} \subset \Cone^\circ(\overnorm M)$ be the set of valuations in which $n \delta < \alpha < m \delta$.  By assumption, the $U_{n,m}$ cover $\Cone^\circ(\overnorm M)$.  But $\Cone^\circ(\overnorm M)$ is quasicompact, so a finite collection $U_{n_i,m_i}$ suffices to cover it.  Taking $n = \min \{ n_i \}$ and $m = \max \{ m_i \}$ we find that all $U_{n_i,m_i}$ are contained in $U_{n,m}$, so $U_{n,m} = \Cone^\circ(\overnorm M)$.  Thus we have $n \delta \leq \alpha \leq m \delta$ in every sharp valuative extension of $\overnorm M$.  But $\overnorm M$ is saturated, so an inequality holds in $\overnorm M$ if and only if it holds in every valuative extension of $\overnorm M$.  Therefore $n \delta \leq \alpha \leq m \delta$ and $\alpha \prec \delta$, as required.
\end{proof}

\begin{definition} \label{def:bdd-mono}
We say that a homomorphism $H_1(\tropfont X) \to \overnorm M^{\rm gp}$ on $\tropfont X$ has \emph{bounded monodromy} if it satisfies the equivalent conditions of Corollary~\ref{cor:bdd-mono}.  We indicate a subgroup of bounded monodromy by decoration with a dagger ($\dagger$).
\end{definition}

\subsection{The tropical Jacobian}
\label{sec:trojac}

Let $\tropfont X$ be a tropical curve metrized by a monoid $\overnorm M$.  We construct the tropical Jacobian of $\tropfont X$ in a manner covariantly functorial in $\overnorm M$.  This effectively constructs the tropical Jacobian relative to the moduli space of tropical curves. 

\begin{definition} \label{def:trojac}
We define the tropical Jacobian by~\eqref{eqn:85}, where the dagger ($\dagger$) indicates the subgroup of elements with bounded monodromy (Definition~\ref{def:bdd-mono}):
\begin{equation} \label{eqn:85}
\TroJac(\tropfont X) = \Hom\bigl(H_1(\tropfont X), \overnorm M^{\rm gp}\bigr)^\dagger / H_1(\tropfont X)
\end{equation}
\end{definition}

\begin{example} \label{ex:trojac}
When $\overnorm M = \mathbf R_{\geq 0}$, the tropical Jacobian is a real torus.  Over a general monoid (that is integral, saturated, and finitely generated), we imagine a family of real tori, parameterized by homomorphisms $\overnorm M \to \mathbf R_{\geq 0}$, and $\TroJac(\tropfont X)$ represents the space of sections of this family respecting appropriately defined integral structure.
\end{example}

Now suppose that we have a monoid homomorphism $\overnorm M \to \overnorm N$.  This induces an edge contraction $\tropfont Y$ of $\tropfont X$.  We wish to produce a morphism:
\numberwithin{equation}{subsection}
\setcounter{equation}{\value{theorem}}
\begin{equation} \label{eqn:44}
\TroJac(\tropfont X) \to \TroJac(\tropfont Y)
\end{equation}
The edge contraction $\tropfont X \to \tropfont Y$ induces a homomorphism $H_1(\tropfont X) \to H_1(\tropfont Y)$.  Note that if $\mu \in \TroJac(\tropfont X)$ has bounded monodromy then, by definition, the composition
\begin{equation} \label{eqn:86}
H_1(\tropfont X) \to \overnorm M^{\rm gp} \to \overnorm N^{\rm gp}
\end{equation}
takes the value zero on all loops of $\tropfont X$ contracted in $\tropfont Y$.  Therefore the homomorphism factors through $H_1(\tropfont Y)$, and does so uniquely because $H_1(\tropfont X) \to H_1(\tropfont Y)$ is surjective.  The factorization still has bounded monodromy, since if $\alpha$ is bounded by $\delta$ in $\overnorm M^{\rm gp}$ then its image in $\overnorm N^{\rm gp}$ is bounded by the image of $\delta$.  We obtain the desired morphism~\eqref{eqn:44}.

It is clear from the construction that it respects compositions of monoid homomorphisms.  Following the procedure described in Section~\ref{sec:trop-mod}, we may extend the definition of $\TroJac(\tropfont X)$ to families.  That is, given a family of tropical curves $\tropfont X$ over a logarithmic scheme $S$ we obtain an \'etale sheaf on the category of logarithmic schemes over $S$ by either of the following equivalent procedures:
\begin{enumerate}
\item If $T$ is an atomic neighborhood of a geometric point $t$, then we define $\TroJac(\tropfont X/S)(T) = \TroJac(\tropfont X_t)$ and sheafify the resulting presheaf.
\item If $T$ is a logarithmic scheme over $S$ of finite type then an object of $\TroJac(\tropfont X/S)(T)$ consists of objects of $\TroJac(\tropfont X_t)$ for each geometric point $t$ that are compatible along geometric generizations.  We extend from logarithmic schemes that are of finite type to all logarithmic schemes by the approximation procedure of~\cite[IV.8]{EGA}.
\end{enumerate}
If $X$ is a proper, vertical logarithmic curve over $S$ with tropicalization $\tropfont X$ then we pose $\TroJac(X/S) = \TroJac(\tropfont X/S)$.

\begin{example} \label{ex:trojac2}
If $S$ is an atomic logarithmic scheme and $\tropfont X$ is a family of tropical curves over $S$ then $\TroJac(\tropfont X/S) = \Hom(H_1(\tropfont X), \ologGm)^\dagger / H_1(\tropfont X)$.  In particular, $\TroJac(\tropfont X/S)$ is typically not representable by an algebraic stack with a logarithmic structure (or a cone stack, in the language of \cite{cavalieri2017moduli}):  if it were then the tropical torus $\Hom(H_1(\tropfont X), \ologGm)^\dagger$, which is a $H_1(\tropfont X)$-torsor over $\TroJac(\tropfont X/S)$, would be representable.  The proof Proposition~\ref{prop:nonrep} shows that it is not.
\end{example}

\begin{example}

We illustrate the necessity of the bounded monodromy condition from a tropical perspective; in Section~\ref{sec:unint} we will see an algebraic version of the same idea.  Our discussion will be somewhat informal, as we simply wish to convey some intuition.

Let $C$ be the nodal curve of genus $2$ consisting of two rational curves joined at $3$ nodes. Let $S$ be a smooth surface with a normal crossings divisor $D=D_1 \cup D_2$ consisting of two smooth components meeting at a single point $s$, and let $X \to S$ be a family of curves smoothing $C \to s$ to a smooth family of genus $2$ curves over $S-D$. We choose $X$ to have smooth total space, and assume that the first and third node smooth out together over $D_2$, while the second smooths out indepenently over $D_1$. Such a family can be constructed for example by restricting to a sufficiently small two parameter sub-family of a versal family of $C$.

We give $X \to S$ its minimal logarithmic structure, that is, the log structure pulled back from the moduli space of curves and its universal family $\overline{\mathcal{C}}_2 \to \overline{\mathcal{M}}_2$. The logarithmic scheme $S$ has four strata, given by $s,D_1-s,D_2-s$, and $S-D$. Restricting to a sufficiently small neighborhood of $s$ if necessary (an atomic neighborhood in the terminology above) the logarithmic structure is globally generated, with 
$$
\Gamma(S,\overline{M}_S) = \overline{M}_{S,s} = \mathbf{N}^2
$$   

The precise construction of the geometric family $X \to S$ is not important: we chose it because it leads to a simple yet interesting tropicalization, which we could have taken as our original input. The tropicalization $\tropfont X \to S$ is a family of polyhedral complexes over $S$, constant on the strata of $S$, and metrized by the various sheaves of monoids $\overnorm{M}_{S,t}$ for $t \in S$. Its most interesting fiber is the fiber over the deepest stratum $s$, where it consists of the dual graph of $C$, with two vertices $v_1,v_2$ joined by three edges $e_1,e_2,e_3$. Each edge has a length $\delta_i$ in the monoid $\overnorm{M}_{S,s} = \mathbf{N}^2$. A picture is shown in Figure~\ref{fig:genus-2} for an arbitrary monoid $\overline{M}$. Our assumption that the total space of $X$ is smooth and two of the nodes smooth together, while the other one smooths independently mean that the lengths $\delta_1$ and $\delta_3$ of $e_1$ and $e_3$ are equal to each other and one of the generators of $\mathbf{N}^2$, while the length $\delta_2$ of $e_2$ is the other generator of $\mathbf{N}^2$. We can think of this tropical curve equivalently\footnote{To get equivalence, we also need to remember the integral structure $\Hom(\overnorm{M}_{S,s}, \mathbf{N}) \subset \sigma_s$.} as a family of traditional tropical curves, with edge lengths in $\mathbf{R}_{\ge 0}$, varying over the dual cone $\sigma_s: = \Hom(\overnorm{M}_{S,s}, \mathbf{R}_{\ge 0}) \cong \mathbf{R}_{\ge 0}^2$. The total space of this family thus has three real dimensions. The length of the edge $e_i$ in the fiber over $h = (a,b) \in \sigma$ is simply the evaluation $h(\delta_i)$.  With our choices these lengths are $a$ for $e_1,e_3$ and $b$ for $e_2$.    

As the stratum $s$ generizes to $D_1$, the edge $e_2$ and its length $\delta_2$ contract, and we obtain the tropical curve metrized by $\overline{M}_{S,t_1} \cong \mathbf{N}$ (with $t_1 \in D_1-s$) consisting of a single vertex $v$ with two edges $e_1,e_3$ joining $v$ to itself, both of length the generator of $\mathbf{N}$. Again, we can think of this curve as the total space of a family over $\sigma_{t_1} = \Hom(\overnorm{M}_{S,t},\mathbf{R}_{\ge 0}) = \mathbf{R}_{\ge 0}$. These data are redundant: $\sigma_{t_1}$ appears as a face of $\sigma_s$, and compatibility with generization means that the restriction of $\tropfont X_s$ to the face $\sigma_{t_1}$ agrees with $\tropfont X_{t_1}$. Similarily, generizing $s$ to $D_2$ contracts $e_1,e_3$, and we have the tropical curve metrized over $\mathbf{N}$ consisting of a single vertex $v$ with one edge $e_2$, of length the generator again. This can be thought of as the fiber of $\tropfont X_s$ over the other one dimensional face of $\sigma_s$. Generizing to $S-D$, the fiber of $\tropfont X \to S$ reduces to a point, which we can think of as a point over the $0$ face of $\sigma_s$. A picture of the tropicalization is shown in Figure~\ref{fig:4}.  

\begin{figure}
  \begin{tikzpicture}
\tikzset{insert |/.style={decoration={markings,
  mark=at position #1 with {%
   \draw[line cap=round,mark segment] 
    (0,-\pgfkeysvalueof{/tikz/mark
    segment length}/2) -- (0,\pgfkeysvalueof{/tikz/mark
    segment length}/2);}
  }},
  | mark/.style={postaction=decorate,insert |=#1},
  insert ||/.style={decoration={markings,
  mark=at position #1 with {%
   \draw[line cap=round,mark segment] 
    (-\pgfkeysvalueof{/tikz/mark segment distance}/2,-\pgfkeysvalueof{/tikz/mark
    segment length}/2) -- (-\pgfkeysvalueof{/tikz/mark segment distance}/2,\pgfkeysvalueof{/tikz/mark
    segment length}/2);
   \draw[line cap=round,mark segment] (\pgfkeysvalueof{/tikz/mark segment distance}/2,-\pgfkeysvalueof{/tikz/mark
    segment length}/2) -- (\pgfkeysvalueof{/tikz/mark segment distance}/2,\pgfkeysvalueof{/tikz/mark
    segment length}/2);}
  }},
  || mark/.style={postaction=decorate,insert ||=#1},
 insert |||/.style={decoration={markings,
  mark=at position #1 with {%
   \draw[line cap=round,mark segment] 
    (-\pgfkeysvalueof{/tikz/mark segment distance},-\pgfkeysvalueof{/tikz/mark
    segment length}/2) -- (-\pgfkeysvalueof{/tikz/mark segment distance},\pgfkeysvalueof{/tikz/mark
    segment length}/2);
   \draw[line cap=round,mark segment] 
    (0,-\pgfkeysvalueof{/tikz/mark
    segment length}/2) -- (0,\pgfkeysvalueof{/tikz/mark
    segment length}/2); 
   \draw[line cap=round,mark segment] 
   (\pgfkeysvalueof{/tikz/mark segment distance},-\pgfkeysvalueof{/tikz/mark
    segment length}/2) -- (\pgfkeysvalueof{/tikz/mark segment distance},\pgfkeysvalueof{/tikz/mark
    segment length}/2);}
  }},
  ||| mark/.style={postaction=decorate,insert |||=#1},
 mark segment/.style={thick},
 mark segment options/.code=\tikzset{mark segment/.style={#1}},
 mark segment distance/.initial=2pt,
 mark segment length/.initial=4pt,
 angle deco |/.style={insert |=0.5,
      pic actions/.append code=\tikzset{postaction=decorate}},
 angle deco ||/.style={insert ||=0.5,
      pic actions/.append code=\tikzset{postaction=decorate}},
 angle deco |||/.style={insert |||=0.5,
      pic actions/.append code=\tikzset{postaction=decorate}}}
    
\coordinate (v1) at (4,0,0){};
    \node[right] at (v1){$(0,0)$};
    \node[draw,circle,inner sep=1pt,fill] at (v1){};
    \coordinate (v2) at (4,2,-1){};
    \node[above] at (v2){$(a+b,b)$};
    \node[draw,circle,inner sep = 1pt,fill] at (v2){};
    \coordinate (v3) at (4,-1,2){};
    \node[left] at (v3){$(b,a+b)$}; 
    \node[draw,circle,inner sep = 1pt, fill] at (v3){};
    \coordinate (v4) at (4,1,1){};
    \node[left] at (v4){$(a+2b,a+2b)$}; 
    \node[draw,circle,inner sep=1pt,fill] at (v4){};

\draw[-, red, thick] (v1) -- (v2) -- (v4) -- (v3) -- (v1) -- cycle; 
\path[|| mark =0.5] (v1) -- (v2);
\path[|| mark =0.5] (v3)-- (v4);
\path[| mark =0.5] (v1) -- (v3);
\path[| mark =0.5] (v2) -- (v4);
\fill[red!20,nearly transparent] (v1) -- (v2) -- (v4) -- (v3) -- (v1) -- cycle;

\coordinate (w1) at (0,0,0){};
    \node[right] at (w1){$(0,0)$};
    \node[draw,circle,inner sep=1pt,fill] at (w1){};
    \coordinate (w2) at (0,2,0){};
    \node[above] at (w2){$(a+b,0)$};
    \node[draw,circle,inner sep = 1pt,fill] at (w2){};
    \coordinate (w3) at (0,0,2){};
    \node[left] at (w3){$(0,a+b)$}; 
    \node[draw,circle,inner sep = 1pt, fill] at (w3){};
    \coordinate (w4) at (0,2,2){};
    \node[left] at (w4){$(a+b,a+b)$}; 
    \node[draw,circle,inner sep=1pt,fill] at (w4){};

\draw[-, blue, thick] (w1) -- (w2) -- (w4) -- (w3) -- (w1) -- cycle; 
\path[|| mark =0.5] (w1) -- (w2);
\path[|| mark =0.5] (w3)-- (w4);
\path[| mark =0.5] (w1) -- (w3);
\path[| mark =0.5] (w2) -- (w4);
\fill[blue!20,nearly transparent] (w1) -- (w2) -- (w4) -- (w3) -- (w1) -- cycle;

\coordinate (u1) at (8,0,0){};
    \node[right] at (u1){$(0,0)$};
    \node[draw,circle,inner sep=1pt,fill] at (u1){};
    \coordinate (u2) at (8,2,-2){};
    \node[below] at (u2){};
    \node[draw,circle,inner sep = 1pt,fill] at (u2){};
    \coordinate (u3) at (8,-2,2){};
    \node[right] at (u3){}; 
    \node[draw,circle,inner sep = 1pt, fill] at (u3){};
    \coordinate (u4) at (8,0,0){};
    \node[above] at (u4){}; 
    \node[draw,circle,inner sep=1pt,fill] at (u4){};

\draw[-, green, thick] (u1) -- (u2) -- (u4) -- (u3) -- (u1) -- cycle; 
\path[|| mark =0.5] (u1) -- (u2);
\path[|| mark =0.5] (u3)-- (u4);

\draw[-,dashed] (w1) -- (v1) -- (u1);
\draw[-,dashed] (w2) -- (v2) -- (u2);
\draw[-,dashed] (w3) -- (v3) -- (u3); 
\draw[-,dashed] (w4) -- (v4) -- (u4);

\coordinate (t1) at (4,-5,2){}; 
\node [draw, circle,fill=red,inner sep=1pt] at (t1){};
\coordinate (t2) at (4,-5,-2){};
\node [draw, circle,fill=red,inner sep=1pt] at (t2){};
\node[above] at (4,-5,0){$b$};
\node[above] at (4,-4.2,0){$a$};
\node[below] at (4,-5.8,0){$a$};

\draw [red] (t1) to[out=135,in=135] (t2);
\draw[-, red] (t1) -- (t2);
\draw [red] (t1) to[out=-45,in=-45] (t2);  

\coordinate (s1) at (0,-5,0){};
\node[draw, circle, fill=blue, inner sep =1pt] at (s1){};
\coordinate (s2) at (0,-5,3){};
\node[left] at (s2){$a+b$};
\coordinate (s3) at (0,-5,-3){}; 
\node[right] at (s3){$a+b$};

\draw[blue] (s1) to[out=135,in=135] (s2) to[out=-45, in=-45] (s1);
\draw[blue] (s1) to[out=135,in=135] (s3) to[out=-45, in=-45] (s1);

\coordinate (r1) at (8,-5,0){};
\node[draw, circle, fill=green, inner sep =1pt] at (r1){};
\coordinate (r2) at (8,-5,3){};
\node[left] at (r2){$a+b$};

\draw[green] (r1) to[out=135,in=135] (r2) to[out=-45, in=-45] (r1);

\draw[-,dashed](s1) -- (t1) -- (r1); 
\draw[-,dashed](s1) -- (t2) -- (r1);

\coordinate(b1) at (0,-10,0){};
\node[above] at (b1){$(a+b,0)$};
\node[draw, blue, circle,inner sep=1pt,fill] at (b1){};
\coordinate(b2) at (8,-10,0){}; 
\node[above] at (b2){$(0,a+b)$};
\node[draw, green, circle, inner sep=1pt,fill] at (b2){};
\coordinate(b3) at (4,-10,0){};
\node[above] at (b3) at (b3){$(a,b)$}; 
\node[draw,red,circle,inner sep=1pt,fill] at (b3){}; 

\draw[-,thick](b1)--(b2)--(b3);

\end{tikzpicture}
\caption{The Tropical Jacobian $\TroJac(X/S)$ (top) and the tropicalization $\tropfont{X}$ (middle), over the cross-section $x+y = a+b$ of $\sigma_s = \mathbf{R}_{\ge 0}^2$ (bottom). }
\label{fig:4}
\end{figure}

We write $J = \TroJac(\tropfont X/S)$.  In order to motivate the bounded monodromy condition, let us accept that for each point $(a,b) \in \sigma_s$, corresponding dually to a homomorphism $\mathbf{N}^2 \to \mathbf R$, that the fiber of $J$ over $(a,b)$ should be a real torus, constructed as follows:
\begin{equation*}
J_{(a,b)} = \Hom\bigl(H_1(\tropfont{X}_{(a,b)}), \mathbf{R}_{\ge 0}\bigr)/ \partial H_1(\tropfont X_{(a,b)})
\end{equation*}
More specifically, if $a \neq 0$ then we have the presentation~\eqref{eqn:127}, whereas if $a = 0$ and $b \neq 0$, we have~\eqref{eqn:128} (if $a = b= 0$ then we have $J_{(0,0)} = 0$, of course):
\begin{align} 
J_{(a,b)} & = \Hom( \mathbf Z^2, \mathbf R_{\geq 0} ) / \mathbf Z (a + b, -b) + \mathbf Z (-b, a+b)\label{eqn:127} \\
J_{(0,b)} & = \Hom( \mathbf Z, \mathbf R_{\geq 0} ) / \mathbf Z b \label{eqn:128}
\end{align}

Whatever the definition of the $S$-points of $J$, it should have natural specialization maps to $J_{(a,b)}$ for all $(a,b) \in \sigma_s$.  The bounded monodromy condition is imposed to guarantee this.  In its absence, the ``tropical Jacobian''
\begin{equation*}
\Hom( H_1(\tropfont X_s), \Gamma(S, \overnorm M_S^{\rm gp}) ) / \partial H_1(\tropfont X_s)
\end{equation*}
contains a homomorphism $\phi$ with $\phi(e_1 - e_2) = 0$ and $\phi(e_2 - e_3) = \delta_2$.  Specializing to $a=0$ (i.e., setting $\delta_1 = \delta_3 = 0$), the loop $e_1 - e_3$ is collapsed, yet $\phi(e_1 - e_3)$ is not zero modulo the subgroup generated by $\delta_1 = \delta_3$.  In other words, $\phi$ does not descend to be well-defined on $H_1(\tropfont X_t)$ when $t$ is a generization of $s$ in $D_2$.

To correct for this, we demand that when the homomorphism $\phi : H_1(\tropfont X_s) \to \overnorm M_S^{\rm gp}$ is evaluated on a loop $\gamma$ the value lies in the subgroup of elements of $\overnorm M_S^{\rm gp}$ that are mapped to $0$ when $\gamma$ is smoothed.  This is precisely the bounded monodromy condition.  We observe that, in the example above, the monodromy of $\phi$ around the loop $\gamma = e_1 - e_3$ is not bounded by the length, $\delta_1 + \delta_3 = 2 \delta_1$, of $\gamma$ hence remains nonzero when $\gamma$ is smoothed.
\end{example}

\subsection{The tropical Picard group}
\label{sec:tropic-sheaf}

\begin{definition} \label{def:tropic}
Let $\tropfont X$ be a tropical curve metrized by a monoid $\overnorm M$.  We say that an element of $H^1(\tropfont X, \L)$ has bounded monodromy if its image in $H^1(\tropfont X, \PL)$ has bounded monodromy (which means that it is the image of a class of bounded monodromy in $H^1(\tropfont X, \overnorm M^{\rm gp}) = \Hom(H_1(\tropfont X), \overnorm M^{\rm gp})$).  For each $d \in H_0(\tropfont X)$, we write $\TroPic^d(\tropfont X)$ for the preimage of $d$ under the degree homomorphism $\TroPic(\tropfont X) \subset H^1(\tropfont X, \L) \xrightarrow{\deg} H_0(\tropfont X)$ from Corollary~\ref{cor:tropic}.

We define $\bTroPic(\tropfont X)$ to be category of $\L$-torsors whose classes in $H^1(\tropfont X, \L)$ have bounded monodromy, and we define $\TroPic(\tropfont X)$ to be the set of isomorphism classes of $\TroPic(\tropfont X)$.  Objects of $\bTroPic(\tropfont X)$ are called \emph{tropical line bundles} on $\tropfont X$.
\end{definition}

The main task of this section is to describe the functoriality of $\TroPic(\tropfont X)$ with respect to the monoid~$\overnorm M$.

\begin{lemma} \label{lem:decomp}
A class in $H^1(\tropfont X, \L)$ has bounded monodromy if and only if it is the sum of a class in the image of $H^0(\tropfont X, \mathsf V)$ and a class of bounded monodromy in $H^1(\tropfont X, \overnorm M^{\rm gp}) = \Hom(H_1(\tropfont X), \overnorm M^{\rm gp})$ (under the maps induced from diagram~\eqref{eqn:34}).
\end{lemma}
\begin{proof}
This follows from the commutativity of the diagram~\eqref{eqn:50}, below, and its exactness in the second row (which is the long exact sequence associated to the middle column of~\eqref{eqn:34}):
\begin{equation} \vcenter{\xymatrix{
& H^1(\tropfont X, \overnorm M^{\rm gp}) \ar@{=}[r] \ar[d] & H^1(\tropfont X, \overnorm M^{\rm gp}) \ar[d] \\
H^0(\tropfont X, \mathsf V) \ar[r] & H^1(\tropfont X, \L) \ar[r] & H^1(\tropfont X, \PL)
}} \end{equation}
\end{proof}

We will obtain the functoriality of $\TroPic(\tropfont X)$ from naturally defined functorial operations on $\PL$ and $\mathsf V$.  We begin by summarizing these. 

\begin{proposition} \label{prop:gen-P}
Let $\tropfont X$ be a tropical curve metrized by $\overnorm M$, let $u : \overnorm M \to \overnorm N$ be a homomorphism of monoids, and let $\sigma : \tropfont X \to \tropfont Y$ be the induced edge contraction.  Let $\mathsf D_{\tropfont X}$ and $\mathsf D_{\tropfont Y}$ be the diagrams~\ref{eqn:34} on $\mathscr X$ and on $\mathscr Y$, respectively.
\begin{enumerate}[label=(\roman{*})]
\item There is a unique homomorphism $\PL(\tropfont X) \to \PL(\tropfont Y)$, sending $f \in \PL(\tropfont X)$ to $\overnorm f \in \PL(\tropfont Y)$ such that $\overnorm f(\sigma( x)) = f(x)$ whenever $x$ is not contracted in $\tropfont Y$.
\item There is a unique homomorphism $\mathsf V(\tropfont X) \to \mathsf V(\tropfont Y)$ by sending the basis vector $[x]$ to $[\sigma(x)]$.
\item The homomorphisms above commute with the quotient map $\PL \to \mathsf V$.
\end{enumerate}
\end{proposition}
\numberwithin{equation}{theorem}
\begin{proof}\ 
\begin{enumerate}[label=(\roman{*})]
\item The uniqueness is evident, since every $y \in \tropfont Y$ is the image of some $x \in \tropfont X$ that is not contracted.  To check the existence, assume that $x$ is a flag of $\tropfont X$ that is contracted in $\tropfont Y$ and $f(x) = (\alpha, \mu)$.  Then $\alpha(i(x)) - \alpha(x) \in \mathbf Z \ell(x)$ and $\ell(x)$ lies in the kernel of $u$ (because $x$ is contracted), so $u(\alpha(i(x)) = u(\alpha(x))$.  Therefore $u \circ f$ is constant on the regions contracted by $\sigma$ and descends to $\tropfont Y$.
\item Immediate.
\item We argue that the diagram~\eqref{eqn:40} commutes:
\begin{equation} \label{eqn:40} \vcenter{\xymatrix{
\PL(\tropfont X) \ar[r] \ar[d] & \mathsf V(\tropfont X) \ar[d] \\
\PL(\tropfont Y) \ar[r] & \mathsf V(\tropfont Y)
}} \end{equation}
Let $f = (\alpha, \mu)$ be a piecewise linear function on $\tropfont X$.  The coefficient of $v$ in the image of $f$ in $\mathsf V(\tropfont X)$ is $\sum_{r(e) = v} \mu(e)$.  Therefore the image of $f$ in $\mathsf V(\tropfont Y)$, going around the top and right of~\eqref{eqn:40}, is $\sum_{f(v) = w} \sum_{r(e) = v} \mu(e)$.  In this sum, each edge of the contracted locus appears twice, with opposite orientations, and each edge exiting the contracted locus appears once, oriented out.  The sum therefore reduces to $\sum_{r(e) = w} \mu(e)$, which is what we get from following $f$ around the bottom and left of the diagram.
\end{enumerate}
\end{proof}

\begin{corollary} \label{cor:natural}
Diagram~\ref{eqn:34} is natural with respect to weighted edge contractions.
\end{corollary}

\numberwithin{equation}{theorem}
\begin{proposition} \label{prop:generization}
Let $\overnorm M \to \overnorm N$ be a homomorphism of commutative monoids inducing an edge contraction $\tropfont X \to \tropfont Y$ of tropical curves.  Then the maps
\begin{align}
H^0(\tropfont X, \mathsf V) & \to H^0(\tropfont Y, \mathsf V) \to H^1(\tropfont Y, \overnorm N)^\dagger \label{eqn:111} \\
H^1(\tropfont X, \overnorm M)^\dagger & \to H^1(\tropfont Y, \overnorm N)^\dagger \notag
\end{align}
agree on their common domain of definition and combine to define a map:
\begin{equation} \label{eqn:110}
H^1(\tropfont X, \L)^\dagger \to H^1(\tropfont Y, \L)^\dagger
\end{equation}
\end{proposition}
\begin{proof}
Diagram~\eqref{eqn:34} induces a commutative square~\eqref{eqn:50}, below:
\begin{equation} \label{eqn:50} \vcenter{\xymatrix{
H^0(\tropfont X, \PL) \ar[r] \ar@{=}[d] & H^0(\tropfont X, \mathsf E) \ar[r] \ar[d] & H^1(\tropfont X, \overnorm M^{\rm gp})^\dagger \ar[d] \ar[r] & H^1(\tropfont X, \PL)^\dagger \ar@{=}[d] \\
H^0(\tropfont X, \PL) \ar[r] & H^0(\tropfont X, \mathsf V) \ar[r] & H^1(\tropfont X, \L)^\dagger \ar[r] & H^1(\tropfont X, \PL)^\dagger
}} \end{equation}  
Suppose that $u \in H^1(\tropfont X, \L)$ is the image of some $v \in H^1(\tropfont X, \overnorm M^{\rm gp})^\dagger$ and $w \in H^0(\tropfont X, \mathsf V)$.  Then the image of $u$ in $H^1(\tropfont X, \PL)$ must vanish.  This is also the image of $v$, so that $v$ is the image of some $x \in H^0(\tropfont X, \mathsf E)$.  The difference between $w$ and the image of $x$ maps to $0$ in $H^1(\tropfont X, \L)$, hence is the image of some $y \in H^0(\tropfont X, \PL)$.  Replacing $w$ by $w - y$ we discover that we must show the two maps in question agree on $H^0(\tropfont X, \mathsf E)$.

We can define a map~\eqref{eqn:87} sending an edge $x$ to itself if it is not contracted in $\tropfont Y$ and to $0$ if it is contracted.
\begin{equation} \label{eqn:87}
H^0(\tropfont X, \mathsf E) \to H^0(\tropfont Y, \mathsf E)
\end{equation}
This commutes with the maps to $H^0(\tropfont X, \mathsf V)$ and $H^1(\tropfont X, \overnorm M^{\rm gp}) = \Hom(H_1(\tropfont X), \overnorm M^{\rm gp})$.
\end{proof}

\subsection{The tropical Picard stack}
\label{sec:tropic-stack}

The construction in Proposition~\ref{prop:generization} can be categorified to operate on $\L$-torsors with bounded monodromy, and not merely their isomorphism classes.  Given an edge contraction $\sigma : \tropfont X \to \tropfont Y$ associated to a homomorphism of monoids $\overnorm M \to \overnorm N$ and an $\L$-torsor $Q$ on $\tropfont X$ with bounded monodromy, we wish to produce an $\L$-torsor on $\tropfont Y$ in a canonical way.  

Using the following lemma, we may promote $\sigma$ to be a morphism of sites.

\begin{lemma}
Let $\sigma : \tropfont X \to \tropfont Y$ be an edge contraction induced from a homomorphism $\overnorm M \to \overnorm N$.  Let $\tropfont V \to \tropfont Y$ be a local isomorphism.  Then the set-theoretic fiber product $\tropfont U = \tropfont V \mathop\times_{\tropfont Y} \tropfont X$ is naturally equipped with the structure of a tropical curve and the projection $\tropfont U \to \tropfont X$ is a local isomorphism.
\end{lemma}
\begin{proof}
The involution $i$ and the partially defined function $r$ on $\tropfont U$ are induced from those on $\tropfont X$, $\tropfont Y$, and $\tropfont V$ and their compatibility.  The metric $\ell$ is induced from the projection to $\tropfont X$.  We must verify that $r(u) = u$ if and only if $i(u) = u$ if and only if $\ell(u) = 0$ for all $u \in \tropfont U$.  Indeed, $r(u) = u$ if and only if $i(u) = u$ because this property holds in $\tropfont V$ and in $\tropfont X$.  If $\ell(u) = 0$ in $\mathscr U$ then by definition, $\ell(u) = 0$ in $\mathscr X$ and therefore in $\mathscr Y$ as well;  since $\mathscr V \to \mathscr Y$ is a local isomorphism, this implies $\ell(u) = 0$ in $\mathscr V$, so $r(u) = u$ in both $\mathscr X$ and in $\mathscr V$;  by definition, this implies $r(u) = u$ in $\mathscr V$.

To see that $\tropfont U \to \tropfont X$ is a local isomorphism, let $u$ be a vertex of $\tropfont U$ and denote by $x$, $y$, and $v$ its images in $\tropfont X$, $\tropfont Y$, and $\tropfont V$.  Let $\tropfont S_x$, $\tropfont S_y$, $\tropfont S_u$, and $\tropfont S_v$ denote their stars.  Then $\tropfont S_v$ bijects onto $\tropfont S_y$, by the assumption that $\tropfont V \to \tropfont Y$ is a local isomorphism.  Therefore $\sigma^{-1} \tropfont S_v$ maps isomoprhically onto $\sigma^{-1} \tropfont S_y$.  But $\tropfont S_u \subset \sigma^{-1} \tropfont S_v$ and $\tropfont S_x \subset \sigma^{-1} \tropfont S_y$, so that $\tropfont S_u$ maps isomorphically onto $\tropfont S_x$, as required.
\end{proof}

The lemma shows that if $\tropfont V \to \tropfont Y$ is a local isomorphism, then $\sigma^{-1} \tropfont V \to \tropfont X$ is also a local isomorphism.  It is immediate that $\sigma^{-1}$ respects fiber products and covers so we obtain a morphism of sites $\sigma : \tropfont X \to \tropfont Y$.

We construct the desired functor $\bTroPic(\tropfont X) \to \bTroPic(\tropfont Y)$ by working locally.  We write $\sigma_\ast \mathrm B \L_{\tropfont X}^\dagger$ for the substack of those $Q \in \sigma_\ast \mathrm B \L_{\tropfont X}(\tropfont V)$ such that $Q$ has bounded monodromy on $\sigma^{-1} \tropfont U$ for each local isomorphism $\tropfont U \to \tropfont Y$.  Note, however, that bounded monodromy is not a local condition in general.

\begin{proposition} \label{prop:tropic-func}
There is a morphism
\begin{equation} \label{eqn:51}
\sigma_\ast \mathrm B \L_{\tropfont X}^\dagger \to \mathrm B\L_{\tropfont Y}
\end{equation}
inducing the morphisms in Proposition~\ref{prop:generization}.
\end{proposition}
\begin{proof}
Provided we do so compatibly with restriction, it is sufficient to work locally in $\tropfont Y$.  We can therefore assume that $\tropfont Y$ is either a single edge or has a single vertex with a number of edges attached to it at only one side.  In the former case, $\tropfont X$ is also a single edge and $\sigma : \tropfont X \to \tropfont Y$ is an isomorphism, because $\sigma$ is an edge contraction.  We therefore assume that $\tropfont Y$ is a single vertex with edges radiating from it.  We note that in this case, $\tropfont Y$ has no nontrivial covers, so that we only need to construct~\eqref{eqn:51} on global sections:
\begin{equation} \label{eqn:89}
\bTroPic(\tropfont X) \to \bTroPic(\tropfont Y)
\end{equation}

Let us write $\L'$ for the sheaf of $\overnorm N$-valued linear functions on $\tropfont X$ and $\bTroPic(\tropfont X)'$ for the category of $\L'$-torsors on $\tropfont X$ of bounded monodromy.  Since every edge of $\tropfont X$ that is contracted by $\sigma$ has length~$0$ in $\overnorm N$ (by definition), $\L' = \sigma^\ast \L_{\tropfont Y}$.  In particular, $\L'$ is constant with value $\overnorm N^{\rm gp}$ on the preimage of the vertex of $\tropfont Y$.  The quotient $\L' / \overnorm N^{\rm gp}$ is a constant $\mathbf Z$ on the edges of $\tropfont X$ not contracted by $\sigma$ and therefore has vanishing $H^1$.  We may therefore make the identifications~\eqref{eqn:88}:
\begin{equation} \label{eqn:88}
H^1(\tropfont X, \L') = H^1(\tropfont X, \overnorm N^{\rm gp}) = \Hom(H_1(\tropfont X), \overnorm N^{\rm gp})
\end{equation}
But every loop of $\tropfont X$ has length~$0$ when measured in $\overnorm N$, so that a homomorphism $H_1(\tropfont X) \to \overnorm N^{\rm gp}$ of bounded monodromy must be~$0$.  Therefore $H^1(\tropfont X, \L')^\dagger = 0$ and $\bTroPic(\tropfont X)' = \mathrm B \Gamma(\tropfont X, \L')$.  Observing now that $\L'(\tropfont X) = \L(\tropfont Y)$, we conclude that $\bTroPic(\tropfont X)' = \bTroPic(\tropfont Y)$.  The sought after morphism~\eqref{eqn:89} now arises as the composition~\eqref{eqn:90}:
\begin{equation} \label{eqn:90} \vcenter{ \xymatrix@R=10pt{
\bTroPic(\tropfont X) \ar@{=}[d] \ar[r] & \bTroPic(\tropfont X)' \ar@{=}[d] & \bTroPic(\tropfont Y) \ar[l]_-{\sim} \ar@{=}[d] \\
\Gamma(\tropfont X, \mathrm B \L)^\dagger \ar[r] & \Gamma(\tropfont X, \mathrm B \sigma^\ast \L)^\dagger & \ar[l] \Gamma(\tropfont Y, \mathrm B \L)^\dagger
}} \end{equation}
\end{proof}

We leave it to the reader to verify that the morphism of Proposition~\ref{prop:tropic-func} is compatible with composition of homomorphisms of monoids.  We can now define the tropical Picard group in families, using the process described in Section~\ref{sec:trop-mod}.  If $\tropfont X$ is a family of tropical curves over a logarithmic scheme $S$, we obtain a stack $\bTroPic(\tropfont X/S)$ on the large \'etale site of $S$ characterized by either of the following two descriptions:
\begin{enumerate}
\item If $T$ is an atomic neighborhood of a geometric point $t$, then $\bTroPic(\tropfont X/S)(T) = \bTroPic(\tropfont X_t)$, and in general $\bTroPic(\mathscr X)$ is the stackification of the category fibered in groupoids arising from this definition.
\item If $T$ is a logarithmic scheme over $S$ and $T$ is of finite type then the objects of $\bTroPic(\tropfont X/S)(T)$ consists of a tropical line bundle $Q_t$ on $\tropfont X_t$ for each geometric point $t$ such that, for any geometric specialization $t \leadsto t'$, the line bundle $Q_{t'}$ induces $Q_t$ by way of the edge contraction $\tropfont X_{t'} \to \tropfont X_t$ and Proposition~\ref{prop:tropic-func}.  One extends to general logarithmic schemes using finite type approximations.
\end{enumerate}

\subsection{Prorepresentability and subdivisions}
\label{sec:trojac-subdiv}

Let $\tropfont X$ be a tropical curve metrized by a monoid $\overnorm M$.  We saw in Section~\ref{sec:trojac} that the tropical Jacobian can be regarded as a functor of pairs $(\overnorm N, u)$ where $u : \overnorm M \to \overnorm N$ is a homomorphism of monoids.  This functor is not representable, as we saw in Example~\ref{ex:trojac2}.  However, it is not that far from being representable:  it is the quotient of a prorepresentable functor by a discrete group.

\begin{proposition} \label{prop:prorep}
	The functor $\Hom(H_1(\tropfont X), \PL)^\dagger$ is prorepresentable on $\overnorm M / \Mon$ by the system of all submonoids $\overnorm P$ of $\overnorm M^{\rm gp} + H_1(\tropfont X)$ (direct sum) with the following properties:
	\begin{enumerate}
		\item $\overnorm P$ is finitely generated over $\overnorm M$;
		\item for each $\gamma \in H_1(\tropfont X)$ we have $\gamma \prec \ell(\gamma)$ in $\overnorm P$.
	\end{enumerate}
\end{proposition}
\begin{proof}
	Note that the second property implies that $\overnorm P$ generates $\overnorm M^{\rm gp} + H_1(\tropfont X)$ as a group.  Indeed, if $\gamma \prec \ell(\gamma)$ then $\gamma - n \ell(\gamma) \in \overnorm P$ for some integer $n$; as $\ell(\gamma) \in \overnorm M \subset \overnorm P$, this implies $\gamma \in \overnorm P$.

	Let $I$ be the diagram of all $\overnorm P$ with the indicated properties.  Let $F = \varinjlim_{\overnorm P \in I} \Hom(\overnorm P, -)$ be the pro-object they represent.  Certainly, if $\overnorm P \in I$ then a homomorphism $\overnorm P \to \overnorm N$ commuting with the morphisms from $\overnorm M$ induces an object of $\Hom(H_1(\tropfont X), \overnorm N)^\dagger$ by passing to the associated group.  This gives us a morphism $F \to \Hom(H_1(\tropfont X), \PL)^\dagger$ that we would like to show is an isomorphism.

	Suppose that $\mu : H_1(\tropfont X) \to \overnorm N^{\rm gp}$ is a homomorphism with bounded monodromy.  Combining this with the structural homomorphism $\overnorm M \to \overnorm N$ we get a homomorphism of monoids $\nu : \overnorm M + H_1(\tropfont X) \to \overnorm N^{\rm gp}$.  Choose a basis $e_1, \ldots, e_g$ of $H_1(\tropfont X)$.  For each $i$, there are integers $n$ and $m$ such that $n \nu(\ell(e_i)) \leq \nu(e_i) \leq m \nu(\ell(e_i))$ in $\overnorm N^{\rm gp}$.  That is $e_i - n \ell(e_i)$ and $m \ell(e_i) - e_i$ both lie in the preimage of $\overnorm N$ under $\nu$.  We take $\overnorm P$ to be the submonoid of $\overnorm M + H_1(\tropfont X)$ generated by $\overnorm M$ and the $e_i - n \ell(e_i)$ and $m \ell(e_i) - e_i$.  Then, by construction, $\overnorm P$ is finitely generated over $\overnorm M$, generates $\overnorm M + H_1(\tropfont X)$ as a group, has bounded monodromy, and induces $\mu$ via $\nu$. 

	This shows that $F \to \Hom(H_1(\tropfont X), \PL)^\dagger$ is surjective.  To see that it is also injective, consider a second map $\overnorm Q \to \overnorm N$ inducing $\mu$ as above, with $\overnorm Q \in I$.  Then $\overnorm Q \cap \overnorm P$ is also in $I$ and the map $\overnorm Q \cap \overnorm P \to \overnorm N$ induced from either $\overnorm Q \to \overnorm N$ or $\overnorm P \to \overnorm N$ --- they must be the same because the induced maps on associated groups is the same --- represents the same object of $F(\overnorm N)$.  This proves the injectivity and completes the proof.
\end{proof}

Let us now assume that $\overnorm M$ is finitely generated.  There is no loss of generality in doing so, since we only care about the set of lengths of the edges of $\tropfont X$, which is in any case a finitely generated submonoid.

It is then dual to a rational polyhedral cone $\sigma$, and the category of monoids that are finitely generated relative to $\overnorm M$ is contravariantly equivalent to the category $\RPC / \sigma$ of rational polyhedral cones over $\sigma$.  These observations permit us to reinterpret Proposition~\ref{prop:prorep} dually, to the effect that $\Hom(H_1(\tropfont X), \PL)^\dagger$ is representable by an ind-object of $\RPC / \sigma$.  

Rational polyhedral cones are finitely generated, saturated, convex regions in lattices, so we can interpret ind-rational polyhedral cones as not-necessarily-finitely generated, saturated, convex regions in torsion-free abelian groups.  Actually, Proposition~\ref{prop:prorep} gives a pro-object of $\overnorm M / \Mon$ whose associated group is constant, so that it is represented dually by a saturated, convex region in the lattice $\Hom(\overnorm M^{\rm gp}, \mathbf Z) \times H^1(\tropfont X)$.  The following corollary specifies which:

\begin{corollary} \label{cor:indrep}
The functor $\Hom(H_1(\tropfont X), \PL)^\dagger$ is ind-representable by the collection $\tau$ of pairs $(u,v) \in \Hom(\overnorm M^{\rm gp}, \mathbf Z) \times \Hom(H_1(\tropfont X), \mathbf Z)$ such that $u(\overnorm M) \geq 0$ and whenever $u(\ell(\gamma)) = 0$ for some $\gamma \in H_1(\tropfont X)$, we also have $v(\gamma) = 0$.
\end{corollary}
\begin{proof}
Let $I$ denote the pro-$\overnorm M$-monoid consisting of all $\overnorm P \subset \overnorm M^{\rm gp} \times H_1(\tropfont X)$ such that $\overnorm P$ is finitely generated over $\overnorm M$ and $\gamma \prec \ell(\gamma)$ in $\overnorm P$ for all $\gamma \in H_1(\tropfont X)$ (as in Proposition~\ref{prop:prorep}).  Let $J$ denote the ind-rational polyhedral cone consisting of all $(u,v)$ such that $u(\overnorm M) \geq 0$ and $u(\ell(\gamma)) = 0$ implies $v(\gamma) = 0$.  We wish to show that $I$ and $J$ are dual.

Since $I$ is closed under finite intersections and $J$ is closed under finite unions, it is sufficient to demonstrate the duality on the level of rays in $\Hom(\overnorm M^{\rm gp} \times H_1(\tropfont X), \mathbf Z)$ and the corresponding half-spaces in $\overnorm M^{\rm gp} \times H_1(\tropfont X)$.  That is, we need to show that $(u,v) \in \Hom(\overnorm M^{\rm gp} \times H_1(\tropfont X), \mathbf Z)$ has the properties $u(\overnorm M) \geq 0$ and $u(\ell(\gamma)) = 0$ implies $v(\gamma) = 0$ if and only if $\overnorm M \subset (u,v)^\vee$ and $\gamma \prec \ell(\gamma)$ in $(u,v)^\vee$.  But this is immediate:  $u(\overnorm M) \geq 0$ means precisely that $\overnorm M \subset (u,v)^\vee$; likewise, $\gamma \prec \ell(\gamma)$ in the half-space $(u,v)^\vee$ means either that $u(\ell(\gamma)) = v(\gamma) = 0$ or that $u(\ell(\gamma)) > 0$, which is equivalent to the property that $u(\ell(\gamma)) = 0$ implies $v(\gamma) = 0$.
\end{proof}

\begin{example}
Consider a loop $\tropfont X$ of circumference $\delta$, metrized by the monoid $\mathbf N \delta$.  Figure~\ref{fig:trojac} gives a visual representation of the ind-rational polyhedral cone representing $\Hom(H_1(\tropfont X), \ologGm)^\dagger$ on the left side, and of the pro-monoid representing it on the right.  On the left, the cone is the union of the origin and the strict right half plane; on the right, it is an infinitesimal thickening of the positive horizontal axis.

The real points of the tropical Jacobian can be seen by dividing the picture on the left by $H_1(\tropfont X) \simeq \mathbf Z$, which acts by vertical translation:  $(x,y) \mapsto (x,y+x)$.
\end{example}

\begin{figure}
\begin{tikzpicture}
\foreach \y in {-2,...,2} 
  \foreach \x in { -2,...,2 }
    \filldraw[color=black] (\x,\y) circle (2pt);
\begin{scope}
\clip (0,-2) rectangle (2,2);
\shade[inner color=red,outer color=white,opacity=.8] (0,0) circle (1.8);
\end{scope}
\filldraw[color=red] (0,0) circle (2pt);
\draw[thick,dashed,color=red] (0,-2.2) -- (0,2.2);
\end{tikzpicture}
\hskip.2\textwidth
\begin{tikzpicture}
\clip (-2.2,-2.2) rectangle (2.2,2.2);
\node at (1,.4) {$\delta$};
\foreach \y in {-2,...,2} 
  \foreach \x in { -2,...,2 }
    \filldraw[color=black] (\x,\y) circle (2pt);
\filldraw[color=red] (0,0) circle (2pt);
\begin{scope}
\clip (0,0) -- (2.5,.1) -- (2.5,-.1) -- cycle;
\shade[inner color=red,outer color=white,opacity=.8] (0,0) circle (2.2);
\end{scope}
\end{tikzpicture}
\caption{The universal cover of the tropical Jacobian of a loop of circumference $\delta \in \mathbf R_{\geq 0}$ (left) and the pro-monoid that represents it (right).}
\label{fig:trojac}
\end{figure}

The advantage of working with cones instead of monoids is that we can see subdivisions rather explicitly.

\begin{corollary} \label{cor:hom-subdiv}
Subdivisions of $\Hom(H_1(\tropfont X), \PL)^\dagger$ by representable functors are in bijection with subdivisions of the cone $\tau$ defined in Corollary~\ref{cor:indrep}.
\end{corollary}
\begin{proof}
This is entirely a matter of unwinding definitions.  Suppose first that $T$ is a subdivision of $\tau$ by rational polyhedral cones.  Then for every rational polyhedral cone $\sigma$ and morphism $\sigma \to \tau$, the fiber product $T \mathop\times_{\tau} \sigma$ is a subdivision of $\sigma$.  But $\tau$ represents $\Hom(H_1(\tropfont X), \PL)^\dagger$, so that if $h_T$ is the functor represented by $T$ then $h_T \to \Hom(H_1(\tropfont X), \PL)^\dagger$ is representable by subdivisions.

Suppose conversely that $h_T \to \Hom(H_1(\tropfont X), \PL)^\dagger$ is representable by subdivisions, where $T$ is a cone complex.  For any finitely generated subcone $\sigma$ of $\tau$, the fiber product $h_\sigma \mathop\times_{h_\tau} h_T$ is representable by a subdivision of $\sigma$.  It is immediate from this that $T$ is a subdivision of $\tau$.
\end{proof}

\begin{corollary} \label{cor:trojac-subdiv}
Subdivisions of $\TroJac(\tropfont X)$ by cone spaces \cite{cavalieri2017moduli} correspond to $H_1(\tropfont X)$-equivariant subdivisions of the cone $\tau$ of Corollary~\ref{cor:indrep}.
\end{corollary}
\begin{proof}
This is immediate, since subdivisions of $\TroJac(\tropfont X)$ are the same as $H_1(\tropfont X)$-equivariant subdivisions of $\Hom(H_1(\tropfont X), \PL)^\dagger$.
\end{proof}

\subsection{Boundedness of moduli}
\label{sec:trop-qcpt}

Our definition of boundedness is a natural adaptation to logarithmic schemes of the schematic notion~\cite[D\'efinition~1.1]{FGA4}.

\begin{definition} \label{def:bdd}
A moduli problem $F$ over logarithmic schemes over $S$ is said to be \emph{bounded} if, locally in $S$, there is a logarithmic scheme $T$ of finite type over $S$ and a morphism $T \to F$ that is surjective on valuative geometric points.
\end{definition}

In this section and the next, we will work with a bit more generally than necessary for the application to the tropical Jacobian.  The boundedness results we obtain apply to tropical abelian varieties as well.

\begin{definition} \label{def:pos-def}
Let $\partial : H \to \Hom(H, \overnorm M^{\rm gp})$ be a pairing on a finitely generated free abelian group $H$, valued in a partially ordered abelian group $\overnorm M^{\rm gp}$.  We say that $\partial$ is \emph{positive semidefinite} if, $\partial(e).e \geq 0$ for all $e \in H$ and $\partial(e).f \prec \partial(e).e$ for all $e, f \in H$.  We will call it \emph{positive definite} if $\partial(e).e = 0$ only for $e = 0$.  Accordingly, we refer to the quadratic function $\ell(f) = \partial(f).f$ as a \emph{positive semidefinite quadratic form} or a \emph{positive definite quadratic form}.
\end{definition}

\begin{remark} \label{rem:induced-posdef}
If $\partial$ is a positive semidefinite pairing on $H$, valued in $\overnorm M^{\rm gp}$, and $\psi : \overnorm M \to \overnorm N$ is a monoid homomorphism, then $\psi\partial$ is a positive semidefinite pairing valued in $\overnorm N^{\rm gp}$.  Moreover, if $\partial$ is positive semidefinite, then $\partial$ descends to a positive definite pairing on $H / \{ e \in H \:\big|\: \ell(e) = 0 \}$.  Combining these observations, suppose that $\partial$ is positive definite and let $H_\psi = \{ \gamma \in H \: \big| \: \psi(\ell(\gamma)) = 0 \}$.  Then $\partial$ descends to a positive definite pairing on $H/H_\psi$, valued in $\overnorm N^{\rm gp}$.
\end{remark}

\begin{theorem} \label{thm:posdef-bdd}
Let $\partial : H \to \Hom(H, \ologGm)$ be a positive definite bilinear pairing over a logarithmic scheme $S$ of finite type.  Then $\Hom(H, \ologGm)^\dagger / \partial H$ is bounded.
\end{theorem}
\numberwithin{equation}{theorem}
\begin{proof}
We write $\ell$ for the quadratic form $\ell(f) = \partial(f).f$.  The assertion of the theorem is local to the constructible and \'etale topologies on $S$, so we may assume $S$ has constant characteristic monoid.

\begin{sublemma} \label{lem:finite-A}
Let $\ell : H \to \overnorm M$ be a positive definite quadratic form, where $H$ is a lattice of finite rank and $\overnorm M$ is finitely generated.  There is a finite set $C \subset H$ such that $\mu \in \Hom(H, \ologGm)$ lies in $\Hom(H, \ologGm)^\dagger$ if and only if $\mu(f) \prec \ell(f)$ for all $f \in C$.  The set $C$ may be chosen to include a basis of $H_u$ (notation as in Remark~\ref{rem:induced-posdef}) for every monoid homomorphism $u : \overnorm M \to \overnorm N$.
\end{sublemma}
\begin{proof}
The subgroup $H_u$ depends only on the minimal localization homomorphism through which $u$ factors.  Since $\overnorm M$ is finitely generated, there ar only finitely many distinct localization homomorphisms $\psi : \overnorm M \to \overnorm N$.  For each of these localizations, let $H_\psi = \{ x \in H \: | \: \psi(\ell(x)) = 0 \}$.  Let $C_\psi$ be a finite set of generators of $H_\psi$, as an abelian group, and let $C = \bigcup C_\psi$.  Since every $H_u$ appears in this list, $C_u$ contains a basis of every $H_u$.

We will demonstrate that if $u : \overnorm M \to \overnorm N$ is a monoid homomorphism, and $\mu : H \to \overnorm N^{\rm gp}$ is a group homomorphism such that $\mu(x) \prec u(\ell(x))$ for all $x \in C$ then $\mu(x) \prec u(\ell(x))$ for all $x \in H$.  The condition $\mu(x) \prec u(\ell(x))$ is equivalent to the condition that $\mu(x)$ maps to $0$ when $\overnorm N$ is localized by $u(\ell(x))$.  We have a commutative square:
\begin{equation*} \xymatrix{
\overnorm M \ar[r]^u \ar[d]_\psi & \overnorm N \ar[d]^\varphi \\
\overnorm M[-\ell(x)]^\sharp \ar[r] & \overnorm N[-u(\ell(x))]^\sharp
} \end{equation*}
By definition of $H_\psi$, we have $\psi(\ell(H_\psi)) = 0$.  By assumption, we have $\mu(y) \prec \ell(y)$ for all $y \in C_\psi$, so $\varphi(\mu(H_\psi)) = 0$.  We certainly have $x \in H_\psi$, so we conclude $\varphi(\mu(x)) = 0$, as required.
\end{proof}

Choose $C$ as in Lemma~\ref{lem:finite-A} and enlarge it if necessary so that it generates $H$ as an abelian group.  For each pair of integers $m$ and $n$, let $Z_{m,n}$ be the set of all $\mu : H \to \overnorm M^{\rm gp}$ such that $m \ell(f) \leq \mu(f) \leq n \ell(f)$ for all $f \in C$.

\begin{sublemma} \label{lem:Z-bounded}
For all $m$ and $n$, the functor $Z_{m,n}$ is bounded.
\end{sublemma}
\begin{proof}
Let $\mathcal A = [\mathbf A^1/\mathbf G_m]$, with its toric logarithmic structure.  Note that $\mathcal A$ is the locus of $t \in \ologGm$ such that $t \geq 0$.  For each $f \in C$, we obtain a pair of maps $Z_{m,n} \to \mathcal A$:
\begin{gather*}
\alpha_f(\mu) = \mu(f) - m \ell(f) \\
\beta_f(\mu) = n \ell(f) - \mu(f)
\end{gather*}
Since $C$ generates $H$ as an abelian group, the tuple $(\alpha_f, \beta_f)_{f \in C} : Z_{m,n} \to (\mathcal A \times \mathcal A)^{|C|}$ is a monomorphism.  The image is cut out by the relations $\alpha_f(\mu) + \beta_f(\mu) = (n-m) \ell(f)$ and finitely many other equalities induced from the relations among the $f \in C$ determined by the group structure of $H$.  There are finitely many such relations, and each one imposes an open condition on $(\mathcal A \times \mathcal A)^{|C|}$, so $Z_{m,n}$ is representable by an open substack of $(\mathcal A \times \mathcal A)^{|C|}$, and in particular is bounded.
\end{proof}

\begin{subremark}
The proof of the lemma actually shows that $Z_{m,n}$ is representable by an Artin cone.
\end{subremark}

\numberwithin{equation}{theorem}

With $C$ as above, we choose $b \in \mathbf Z$ such that $-b \ell(f) \leq \partial(e) . f \leq b \ell(f)$ for all $e, f \in C$ (such a $b$ exists by the finiteness of $C$ and the definition of a positive definite pairing).  We then take $Z$ to be the set of $\mu \in \Hom(H, \ologGm)$ such that
\begin{equation} \label{eqn:46}
-r (bg + 1) \ell(f) \leq \mu(f) \leq r(bg + 1) \ell(f)
\end{equation}
is satisfied for all $f \in C$, with $r$ being the rank of the abelian group $\overnorm M^{\rm gp}$.  In other words, $Z = Z_{m,n}$ with $m = n = r(bg+1) \ell(f)$, so $Z$ is bounded by Lemma~\ref{lem:Z-bounded}.

To complete the proof of Theorem~\ref{thm:posdef-bdd}, we need to show that the valuative geometric points of some $Z$ surject onto those of $\Hom(H,\ologGm)^\dagger / \partial H$ under the projection $Z \subset \Hom(H, \ologGm)^\dagger \to \Hom(H,\ologGm)^\dagger / \partial H$.  We shall therefore assume that $\overnorm M_S$ is valuative.  Note that passing to a valuation of $\overnorm M_S$ does not change $\overnorm M_S^{\rm gp}$, so $\overnorm M_S^{\rm gp}$ is still finitely generated.  Let
\begin{equation*}
0 = \overnorm N_0 \subset \overnorm N_1 \subset \cdots \subset \overnorm N_k = \overnorm M_S
\end{equation*}
be the filtration guaranteed by Proposition~\ref{prop:arch-filt}.  It is finite because Lemma~\ref{lem:val-gp} implies that each $\overnorm N_i$ is determined by its associated subgroup of $\overnorm M_S^{\rm gp}$, and $\overnorm M_S^{\rm gp}$ is a finitely generated abelian group, hence noetherian.  In fact, the length,~$k$, of this filtration is bounded by the rank,~$r$, of $\overnorm M_S^{\rm gp}$.  For each $i$, let $H_i$ be the subgroup of $\gamma \in H$ such that $\ell(\gamma) \in \overnorm N_i$.

We now proceed by induction on the length of this filtration.  We argue that \emph{if $\mu$ is an element of $\Hom(H, \overnorm N_i^{\rm gp})$ with bounded monodromy then there is some $\gamma \in H$ and some $\zeta$ such that $-(g+1) \ell(f) \leq \zeta(f) \leq (g+1)\ell(f)$ for all $f \in C$, and $\mu - \zeta - \partial(\gamma)$ takes values in $\overnorm N_{i-1}$}.

By composition with the homomorphism $q : \overnorm N_i \to \overnorm N_i / \overnorm N_{i-1}$, we obtain a map
\begin{equation} \label{eqn:45}
\Hom(H, \overnorm N_i^{\rm gp})^\dagger \to \Hom(H, \overnorm N_i^{\rm gp} / \overnorm N_{i-1}^{\rm gp})^\dagger .
\end{equation}
The important point here is that if $\mu \in \Hom(H, \overnorm N_i^{\rm gp})$ has bounded monodromy then $q  \mu$ also has bounded monodromy, in the sense that $q \mu(\alpha) \prec q \ell(\alpha)$ where $\ell(\alpha) \in \overnorm M_S$ denotes the length of $\alpha$.  Let us write $\overnorm\mu$ for the image of $\mu$ in $\overnorm N_i^{\rm gp} / \overnorm N_{i-1}^{\rm gp}$, as well as $\overnorm\ell$ for the length function taking values in $\overnorm N_i^{\rm gp} / \overnorm N_{i-1}^{\rm gp}$, and $\overnorm\partial$ for the reduction of the pairing $\partial$ modulo $\overnorm N_{i-1}^{\rm gp}$.  Finally, let $\overnorm H = H_i / H_{i-1}$ and note that $\overnorm\partial$ and $\overnorm\ell$ are well-defined and positive definite on $\overnorm H$.

By Proposition~\ref{prop:arch-filt}, the totally ordered abelian group $\overnorm N_i^{\rm gp} / \overnorm N_{i-1}^{\rm gp}$ is archimedean, hence admits an order preserving inclusion in $\mathbf R$ by Theorem~\ref{thm:archimedean}.  Since $\overnorm\mu \in \Hom(\overnorm H, \mathbf R)$, and $\overnorm\partial(\overnorm H) \subset \Hom(H, \mathbf R)$ is a lattice (because $\partial$ is positive definite), and $C$ contains a set of generators of $\overnorm H$, it is possible to write $\overnorm\mu = \alpha + \partial(\gamma)$ for some $\gamma \in H$ and some $\alpha = \sum a_i \partial(e_i)$ with $0 \leq a_i \leq 1$ for all $i$, and with the $e_i \in C$.  Now, evaluating $\alpha$ on $f \in C$, we get
\begin{equation}
\alpha ( f ) = \sum_{i=1}^g a_i \partial(e_i) . f \label{eqn:91} 
\end{equation}
But we have $- b \ell(f) \leq \partial(e_i) . f \leq b \ell(f)$ for all $f \in C$, so we obtain
\begin{equation} \label{eqn:20}
- b g \overnorm\ell(f) \leq \alpha ( f ) \leq b g \overnorm\ell(f) \\
\end{equation}

Note now that $\alpha = \overnorm \mu - \partial(\gamma)$, which is in $\Hom(\overnorm H, \mathbf R)$ by construction, is actually in the image of $\Hom(\overnorm H, \overnorm N_i^{\rm gp} / \overnorm N_{i-1}^{\rm gp})$.  Using the fact that $\overnorm H$ is free, we can lift and extend $\alpha$ to some $\zeta \in \Hom(H, \overnorm N^{\rm gp}) = \Hom(H, \overnorm M_S^{\rm gp})$ such that $\zeta(H) \subset \overnorm N_i$ and $\zeta(H_{i-1}) = 0$.  This ensures that $\zeta$ lies in the bounded monodromy subgroup $\Hom(H, \overnorm N_i^{\rm gp})^\dagger \subset \Hom(H, \overnorm N_i^{\rm gp})$: 

\begin{sublemma}
The lift $\zeta$ chosen above lies in $\Hom(H, \overnorm N_i^{\rm gp})^\dagger$.
\end{sublemma}
\begin{proof}
If $\ell(f) \in \overnorm N_{i-1}$ we have $\zeta(f) = 0$ so certainly $\zeta(f)$ is bounded by \emph{all} of $\overnorm N$ and in particular by $\ell(f)$.  If $\ell(f) \not\in \overnorm N_{i-1}$ then Proposition~\ref{prop:arch-filt} implies that \emph{all} of $\overnorm N_i^{\rm gp}$, and in particular $\zeta(f)$, is bounded by $\ell(f)$.
\end{proof}

Suppose that $f \in C$.  The inequality~\eqref{eqn:20} lifts to
\begin{equation} \label{eqn:21}
- b g \ell(f) - u \leq \zeta(f) \leq b g \ell(f) + v
\end{equation}
for some $u, v \in \overnorm N_{i-1}$.  If $\zeta(f) \neq 0$ then $\overnorm\ell(f)$ is a positive element of $\overnorm N_i^{\rm gp} / \overnorm N_{i-1}^{\rm gp}$.  Both $u$ and $v$ lie in $\overnorm N_{i-1}$, so $u$ and $v$ are both dominated by $\ell(f)$ by Proposition~\ref{prop:arch-filt}.  In particular, $u \leq \ell(f)$ and $v \leq \ell(f)$.  Substituting this into~\eqref{eqn:21}, we obtain~\eqref{eqn:52},
\begin{equation} \label{eqn:52}
- (bg+1) \ell(f) \leq \zeta(f) \leq (bg+1) \ell(f)
\end{equation}
as desired.

We have now shown that $\mu - \partial(\gamma) - \zeta$ takes values in $\overnorm N_{i-1}^{\rm gp}$.  Repeating this process once for each of the~$k$ steps of the filtration~\ref{eqn:45}, we obtain $\mu - \sum \partial(\gamma_i) - \sum_{i = 1}^k \zeta_i = 0$.  Thus $\zeta = \sum \zeta_i$ represents $\mu$ in $\TroJac(\tropfont X/S)$ and, as each $\zeta_i$ satisfies~\eqref{eqn:52} and $k \leq r$, their sum satisfies~\eqref{eqn:46}, so $\zeta \in Z(S)$.
\end{proof}

\begin{corollary} \label{cor:trop-bdd}
If $\tropfont X$ is a compact tropical curve over $S$ then $\TroJac(\tropfont X/S)$ is bounded over~$S$.
\end{corollary}
\begin{proof}
  Let $\tropfont X$ be the tropicalization of $X$.  The assertion is local to the constructible topology and to the \'etale topology on $S$, so we can assume that the logarithmic structure on $S$ has constant characteristic monoid and that the dual graph of $X$ is also constant.  After these restrictions, we have the exact sequence~\eqref{eqn:19} by definition of the tropical Jacobian:
\begin{equation} \label{eqn:19}
0 \to H_1(\tropfont X) \xrightarrow{\partial} \Hom(H_1(\tropfont X), \ologGm)^\dagger \to \TroJac(\tropfont X/S) \to 0
\end{equation} 
Since $\partial$ requires only a finite number of elements of $\overnorm M_S$ to describe, we may assume that $\overnorm M_S^{\rm gp}$ is a finitely generated abelian group.  Since the intersection pairing of a tropical curve is positive definite, we may now apply Theorem~\ref{thm:posdef-bdd} to conclude.
\end{proof}

\begin{corollary} \label{cor:trojac-qcpt}
Let $\tropfont X$ be a compact tropical curve over $S$.  Then $\TroPic^d(\tropfont X/S)$ is quasicompact over~$S$ for all $d \in H_0(\tropfont X)$.
\end{corollary}
\begin{proof}
As $\TroPic^d(\tropfont X/S)$ is a torsor under $\TroPic^0(\tropfont X/S)$ by Corollary~\ref{cor:tropic}, it is sufficient to assume $d = 0$.  But $\TroPic^0(\tropfont X/S) = \TroJac(\tropfont X/S)$ by Corollary~\ref{cor:trojac-qcpt}, so the conclusion follows from Corollary~\ref{cor:trop-bdd}.
\end{proof}

\subsection{Boundedness of the diagonal}
\label{sec:trojac-bdd-diag}

The main point of this section is to demonstrate that the lattice defined by a positive definite matrix of real numbers is discrete and that this is also valid as the lattice varies in a tropical family.  We make use of the tropical topology defined in Section~\ref{sec:conetop}.

These results are also demonstrated by a different method as part of the proof of \cite[Proposition~4.5]{kajiwara_kato_nakayama_2008b}.  The proof appears in \cite[Lemma~5.2.7]{kajiwara_kato_nakayama_2008a} and \cite[Section~9.4]{kajiwara_kato_nakayama_2008b}.  Unlike the present proof, that proof does not rely on the tropical topology, but it ultimately comes down to a compactness argument, as this one does.

\begin{theorem} \label{thm:qcpt-diag}
Let $H$ be a finitely generated free abelian group, let $\overnorm M^{\rm gp}$ be a finitely generated, partially ordered abelian group, and let $\partial : H \to \Hom(H, \overnorm M^{\rm gp})$ be a positive definite pairing.  For each $\phi : H \to \overnorm M^{\rm gp}$ that is bounded by $\partial$, there are at most finitely many $\gamma \in H$ for which there exists a sharp homomorphism $\overnorm M \to \overnorm N$ with $\partial_{\overnorm N}(\gamma) = \phi_{\overnorm N}$.
\end{theorem}
\begin{proof}
As usual, we write $\ell : H \to \overnorm M^{\rm gp}$ for the quadratic form associated with $\partial$.  Since every monoid has a sharp homomorphism to a valuative monoid, we can assume that $\overnorm N$ is valuative.

Since $\ShpVal(\overnorm M)$ is quasicompact, it is sufficient to show that every $V \in \ShpVal(\overnorm M)$ has an open neighborhood $U$ such that there are only finitely elements in $H$ that represent $\phi_W$ for any valuation $\overnorm M \to W$ lying in $U$.  We fix one $V \in \ShpVal(\overnorm M)$.  Beginning with $U = \ShpVal(\overnorm M)$ we will replace $U$ by a smaller open neighborhood of $V$ finitely many times until we arrive at a neighborhood where we can be sure $\phi$ has only finitely many representatives.

Since the underlying abelian group of $V$ is finitely generated, $V$ has a finite filtration by totally ordered subgroups $V_p$ such that $V_p / V_{p-1}$ are all archimedean; we choose embeddings $V_p / V_{p-1} \subset \mathbf R$.  These, together with $\partial$, induce a filtration of $H$ and positive definite pairings $\partial_p$ on $H_p / H_{p-1}$, valued in $V_p / V_{p-1}$.

The proof will be by induction on the index $p$ of the subgroup $V_p$ in which $\partial_V$ and $\phi_V$ take their values.  Assume that $\phi_V$ and $\partial_V$ both take values in $V_p \subset V$.  Since $\phi_V$ is bounded by $\partial$, it descends to a map $\phi_p : H / H_{p-1} \to V_p / V_{p-1} \subset \mathbf R$.  

\begin{lemma}
There is an open neighborhood of $V$ over which the representatives of $\phi$ in $H$ lie in at most finitely many distinct cosets of $H_{p-1}$.
\end{lemma}
\begin{proof} 
Choose a subdivision of $\mathbf R H_p / \mathbf R H_{p-1}$ into a finite number of rational polyhedral cones $\sigma$ such that $\partial_p(\bar\beta).\bar\gamma > 0$ whenever $\bar\beta$ and $\bar\gamma$ are elements of $H_p / H_{p-1}$ lying in the same cone $\sigma$.  This means that if $\beta$ and $\gamma$ are lifts of $\bar\beta$ and $\bar\gamma$ to $H_p$ then $\partial_V(\beta).\gamma \succ V_{p-1}$.  In particular, $\partial_V(\beta).\gamma > 0$ for all $\beta, \gamma \in H_p - H_{p-1}$ and, for each $\gamma \in H_p - H_{p-1}$, there is a positive $n$ (depending on $\gamma$) such that $n \ell_V(\gamma) > \phi_V(\gamma)$.

Choose a finite set of generators $B_\sigma$ in $H_p$ for each $\sigma$.  There is an open neighborhood $U_\sigma \subset U$ of $V$ such that $\partial_{U_\sigma}(\beta).\gamma > 0$ and $n \ell_{U_\sigma}(\gamma) > \phi_{U_\sigma}(\gamma)$ for all $\beta, \gamma \in B_\sigma$.  Since there are only finitely many cones $\sigma$ and finitely many generators in each $B_\sigma$, the positive integer $n$ can be chosen independent of $\sigma$, $\beta$, and $\gamma$.  Replacing $U$ by the intersection of the (finitely many) $U_\sigma$, we can assume that these inequalities hold on $U$.

Now suppose $\gamma \in H$ is a putative representative of $\gamma$.  The reduction of $\gamma$ modulo $H_{p-1}$ lies in some cone $\sigma$.  We can therefore write $\gamma \equiv \sum_{\beta \in B_\sigma} a_\beta \beta \pmod{H_{p-1}}$ with all $a_\beta \geq 0$.  Evaluating on $\beta \in B_\sigma$, we have the following inequality over $U$:
\begin{equation*}
\partial_U(\gamma).\beta = \sum_{\beta' \in B_\sigma} a_{\beta'} \partial_U(\beta).\beta' > a_\beta \ell_U(\beta)
\end{equation*}
Since the $a_\beta$ are all positive integers, and $n \ell_U(\beta) > \phi_U(\beta)$, we will have $\partial_U(\gamma).\beta > \phi_U(\beta)$ if $a_\beta > n$.  In particular, we deduce that, for each $\sigma$, there are at most finitely many possibilities for the $a_\beta$ if $\partial(\gamma)$ is to have a chance of representing $\phi$ anywhere in $U$.  Since there are only finitely many cones $\sigma$, we conclude.
\end{proof}

Suppose $\gamma$ represents one of the cosets invoked in the lemma.  Then the representatives of $\phi$ in $\gamma + H_{p-1}$ correspond bijectively to the representatives of $\phi - \partial(\gamma)$ in $H_{p-1}$.  Replacing $\phi$ successively by $\phi - \partial(\gamma)$ for representatives $\gamma$ of each of the finitely many cosets guaranteed by the lemma, it therefore suffices to show that $\phi$ has at most finitely many representatives in $H_{p-1}$ in a neighborhood of $V$.  

We can now replace $H$ with $H_{p-1}$ and $\phi$ and $\partial$ with their restrictions to $H_{p-1}$.  If there are finitely many potential representatives of $\phi \big|_{H_{p-1}}$ in $H_{p-1}$ in a neighborhood of $V$, then of course there will be finitely many potential representatives of $\phi$ in $H_{p-1}$ as well.  With this reduction, both $\phi_V$ and $\partial_V$ take values in $V_{p-1}$ and we can induct.
\end{proof}

\begin{corollary} \label{cor:pairing-bdd}
Let $H$ be a finitely generated abelian group, let $\overnorm M$ be a sharp monoid, and let $\partial : H \to \Hom(H, \overnorm M^{\rm gp})$ be a positive definite pairing.  Then $H \to \Hom(H, \ologGm)$ is of finite type and affine.
\end{corollary}
\begin{proof}
The assertion is, in other words, that for any logarithmic scheme of finite type and any morphism $\phi : S \to \Hom(H, \ologGm)$, the fiber product $H \mathop\times_{\Hom(H, \ologGm)} S$ is quasicompact.  This assertion is local in the constructible and \'etale topologies on $S$, so we can assume $S$ is connected and has a constant logarithmic structure, the stalk of whose characteristic monoid is $\overnorm M$.  We therefore regard $\phi$ as a homomorphism $H \to \overnorm M^{\rm gp}$.

By Theorem~\ref{thm:qcpt-diag}, there are only finitely many $\gamma \in H$ such that there is a nonempty subfunctor of that represented by $S$ in which $\partial(\gamma)$ represents $\phi$.  These are mutually disjoint, and treating these one at a time, it is sufficient to show that the subfunctor of $S$ in which a fixed $\partial(\gamma)$ represents $\phi$ is of finite type and affine.

This functor is described by the finitely many equations $\partial(\gamma).e_i = \phi(e_i)$, for $e_i$ lying in a basis of $H$.  By Proposition~\ref{prop:ologGm-zero}, this locus is representable by an affine scheme of finite type.
\end{proof}

\begin{corollary} \label{cor:trop-diag-qcpt}
Let $\tropfont X$ be a tropical curve metrized by $\overnorm M$.  Then the intersection pairing
\begin{equation} \label{eqn:68}
\partial : H_1(\tropfont X) \to \Hom(H_1(\tropfont X), \ologGm) 
\end{equation}
is bounded.
\end{corollary}
\begin{proof}
The intersection pairing is positive definite.
\end{proof}

\numberwithin{theorem}{subsection}
\section{The logarithmic Picard group}
\label{sec:logpic}

Suppose that $X$ is a proper, vertical logarithmic curve over $S$ where the underlying scheme of $S$ is the spectrum of an algebraically closed field, and let $\tropfont X$ be the tropicalization of $X$.  Then $H^1(X, \ologGm) = H^1(X, \overnorm M_X^{\rm gp}) = H^1(\tropfont X, \PL)$ because $\overnorm M_X^{\rm gp}$ is a sheaf of torsion-free abelian groups.  If $\overnorm Q$ is an $\overnorm M_X^{\rm gp}$-torsor on $X$ then we say $\overnorm Q$ has \emph{bounded monodromy} if the corresponding $\PL$-torsor on $\tropfont X$ does.  If $Q$ is an $\logGm$-torsor on $X$ then we say $Q$ has bounded monodromy if its induced $\overnorm M_X^{\rm gp}$-torsor has bounded monodromy.

\numberwithin{theorem}{section}
\begin{definition} \label{def:logpic}
Let $X$ be a proper, vertical logarithmic curve over $S$.  A \emph{logarithmic line bundle} on $X$ is a $\logGm$-torsor on $X$ in the strict \'etale topology whose fibers over $S$ have bounded monodromy.  Let $\bLogPic(X/S)$ be the category fibered in groupoids on logarithmic schemes over $S$ whose $T$-points are the logarithmic line bundles on $X_T$.  We write $\LogPic(X/S)$ for its associated sheaf of isomorphism classes.
\end{definition}

\setcounter{subsection}{\value{theorem}}
\numberwithin{theorem}{subsection}

\subsection{Local finite presentation}
\label{sec:lfp}

\begin{definition} \label{def:lfp}
A category fibered in groupoids $F$ on the category of logarithmic schemes over $S$ is said to be \emph{locally of finite presentation} if, for any cofiltered system of affine logarithmic $S$-schemes $S_i$, the map
\begin{equation*}
\varinjlim F(S_i) \to F(\varprojlim S_i)
\end{equation*}
is an equvialence of categories.
\end{definition}

The definition of local finite presentation should be compared with Lemma \ref{lem:lfp-morphism}. Local finite presentation is important because it allows us to limit our attention to logarithmic schemes of finite type.

\begin{proposition} \label{prop:lfp}
Suppose $X$ is a proper, vertical logarithmic curve over $S$.  Then $\bLogPic(X/S)$ is locally of finite presentation over $S$.
\end{proposition}
\begin{proof} 

We prove the essential surjectivity part of Definition~\ref{def:lfp} for the functor $\pi_\ast \mathrm B\logGm$.  The full faithfulness is similar but easier, and we omit it.  Then we prove that the bounded monodromy condition defining $\bLogPic(X/S)$ inside $\pi_\ast \mathrm B\logGm$ is locally of finite presentation.

The assertion is local in $S$, so we assume $S$ is quasicompact and quasiseparated.  Consider a cofiltered inverse system of affine logarithmic schemes $S_i$ over $S$.  Let $X_i$ be the base change of $X$ to $S_i$.  Let $L$ be a logarithmic line bundle over $Y = \varprojlim X_i$.  Then there is an \'etale cover $U_j$ of $Y$ over which $L$ can be trivialized.  We can assume that the $U_j$ are all quasicompact and quasiseparated.  We note that $Y$ is quasicompact and quasiseparted because all the $X_i$ were.  In particular, we can arrange for the $U_j$ to be finite in number.  By \cite[Th\'eor\`eme~IV.8.8.2]{EGA}, they are induced by maps $U_{ij} \to X_i$ for some index $i$.  These maps can be assumed \'etale by \cite[Proposition~IV.17.7.8]{EGA} and surjective by \cite[Th\'eor\`eme~IV.8.10.5]{EGA}.

The transition functions defining $L$ come from $\Gamma(U_{jk}, M_Y^{\rm gp}) = \varinjlim_i \Gamma(U_{ijk}, M_{X_i}^{\rm gp})$, so are induced from transition functions over $U_{ijk}$ for some sufficiently large $i$.  Likewise, the cocycle condition is checked in $\Gamma(U_{jk\ell}, M_Y^{\rm gp}) = \varinjlim_i \Gamma(U_{ijk\ell}, M_{X_i}^{\rm gp})$ and is therefore valid for a sufficiently large $i$.  Then $L$ is induced from $X_i$. 

It remains to verify that the bounded monodromy condition is locally of finite presentation.  That is, we assume that we have a cofiltered inverse system of affine logarithmic schemes $S_i$ over $S$, as before, and that $\alpha_i \in H^1(X_i, \overnorm M_{X_i}^{\rm gp})$.  We assume that their limit $\beta \in H^1(Y, \overnorm M_Y^{\rm gp})$ has bounded monodromy and we prove the same for a sufficiently large $\alpha_i$.

There is a finite stratification of $S$ into locally closed subschemes such that $\overnorm M_S$ is locally constant on each stratum.  Since the bounded monodromy condition is checked on geometric points, we can replace $S$ with one of its strata and assume $\overnorm M_S$ is constant.  Now replacing $S$ by an \'etale cover, we can assume $\overnorm M_S$ is constant and that the dual graph $\tropfont X$ of $X$ is constant as well.

Using the exact sequence~\eqref{eqn:57}
\begin{equation} \label{eqn:57}
R^1 \pi_\ast \pi^\ast \overnorm M_S^{\rm gp} \to R^1 \pi_\ast \overnorm M_X \to R^1 \pi_\ast \overnorm M_{X/S} = 0
\end{equation}
we can lift $\alpha$ to $\tilde \alpha \in H^1(X, \pi^\ast \overnorm M_S^{\rm gp}) = \Hom(H_1(\tropfont X), \overnorm M_S^{\rm gp})$.  The bounded monodromy condition for $\tilde\alpha$ can be checked by evaluating it on each of the finitely many generators of $H_1(\tropfont X)$, and for any one $\gamma$ in $H_1(\tropfont X)$, we can see that $\tilde\alpha(\gamma)$ is bounded by $\ell(\gamma)$ in $\varinjlim \Gamma(S_i, \overnorm M_{S_i}^{\rm gp})$ if and only if it is bounded in $\Gamma(S_i, \overnorm M_{S_i}^{\rm gp})$ at some finite stage.  This completes the proof.
\end{proof}

\begin{corollary} \label{cor:lfp}
Suppose $X$ is a proper, vertical logarithmic curve over $S$.  Then the sheaf $\LogPic(X/S)$ is locally of finite presentation over $S$.
\end{corollary}
\begin{proof}
We can assume without loss of generality that $X$ has connected fibers over $S$.  Then $\bLogPic(X/S)$ is a gerbe over $\LogPic(X/S)$ banded by $\logGm$.  Locally in $S$, this gerbe admits a section, making $\LogPic(X/S)$ into a $\logGm$-torsor over $\bLogPic(X/S)$.  But $\logGm$ is certainly locally of finite presentation and $\bLogPic(X/S)$ is locally of finite presentation over $S$ by Proposition~\ref{prop:lfp}, so $\LogPic(X/S)$ is locally of finite presentation over $S$, as required.
\end{proof}

\subsection{Line bundles on subdivisions}
\label{sec:linesub}

The following statement is a corollary of Proposition~\ref{prop:bounded}.  It says, effectively, that logarithmic line bundles can be represented by line bundles locally in the logarithmic \'etale topology.

\begin{corollary} \label{cor:bdd-nbhd}
Let $X$ be a proper, vertical logarithmic curve over a logarithmic scheme $S$.  A class $\alpha \in H^1(X, \overnorm M_X^{\rm gp})$ has bounded monodromy in the geometric fibers if and only if, \'etale-locally in $S$, we can find a logarithmic modification $\tilde S \to S$ and a model $\tilde X$ of $X$ over $\tilde S$ such that $\alpha \big|_{\tilde X} = 0$.
\end{corollary}
\begin{proof}
Replacing $S$ by an \'etale cover, we can assume $S$ is affine.  We can then assume $S$ is of finite type because the moduli space of logarithmic curves is locally of finite presentation and $\bLogPic(X/S)$ is locally of finite presentation (Proposition~\ref{prop:lfp}).  Passing to a finer \'etale cover if necessary, we can arrange for $S$ to be atomic (Proposition~\ref{prop:atomic}) and for the dual graph of $X$ to be constant over the closed stratum.  In particular, $S$ is quasicompact.

Let $S^{\rm val}$ be the limit over all logarithmic modifications of $S$.  This is a locally ringed space and its logarithmic structure is valuative.  By Proposition~\ref{prop:bounded}, for each point $s$ of $S^{\rm val}$, we can find a subdivision $\tropfont Y_s$ of $\tropfont X_s$ to which the restriction of $\alpha$ is zero.  If $Y_s$ denotes the corresponding logarithmic modification of $X_s$ then $\alpha$ restricts to $0$ on $Y_s$.

The subdivision $\tropfont Y_s$ only requires a finite number of elements of $\overnorm M_{S^{\rm val},s}$ that are not already in $\overnorm M_{S,s}$, so it is possible to recover $\tropfont Y_s$ and $Y_s$ as pulled back from a logarithmic modification $Y_1$ of $X$ over a logarithmic modification $S_1$ of $S$.  Moreover, there is an open neighborhood $U_1$ of $s$ in $S_1$ where $\alpha \big|_{Y_1 \mathop\times_{S_1} U_1} = 0$.

Since $S^{\rm val}$ is quasicompact, the preimages of finitely many of these open neighborhoods $U_i$ suffice to cover $S^{\rm val}$.  Let $T$ be the fiber product of the finitely many logarithmic modifications $S_i$ of $S$.  Let $\tropfont Z$ and $Z$ be the common subdivision of the $\tropfont Y_i \big|_T$ and the corresponding logarithmic modification of $X$ over $T$, respectively.  Then the $U_i$ pull back to an open cover of $T$, from which it follows that $\alpha \big|_Z = 0$.
\end{proof}

\begin{proposition} \label{prop:logpic-bdd-mono}
Let $X$ be a logarithmic curve over $S$.  The following conditions are equivalent for a $\logGm$-torsor $P$ on $X$:
\begin{enumerate}
\item $P$ has bounded monodromy.
\item In each valuative geometric fiber of $S$, there is a model $Y$ of $X$ where $P$ is induced from a $\mathcal O_Y^\ast$-torsor.
\item \'Etale-locally in $S$ there is a logarithmic modification $\tilde S \to S$ and a model $\tilde X$ of $X$ over $\tilde S$ such that the restriction of $P$ to $\tilde X$ is representable by a $\mathcal O_X^\ast$-torsor.
\end{enumerate}
\end{proposition}
\numberwithin{equation}{theorem}
\begin{proof}
From the exact sequence~\eqref{eqn:27},
\begin{equation} \label{eqn:27}
H^1(X, \mathcal O_X^\ast) \to H^1(X, M_X^{\rm gp}) \to H^1(X, \overnorm M_X^{\rm gp})
\end{equation}
to find a $Y \to X$ where $P$ is representable by a $\mathcal O_Y^\ast$-torsor is equivalent to finding a cover where the class of $P$ in $H^1(X, \overnorm M_X^{\rm gp})$ is trivial.  With this observation, the equivalence of the first two conditions is Proposition~\ref{prop:bounded} and the equivalence of the first and last conditions is Corollary~\ref{cor:bdd-nbhd}.
\end{proof}
\counterwithout{equation}{theorem}

\subsection{Logarithmic \'etale descent}

By definition, $\bLogPic(X/S)$ is a stack in the \'etale topology.  We show here that, in certain situations, it is in fact a stack in the \emph{logarithmic} \'etale topology.  As the logarithmic \'etale topology is generated by \'etale covers, logarithmic modifications, and root stack constructions~\cite[Lemma~3.11]{nakayama2017logarithmic} (prime to the characteristic) we still need to check descent along logarithmic modifications and root stacks.  Descent along logarithmic modifications (not necessarily prime to the characteristic) was proved by K.\ Kato~\cite{kato2019logarithmic}, as we summarize below, so the main topic of this section will be descent along logarithmic modifications.

\begin{theorem} \label{thm:log-mod-inv}
The fibered category of \'etale $\logGm$-torsors is a stack in the Kummer logarithmic flat topology on logarithmic schemes.  It is a stack in the full logarithmic \'etale topology on logarithmic schemes whose structure sheaves are sheaves in the logarithmic \'etale topology.
\end{theorem}

\begin{remark}
We note that under either logarithmic modifications or root stacks, the cohomology groups of $H^i(X,\overnorm M_X^{\rm gp})$ of the characteristic monoid and the Picard group of $X$ change. Theorem \ref{thm:log-mod-inv} asserts that nevertheless, for a sufficiently well behaved $X$, these groups change in the same way, and thus the cohomology groups $H^i(X,M_X^{\rm gp})$ of the logarithmic structure itself remain constant. 
\end{remark}

\begin{remark}
	It is not true in general, as was claimed in an earlier draft of this paper, that $\bLogPic(X/S)$ is a stack in the large full logarithmic \'etale topology on $S$, nor that it is a stack on the small logarithmic \'etale topology for all bases $S$.  Nakayama has given an example of a logarithimic \'etale cover of a logarithmic scheme $S$ with respect to which $\Gm$ does not form a separated presheaf \cite[Remark following Proposition~2.6]{Nakayama}.  Since one can certainly find logarithmic curves $X$ over $S$ whose logarithmic Jacobians contain $\Gm$ as a subgroup, it follows that $\LogPic(X/S)$ cannot be a sheaf in the (small) full logarithmic \'etale topology on this $S$.  Since $\bLogPic(X/S)$ is, \'etale-locally in $S$, a product of $\LogPic(X/S)$ with $\mathrm B \logGm$, it follows as well that $\bLogPic(X/S)$ is not a stack in the (small) full logaithmic \'etale topology on this $S$.
\end{remark}

We will give the full proof of Theorem~\ref{thm:log-mod-inv} elsewhere.  For right now, we will only demonstrate the few cases that we need right now.  We require the following definition to formulate the hypotheses of those cases:

\begin{definition} \label{def:saturated}
	Let $\overnorm M$ be a finitely generated, saturated, sharp monoid.  An element $\alpha \in \overnorm M^{\rm gp}$ will be called \emph{saturated} (with respect to $\overnorm M$) if $\alpha$ generates the associated group of every rank~$1$ localization of $\overnorm M$ in which it is nonzero.  If $X$ is a logarithmic scheme, then a section of $\overnorm M_X^{\rm gp}$ is saturated if its image in $\overnorm M_{X,x}^{\rm gp}$ is saturated with respect to $\overnorm M_{X,x}$ for every geometric point $x$ of $X$.
\end{definition}

In other words, $\alpha$ is saturated if $\overnorm N / \mathbf Z \alpha$ is torsion-free for every rank~$1$ localization $\overnorm N$ of $\overnorm M$.

Here are the versions of Theorem~\ref{thm:log-mod-inv} that we will prove here:

\begin{proposition} \label{prop:fppf}
	Let $X$ be a logarithmic scheme.  Every $M_X^{\rm gp}$-torsor in the strict fppf topology descends uniquely to an \'etale $M_X^{\rm gp}$-torsor.  In particular, $M_X^{\rm gp}$ is an fppf sheaf.
\end{proposition}

\begin{proposition} \label{prop:root}
Let $X$ be a logarithmic scheme and let $\tau : Y \to X$ be a root stack.  Then every \'etale $M_Y^{\rm gp}$-torsor descends uniquely to an $M_X^{\rm gp}$-torsor.  In particular, $\tau_\ast M_Y^{\rm gp} = M_X^{\rm gp}$.
\end{proposition}

\begin{proposition} \label{prop:sat-hyp}
Let $X$ be a logarithmic scheme and let $\tau : Y \to X$ be a logarithmic modification.  Assume $\tau$ is, \'etale-locally in $X$, the subdivision associated with a saturated section of $\overnorm M_X^{\rm gp}$.  Then every \'etale $M_Y^{\rm gp}$-torsor on $Y$ is \'etale-locally trivial on $X$.
\end{proposition}

Proposition~\ref{prop:sat-hyp} will suffice for the applications in this paper.  It is also the main technical piece of the full proof of Theorem~\ref{thm:log-mod-inv}.  The following proposition is not used in this paper, but will be needed for some applications and can be proved more quickly from Proposition~\ref{prop:sat-hyp} than Theorem~\ref{thm:log-mod-inv} can. 

\begin{proposition} \label{prop:log-flat}
	Let $X$ be a logarithmic scheme and let $\tau : Y \to X$ be a logarithmic modification.  Assume that $X$ is logarithmically flat over a valuative logarithmic scheme.  Then every \'etale $M_Y^{\rm gp}$-torsor on $Y$ descends to an $M_X^{\rm gp}$-torsor on $X$.  In particular, $\tau_\ast M_Y^{\rm gp} = M_X^{\rm gp}$.
\end{proposition}

\setcounter{subsubsection}{\value{theorem}}
\subsubsection{Proof of Proposition~\ref{prop:fppf}.}

An $M_X^{\rm gp}$-torsor is a $\mathcal O_X^\ast$-torsor over an $\overnorm M_X^{\rm gp}$-torsor.  We can view $\overnorm M_X^{\rm gp}$-torsors on the strict fppf site as algebraic spaces via the espace \'etal\'e, so these satisfy strict fppf descent, and $\mathcal O_X^\ast$-torsors satisfy fppf descent by Hilbert's Theorem 90.

\subsubsection{Proof of Proposition~\ref{prop:root}.}

This is really a theorem of K.\ Kato.  We merely summarize how to deduce it from various statements in \cite{kato2019logarithmic}.

It is shown in \cite[Theorem~3.2]{kato2019logarithmic} that $\logGm$ is a sheaf in the Kummer logarithmic flat topology and in \cite[Corollary~5.2]{kato2019logarithmic} that $\logGm$-torsors in the Kummer logarithmic flat topology are \'etale-locally trivial.  Thus any descent datum for a $\logGm$-torsor in the Kummer logarithmic flat topology descends to a descent datum in the strict \'etale topology, in which it is effective by definition.  Therefore $\mathrm B \logGm$ is a stack in the Kummer logarithmic flat topology.

\numberwithin{theorem}{subsubsection}
\subsubsection{Observations about saturated sections of the characteristic monoid}

The idea of Definition~\ref{def:saturated} is that if $X$ is a logarithmic scheme and $\alpha \in \Gamma(X, \overnorm M_X^{\rm gp})$ then $\alpha$ determines a morphism $X \to \tropGm$, and we should think of this as a saturated morphism if $\alpha$ satisfies the condition of the definition.

\begin{lemma} \label{lem:saturated}
Let $\overnorm M$ be a fine and saturated monoid.  Suppose $\alpha \in \overnorm M^{\rm gp}$ is saturated.  Then then $\overnorm M + \mathbf N\alpha \subset \overnorm M^{\rm gp}$ is saturated.
\end{lemma}
\begin{proof}
Let $\sigma$ be the rational polyhedral cone dual to $\overnorm M$.  Let $\tau \subset \sigma$ be the subcone where $\alpha \geq 0$.  Let $\overnorm N \subset \overnorm M^{\rm gp}$ be the set of $\gamma \in \overnorm M^{\rm gp}$ such that $\gamma(\tau) \geq 0$.  Then $\overnorm N$ is the saturation of $\overnorm M + \mathbf N\alpha$.  We wish to show that $\overnorm N = \overnorm M + \mathbf N\alpha$.

	Suppose that $\gamma \in \overnorm N$.  We can write $\gamma$ as $c \alpha + \beta$ where $\beta \in \mathbf Q_{\geq 0} \overnorm M$ and $c \in \mathbf Q_{\geq 0}$.  We argue that $c$ must be an integer and $\beta$ must be in $\overnorm M$.  If $c \geq 1$ then $(c-1) \alpha + \beta$ is also $\geq 0$ on $\xi$ and is in $\overnorm M^{\rm gp}$.  We therefore assume that $0 \leq c < 1$ and will show that $\gamma \in \overnorm M$.

	Suppose that $\xi$ is any ray of $\sigma$.  Since $\beta \in \overnorm M$, the slope of $\beta$ on $\xi$ must be $\geq 0$.  The slope of $\alpha$ on $\xi$ is $\geq -1$ because $\alpha$ is saturated, and $c < 1$, so the slope of $\gamma$ on $\xi$ must be $> -1$.  On the other hand, $\gamma$ must have integral slope on $\xi$, so this slope must be $\geq 0$.  This holds for all rays $\xi$ of $\sigma$, so we deduce that $\gamma(\sigma) \geq 0$.  That is $\gamma \in \overnorm M$.
\end{proof}

\begin{corollary} \label{cor:A1-fiber}
Suppose that $X$ is a logarithmic scheme whose underlying scheme is the spectrum of a field.  Let $\alpha$ be a section of $\overnorm M_X^{\rm gp}$.  Assume that neither $\alpha$ nor $-\alpha$ lies in $\overnorm M_X$.  Then there is a universal scheme with a fine logarithmic structure, $Y$, over $X$ such that the restriction of $\alpha$ to $Y$ lies in $\overnorm M_Y$ and the underlying scheme of $Y$ is isomorphic to $\mathbf A^1_X$.  If $\alpha$ is saturated then $Y$ is saturated.
\end{corollary}
\begin{proof}
Let $Y$ have underlying scheme $\Spec \Sym \mathcal O_X(-\alpha)$.  Write $\tau : Y \to X$ for the projection.  Let $M_Y$ be the submonoid of $\tau^\ast(M_X)^{\rm gp}$ (the associated group of the pullback logarithmic structure) generated by $\tau^\ast M_X$ and $\mathcal O_Y^\ast(-\alpha) \subset \tau^\ast(M_X)^{\rm gp}$.  We have a homomorphism $\varepsilon_Y : M_Y \to \mathcal O_Y$ by the following formula, for all local choices of $a \in \mathcal O_Y^\ast(-\alpha)$ and all $b \in \tau^\ast M_X$:
\begin{equation} \label{eqn:132}
\varepsilon_Y(a^n b) = a^n \varepsilon_X(b)
\end{equation}
Note that if $a^n b \in M_Y$ has a second local representation as $c^m d$ for $c \in \mathcal O_Y^\ast(-\alpha)$ and $d \in \tau^\ast M_X$ then $b$ and $d$ are either both units of $\tau^\ast M_X$ or are both nonunits.%
\footnote{%
	Indeed, write $\beta$ and $\delta$ for the images of $b$ and $d$ in $\overnorm M_X$.  Then $a^n b = c^m d$ entails $n \alpha + \beta = m \alpha + \delta$.  But $\mathbf Z \alpha \cap \overnorm M_X = 0$ since $\overnorm M_X$ is saturated and neither $\alpha$ nor $-\alpha$ is in $\overnorm M_X$.  Thus if one of $\beta$ or $\delta$ is zero --- say $\delta$ --- then $(m-n)\alpha = \beta$, so $n=m$ and $\beta = 0$ as well.  So either $\beta$ and $\delta$ are both zero, or are both nonzero.  If $\beta$ and $\delta$ are both zero then $b$ and $d$ are units, and if they are both nonzero then $\varepsilon_X(b) = \varepsilon_X(d) = 0$.%
}
If $b$ and $d$ are both units then $a^n b = c^m d$ lies in $\bigcup_{n \geq 0} \mathcal O_Y^\ast(-n \alpha)$, and $\varepsilon_Y$ restricts to the canonical injection of this submonoid into $\Sym \mathcal O_X(-\alpha)$, so it is certainly well-defined; if $b$ and $d$ are both nonunits then $\varepsilon_X(b) = \varepsilon_X(d) = 0$, since $X$ is the spectrum of a field, hence $\varepsilon_Y(a^n b) = \varepsilon_Y(c^n d)$.  

Thus $\varepsilon_Y$ is well-defined and makes $M_Y$ into a logarithmic structure on $Y$.  This logarithmic structure is fine, since \'etale-locally $X$ has a global chart by $\Gamma(X, \overnorm M_X)$ and then $Y$ has a global chart by $\overnorm M_X + \mathbf N \alpha$.  The characteristic monoid of $Y$ is $\overnorm M_X + \mathbf N \alpha$ at the closed point where $\mathcal O_Y(-\alpha) \to \mathcal O_Y$ vanishes, and it is $\overnorm M_X + \mathbf Z\alpha  / \mathbf Z \alpha$ elsewhere.  In particular, it is integral and therefore fine.  By Lemma~\ref{lem:saturated}, it is also saturated if $\alpha$ is saturated.

To conclude, we verify that $Y$ has the required universal property.  If $f : Z \to X$ is any morphism such that $f^\ast \alpha$ lies in $\overnorm M_Z$ then we obtain a morphism $\mathcal O_Z(-\alpha) \to \mathcal O_Z$.  This extends to a ring homomorphism $\Sym \mathcal O_Z(-\alpha) \to \mathcal O_Z$ and therefore induces a morphism of schemes $g : Z \to Y$.  Since $M_Y$ is the submonoid of $\tau^\ast M_X^{\rm gp}$ generated by $\mathcal O_Y^\ast(-\alpha)$ and $\tau^\ast M_X$, the map $g^\ast M_Y^{\rm gp} \to M_Z^{\rm gp}$ induced from $f$ carries $g^\ast M_Y$ into $M_Z$.  Finally, $g$ was chosen so that $g^\ast \mathcal O_Y(-\alpha) \to \mathcal O_Z$ coincides with $\mathcal O_Z(-\alpha) \to \mathcal O_Z$ so $g$ is actually a morphism of logarithmic schemes.
\end{proof}

\begin{corollary} \label{cor:P1-fiber}
Suppose $X$ is a logarithmic scheme whose underlying scheme is the spectrum of a field.  Suppose $\alpha$ is a saturated section of $\overnorm M_X^{\rm gp}$.  Let $Y$ be the universal logarithmic scheme over $X$ where $\alpha$ is locally comparable to $0$ in the partial order on $\overnorm M_X^{\rm gp}$ induced from $\overnorm M_X$.  Then the underlying scheme of $Y$ is isomorphic to $\mathbf P^1_X$.
\end{corollary}
\begin{proof}
The charts where $\alpha \geq 0$ and where $\alpha \leq 0$ are each isomorphic to $\mathbf A^1_X$ by Corollary~\ref{cor:A1-fiber}.  These are glued along the locus where $\alpha = 0$, which is isomorphic to $\Spec \sum_{n \in \mathbf Z} \mathcal O_X(n \alpha)$.
\end{proof}

\begin{corollary} \label{cor:saturated}
	Suppose $X$ is a logarithmic scheme and $\alpha \in \Gamma(X, \overnorm M_X^{\rm gp})$.  Let $Y$ be the universal scheme with a fine logarithmic structure where $\alpha$ is locally comparable to $0$ in the partial order of $\overnorm M_Y^{\rm gp}$ induced from $\overnorm M_Y$.  Then $Y$ is saturated and the underlying scheme of the fiber of $Y$ over a each point $x$ of $X$ is either empty, isomorphic to $x$, or isomorphic to $\mathbf P^1_x$.
\end{corollary}
\begin{proof}
It is immediate from Lemma~\ref{lem:saturated} that $Y$ is saturated.

Since the construction of $Y$ relative to $X$ is compatible with strict base change, it is sufficient to prove the corollary after base change to a point.  We therefore assume that $X$ is the spectrum of a field.  There are now $3$ possibilities:  if $\alpha \in \overnorm M_X$ then $Y = X$ and we are done; if $-\alpha \in \overnorm M_X$ then $Y = \varnothing$ and we are done; if neither $\alpha$ nor $-\alpha$ is in $\overnorm M_X$ then $Y \simeq \mathbf P^1_X$ by Corollary~\ref{cor:P1-fiber}.
\end{proof}

\begin{lemma} \label{lem:root-sat}
Suppose that $X$ is a quasicompact logarithmic scheme and $\alpha \in \Gamma(X, \overnorm M_X^{\rm gp})$.  Then there is a root stack $X'$ of $X$ and an integer $n > 0$ such that $n^{-1} \alpha$ is saturated on $X'$.
\end{lemma}
\begin{proof}
Since $X$ is quasicompact, there is a finite collection of rank~$1$ localizations of the characteristic monoids $\overnorm M_{X,x}$ at geometric points $x$ of $X$.  Each of these localizations is (uniquely) isomorphic to $\mathbf N$.  Let $n_x$ be the image of $\alpha$ in $\overnorm M_{X,x} \simeq \mathbf N$.  Let $n$ be the least common multiple of the nonzero $n_x$ and let $X'$ be the extension of $\overnorm M_X$ in $\mathbf Q \overnorm M_X$ generated by $\overnorm M_X$ and $n^{-1} \alpha$.  Then $n^{-1} \alpha$ generates $\overnorm M_{X,x}$ at every geometric point $x$ of $X$ where $\alpha$ is nonzero.
\end{proof}

\setcounter{secnumdepth}{5}
\subsubsection{Proof of Proposition~\ref{prop:sat-hyp}}

The hypotheses are preserved by \'etale localization in $X$, and the conclusion satisfies \'etale descent in $X$.  We can therefore assume that $X$ is quasicompact and quasiseparated and has a global chart by a sharp monoid $P$, and that there is a saturated $\alpha \in P$ of such that $Y$ is the universal logarithmic modification of $X$ where the image of $\alpha$ in $\overnorm M_Y^{\rm gp}$ is locally comparable to $0$.  

In this situation, we prove two slightly more precise statements that imply Proposition~\ref{prop:sat-hyp}:

\begin{lemma}
We have $\mathrm R^1 \tau_\ast \overnorm M_Y^{\rm gp} = 0$.
\end{lemma}
\begin{proof}
By proper base change for \'etale cohomology,  which implies that the base change map is injective for $H^1$ \cite[Th\'eor\`eme~5.1~(ii)]{sga4-XII}, it is sufficient to prove that 
\begin{equation*}
H^1(\tau^{-1} x, \overnorm M_Y^{\rm gp}) = 0
\end{equation*}
for all geometric points $x$ of $X$.  By Corollary~\ref{cor:saturated}, the underlying scheme of the fiber $Z = \tau^{-1} x$ is either empty, isomorphic to $x$, or isomorphic to $\mathbf P^1_x$.  In the first two cases, the conclusion is trivial.  If $Z \simeq \mathbf P^1_x$, we have an exact sequence:
\begin{equation*}
0 \rightarrow \pi^{-1} \overnorm M_x^{\rm gp} \rightarrow \overnorm M_Z^{\rm gp} \rightarrow \overnorm M_{Z/x}^{\rm gp} \rightarrow 0
\end{equation*}
We have $H^1(Z, \pi^{-1} \overnorm M_x^{\rm gp}) = 0$ since $\pi^{-1} \overnorm M_x^{\rm gp}$ is constant and $Z$ is simply connected.  We have $H^1(Z, \overnorm M_{Z/x}^{\rm gp}) = 0$ since $\overnorm M_{Z/x}$ is concentrated in dimension~$0$.  Combined with proper base change, this gives $\mathrm R^1 \tau_\ast \overnorm M_Y^{\rm gp} = 0$.
\end{proof}

\begin{lemma}
The map $\tau_\ast (\overnorm M_Y^{\rm gp}) / \overnorm M_X^{\rm gp} \to \mathrm R^1 \tau_\ast \mathcal O_Y^\ast$ is an isomorphism.
\end{lemma}
\begin{proof}
We make some preliminary reductions.  It is sufficient to show that the map of \'etale stalks $\tau_\ast(\overnorm M_Y^{\rm gp}) / \overnorm M_{X,x}^{\rm gp} \to \mathrm R^1 \tau_\ast (\mathcal O_Y^\ast)_x$ is an isomorphism for every geometric point $x$ of $X$.  Since $\tau$ is of finite presentation, the formation of $\mathrm R^1 \tau_\ast \mathcal O_Y^\ast$ commutes with localization, as does the formation of $\tau_\ast(\overnorm M_Y^{\rm gp}) / \overnorm M_{X}^{\rm gp}$.  It therefore suffices to assume that $X$ is the spectrum of a henselian local ring.  We write $x$ for the closed point of $X$.

Under these assumptions, we need to show that the following map is an isomorphism:
\begin{equation} \label{eqn:133} \tag{$\ast$}
H^0(Y, \overnorm M_Y^{\rm gp}) / H^0(X, \overnorm M_X^{\rm gp}) \to H^1(Y, \mathcal O_Y^\ast)
\end{equation}

Since $X$ has no nontrivial \'etale covers, we can find $\gamma$ and $\delta$ in $\overnorm M_X$ such that $\gamma - \delta = \alpha$.  These lift to sections $\tilde\gamma$ and $\tilde\delta$ of $M_X$.  Then $\alpha$ factors through $\varphi = (\tilde\gamma, \tilde\delta) : X \to \mathbf A^2$ (where $\mathbf A^2$ has the toric logarithmic structure), and $Y$ is the pullback of the blowup of $\mathbf A^2$ at the origin.  We write $X' = \mathbf A^2$ and $Y'$ for the blowup of $X'$ at the origin.

We will now prove the proposition by considering successively more general examples of~$X$.

\renewcommand{\theparagraph}{\textsc{Step \Roman{paragraph}}}
\vskip1ex
	\paragraph{$X$ is the spectrum of a field.} \label{step:1}
Then $Y$ is either empty, isomorphic to $X$, or isomorphic to $\mathbf P^1_X$.  The conclusion is trivial in the first two cases, so assume $Y \simeq \mathbf P^1_X$.  Then $H^0(Y, \overnorm M_Y^{\rm gp}) = H^0(X, \overnorm M_X^{\rm gp}) + \mathbf Z \beta$ where $\beta$ is the section $\max \{ 0, \alpha \}$.  Direct calculation shows that $\mathcal O_Y(\beta) \simeq \mathcal O_{\mathbf P^1_X}(1)$ under the isomorphism $Y \simeq \mathbf P^1_X$.  Since $\Pic(\mathbf P^1_X)$ is generated by $\mathcal O_{\mathbf P^1_X}(1)$, this shows that~\eqref{eqn:133} is an isomorphism.

\vskip1ex
	\paragraph{$X$ is an artinian local ring and $\varphi : X \to X'$ factors \emph{schematically} through the origin.}
	Let $x$ be the closed point of $X$.  We will show that $H^0(Y, \mathcal O_Y^\ast) = H^0(\tau^{-1} x, \mathcal O_{\tau^{-1} x}^\ast)$.  Since the same identity holds for $H^0(Y, \overnorm M_Y^{\rm gp}) / H^0(X, \overnorm M_X^{\rm gp})$ (since $\overnorm M_Y^{\rm gp}$ and $\overnorm M_X^{\rm gp}$ are \'etale sheaves and the inclusions $x \to X$ and $\tau^{-1} x \to Y$ induce equivalence of \'etale sites), this will be enough to complete this case.

Every artinian local ring is an iterated extension of the residue field $k = k(x)$ by square-zero ideals isomorphic to $k$.  By induction, we may therefore assume that $X$ such an extension of some $X_0$, and the conclusion of the proposition is already known for the restriction $Y_0$ of $Y$ to $X_0$.  The inductive step is to show that the restriction map $H^1(Y, \mathcal O_Y^\ast) \to H^1(Y_0, \mathcal O_{Y_0}^\ast)$ is an isomorphism.

Let $J$ be the ideal of $Y_0$ in $Y$.  This is also the kernel of $\mathcal O_Y^\ast \to \mathcal O_{Y_0}^\ast$.  We have an exact sequence:
\begin{equation*}
H^1(\tau^{-1} x, J) \to H^1(Y, \mathcal O_Y^\ast) \to H^1(Y_0, \mathcal O_{Y_0}^\ast) \to H^2(\tau^{-1} x, J)
\end{equation*}
Since $J$ is supported on $\tau^{-1} x \simeq \mathbf P^1_{\tau^{-1} x}$, we have $H^2(\tau^{-1} x, J) = 0$ for dimension reasons.  We show next that $H^1(\tau^{-1} x, J)$ also vanishes.

Let $Z$ be the fiber product $X \mathop\times_{X'} Y'$ as a scheme with a logarithmic structure that is not necessarily integral.  Since $\varphi(X) = 0$, we have $Z \simeq \mathbf P^1_X$.  Then $Y$ is the universal scheme over $Z$ over which the pullback of $M_Z$ is integral (it is automatically saturated by Corollary~\ref{cor:saturated}).  Since $Z \simeq \mathbf P^1_X$ is flat over $X$, the ideal of $Z_0$ in $Z$ is $\mathcal O_{\tau^{-1} x}$.  Therefore $J$, which is the ideal of $Y_0 = Y \cap Z_0$ in $Y$ is a quotient of $\mathcal O_{\tau^{-1} x}$.  Let $K$ be the kernel of $\mathcal O_{\tau^{-1} x} \to J$.  Then we have an exact sequence:
\begin{equation*}
H^1(\tau^{-1} x, \mathcal O_{\tau^{-1} x}) \to H^1(\tau^{-1} x, J) \to H^2(\tau^{-1} x, K)
\end{equation*}
	Since $\tau^{-1} x \simeq \mathbf P^1_x$, both $H^1(\tau^{-1} x, \mathcal O_{\tau^{-1} x})$ and $H^2(\tau^{-1} x, K)$ vanish.  This completes the proof that $H^1(Y, \mathcal O_Y^\ast) = H^1(\tau^{-1} x, \mathcal O_{\tau^{-1} x}^\ast)$, and implies that~\eqref{eqn:133} is an isomorphism in this case.

\vskip1ex
\paragraph{$X$ is the spectrum of an artinian local ring.}
	For each $n$, let $X'_n$ be the $n$-th infinitesimal neighborhood of the origin in $X'$ and let $Y'_n$ be its preimage in $Y'$.  Let $X_n = \varphi^{-1} X'_n$ be the preimage in $X$ of $X'_n$, and let $Y_n$ be the preimage in $Y$.  Let $Z_n$ be the pullback of $Y'$ to $X_n$ as a scheme with a logarithmic structure that is not necessarily integral.  The ideal of $Y'_{n-1}$ in $Y'_n$ is $\mathcal O_{Y'_0}(n)$.  Therefore the ideal of $Z_{n-1}$ in $Z_n$ is a quotient of $\mathcal O_{Z_0}(n)$.  Since $Y_{n-1} = Z_{n-1} \cap Y_n$ the ideal of $Y_{n-1}$ in $Y_n$ is thus also a quotient of $\mathcal O_{Z_0}(n)$; we denote this quotient by $J$ and write $K$ for the kernel.  We have an exact sequence:
\begin{equation*}
	H^1(Z_0, \mathcal O_{Z_0}(n)) \to H^1(Y_0, J) \to H^2(Z_0, K)
\end{equation*}
	Since $Z_0$ is $1$-dimensional, $H^2(Z_0, K)$ vanishes. But $H^1(Z_0, \mathcal O_{Z_0}(n))$ also vanishes, because $Z_0 \simeq \mathbf P^1_{X_0}$ and $n \geq 0$ (in fact $n \geq 1$).  Since we also have $H^2(Y_0, J) = 0$ by dimension considerations, we deduce that $H^1(Y, \mathcal O_Y^\ast) = H^1(Y_0, \mathcal O_{Y_0}^\ast) = H^1(\tau^{-1} x, \mathcal O_{\tau^{-1} x}^\ast)$, as before.  Thus~\eqref{eqn:133} is an isomorphism in this case, by the same argument as in the last step.

\vskip1ex
\paragraph{$X$ is the spectrum of a complete noetherian local ring.}
Suppose $X = \Spec A$ and $I$ is the maximal ideal of $A$.  Write $X_n$ for the vanishing locus of $I^{n+1}$ and $Y_n = \tau^{-1} X_n$.  By Grothendieck's existence theorem, $H^1(Y, \mathcal O_Y^\ast) = \varprojlim H^1(Y_n, \mathcal O_{Y_n}^\ast)$.  By proper base change $H^0(Y, \overnorm M_Y^{\rm gp}) / H^0(X, \overnorm M_X^{\rm gp}) = H^0(Y_0, \overnorm M_{Y_0}^{\rm gp}) / H^0(X_0, \overnorm M_{X_0}^{\rm gp})$.  Therefore we conclude by the previous case.

\vskip1ex
\paragraph{$X$ is the henselization of scheme of finite type at a geometric point.}
As before, let $X_n$ denote the $n$-th order infinitesimal neighborhood of $x$ in $X$ and let $Y_n$ be its preimage in $Y$.  We have a commutative diagram:
	\begin{equation*} \xymatrix{
			H^0(Y, \overnorm M_Y^{\rm gp}) / H^0(X, \overnorm M_X^{\rm gp}) \ar[r] \ar[d] & H^1(Y, \mathcal O_Y^\ast) \ar[d] \\
			H^0(\tau^{-1} x, \overnorm M_{\tau^{-1} x}^{\rm gp}) / \overnorm M_{X,x}^{\rm gp} \ar[r] & H^1(\tau^{-1} x, \mathcal O_{\tau^{-1} x}^\ast)
	} \end{equation*}
	The left vertical arrow is an isomorphism by proper base change, and the bottom horizontal arrow is an isomorphism by the special case of the lemma for a point.  Therefore the upper horizontal arrow is injective.
	
	We prove it is also surjective.  Suppose that $L$ is an invertible sheaf on $Y$.  Let $\hat X$ be the completion of $X$ and let $\hat Y$ be the base change of $Y$ to $\hat X$.  Let $\hat L$ be the restriction of $L$ to $\hat Y$.  Then by the previous case, there is some $\lambda \in H^0(Y, \overnorm M_Y^{\rm gp})$ such that $\hat L \simeq \mathcal O_{\hat Y}(\lambda)$.  Since the problem of specifying an isomorphism $L \simeq \mathcal O_Y(\lambda)$ is locally of finite presentation, as a functor of $X$, Artin's approximation theorem implies that there is an isomorphism $L \simeq \mathcal O_Y(\lambda)$.

\vskip1ex
\paragraph{$X$ is the spectrum of a henselian local ring.}
	We can present $X$ as a cofiltered inverse limit of spectra $X_i$ of henselian local rings of finite type with closed points $x_i$.  Then $\overnorm M_{X,x}^{\rm gp} = \varinjlim \overnorm M_{X_i,x_i}^{\rm gp}$ so we can assume that the element $\alpha \in \overnorm M_{X,x}^{\rm gp}$ used to construct $Y$ is induced from elements $\alpha_i \in \overnorm M_{X_i,x_i}^{\rm gp}$.  For each $i$, let $Y_i$ be the universal logarithmic scheme over $X_i$ where $\alpha_i$ is locally comparable to $0$ in $\overnorm M_{Y_i}^{\rm gp}$.  Then $H^1(Y, \mathcal O_Y^\ast) = \varinjlim H^1(Y_i, \mathcal O_{Y_i}^\ast)$ and $H^0(Y, \overnorm M_Y^{\rm gp}) / H^0(X, \overnorm M_X^{\rm gp}) = \varinjlim H^0(Y_i, \overnorm M_{Y_i}^{\rm gp}) / H^0(X_i, \overnorm M_{X_i}^{\rm gp})$.  The previous case shows that $H^0(Y_i, \overnorm M_{Y_i}^{\rm gp}) / H^0(X_i, \overnorm M_{X_i}^{\rm gp}) \to H^1(Y_i, \mathcal O_{Y_i}^\ast)$ is an isomorphism for every $i$, so we conclude by passage to the limit.
\end{proof}

\subsubsection{The logarithmically flat case}

\begin{proposition} \label{prop:injective}
	Suppose that $X$ and $Y$ are logarithmic schemes over $S$ and $\tau : Y \to X$ is a logarithmic modification.  Assume that both $X$ and $Y$ are integral over $S$ and $X$ is logarithmically flat over $S$.  Then the map $\mathcal O_X \to \mathrm R \tau_\ast \mathcal O_Y$ is a quasi\"isomorphism.
\end{proposition}
\begin{proof}
The assertion is \'etale-local in $S$ and in $X$, so we may assume that there is a cartesian diagram~\eqref{eqn:131} in which $X'$, $Y'$, and $S'$ are toric varieties and the maps between them are toric, and all horizontal arrows are strict:
\begin{equation} \label{eqn:131} \vcenter{\xymatrix{
Y \ar[r] \ar[d]_\tau & Y' \mathop\times_{S'} S \ar[r] \ar[d] & Y' \ar[d]^\rho \\
X \ar[r] \ar[dr] & X' \mathop\times_{S'} S \ar[r] \ar[d] & X' \ar[d] \\
& S \ar[r] & S'
}} \end{equation}
Since $X$ and $Y$ are integral over $S$, the toric varieties $X'$ and $Y'$ are flat over $S'$.  Since $X$ is logarithmically flat over $S$, the map $X \to X' \mathop\times_{S'} S$ is also flat.

	As $Y' \to X'$ is a toric modification, we have $\mathrm R \rho_\ast \mathcal O_{Y'} = \mathcal O_{X'}$ \cite[\S3.5, p.~76, Proposition]{Fulton}.  Since $X'$ and $Y'$ are flat over $S'$, this equality remains true after base change to $S$.%
	\footnote{The assertion is Zariski-local in $X'$, so we may assume that $X'$ is affine.  Then $\Gamma(X', \mathrm R \rho_\ast \mathcal O_{Y'}) = \mathrm R \Gamma(Y', \mathcal O_{Y'})$ so we may apply flat base change \cite[\href{https://stacks.math.columbia.edu/tag/02KH}{Tag 02KH}]{stacks-project}.}
Thus $\mathcal O_{X' \mathop\times_{S'} S} \to \mathrm R \rho_\ast \mathcal O_{Y' \mathop\times_{S'} S}$ is also a quasi\"isomorphism.  Then by flat base change along $X \to X' \mathop\times_{S'} S$, we obtain the quasi\"isomorphism $\mathcal O_X \to \mathrm R \tau_\ast \mathcal O_Y$ that we require.
\end{proof}

\begin{lemma} \label{lem:log-str-pf}
	Suppose $X$ and $Y$ are logarithmic schemes over $S$.  Assume that $X$ and $Y$ are both integral over $S$ and that $X$ is logarithmically flat over $S$.  Let $\tau : Y \to X$ be the logarithmic modification associated with a saturated section $\alpha$ of $\overnorm M_X^{\rm gp}$.  Then $\tau_\ast M_Y^{\rm gp} = M_X^{\rm gp}$.
\end{lemma}
\begin{proof}
	We have a commutative diagram with exact rows:
	\begin{equation*} \xymatrix{
			0 \ar[r] & \mathcal O_X^\ast \ar[r] \ar[d] & M_X^{\rm gp} \ar[r] \ar[d] & \overnorm M_X^{\rm gp} \ar[r] \ar[d] & 0 \\
			0 \ar[r] & \tau_\ast \mathcal O_Y^\ast \ar[r] & \tau_\ast M_Y^{\rm gp} \ar[r] & \tau_\ast \overnorm M_Y^{\rm gp} \ar[r] & \mathrm R^1 \tau_\ast \mathcal O_Y^\ast
	} \end{equation*}
	The map $\mathcal O_X^\ast \to \tau_\ast \mathcal O_Y^\ast$ is an isomorphism by Proposition~\ref{prop:injective}.  The map $\tau_\ast(\overnorm M_Y^{\rm gp}) / \overnorm M_X^{\rm gp} \to \mathrm R^1 \tau_\ast \mathcal O_Y^\ast$ is an isomorphism by Proposition~\ref{prop:sat-hyp}.  We conclude that $M_X^{\rm gp} \to \tau_\ast M_Y^{\rm gp}$ is an isomorphism by snake lemma and the 5-lemma.
\end{proof}

\begin{proof}[Proof of Proposition~\ref{prop:log-flat}]
	We make an observation that we will use repeatedly in the proof.  We will say that `the proposition holds for $\tau : Y \to X$' to mean that if $X$ is logarithmically flat over $S$ then every \'etale $M_Y^{\rm gp}$-torsor descends uniquely along $\tau$ to an $M_X^{\rm gp}$-torsor.  Suppose that we have logarithmic modifications $Z \xrightarrow{\rho} Y \xrightarrow{\tau} X$ and that the proposition is known to hold for $\rho$.  If $X$ is logarithmically flat over $S$ then so is $Y$, because $Y$ is a logarithmic modification of $X$.  Therefore \'etale $M_Y^{\rm gp}$-torsors descend uniquely along $\tau$ to $M_X^{\rm gp}$-torsors if and only if \'etale $M_Z^{\rm gp}$-torsors descend uniquely along $\tau\rho$ to $M_X^{\rm gp}$-torsors.  Indeed, if $L$ is an $M_Y^{\rm gp}$-torsor then it is the unique descent of $\rho^\ast L$  to $Y$ (since the proposition is known for $\rho$ and $Y$ is logarithmically flat over $S$), so $\rho^\ast L$ descends uniquely to $X$ if and only if $L$ descends uniquely to $X$.

Assume that $X$ is logarithmically flat over a valuative logarithmic scheme $S$ and $\tau : Y \to X$ is a logarithmic modification.  We wish to show that every $M_Y^{\rm gp}$-torsor descends uniquely to an $M_X^{\rm gp}$-torsor.  This is a local question in the \'etale topology of $X$, so we assume that $X$ is quasicompact and has a global chart by a sharp monoid $P$.  Let $\sigma$ be the rational polyhedral cone dual to $P$.  Since $X$ is quasicompact, every logarithmic modification of $X$ has a refinement that is induced from a subdivision of $\sigma$.  Every subdivision of $\sigma$  has a refinement that is an iterated subdivision along hyperplanes.  We can therefore find a logarithmic modification $\rho : Z \to Y$ such that both $\rho$ and $\tau\rho$ are iterated subdivisions along hyperplanes.  By the observation above, it is therefore sufficient to prove the proposition when $Y$ is an iterated subdivision of $X$ by hyperplanes.  By induction and another application of the observation, it suffices to assume $Y$ is a subdivision along a single hyperplane.

	We now assume $X$ has a global chart and that $\tau : Y \to X$ is a hyperplane subdivision of $X$ associated with some $\alpha \in \Gamma(X, \overnorm M_X^{\rm gp})$.  By Lemma~\ref{lem:root-sat}, there is a root stack $\rho : X' \to X$ and an integer $n > 0$ such that $n^{-1} \alpha \in \Gamma(X', \overnorm M_{X'}^{\rm gp})$ and $n^{-1} \alpha$ is saturated.  Let $\tau' : Y' \to X'$ and $\rho' : Y' \to Y$ be the morphisms induced by base change.  Note that $Y'$ is the logarithmic modification of $X'$ associated with the section $n^{-1} \alpha \in \Gamma(X', M_{X'}^{\rm gp})$.

	We know by Proposition~\ref{prop:root} that $M_{Y'}^{\rm gp}$-torsors descend uniquely to $Y$.  Therefore, by the observation from the beginning, our conclusion holds for $\tau$ if and only if it holds for $\tau \rho'$.  But we also know by Proposition~\ref{prop:sat-hyp} that every $M_{Y'}^{\rm gp}$-torsor descends uniquely to an $\tau'_\ast M_{Y'}$-torsor on $X'$.  On the other hand, both $X'$ and $Y'$ are integral over $S$ since $S$ is valuative and all morphisms to a valuative logarithmic scheme are integral, so $\tau'_\ast M_{Y'}^{\rm gp} = M_{X'}^{\rm gp}$ by Lemma~\ref{lem:log-str-pf}.  Finally, every $M_{X'}^{\rm gp}$-torsor descends uniquely to a $M_X^{\rm gp}$-torsor by Proposition~\ref{prop:root}, again.
\end{proof}

\subsubsection{Applications}

\begin{corollary} \label{cor:subdiv-curve}
	Suppose that $\tau : Y \to X$ is a subdivision of logarithmic curves over $S$.  Then every $M_Y^{\rm gp}$-torsor on $Y$ descends uniquely to an $M_X^{\rm gp}$-torsor on $X$.
\end{corollary}
\begin{proof}
	Since a subdivision of logarithmic curves $\tau : Y \to X$ can be presented \'etale-locally in $X$ as an iterated subdivision associated with a saturated section of $\overnorm M_X^{\rm gp}$, it suffices to consider only the case where $Y$ is the subdivision of $X$ associated with a single global section of $\overnorm M_X^{\rm gp}$.  In this case, Proposition~\ref{prop:sat-hyp} implies that every $M_Y^{\rm gp}$-torsor descends uniquely to a $\tau_\ast M_Y^{\rm gp}$-torsor on $X$.  By Lemma~\ref{lem:log-str-pf}, we have $\tau_\ast M_Y^{\rm gp} = M_X^{\rm gp}$, so we may conclude.
\end{proof}

\begin{corollary} \label{cor:log-etale-desc}
	Suppose that $X$ is a logarithmic curve over a logarithmic scheme $S$.  Then $\bLogPic(X/S)$ and $\LogPic(X/S)$ satisfy descent with respect to fppf covers and root stacks of $S$.  If $S$ is logarithmically flat then $\bLogPic(X/S)$ and $\LogPic(X/S)$ also satisfy descent with respect to logarithmic modifications.  In particular $\bLogPic(X/S)$ is a stack, and $\LogPic(X/S)$ is a sheaf, in the Kummer logarithmic flat topology on $S$.  If $S$ is logarithmically flat then $\bLogPic(X/S)$ is a stack, and $\LogPic(X/S)$ is a sheaf, in the small full logarithmic \'etale topology on $S$.
\end{corollary}
\begin{proof}
The Kummer logarithmic flat topology is generated by fppf covers and root stacks.  The logarithmic \'etale topology is generated by \'etale covers, logarithmic modifications, and root stacks of order prime to the characteristic \cite[Lemma~3.11]{nakayama2017logarithmic}.  Therefore it will suffice to show descent with respect to fppf covers, root stacks, and logarithmic modifications.

	We know from Propositions~\ref{prop:fppf},~\ref{prop:root} and~\ref{prop:log-flat} that $\mathrm B \logGm$ satisfies descent with respect to fppf covers, root stacks, and (if $S$ is logarithmically flat) logarithmic modifications.  Since $\bLogPic(X/S)$ is isomorphic to $\LogPic(X/S) \times \mathrm B\logGm$ \'etale-locally in $S$, the statements concerning $\LogPic(X/S)$ and $\bLogPic(X/S)$ are equivalent.  We will prove the statements involving $\bLogPic(X/S)$.

	Any fppf cover, root stack, or logarithmic modification of $S$ pulls back to an fppf cover, root stack, or logarithmic modification of $X$.  Since $X$ is logarithmically flat over $S$, it will be logarithmically flat if $S$ is.  We have just seen in the last paragraph (with $S$ replaced by $X$) that $M_X^{\rm gp}$-torsors satisfy descent along all of these kinds of covers.  Therefore all that is left is to see that boundedness of monodromy satisfies descent with respect to strict fppf covers, root stacks, and logarithmic modifications of $S$.

	By Proposition~\ref{prop:logpic-bdd-mono}, boundedness of monodromy can be verified at the valuative geometric points of $S$, so we are free to assume $S$ is a valuative geometric point.  It is immediate that boundedness of monodromy descents along fppf covers, since the condition only depends on the characteristic monoid.  Valuative geometric points have no nontrivial logarithmic modifications, so descent in that case is also trivial.  Finally, Lemma~\ref{lem:root-bdd} says that boundedness of monodromy descends along root stacks.
\end{proof}

The following corollary of Proposition~\ref{prop:sat-hyp} complements Theorem~\ref{thm:curve-val}:

\begin{corollary} \label{cor:r1}
Let $S$ be the spectrum of a valuation ring with generic point $\eta$ and let $X$ be a family of nodal curves over $S$.  Assume that $X_\eta$ and $\eta$ have been given logarithmic structures $M_{X_\eta}$ and $M_\eta$ making $X_\eta$ into a logarithmic curve over $\eta$, with $M_\eta$ valuative.  Let $M_X$ and $M_S$ be the maximal extensions, respectively, of $M_{X_\eta}$ and $M_\eta$ to $X$ and to $S$.  Let $j : X_\eta \to X$ be the inclusion of the generic fiber.  Then $R^1 j_\ast M_{X_\eta}^{\rm gp} = 0$.
\end{corollary}
\begin{proof}
	We wish to show that any $M_{X_\eta}^{\rm gp}$-torsor can be trivializated \'etale-locally on $X$.  Suppose that $L_\eta$ is an $M_{X_\eta}^{\rm gp}$-torsor.  Since the base has a valuative logarithmic structure, Corollary~\ref{cor:bdd-nbhd} implies that there is a logarithmic modification $\tilde X_\eta$ of $X_\eta$ on which $L_\eta$ is representable by an invertible sheaf $\tilde L_\eta$.  We can extend this logarithmic modification to a logarithmic modification $\tilde X$ of $X$.  Now, $\tilde L_\eta$ can be represented by a divisor $D_\eta$ in the smooth locus of $\tilde X_\eta$.  After replacing $\tilde X$ by a further logarithmic modification, we can assume that the closure of $D_\eta$ in $\tilde X$ is contained in the smooth locus.  Thus $\tilde L_\eta$ extends to an invertible sheaf on $\tilde X$.  Note that $\tilde X$ is an iterated subdivision of $X$ associated with saturated sections of $\overnorm M_X^{\rm gp}$.  Passing to the associated $M_{\tilde X}^{\rm gp}$-torsor, and applying Corollary~\ref{cor:subdiv-curve}, every $M_{\tilde X}^{\rm gp}$-torsor on $\tilde X$ descends uniquely to a $M_X^{\rm gp}$-torsor on $X$.  It follows that the $M_{X_\eta}^{\rm gp}$-torsor $L_\eta$ extends to an $M_X^{\rm gp}$-torsor $L$.  Since the $M_X^{\rm gp}$-torsor $L$ can be trivialized \'etale-locally, it follows that $L_\eta$ can be trivialized \'etale-locally in $X$ as well.  
\end{proof}

\numberwithin{theorem}{subsection}
\subsection{Degree}
\label{sec:degree}

Let $X$ be a proper, vertical logarithmic curve over $S$, whose underlying scheme is the spectrum of an algebraically closed field with a valuative logarithmic structure.  We construct a dashed arrow making diagram~\eqref{eqn:9} commute:
\numberwithin{equation}{subsection}
\begin{equation} \label{eqn:9} \vcenter{\xymatrix{
	H^0(X, \overnorm M_{X}^{\rm gp}) \ar[r] \ar[dr] & H^1(X, \mathcal O_X^\ast) \ar[r] \ar[d] & \LogPic(X) \ar@{-->}[d]^{\deg} \\
	& \mathbf Z^V \ar[r]^\Sigma & \mathbf Z
}} \end{equation}
Here, $V$ is the set of vertices of the dual graph of $X$, so $\mathbf Z^V$ is the N\'eron--Severi group of $X$, and the solid vertical arrow is the multidegree.  The map $\Sigma : \mathbf Z^V \to \mathbf Z$ is the sum.  We regard a section of $\overnorm M_X^{\rm gp}$ as a piecewise linear, $\overnorm M_S^{\rm gp}$-valued function on the dual graph of $X$, with integer slopes along the edges.  The diagonal map to $\mathbf Z^V$ sends such a function to tuple whose each component is the sum of the outgoing slopes from the corresponding vertex of the dual graph.  The composed map to $\mathbf Z$ therefore takes the sum of the outgoing slopes from every vertex; since each edge gets counted twice with opposite orientations ($X$ is vertical, so its dual graph is compact) the composition is zero.  This gives the vertical arrow on the image of $H^1(X, \mathcal O_X^\ast)$.

Given any logarithmic line bundle $L$ on $X$, Corollary~\ref{cor:bdd-nbhd} implies that there is a logarithmic modification $\tilde X$ of $X$ such that the restriction $\tilde L$ of $L$ to $\tilde X$ lies in the image of $H^1(\tilde X, \mathcal O_{\tilde X}^\ast)$.  We define the degree of $L$ to be the degree of any invertible sheaf $\tilde L$ representing $L$ on any logarithmic modification of $X$.  

This defines the degree at all \emph{valuative, geometric} points of $\bLogPic(X/S)$.  The following proposition extends the definition to families:

\setcounter{theorem}{\value{equation}}
\numberwithin{equation}{theorem}
\begin{proposition}
	The degree of a logarithmic line bundle well-defined and locally constant on $\bLogPic(X/S)$.
\end{proposition}
\begin{proof}
	Suppose that $L$ is a logarithmic line bundle on a logarithmic curve $X$ over $S$.  We can find a logarithmic modification $\tilde S$ and a semistable model $\tilde X$ of $X \mathop\times_S \tilde S$ such that $L$ can be represented by an invertible sheaf on $\tilde X$.  By construction, the total degree of this invertible sheaf is the degree of $\tilde L$, the pullback of $L$ to $\tilde X$.  Since the total degree of an invertible sheaf is locally constant, so is the total degree of $\tilde L$.  It follows that the degree, as defined above, is locally constant and well-defined on $\tilde S$.  But $\tilde S \to S$ is surjective, and every valuative geometric point of $S$ lifts to $\tilde S$, so the degree is also well-defined on $S$.  Furthermore, $\tilde S \to S$ is closed, so a function on $S$ that pulls back to a locally constant function on $\tilde S$ must have been locally constant on $S$.  Therefore the total degree of $L$ on $S$ is also locally constant.
\end{proof}

\begin{definition}
We write $\LogPic^d(X/S)$ for the open and closed substack of $\LogPic(X/S)$ parameterizing isomorphism classes of $\logGm$-torsors with bounded monodromy and degree $d$.
\end{definition}

\numberwithin{equation}{subsection}
\subsection{Quotient presentation}
\label{sec:log-av}

We construct a quotient presentation of $\LogPic^0(X/S)$.  Over the strata of $S$, this produces a logarithmic abelian variety with \emph{constant degeneration}, in the terminology of Kajiwara, Kato, and Nakayama \cite{kajiwara_kato_nakayama_2008a,kajiwara_kato_nakayama_2008b,kajiwara_kato_nakayama_2013,kajiwara_kato_nakayama_2015} (see Section~\ref{sec:strathom}).  Our presentation is inspired by Kajiwara's~\cite{kajiwara1993logarithmic}.

Let $X$ be a proper, vertical logarithmic curve over $S$, with connected geometric fibers.  Write $\Pic^{[0]}(X/S)$ for the multidegree zero part of $\Pic^0(X/S)$. 

\setcounter{theorem}{\value{equation}}
\numberwithin{equation}{theorem}
\begin{lemma} \label{lem:global-sections}
Let $X$ be a proper, vertical logarithmic curve over $S$ with connected geometric fibers.  Then the natural map $M_S^{\rm gp} \to \pi_\ast M_X^{\rm gp}$ is an isomorphism.
\end{lemma}
\numberwithin{equation}{theorem}
\begin{proof}
This assertion is \'etale-local in $S$.  We can therefore assume that $S$ is atomic and that the dual graph of $X$ is constant on the closed stratum.  We denote it $\tropfont X$.  Now $H^0(X, \overnorm M_X^{\rm gp})$ is the group of piecewise linear function on $\tropfont X$ having integer slopes along the edges and taking values in $\overnorm M_S^{\rm gp}$.  Since sections of $M_X^{\rm gp}$ correspond generically on $X$ to rational functions, the associated piecewise linear function on $\tropfont X$ of such a section will be \emph{linear}.  That is, the sum of the outgoing slopes along the edges incident to any vertex of $\tropfont X$ will be zero.

On the other hand $\tropfont X$ is compact so every linear function on $\tropfont X$ is constant by Lemma~\ref{lem:harmonic}.  Therefore the section of $\overnorm M_X^{\rm gp}$ induced from any section of $M_X^{\rm gp}$ lies in the image $\overnorm M_S^{\rm gp}$, which is to say that there is a diagonal arrow as shown in the commutative diagram of exact sequences~\eqref{eqn:11}:
\begin{equation} \label{eqn:11} \vcenter{\xymatrix{
0 \ar[r] & \mathcal O_S^\ast \ar[r] \ar[d] & M_S^{\rm gp} \ar[r] \ar[d] & \overnorm M_S^{\rm gp} \ar[r] \ar[d] & 0 \\
0 \ar[r] & \pi_\ast \mathcal O_X^\ast \ar[r] & \pi_\ast M_X^{\rm gp} \ar[r] \ar[ur] & \pi_\ast \overnorm M_X^{\rm gp} \ar[r] & R^1 \pi_\ast \mathcal O_X^\ast
}} \end{equation}
As $X$ is proper and flat over $S$ with reduced, connected fibers, the map $\mathcal O_S^\ast \to \pi_\ast \mathcal O_X^\ast$ is an isomorphism.  We may therefore conclude by the $5$-lemma, applied to the induced diagram~\eqref{eqn:43}:
\begin{equation} \label{eqn:43} \vcenter{\xymatrix{
0 \ar[r] & \mathcal O_S^\ast \ar[r] \ar[d]^\wr & M_S^{\rm gp} \ar[r] \ar[d] & \overnorm M_S^{\rm gp} \ar[r] \ar[d]^\wr & 0 \\
0 \ar[r] & \pi_\ast \mathcal O_X^\ast \ar[r] & \pi_\ast M_X^{\rm gp} \ar[r] & \pi_\ast \overnorm M_X^{\rm gp} 
}} \end{equation}
\end{proof}

\begin{proposition} \label{prop:kernel}
The map $R^1 \pi_\ast \pi^\ast M_S^{\rm gp} \to R^1 \pi_\ast M_X^{\rm gp}$ induces a surjection from the multidegree~$0$ part onto the degree~$0$ part, with kernel~$H_1(\tropfont X)$.
\end{proposition}
\begin{proof}
We may assume without loss of generality that $X$ has connected geometric fibers over $S$.  We use the exact sequence~\eqref{eqn:15}
\begin{equation} \label{eqn:15}
0 \to \pi^\ast M_S^{\rm gp} \to M_X^{\rm gp} \to \overnorm M_{X/S}^{\rm gp} \to 0
\end{equation}
and its associated long exact sequence in the top row of~\eqref{eqn:16}:  
\begin{equation} \label{eqn:16} \vcenter{\xymatrix{
0 \ar[r] & \pi_\ast \overnorm M_{X/S}^{\rm gp} \ar[r] \ar[d]^\wr & \mathrm R^1 \pi_\ast \pi^\ast M_S^{\rm gp} \ar[r] \ar[d] & \mathrm R^1 \pi_\ast M_X^{\rm gp} \ar[r] \ar@{-->}[d] & 0 \\
& \mathbf Z^E \ar[r] & \mathbf Z^V \ar[r] & \mathbf Z \ar[r] & 0
}} \end{equation}
Here $\mathbf Z^V$ is the sheaf of abelian groups freely generated by the irreducible components of the fibers, and $\mathbf Z^E$ is the sheaf whose stalks are freely generated by the nodes.  When a node is smoothed in $X$, the corresponding generator of the stalk of $\mathbf Z^E$ maps to zero under the generization map.

	Note that the first map in the first row of~\eqref{eqn:15} is injective because $M_S^{\rm gp} \to \pi_\ast M_X^{\rm gp}$ is an isomorphism by Lemma~\ref{lem:global-sections}.  A section of $\mathrm R^1 \pi_\ast \pi^\ast M_S^{\rm gp}$ induces isomorphism classes of line bundles on the components of $X$ and therefore has a well-defined multidegree.  This gives the vertical homomorphism in the middle term of diagram~\eqref{eqn:16}.  

By an explicit calculation, the map $\pi_\ast \overnorm M_{X/S}^{\rm gp} \to R^1 \pi_\ast \pi^\ast M_S^{\rm gp}$ commutes with the boundary map $\mathbf Z^E \to \mathbf Z^V$ computing the homology of the dual graph of $X$.  Therefore we recover the degree homomorphism by passing to cokernels, as indicated by the dashed arrow in~\eqref{eqn:16}.

	We write $\mathrm R^1 \pi_\ast (\pi^\ast M_S^{\rm gp})^{[0]}$ for the multidegree~$0$ part of $\mathrm R^1 \pi_\ast (\pi^\ast M_S^{\rm gp})$ and $\mathrm R^1 \pi_\ast (M_X^{\rm gp})^0$ for the degree zero part of $\mathrm R^1 \pi_\ast M_X^{\rm gp}$ (in other words, the kernels of the center and right vertical arrows of diagram~\eqref{eqn:16}).  It follows from the snake lemma that the map
\begin{equation} \label{eqn:53}
\mathrm R^1 \pi_\ast (\pi^\ast M_S^{\rm gp})^{[0]} \to \mathrm R^1 \pi_\ast (M_X^{\rm gp})^0
\end{equation}
is surjective with kernel $H_1(\tropfont X)$.
\end{proof}

\begin{corollary} \label{cor:quotient}
Let $X$ be a proper, vertical logarithmic curve over $S$.  Let $R^1 \pi_\ast (\pi^\ast \logGm)$ denote the sheaf on logarithmic schemes over $S$ whose value on a logarithmic scheme $T$ over $S$ is $R^1 \pi_\ast \pi^\ast M_T^{\rm gp}$ where $\pi$ abusively denotes the projection $X_T \to T$.  There is an exact sequence
\begin{equation*}
0 \to H_1(\tropfont X) \to {R^1 \pi_\ast (\pi^\ast \logGm)^{[0]}}^\dagger \to \LogPic^0(X/S) \to 0
\end{equation*}
where ${R^1 \pi_\ast (\pi^\ast \logGm)^{[0]}}^\dagger$ is the bounded monodromy, multidegree~$0$ subsheaf of $R^1 \pi_\ast (\pi^\ast \logGm)$.
\end{corollary}

\numberwithin{equation}{subsection}
\subsection{Semiabelian structure}
\label{sec:strathom}

We assume that $X$ is a family of logarithmic curves over $S$ with constant degeneracy.  That is, the characteristic monoid of $S$ is constant, as is the dual graph of $X$.  Let $X^\nu$ be the normalization of the nodes of $X$.  We have an exact sequence~\eqref{eqn:12}
\begin{equation} \label{eqn:12} 
0 \to T  \to \Pic^{[0]}(X/S) \to \Pic^{[0]}(X^\nu/S) \to 0
\end{equation}
where $T$ is the torus $\Hom(H_1(\tropfont X), \Gm)$ and $\tropfont X$ is the dual graph of $X$.  

We obtain a similar sequence with $\pi^\ast M_S^{\rm gp}$ in place of $\mathcal O_X^\ast$.  The short exact sequence~\eqref{eqn:28} yields the long exact sequence~\eqref{eqn:29}:
\begin{gather} 
	\label{eqn:28}
0 \to \mathcal O_{X^\nu}^\ast \to \nu^\ast \pi^\ast M_S^{\rm gp} \to \nu^\ast \pi^\ast \overnorm M_S^{\rm gp} \to 0
\\
	\label{eqn:29} 
\pi_\ast \nu_\ast \nu^\ast \pi^\ast M_S^{\rm gp} \to   \pi_\ast \nu_\ast \nu^\ast \pi^\ast \overnorm M_S^{\rm gp} \to   R^1 (\pi_\ast \nu_\ast) \mathcal O_{X^\nu}^\ast \to   R^1 (\pi_\ast \nu_\ast) \nu^\ast \pi^\ast M_S^{\rm gp} \to   R^1 (\pi_\ast \nu_\ast) \nu^\ast \pi^\ast \overnorm M_S^{\rm gp}
\end{gather}
As $M_S^{\rm gp} \to \overnorm M_S^{\rm gp}$ is surjective, so is $\pi_\ast \nu_\ast \nu^\ast \pi^\ast M_S^{\rm gp} \to \pi_\ast \nu_\ast \nu^\ast \pi^\ast \overnorm M_S^{\rm gp}$.  Furthermore, the components of $X^\nu$ are irreducible curves over $S$, so they have no first cohomology valued in $\overnorm M_S^{\rm gp}$ because it is torsion-free and constant on the fibers.  The sequence therefore reduces to an isomorphism between $R^1(\pi_\ast \nu_\ast) \mathcal O_{X^\nu}^\ast$ an $R^1(\pi_\ast \nu_\ast)\nu^\ast \pi^\ast M_S^{\rm gp}$.  That is, we have an isomorphism between $\Pic(X^\nu/S)$ and the functor $T \mapsto \Gamma(T, R^1 \pi_\ast \nu_\ast \nu^\ast \pi^\ast M_T^{\rm gp})$ on logarithmic schemes over~$S$.

By pullback, we therefore obtain a morphism~\eqref{eqn:95}:
\begin{equation} \label{eqn:95}
R^1 \pi_\ast \pi^\ast M_S^{\rm gp} \to R^1 \pi_\ast \nu_\ast \nu^\ast \pi^\ast M_S^{\rm gp} \simeq \Pic(X^\nu / S)
\end{equation}
The kernel of this morphism consists of those $M_S^{\rm gp}$-torsors on $X$ that are trivial when restricted to $X^\nu$.  Such a torsor is specified by transition functions in $M_S^{\rm gp}$ along the nodes of $X$ and the kernel may therefore be identified with $T^{\log} = \Hom(H_1(\tropfont X), \logGm)$.  

Passing to the multidegree~$0$ parts of $R^1 \pi_\ast \pi^\ast M_S^{\rm gp}$ and $\Pic(X^\nu / S)$, we get an exact sequence~\eqref{eqn:30}:
\begin{equation} \label{eqn:30}
0 \to \Hom(H_1(\tropfont X), \logGm) \to R^1 \pi_\ast (\pi^\ast M_S^{\rm gp})^{[0]} \to \Pic^{[0]}(X^\nu / S) \to 0
\end{equation}

\subsection{Local description of the homology action}
\label{sec:lochom}

We retain the assumptions of Section~\ref{sec:strathom} and permit further \'etale localization in $S$.

Because we have assumed the logarithmic structure of $S$ is constant, $\overnorm M_S^{\rm gp}$ is a constant sheaf of finitely generated free abelian groups.  Working locally in $S$, we can assume that $M_S^{\rm gp} \to \overnorm M_S^{\rm gp}$ is split, and therefore that $M_S^{\rm gp} \simeq \mathcal O_S^\ast \times \overnorm M_S^{\rm gp}$.  We fix one such splitting $m : \overnorm M_S^{\rm gp} \to M_S^{\rm gp}$ splitting the surjection $M_S^{\rm gp} \to \overnorm M_S^{\rm gp}$.  Using this we get a splitting $\pi^\ast M_S^{\rm gp} = \mathcal O_X^\ast \times \pi^\ast \overnorm M_S^{\rm gp}$, and therefore also a splitting~\eqref{eqn:101}:
\begin{equation} \label{eqn:101}
R^1 \pi_\ast \pi^\ast \logGm \simeq \Pic(X/S) \times \Hom(H_1(\tropfont X), \ologGm)
\end{equation}
We have used the canonical identification $R^1 \pi_\ast \pi^\ast \ologGm \simeq \Hom(H_1(\tropfont X), \ologGm)$.

Our goal in this section is to explain the map $H_1(\tropfont X) \to R^1 \pi_\ast \pi^\ast (\logGm)^{[0]}$ from Corollary~\ref{cor:quotient}, which is induced from $\mathbf Z^E \to R^1 \pi_\ast \pi^\ast \logGm$, in terms of this splitting.  Given $\alpha \in \Gamma(X, M_{X/S}^{\rm gp}) = \mathbf Z^E$, we write $\pi^\ast M_S(\alpha)$ for its image in $R^1 \pi_\ast \pi^\ast \logGm$.

We work out the pullback of $\pi^\ast M_S(\alpha)$ to the normalization $X^\nu$ of $X$ along its nodes.  We let $\tropfont X^\nu$ be the union of the stars of $\tropfont X$.  In a sense that we do not make precise here, this is the tropicalization of $X^\nu$ when $X^\nu$ is given the logarithmic structure pulled back from $X$.  Every section $\alpha$ of $\mathbf Z^E = \Gamma(X, \overnorm M_{X/S}^{\rm gp})$ can be lifted to a section $\tilde\alpha$ of $\overnorm M_{X^\nu}^{\rm gp}$, which can also be regarded as a piecewise linear function on $\tropfont X^\nu$.  Then $\nu^\ast \pi^\ast M_S^{\rm gp}(\alpha)$ is represented by the line bundle $\mathcal O_{X^\nu}(\tilde\alpha)$.  Note that the isomorphism class of $\mathcal O_{X^\nu}(\tilde\alpha)$ depends only on $\alpha$ because $\tilde\alpha$ is uniquely determined up to the addition of a constant from $\overnorm M_S^{\rm gp}$ on each component, and the addition of a constant only changes $\mathcal O_{X^\nu}(\tilde\alpha)$ by a line bundle pulled back from $S$. 

Suppose that $X_0 \subset X^\nu$ is a component and $\tropfont X_0$ is the corresponding component of $\tropfont X^\nu$.  Then~\eqref{eqn:97} computes $\mathcal O_{X_0}(\tilde\alpha)$:
\begin{equation} \label{eqn:97}
\mathcal O_{X_0}(\tilde\alpha) = \mathcal O_{X_0}(\sum \alpha_e D_e)
\end{equation}
The sum is taken over the edges $e$ of $\tropfont X_0$, with $D_e$ denoting the node of $X$ corresponding to $e$, and $\alpha_e$ denoting the slope of $\alpha$ along $e$ when $e$ is oriented away from the central vertex of $\tropfont X_0$.  In order to understand $\pi^\ast M_S(\alpha)$, we will need to see how the line bundles $\mathcal O_{X_i}(\tilde\alpha)$ on the components of $X^\nu$ are glued to one another.

Recall that $m : \overnorm M_S^{\rm gp} \to M_S^{\rm gp}$ denotes a fixed splitting.  For each $\delta \in \overnorm M_S^{\rm gp}$, we write $m^{\delta} : \mathcal O_X \to \mathcal O_X(\delta)$ for map sending $\lambda \in \mathcal O_X$ to $\lambda  m(\delta) \in \mathcal O_X(\delta)$.  Suppose that $D$ is a node of $X$ joining components $X_0$ and $X_1$ and let $e$ be the corresponding edge of $\tropfont X$.  Recall that we have~\eqref{eqn:103},
\begin{equation}
\begin{gathered} \label{eqn:103}
\mathcal O_{X_0}(\tilde\alpha) \big|_D = \mathcal O_D(\tilde\alpha(0)) \otimes \mathcal O_{X_0}(\alpha_e D) \big|_D \\
\mathcal O_{X_1}(\tilde\alpha) \big|_D = \mathcal O_D(\tilde\alpha(1)) \otimes \mathcal O_{X_1}(-\alpha_e D) \big|_D 
\end{gathered}
\end{equation}
where $\alpha_e$ is the slope of $\alpha$ along the edge $e$ of $\tropfont X$ corresponding to $D$, oriented from $0$ to $1$, and $\tilde\alpha(i) \in \overnorm M_S^{\rm gp}$ is the value of $\tilde\alpha$ on the vertex $i$ of $\tropfont X$.  Using the trivializations $m$, we obtain an isomorphism:
\begin{equation} \label{eqn:108}
m^{\alpha_e \delta} : \mathcal O_{X_1}(-\alpha_e D) \big|_{D} \xrightarrow{\sim} \mathcal O_{X_1}(-\alpha_e D + \alpha_e \delta) = \mathcal O_{X_0}(\alpha_e D) \big|_D
\end{equation}
Combined with the trivializations $m^{\tilde\alpha(0)}$ and $m^{\tilde\alpha(1)}$ of $\mathcal O_D(\tilde\alpha(0))$ and $\mathcal O_D(\tilde\alpha(1))$, we obtain an isomorphism~\eqref{eqn:124}:
\begin{equation} \label{eqn:124}
m^{\tilde\alpha(1) - \tilde\alpha(0) - \alpha_e \delta} : \mathcal O_{X_0}(\tilde\alpha) \big|_D \xrightarrow{\sim} \mathcal O_{X_1}(\tilde\alpha) \big|_D
\end{equation}
We glue $\mathcal O_{X_0}(\tilde\alpha)$ to $\mathcal O_{X_1}(\tilde\alpha)$ along $D$ by this isomorphism and repeat the same process for each edge of $\tropfont X$ to produce a line bundle $L(\alpha, m)$ on $X$.  Note that $L(\alpha,m)$ depends on $\tilde\alpha$ only up to a canonical isomorphism determined by the trivialization $m$, so we omit the dependence on $\tilde\alpha$ from the notation.

\setcounter{theorem}{\value{equation}}
\begin{remark}
We could have chosen $\tilde\alpha$ canonically to take the value $0$ at every vertex of $\tropfont X^\nu$, but the added flexibility will be useful in the proof of Proposition~\ref{prop:formula}.
\end{remark}

\begin{remark} 
If $\tilde\alpha$ were actually well-defined on $X_{01} = X_0 \cup_D X_1$ then $\tilde\alpha(1) - \tilde\alpha(0) = \alpha_e \delta_e$, where $\delta_e$ is the length of $e$.  Then $m^{\tilde\alpha(1)-\tilde\alpha(0)} = m^{\alpha_e \delta_e}$, so the isomorphisms above agree with the canonical identification $\mathcal O_{X_0}(\tilde\alpha) \big|_D = \mathcal O_{X_{01}}(\tilde\alpha) \big|_D = \mathcal O_{X_1}(\tilde\alpha) \big|_D$.
\end{remark}

\numberwithin{equation}{theorem}
\begin{proposition} \label{prop:formula}
The isomorphism~\eqref{eqn:101} sends $\pi^\ast M_S(\alpha)$ to $(L(\alpha, m), -\partial(\alpha))$.
\end{proposition}
\begin{proof}
The second component of the formula is implied by Lemma~\ref{lem:pairing}.  It can also be deduced from the argument below.

Let $\tilde {\tropfont X}$ be the universal cover of $\tropfont X$ and let $\rho : \tilde X \to X$ be the corresponding \'etale cover.  The fundamental group of $\tropfont X$ acts by deck transformations on $\tilde X$.  Since $H_1(\tilde{\tropfont X}) = 0$, we can find a lift of $\tilde\alpha$ to $\overnorm M_{\tilde X}$.  Without loss of generality, we can assume that the function on $X^\nu$ constructed before the statement of the proposition is induced from this $\tilde\alpha$ by restriction along some embedding $X^\nu \subset \tilde X$.

By construction $\rho^\ast L(\alpha, m)$ induces $\rho^\ast \pi^\ast M_S(\alpha)$.  We will prove that $\pi^\ast M_S(\alpha) = (L(\alpha, m), \partial(\alpha).\gamma)$ by comparing their transition data on the cover $\tilde X$.

If $\gamma \in \pi_1(\tropfont X)$ then $\gamma$ acts by deck transformations on $\tilde X$ and we have a canonical identification:
\begin{equation*}
\gamma^\ast \mathcal O_{\tilde X}(\tilde\alpha) = \mathcal O_{\tilde X}(\gamma^\ast \tilde\alpha) = \mathcal O_{\tilde X}(\tilde\alpha) \otimes \mathcal O_X(\partial(\alpha) . \gamma)
\end{equation*}
By definition, we have an inclusion $\mathcal O_X^\ast(\partial(\alpha).\gamma)$ inside $\pi^\ast M_S^{\rm gp}$ as the fiber over $-\partial(\alpha).\gamma \in \pi^\ast \overnorm M_S^{\rm gp}$.  This gives us a canonical identification $\gamma^\ast \rho^\ast \pi^\ast M_S(\tilde\alpha) = \rho^\ast \pi^\ast M_S(\tilde\alpha)$ that serves as a descent datum for $\rho^\ast \pi^\ast M_S(\tilde\alpha)$ from $\tilde X$ to $\pi^\ast M_S(\tilde\alpha)$ on $X$.

In terms of the splitting $m$, the map from $\mathcal O_{X}^\ast(\partial(\alpha).\gamma)$ to $\pi^\ast M_S^{\rm gp}$ is given by~\eqref{eqn:102}:
\begin{equation} \label{eqn:102}
(m^{-\partial(\alpha).\gamma}, -\partial(\alpha).\gamma) : \mathcal O_X^\ast(\partial(\alpha).\gamma) \to \mathcal O_X^\ast \times \pi^\ast \overnorm M_S^{\rm gp} 
\end{equation}
The second component of this formula gives the homomorphism $H_1(\tropfont X) \to \overnorm M_S^{\rm gp}$ that makes up the second component of~\eqref{eqn:101}.

The transition function for $L(\alpha, m)$ around the loop $\gamma$ is given by~\eqref{eqn:105}:
\begin{equation} \label{eqn:105}
\prod_e ( m^{-\delta_e\alpha_e} ) ^{\gamma_e} = m^{-\sum \alpha_e \gamma_e \delta_e}
\end{equation}
By definition of the intersection pairing, $\sum \alpha_e \gamma_e \delta_e = \partial(\alpha).\gamma$, so~\eqref{eqn:105} agrees with the first component of~\eqref{eqn:102}.
\end{proof}

\numberwithin{equation}{subsection}
\subsection{Tropicalizing the logarithmic Jacobian}
\label{sec:tropicalization}

For any proper, vertical logarithmic curve $X$ over $S$, we construct a morphism~\eqref{eqn:17}
\begin{equation} \label{eqn:17}
\LogPic^0(X/S) \to \TroJac(X/S)
\end{equation}
over $S$.  For each logarithmic scheme $T$ and object of $\LogPic^0(X/S)$, we must produce a section of $\TroJac(X/S)$.  By Corollary~\ref{cor:lfp}, it is sufficient to do this when $T$ is of finite type.  Under this assumption, the $T$-points of $\TroJac(X/S)$ are generization-compatible objects of $\TroJac(\tropfont X_t)$, for each geometric point $t$ of $T$.  We therefore describe the morphism first under the assumption that $X$ has constant dual graph over $S$ and $S$ has constant characteristic monoid (which covers the case of a geometric point) and then discuss generization.

If $X$ has constant dual graph and $S$ has constant characteristic monoid, we use the commutative diagram of exact sequences~\eqref{eqn:31}:
\begin{equation} \label{eqn:31} \vcenter{\xymatrix{
		0 \ar[r] & H_1(\tropfont X) \ar[r] \ar@{=}[dd] & {H^1(X,\pi^\ast \logGm)^{[0]}}^\dagger \ar[r] \ar[d] & \LogPic^0(X/S) \ar[r] \ar@{-->}[dd] & 0 \\
& & H^1(X, \pi^\ast \ologGm)^\dagger \ar@{=}[d] \\
0 \ar[r] & H_1(\tropfont X) \ar[r] & \Hom(H_1(\tropfont X), \ologGm)^\dagger \ar[r] & \TroJac(\tropfont X/S) \ar[r] & 0
}} \end{equation}
The first row of the diagram comes from Corollary~\ref{cor:quotient} and the bottom row is the definition of the tropical Jacobian from Section~\ref{sec:trojac}.  The identification between $H^1(X, \pi^\ast \ologGm)$ and $\Hom(H_1(\tropfont X), \ologGm)$ comes from the fact that $\overnorm M_S^{\rm gp}$ is a torsion-free sheaf:  since a smooth, proper curve has no nontrivial torsors under such a sheaf, any such torsor on a nodal curve can be trivialized on its normalization, and torsors under $\overnorm M_S^{\rm gp}$ on $X$ are determined uniquely by monodromy around the loops of the dual graph.  A unique dashed arrow exists by the universal property of the cokernel.

We show now that this morphism is compatible with generizations.  Any specialization $s \leadsto t$ in $\LogPic^0(X/S)$ can be represented by a map $T \to \LogPic^0(X/S)$ where $T$ is a strictly henselian valuation ring with some logarithmic structure, $s$ is its generic point, and $t$ is its closed point.  This map gives a logarithmic curve $X_T = X \mathop\times_S T$ over $T$ and an $M_{X_T}^{\rm gp}$-torsor $P$ on $T$ with bounded monodromy and degree~$0$.  Since $\LogPic^0(X/S)$ is the quotient of $H^1(X, \pi^\ast \logGm)$ by a discrete group, we can lift $P$ to a $\pi^\ast M_T^{\rm gp}$-torsor, $Q$, on~$X_T$.

We now have a commutative diagram~\eqref{eqn:18}:
\begin{equation} \label{eqn:18} \vcenter{\xymatrix{
H^1(X_t, \pi^\ast M_t^{\rm gp})^\dagger \ar[d] & H^1(X_T, \pi^\ast M_T^{\rm gp})^\dagger \ar@{..>}[r] \ar[d] \ar@{..>}[ddl] \ar[l] & H^1(X_s, \pi^\ast M_s^{\rm gp})^\dagger \ar@{..>}[d] \\
H^1(X_t, \pi^\ast \overnorm M_t^{\rm gp})^\dagger \ar@{=}[d] & H^1(X_T, \pi^\ast \overnorm M_T^{\rm gp})^\dagger \ar[l] \ar[r] & H^1(X_s, \pi^\ast \overnorm M_s^{\rm gp})^\dagger  \ar@{=}[d] \\
\Hom(H_1(\tropfont X_t), \overnorm M_t^{\rm gp})^\dagger \ar@{..>}[rr] & & \Hom(H_1(\tropfont X_s), \overnorm M_s^{\rm gp})^\dagger
}} \end{equation}
\vskip5mm
The commutativity of the trapezoid rendered in dotted arrows is precisely the compatibility of our map with generization.

\setcounter{theorem}{\value{equation}}
\numberwithin{equation}{theorem}
\begin{theorem} \label{thm:trop-ex-seq}
Let $X$ be a proper, vertical logarithmic curve over $S$.  There is an exact sequence:
\begin{equation} \label{eqn:25}
0 \to \Pic^{[0]}(X/S) \to \LogPic^0(X/S) \to \TroJac(X/S) \to 0
\end{equation}
\end{theorem}
\begin{proof} 
	Applying the snake lemma to diagram~\eqref{eqn:31}, and identifying $H^1(X, \pi^\ast \ologGm)  = \Hom(H_1(\tropfont X), \ologGm)$, we see that the exactness of~\eqref{eqn:25} is equivalent to that of~\eqref{eqn:62}:
	\begin{equation} \label{eqn:62}
		0 \to \Pic^{[0]}(X/S) \to {R^1 \pi_\ast (\pi^\ast \logGm)^{[0]}}^\dagger \to R^1 \pi_\ast ( \pi^\ast \ologGm)^\dagger \to 0
	\end{equation}
	We note that the bounded monodromy subgroup of $R^1 \pi_\ast(\pi^\ast \logGm)^{[0]}$ is simply the preimage of that in $\Hom(H_1(\tropfont X), \ologGm)$, and that the multidegree~$0$ subgroup of $\Pic(X/S)$ is the preimage of the multidegree~$0$ subgroup of $R^1 \pi_\ast (\pi^\ast \logGm)$.  Therefore it will be sufficient to demonstrate the exactness of~\eqref{eqn:63}:
	\begin{equation} \label{eqn:63}
		0 \to \Pic(X/S) \to R^1 \pi_\ast (\pi^\ast \logGm) \to \Hom(H_1(\tropfont X), \ologGm) \to 0
	\end{equation}
	This amounts to showing that, for each logarithmic scheme $T$ over $S$, the sequence~\eqref{eqn:24} is exact, where $Y = X \mathop\times_S T$:
	\begin{equation} \label{eqn:24}
		0 \to R^1 \pi_\ast \mathcal O_Y^\ast \to R^1 \pi_\ast \pi^\ast M_T^{\rm gp}  \to R^1 \pi_\ast \pi^\ast \overnorm M_T^{\rm gp} \to 0
	\end{equation} 
	The exact sequence~\eqref{eqn:24} arises from the long exact sequence~\eqref{eqn:23} associated with the short exact sequence~\eqref{eqn:22}:
	\begin{gather} 
	\label{eqn:22}
		0 \to \mathcal O_Y^\ast \to \pi^\ast M_T^{\rm gp} \to \pi^\ast \overnorm M_T^{\rm gp} \to 0 \\
	\label{eqn:23}
		\pi_\ast \pi^\ast M_T^{\rm gp} \to \pi_\ast \pi^\ast \overnorm M_T^{\rm gp} \to R^1 \pi_\ast \mathcal O_Y^\ast \to R^1 \pi_\ast \pi^\ast M_T^{\rm gp} \to R^1 \pi_\ast \pi^\ast \overnorm M_T^{\rm gp} \to R^2 \pi_\ast \mathcal O_Y^\ast
	\end{gather}
	We have $R^2 \pi_\ast \mathcal O_Y^\ast = 0$ by Tsen's theorem.  As $\pi^\ast \overnorm M_T^{\rm gp}$ is a constant sheaf on the fibers and $M_T^{\rm gp} \to \overnorm M_T^{\rm gp}$ is surjective, the map $\pi_\ast \pi^\ast M_T^{\rm gp} \to \pi_\ast \pi^\ast \overnorm M_T^{\rm gp}$ is surjective as well.  This gives the exactness of~\eqref{eqn:24} and completes the proof.
\end{proof}

\begin{corollary} \label{cor:log-bdd}
Let $X$ be a proper, vertical logarithmic curve over $S$.  For each degree $d$, the sheaf $\LogPic^d(X/S)$ and the stack $\bLogPic^d(X/S)$ are bounded.
\end{corollary}
\begin{proof}
As $\LogPic^d(X/S)$ is a torsor under $\LogPic^0(X/S)$, it is sufficient to prove the corollary for $d = 0$.  By the exact sequence~\eqref{eqn:25}, $\LogPic^0(X/S)$ is a $\Pic^{[0]}(X/S)$-torsor over $\TroJac(X/S)$.  As both $\Pic^{[0]}(X/S)$ and $\TroJac(X/S)$ are bounded --- in the latter case by Corollary~\ref{cor:trop-bdd} --- it follows that $\LogPic^0(X/S)$ is also bounded.  

Finally, we note that $\bLogPic^d(X/S)$ is isomorphic, locally in $S$, to $\LogPic^d(X/S) \times \mathrm B\logGm$, so the conclusion follows from the boundedness of $\mathrm B\logGm$.
\end{proof}

\subsection{The valuative criterion for properness}
\label{sec:val-prop}

\begin{theorem} \label{thm:valcrit}
Let $X$ be a proper, vertical\kern1pt\ logarithmic curve over $S$.  Then $\bLogPic(X/S) \to S$ satisfies the valuative criterion for properness (Theorem~\ref{thm:val-prop}) over $S$.
\end{theorem}
\begin{proof}

Let $R$ be a valuation ring with a valuative logarithmic structure and with field of fractions $K$.  We consider a lifting problem~\eqref{eqn:8} and show it has a unique solution:
\begin{equation} \label{eqn:8} \vcenter{\xymatrix{
	\Spec K \ar[r] \ar[d] & \bLogPic(X/S) \ar[d] \\
	\Spec R \ar[r] \ar@{-->}[ur] & S
}} \end{equation} 
These data give us a logarithmic curve $X_R$ over $R$ and a $M_{X_K}^{\rm gp}$-torsor $P_K$ on $X_K$ with bounded monodromy.  Let $j : X_K \to X_R$ denote the inclusion.

	By Theorem~\ref{thm:curve-val} and Corollary~\ref{cor:r1}, we have $R^1 j_\ast M_{X_K}^{\rm gp} = 0$ and $M_{X_R}^{\rm gp} \to j_\ast M_{X_K}^{\rm gp}$ is an isomorphism.  These imply that the morphism of group stacks $\mathrm B M_{X_R}^{\rm gp} \to j_\ast \mathrm B M_{X_K}^{\rm gp}$ induces isomorphisms on sheaves of isomorphism classes and sheaves of automorphisms, hence is an equivalence.  Pushing forward to $S$ gives~\eqref{eqn:59}:
\begin{equation} \label{eqn:59}
\pi_\ast \mathrm B M_{X_R}^{\rm gp} = j_\ast \pi_\ast \mathrm B M_{X_K}^{\rm gp} 
\end{equation}
But a section of $j_\ast \pi_\ast \mathrm B M_{X_K}^{\rm gp}$ is a commutative square~\eqref{eqn:8} (ignoring bounded monodromy), and a section of $\pi_\ast \mathrm B M_{X_R}^{\rm gp}$ is a diagonal arrow lifting it (again ignoring bounded monodromy).

	To conclude, we check that if $P_K$ has bounded monodromy then so does $P_R$.  Let $\tropfont X_R$ be the tropicalization of $X_R$ and let $\tropfont X_K$ be the tropicalization of $X_K$.  Let $\overnorm P_K$ and $\overnorm P_R$ be the $\PL$-torsors on $\tropfont X_K$ and $\tropfont X_R$ induced by $P_K$ and $P_R$.  Then the pullback of $\overnorm P_K$ along the contraction map $\tropfont X_R \to \tropfont X_K$ coincides with the extension of $\overnorm P_R$ along the map $\overnorm M_R^{\rm gp} \to \overnorm M_K^{\rm gp}$.  That is, $\overnorm P_R \otimes_{\overnorm M_R^{\rm gp}} \overnorm M_K^{\rm gp}$ descends along $\tropfont X_R \to \tropfont X_K$, so it must have trivial monodromy around all loops of $\tropfont X_R$ collapsed in $\tropfont X_K$.  This means that $\overnorm P_R$ has bounded monodromy.
\end{proof}

\begin{remark}
One can also argue using Proposition~\ref{prop:logpic-bdd-mono}.

Since $K$ has a valuative logarithmic structure, we may replace $X_R$ with a semistable model so that $P_K$ is represenatble by an invertible sheaf on $X_K$, and therefore by a divisor $D_K$.  Over an \'etale extension $K'$ of $K$ it is possible to represent $D_{K'}$ as a sum of sections $\sigma_i : \Spec K' \to X_{K'}$.  Let $R'$ be the integral closure of $R$ in $K'$.  For each $i$ there is a \emph{universal} choice of semistable model $X'_{R'}$ such that $\sigma_i$ extends to a section of the strict locus of $X'_{R'}$ over $R'$.  Therefore there is a universal choice of semistable model $X''_{R'} \to X_{R'}$ such that the closure of $D_{K'}$ lies in the strict locus of $X''_{R'}$.  Since this model is characterized by a universal property, it descends to a semistable model $X''_R \to X_R$ such that the closure $D_R$ of $D_K$ lies in the strict locus over $R$.  But then $\mathcal O_{X''_R}(D_R)$ represents a $M_{X_R}^{\rm gp}$-torsor extending $P_K$.  Logarithmic line bundles that are representable by invertible sheaves have bounded monodromy.

This proves the universal closedness part of the valuative criterion, but the choice of $\mathcal O_{X''_R}(D_R)$ depends on the choice of $D_K$, so further argument is necessary to prove separatedness.
\end{remark}

\begin{corollary} \label{cor:valcrit}
The projection $\LogPic(X/S) \to S$ satisfies the valuative criterion for properness.
\end{corollary}
\begin{proof}
Locally in $S$ the projection $\bLogPic(X/S) \to \LogPic(X/S)$ has a section making $\LogPic(X/S)$ into a $\logGm$-torsor over $S$.  But $\logGm$ satisfies the valuative criterion for properness, so $\LogPic(X/S)$ does as well.
\end{proof}

Once we have demonstrated the algebraicity of $\LogPic^d$, we will be able to conclude that it is proper in Corollary~\ref{cor:logpic-proper}.

\subsection{Existence of a smooth cover}
\label{sec:cover}

\begin{definition} \label{def:log-sp}
We call a presheaf $X$ on logarithmic schemes a \emph{logarithmic space} if there is a logarithmic scheme $U$ and a morphism $U \to X$ that is surjective on valuative geometric points and representable by logarithmically smooth logarithmic schemes.
\end{definition}

\begin{theorem} \label{thm:algebraic}
Let $X$ be a proper logarithmic curve over $S$.  Then there is a logarithmic scheme and a universally surjective, logarithmically smooth morphism to $\bLogPic(X)$ that is representable by logarithmic spaces.
\end{theorem}
\begin{proof}
We consider a map $T \to S$ that is a composition of \'etale maps and logarithmic modifications.  Let $Y$ be a semistable model of $X \mathop\times_S T$ over $T$.  Then $\bPic(Y/T)$ is representable by an algebraic stack over $T$.  When equipped with the logarithmic structure pulled back from $T$, we have a morphism to $\bLogPic(X/S)$:
\begin{equation} \label{eqn:10}
\bPic(Y/T) \to \bLogPic(Y/T) \to \bLogPic(X_T/T) \to \bLogPic(X/S)
\end{equation}
We will argue that these maps are a logarithmically \'etale cover of $\LogPic(X/S)$ using the following lemma:

\begin{lemma}  \label{lem:rep-log-et}
For any logarithmic curve $Y$ over $T$, the map $\bPic(Y/T) \to \bLogPic(Y/T)$ is representable by logarithmic spaces and is logarithmically \'etale.
\end{lemma}

\noindent Granting this lemma, we complete the proof of Theorem~\ref{thm:algebraic}.

We show first of all that $\bPic(Y/T) \to \bLogPic(X/S)$ is representable by logarithmic schemes and is logarithmically \'etale.  The first arrow in the sequence~\eqref{eqn:10} has these properties by Lemma~\ref{lem:rep-log-et}, the second is an isomorphism by Corollary~\ref{cor:subdiv-curve} and Lemma~\ref{lem:subdiv-bounded}, and the last arrow is the base change of the logarithmically \'etale morphism $T \to S$, by definition.  Their composition is therefore representable by logarithmic schemes and is logarithmically \'etale.

Proposition~\ref{prop:logpic-bdd-mono} implies that, as $T$ and $Y$ vary over logarithmic modifications of \'etale covers of $S$ and semistable models of $X_T$, respectively, the maps $\bPic(Y/T) \to \bLogPic(X/S)$ are surjective on valuative geometric points.  This completes the proof.
\end{proof}

\begin{proof}[Proof of Lemma~\ref{lem:rep-log-et}]
	We wish to show that, for any logarithmic scheme $T'$ over $T$, and any logarithmic line bundle $L$ on $Y' = Y \mathop\times_T T'$, there is a universal logarithmic scheme $U$ over $T'$ and lift of $L \big|_U$ to an invertible sheaf on $Y \mathop\times_T U$.  Without loss of generality, we can assume that $T = T'$ to lighten the notation.

	The logarithmic line bundle $L$ on $Y$ is a $\logGm$-torsor.  It induces a $\tropGm$-torsor $\overnorm L$ via the map $M_Y^{\rm gp} \to \overnorm M_Y^{\rm gp}$, and this torsor obstructs lifting $L$ to a $\Gm$-torsor, in the sense that sections of $\overnorm L$ are in natural bijection with lifts.  Let us write $\pi_\ast \overnorm L$ for the presheaf on logarithmic schemes over $T$ whose value on $T'$ is $\Gamma(Y', \overnorm L')$, where $Y'$ is the base change of $Y$ to $T'$ and $\overnorm L'$ is the pullback of $\overnorm L$ to a $\tropGm$-torsor on $Y'$.  We need to demonstrate that $\pi_\ast \overnorm L$ is representable by a logarithmic space that is strict and logarithmically \'etale over $T$.

	It is sufficient to prove that $\pi_\ast \overnorm L$ is representable by a logarithmic space that is logarithmically \'etale over $T$ in an \'etale neighborhood of each geometric point $t$ of $T$.  We will therefore assume that $T$ is affine and has a global chart by $\overnorm M_{T,t}$.  Since $\bLogPic(Y/T)$ is locally of finite presentation (Proposition~\ref{prop:lfp}), there is a morphism $T \to T_0$, a logarithmic curve $\pi_0 : Y_0 \to T_0$, and a logarithmic line bundle $L_0$ on $Y_0$ that pulls back to $L$.  Then ${\pi_0}_\ast \overnorm L_0$ pulls back to $\pi_\ast \overnorm L$ by proper base change.  If ${\pi_0}_\ast \overnorm L_0$ is representable by a logarithmically \'etale logarithmic space over $T_0$ then $\pi_\ast \overnorm L$ will be representable by its base change to $T$.  Hence we may replace $T$ and $Y$ by $T_0$ and $Y_0$.  We therefore assume without loss of generality that $T$ is of finite type in addition to being affine. Since $T$ is of finite type, and our problem is \'etale local on $T$, we can also assume that $T$ is an atomic neighborhood of $t$.  That is, we assume $T$ has a global chart by $\overnorm M_{T,t}$ and the logarithmic stratum of $T$ containing $t$ is connected.  After further \'etale localization, we assume as well that the dual graph of $Y$ is constant on the stratum of $T$ that contains $t$.

	We give a combinatorial interpretation of $\pi_\ast \overnorm L$.  Let $Y_t$ be the fiber of $Y$ over $t$ and let $\tropfont Y$ be its dual graph.  We write $\PL_t$ for the sheaf on $\tropfont Y$ that corresponds with $\overnorm M_{Y_t}^{\rm gp}$ on $Y_t$.  Since $\overnorm L$ is constant on the logarithmic strata of $Y$ (since these strata are normal and $\overnorm M_Y^{\rm gp}$ is constant and torsion free), it descends to a $\PL_t$-torsor $\tropfont L_t$ on $\tropfont Y$.  For each geometric point $t'$ of $T$, we have a homomorphism $\overnorm M_{T,t}^{\rm gp} \to \overnorm M_{T,t'}^{\rm gp}$.  This induces a homomorphism $\PL_t \to \PL_{t'}$ on $\tropfont Y$, and this induces a $\PL_{t'}$-torsor, $\tropfont L_{t'}$ on $\tropfont Y$.

	To specify a section of $\pi_\ast \overnorm L$ at $t'$ is the same as to give a section of $\tropfont L_{t'}$.  To specify a section of $\pi_\ast \overnorm L$ on some logarithmic scheme $T'$ over $T$ is the same as to give a section of $\tropfont L_{t'}$ for every geometric point $t'$ of $T'$, in a fashion that is compatible with the maps $\tropfont L_{t'} \to \tropfont L_{t''}$ associated with geometric specializations $t'' \leadsto t'$.  In other words, we may think of $\tropfont L$ as a sheaf on the constant family $\tropfont Y \times T$, and $\pi_\ast \overnorm L = \rho_\ast \tropfont L$, where $\rho : \tropfont Y \times Y \to T$ is the projection.

	Since $\tropfont L$ is a sheaf, we have an exact sequence:
	\begin{equation} \label{eqn:135}
		0 \to \rho_\ast \tropfont L \to \prod_v \tropfont L_v \to \prod_e \tropfont L_e
	\end{equation}
	Choose trivializations $\tropfont L_v \simeq \PL_v$ and $\tropfont L_e \simeq \PL_e$ over each vertex $v$, and each edge $e$, of $\tropfont Y$.  If $e$ is an edge of $\tropfont Y$ connecting vertices $v$ and $w$ then $\PL_e$ is the subgroup of $\PL_v \times \PL_w$ consisting of pairs $(f,g)$ such that $f - g$ lies in the subgroup $\mathbf Z \delta_e$ generated by the length $\delta_e$ of $e$.  Combining this observation with our trivializations and the isomorphism $\pi_\ast \overnorm L \simeq \rho_\ast \tropfont L$, the exact sequence~\eqref{eqn:135} translates into a cartesian square:
	\begin{equation*} \xymatrix{
			\pi_\ast \overnorm L \ar[r] \ar[d] & \displaystyle \prod_{e \in E} \mathbf Z \delta_e \ar[d] \\
			\displaystyle \prod_{v \in V} \tropGm \ar[r] & \displaystyle \prod_{e \in E} \tropGm
	} \end{equation*}
	Thus $\pi_\ast \overnorm L$ is a fiber product of logarithmic spaces that are logarithmically \'etale over $T$, hence is a logarithmic space that is logarithmically \'etale over $T$.
\end{proof}

\begin{corollary} \label{cor:sheaf-algebraic}
There is a logarithmic scheme $W$ and a cover of $W \to \LogPic(X/S)$ that is logarithmically smooth and representable by logarithmic spaces.
\end{corollary}
\begin{proof}
Locally in $S$, we can identify $\bLogPic(X/S) = \LogPic(X/S) \times \mathrm B\logGm$ by identifying $\LogPic(X/S)$ with the sheaf of logarithmic line bundles on $X$ trivialized over a section.  A section $\LogPic(X/S) \to \bLogPic(X/S)$ makes $\LogPic(X/S)$ into a $\logGm$-bundle over $\bLogPic(X/S)$.  If $U \to \bLogPic(X/S)$ is a logarithmically smooth cover by a logarithmic scheme, then its pullback is a logarithmically smooth cover $W \to \LogPic(X/S)$, and $W$ is a $\logGm$-torsor over the logarithmic scheme $U$, hence a logarithmic space.  Replacing $W$ by a logarithmically smooth cover, we can arrange for $W$ to be a logarithmic scheme as required.
\end{proof}

\begin{corollary} \label{cor:diag}
The diagonals of $\LogPic(X/S)$ and $\bLogPic(X/S)$ are representable by logarithmic spaces.
\end{corollary}
\begin{proof}
Let $Z$ be $\LogPic(X/S)$ or $\bLogPic(X/S)$.  We have a logarithmically smooth cover $U \to Z$ that is representable by logarithmic spaces.  We wish to show that $W = V \mathop\times_{Z \times Z} Z$ is representable by logarithmic spaces whenever $V$ is a logarithmic scheme with two maps to $Z$.  But 
\begin{equation*}
W \mathop\times_{Z \times Z} (U \times U) = (V \mathop\times_{Z \times Z} (U \times U)) \mathop\times_{U \times U} (U \mathop\times_Z U)
\end{equation*}
is the fiber product of the logarithmic space $U \mathop\times_Z U$ with the logarithmic space $V \mathop\times_{Z \times Z} (U \times U)$ over the logarithmic scheme $U \times U$, hence is a logarithmic space.
\end{proof}

\subsection{Representability of the diagonal}
\label{sec:diagonal}

Our algebraicity result is slightly stronger for $\LogPic$.

\begin{theorem} \label{thm:diag}
The diagonal of $\LogPic(X/S)$ over $S$ is representable by finite morphisms of logarithmic schemes.
\end{theorem}

In other words, we are to show that if $X$ is a proper, vertical logarithmic curve over $S$ with two logarithmic line bundles $L$ and $L'$ then there is a universal logarithmic scheme $T$ over $S$ such that $L_T \simeq L'_T$ and, moreover, the underlying scheme of $T$ is finite over that of $S$.  This assertion only depends on the difference between $L$ and $L'$ in the group structure of $\LogPic$, so we can assume $L'$ is trivial.  The assertion is also local in the strict \'etale topology on $S$, so we freely replace $S$ by an \'etale cover.  By Corollary~\ref{cor:valcrit}, the diagonal of $\LogPic^d(X/S)$ satisfies the valuative criterion for properness, so it will suffice to prove that the diagonal is schematic, quasicompact, and locally quasifinite.  In fact, morphisms of algebraic spaces that are separated and locally quasifinite are schematic \cite[\href{http://stacks.math.columbia.edu/tag/03XX}{Tag 03XX}]{stacks-project}, so we only need to show the diagonal is representable by algebraic spaces, locally quasifinite, and quasicompact.

\begin{lemma}
The relative diagonal of $\LogPic(X/S)$ over $S$ is quasicompact.
\end{lemma}

\begin{proof}
It is sufficient to demonstrate that $\LogPic^0(X/S)$ has quasicompact diagonal over $S$.  This assertion is local in the constructible topology on $S$, so we assume that the dual graph of $X$ is constant over $S$ and that $\overnorm M_S$ is a constant sheaf on $S$.  Let $\tropfont X$ denote the tropicalization of $X$.  In this situation, Corollary~\ref{cor:quotient} gives an \'etale cover of $\LogPic^0(X/S)$ by $V = {R^1 \pi_\ast(\pi^\ast \logGm)^{[0]}}^\dagger$.  By \'etale descent, it is sufficient to show that $V\mathop\times_{\LogPic^0(X/S)} V \to V \times V$ is quasicompact.

We can recognize this map as the base change to $V \times V$ along $V \times V \to V : (v,w) \mapsto v-w$ of $H_1(\tropfont X) = \ker(V \to \LogPic^0(X/S)) \to V$.  It therefore suffices to demonstrate that $H_1(\tropfont X) \to V$ is quasicompact.

Since $\Pic^{[0]}(X/S)$ is separated, Theorem~\ref{thm:trop-ex-seq} implies that the map $\LogPic^0(X/S) \to \TroJac(X/S)$ is separated, and in particular quasiseparated.  By base change, the compatibility square~\eqref{eqn:31} then shows that $V \to \Hom(H_1(\tropfont X), \ologGm)^\dagger$ is also quasiseparated.  Therefore the quasicompactness of $H_1(\tropfont X) \to V$ follows from the quasicompactness of $H_1(\tropfont X) \to \Hom(H_1(\tropfont X), \ologGm)^\dagger$, which is Corollary~\ref{cor:trop-diag-qcpt}.  
\end{proof}

\begin{lemma} \label{lem:trop-diag-ft}
	Let $S$ be a logarithmic scheme, let $\tropfont X$ be a compact tropical curve over $S$.  Then the zero section of $\TroJac(\tropfont X/S)$ is representable by logarithmic schemes of finite type that are, \'etale-locally in $S$, affine over $H_1(\tropfont X) \mathop\times_S \TroJac(\tropfont X/S)$ (viewing $H_1(\tropfont X)$ as an \'etale sheaf over $S$ and identifying it with its espace \'etal\'e).
\end{lemma}
\begin{proof}
	Suppose we are given a section $S \to \TroJac(\tropfont X/S)$.  Let $Z$ be the pullback of the zero section of $\TroJac(\tropfont X/S)$ to $S$.  We wish to show $Z$ is representable by a logarithmic scheme that is of finite type over $S$ and is affine over $H_1(\tropfont X)$.  This is an \'etale-local assertion on $S$, so we can work locally in $S$ and assume $S \to \TroJac(\tropfont X/S)$ can be lifted to $S \to \Hom(H_1(\tropfont X), \ologGm)$ of $S \to \TroJac(\tropfont X/S)$.  We can realize $Z$ as the pullback of $\partial : H_1(\tropfont X) \to \Hom(H_1(\tropfont X), \ologGm)$ to $S$.  We therefore have~\eqref{eqn:112}:
\begin{equation} \label{eqn:112}
Z = S \mathop\times_{\Hom(H_1(\tropfont X), \ologGm)} H_1(\tropfont X) = \bigcup_{\alpha \in H_1(\tropfont X)} S \mathop\times_{\Hom(H_1(\tropfont X), \ologGm)} \{ \partial(\alpha) \}
\end{equation}
	Here the union ranges over local sections $\alpha$ of $H_1(\tropfont X)$ over $S$.  But $\partial$ is quasicompact by Corollary~\ref{cor:trop-diag-qcpt}, so, locally in $S$, we only need to consider finitely many of the $\alpha \in H_1(\tropfont X)$.  We can therefore assume there is a single $\alpha$.  We write $Z_\alpha$ for the component of $Z$ that corresponds.

Locally in $S$, we can choose a surjection from a finitely generated free abelian group $A$ onto $H_1(\tropfont X)$.  This induces an embedding $\Hom(H_1(\tropfont X), \ologGm) \to \Hom(A, \ologGm)$, which is a product of copies of $\ologGm$.  Applying Proposition~\ref{prop:ologGm-zero} on each copy, we get the result.
\end{proof}

\begin{corollary} \label{cor:affine-ft}
	Let $S$ be a logarithmic scheme and let $X$ be a proper, vertical logarithmic curve over $S$.  Then the zero section of $\LogPic(X/S)$ is representable by logarithmic schemes that are of finite type and, \'etale-locally in $S$, affine over $H_1(\tropfont X) \mathop\times_S \LogPic(X/S)$, where $\tropfont X$ is the tropicalization of $X$.
\end{corollary}
\begin{proof}
	We use the exact sequence from Theorem~\ref{thm:trop-ex-seq}.  By Lemma~\ref{lem:trop-diag-ft}, the map from $\Pic^{[0]}(X/S)$ to $\LogPic^0(X/S)$ is representable by logarithmic schemes of finite type and affine over $H_1(\tropfont X)$.  But the zero section of $\Pic^{[0]}(X/S)$ is a closed embedding because $\Pic^{[0]}(X/S)$ is separated and schematic over $S$; in particular, it is affine and of finite type.  We deduce that the zero section of $\LogPic(X/S)$ is of finite type and is affine over $H_1(\tropfont X)$.
\end{proof}

\begin{proof}[Proof of Theorem~\ref{thm:diag}]
	The diagonal of $\LogPic(X/S)$ is the base change of the embedding of the zero section, so it is sufficient to demonstrate that the embedding of the zero section is finite.  We have seen that it is of finite type in Corollary~\ref{cor:affine-ft}.  Corollary~\ref{cor:affine-ft} also shows that it is affine over an algebraic space that is \'etale over $S$.  In particular, it has affine fibers.  But it also satisfies the valuative criterion for properness by Corollary~\ref{cor:valcrit}.  Therefore the fibers are both affine and proper, hence are finite.  Since the zero section is also of finite type, it is therefore quasifinite.  A quasifinite separated morphism is schematic \cite[Tag 03XX]{stacks-project}, so the zero section is schematic.  Since it is quasifinite and proper, it is finite \cite[Tag 02LS]{stacks-project}.
\end{proof}

\begin{corollary} \label{cor:logpic-proper}
For each integer $d$, the sheaf $\LogPic^d(X/S)$ and the stack $\bLogPic^d(X/S)$ are proper over $S$.
\end{corollary}
\begin{proof}
We have shown that $\LogPic^d(X/S)$ has finite diagonal by Theorem~\ref{thm:diag}, is bounded by Corollary~\ref{cor:log-bdd}, and satisfies the valuative criterion by Theorem~\ref{thm:valcrit}.  The properness of $\bLogPic^d(X/S)$ follows because it is a gerbe banded by the proper group $\logGm$ over $\LogPic^d(X/S)$.
\end{proof}

\begin{remark}
Of course, the statements demonstrated here are all well-known for the Picard group of a proper family of smooth curves.  The separatedness of the Picard group implies that the diagonal is finite, by definition.  However, the Picard stack of a proper family of smooth curves is not separated because line bundles have the non-proper group $\Gm$ acting as automorphisms.  This is resolved in the logarithmic Picard stack because the automorphism group of a logarithmic line bundle is the \emph{proper} logarithmic group, $\logGm$.

For a proper family of nodal curves, the Picard group can fail to be separated because of the possibility of twists by components of the special fiber.
\end{remark}

\subsection{Smoothness}
\label{sec:smooth}

\begin{theorem} \label{thm:smooth}
Let $X$ be a logarithmic curve over $S$.  Then $\LogPic(X/S)$ is logarithmically smooth.
\end{theorem}

There are two parts to smoothness:  the infinitesimal criterion and local finite presentation.  Local finite presentation was addressed in Proposition~\ref{prop:lfp}.

\begin{lemma}
$\bLogPic(X/S)$ satisfies the infinitesimal criterion for smoothness over $S$.  Its logarithmic tangent stack is $\pi_\ast \BGa$, meaning isomorphism classes of deformations are a torsor under $H^1(X, \mathcal O_X)$ and automorphisms are in bijection with $H^0(X, \mathcal O_X)$.
\end{lemma}
\begin{proof}
Consider a lifting problem~\eqref{eqn:33}
\begin{equation} \label{eqn:33} \vcenter{\xymatrix{
T \ar[r] \ar[d] & \bLogPic(X/S) \ar[d] \\
T' \ar@{-->}[ur] \ar[r] & S
}} \end{equation}
in which $T'$ is a strict infinitesimal square-zero extension of $T$.  The lower horizontal arrow gives a logarithmic curve $X'$ over $T'$ with fiber $X$ over $T$, and the upper horizontal arrow gives a logarithmic line bundle $L$ on $X$.  We wish to extend this to $T'$.  It is sufficient to assume that $T'$ is a square-zero extension with ideal $J$.

Let $\overnorm L$ be the $\overnorm M_X^{\rm gp}$-torsor induced from $L$.  As this is a torsor under an \'etale sheaf, and the \'etale sites of $X$ and $X'$ are identical, $\overnorm L$ extends uniquely to $\overnorm L'$.  We therefore assume $\overnorm L'$ is fixed.  We note that the bounded monodromy condition for a putative $L'$ extending $L$ depends only on $\overnorm L'$, and is equivalent to that for $\overnorm L$, hence is automatically satisfied.

We wish to show that $\overnorm L'$ can be lifted to an $M_{X'}^{\rm gp}$-torsor.  Locally in $X$ there is no obstruction to extending $L$ to $L'$.  If we take any two local extensions of $L$, their difference $L' \otimes {L''}^\vee$ is a $M_{X'}^{\rm gp}$-torsor whose restriction to $X$ is trivialized, as is its induced $\overnorm M_{X'}^{\rm gp}$-torsor.  Therefore the $M_{X'}^{\rm gp}$-torsor $L' \otimes {L''}^\vee$ is induced from a uniquely determined $\mathcal O_{X'}^\ast$-torsor exending the trivial one from $X$.

It follows that extensions of $L$ form a gerbe on $X$ banded by $\mathcal O_X \otimes J$.  Obstructions to producing a lift --- in other words, obstructions to producing a section of this gerbe --- lie in $H^2(X, \mathcal O_X \otimes J)$, which vanishes locally in $S$ because $X$ is a curve over $S$.  By the cohomological classification of banded gerbes, deformations form a torsor under $H^1(X, \mathcal O_X \otimes J)$ and automorphisms are in bijection with $H^0(X, \mathcal O_X \otimes J)$.

To get the logarithmic tangent space, we take a trivial extension $T'$ of $T$ by $J = \mathcal O_T$.  Then isomorphisms between a given extension $L'$ and the trivial extension form a torsor under the group of automorphisms of the trivial extension, $\mathbf G_a$.
\end{proof}

\numberwithin{equation}{subsection}
\subsection{Tropicalizing the logarithmic Picard group}
\label{sec:trop-pic}

Let $X$ be a proper, vertical logarithmic curve over $S$ and let $\tropfont X$ denote the tropicalization of $X$.  We construct a tropicalization map:
\begin{equation} \label{eqn:36}
\bLogPic(X/S) \to \bTroPic(\tropfont X/S)
\end{equation}

Since $\bTroPic(\tropfont X/S)$ is locally constant on the logarithmic strata of $S$, our strategy will be to construct~\eqref{eqn:36} stratumwise and then show its compatibility with generization. 

Assume first that $S$ has constant characteristic monoid and that the dual graph of $X$ is constant over $S$.  Under these assumptions, we have an anticontinuous tropicalization map $t : X \to \tropfont X$.

Suppose that $Q$ is a $M_X^{\rm gp}$-torsor on $X$.  Let $\tropfont U \to \tropfont X$ be a local isomorphism and let $U = t^{-1} \tropfont U$.  Let $\NS(U)$ denote the N\'eron--Severi group of $U$.  Then $\NS$ is a functor on finite $X$-schemes and we observe that the sheaf $\mathsf V$ on $\tropfont X$ (whose sections are members of the free abelian group generated by the vertices) is isomorphic to $t_\ast \NS$.  Combined with Lemma~\ref{lem:log-pwl} and the exact sequence in the middle column of~\eqref{eqn:34}, this proves Proposition~\ref{prop:log-lin}:

\setcounter{theorem}{\value{equation}}
\begin{proposition} \label{prop:log-lin}
Let $X$ be a logarithmic curve over $S$, where $S$ has constant characteristic monoid and $X$ has constant dual graph.  Let $\tropfont X$ be the tropicalization of $X$.  Then the sheaf of linear functions $\L$ on $\tropfont X$ is quasi-isomorphic to $t_\ast [ \overnorm M_X^{\rm gp} \to \NS ]$.
\end{proposition}

For the reader's convenience, we recall a few facts about strictly commutative $2$-groups.  Suppose that we have a short exact sequence of sheaves on a site $T$:
\begin{equation*}
\xymatrix{
0 \ar[r] & K \ar[r]^{i} &  G \ar[r]^{q} & H \ar[r] & 0 }
\end{equation*}
This induces a sequence of morphisms of $2$-groups: 
\begin{equation*}
\xymatrix{
\cdots \ar[r]  & H \ar[r] & \mathrm BK \ar[r] &  \mathrm BG \ar[r] & \mathrm BH \ar[r] & \cdots}
\end{equation*} 
The connecting homomorphism $H \to BK$ takes an element $h \in H$ to its preimage $\{g \in G: q(g) = h \}$ in $G$; this is a coset of $K$ and, in particular, a $K$-torsor when the embedding in $G$ is forgotten.  The remaining maps are given by extension of structure group: a $G$-torsor $P$ induces an $H$-torsor $P/K$, and a $K$-torsor $L$ similarily induces a $G$-torsor $L \times^K G = (L \times G)/K$, with the quotient taken by the antidiagonal action of $K$. 

The sequence is `long exact' in the sense that $H$ is the kernel of $\mathrm BK \to \mathrm BG$, and $\mathrm BK$ is the kernel of $\mathrm BG \to \mathrm BH$.  Explicitly, a $K$-torsor with a trivialization of its induced $G$-torsor is the same thing as a $K$-torsor with a $K$-equivariant map to $G$, i.e., a coset of $K$ in $G$, i.e., a section of $H$.  Likewise, a $G$-torsor with a trivialization of its induced $H$-torsor is a $G$-torsor with a $G$-equivariant map to $H$; any such torsor admits a unique reduction to a $K$-torsor by taking the fiber over $0 \in H$.

We now specialize to the cases of interest, which are 
\begin{itemize}
\item $K = \mathbf{G}_m, G = \logGm, H = \ologGm$ on the logarithmic curve $X \to S$, and
\item $K = \L, G = \PL, H = \mathsf{V}$ on the tropicalization $\tropfont X$. 
\end{itemize} 
We will write $\logGm \otimes L$ for the $\logGm$-torsor associated to a $\Gm$-torsor $L$, and $\overnorm{P}$ for the $\ologGm$-torsor associated to a $\logGm$-torsor $P$. The tropicalization map $\bLogPic(X/S) \to \bTroPic(\tropfont X)$ arises from the relationship between $\ologGm$ and $\PL$. Since $t_* \overnorm{M}_X^{\rm{gp}} = \PL$, and $\overnorm M_X^{\rm gp}$-torsors are locally in $S$ trivial on the fibers of $t : X \to \tropfont X$ (namely, the strata of $X$), pushforward along $t$ gives an equivalence:
\begin{equation*}
\mathrm B \ologGm(X) \xrightarrow{\sim} \mathrm B \PL(\tropfont X) : \overnorm P \mapsto t_\ast \overnorm P
\end{equation*}

Furthermore, if $P$ is a $\logGm$-torsor on $X$, with induced $\ologGm$-torsor $\overnorm P$ then the fiber of $P$ over a section $\alpha$ of $\overnorm P$ is a $\mathcal O_X^\ast$-torsor, $P(\alpha)$.  We obtain a map from $\overnorm P$ to the N\'eron--Severi presheaf $\NS$ sending $\alpha$ to the class of $P(\alpha)$.  Since $t_\ast \NS = \mathsf V$, we obtain $t_\ast \overnorm P \to \mathsf V$.  That is, the $\PL$-torsor on $\tropfont X$ associated to a $M_X^{\rm gp}$-torsor on $X$ comes with a canonical trivialization of its induced $\mathsf V$-torsor, hence descends uniquely to a $\L$-torsor that we call $\trop(P)$.  Explicitly, $\trop(P)$ is the $\L$-torsor of multidegree~$0$ lifts of $P$ to an $\mathcal O_X^\ast$-torsor:%
\footnote{The formula for $\trop(P)$ should be interpreted as a groupoid.  Since $\mathcal O_X^\ast \to M_X^{\rm gp}$ is injective, there is at most one isomorphism between any two objects of $\trop(P)$, so this groupoid is equivalent to a set.}
\begin{equation*}
\trop(P) = \{ (L,\alpha: L^{\textup{log}} \xrightarrow{~\sim} P) \: \big| \: \mathrm{multideg}(L) = 0 \}
\end{equation*}

In order to extend this construction to one valid over a general base, we will need to prove its compatibility with the generization maps for $\bTroPic(\tropfont X/S)$, given by Proposition~\ref{prop:tropic-func}.  Note that the bounded monodromy condition has not yet entered into the discussion; indeed, it only becomes necessary when considering families of nonconstant degeneracy.

\begin{proposition} \label{prop:bdd-mono-gen}
Let $X$ be a proper logarithmic curve over $S$ and let $s$ be a geometric point of $S$.  Then $\pi_\ast (\mathrm B \overnorm M_X^{\rm gp})_s \to \Gamma(X_s, \mathrm B \overnorm M_{X_s}^{\rm gp})$ is fully faithful and restricts to an isomorphism on the bounded monodromy subgroups.
\end{proposition}
\begin{proof}
	Full faithfulness follows from proper base change for \'etale cohomology~\cite[Th\'eor\`eme~5.1~(i) and (ii)]{sga4-XII}, so the point is to prove surjectivity on the bounded monodromy subgroup.  The assertion is \'etale-local in $S$, so we may assume that the logarithmic structure of $S$ has a global chart.  The chart gives a stratification of $X$ into finitely many locally closed subschemes, and we can assume without loss of generality that only one is closed and that it contains $s$.

Suppose that $\overnorm L_s$ is an $\overnorm M_X^{\rm gp}$-torsor on $X_s$ with bounded monodromy.  We extend $\overnorm L_s$ to an $\overnorm M_X^{\rm gp}$-torsor on $X$ inductively over the strata of $S$.  By induction, we can assume that $\overnorm L_Z$ has already been constructed on a closed union of strata $Z$ containing $s$ and that the complement of $Z$ in $S$ is an open subset $U$ on which $\overnorm M_S$ is constant.  Let $j$ denote the inclusion of $U$ in~$S$.

The homomorphism $\overnorm M_S^{\rm gp} \to j_\ast \overnorm M_U^{\rm gp}$ induces a homomorphism $\overnorm M_X^{\rm gp} \to \overnorm N_X^{\rm gp}$ by pushout.  Let $\overnorm K_Z$ be the $\overnorm N_X^{\rm gp}$-torsor on $X_Z$ induced from $\overnorm L_s$ along this homomorphism.

Let $\tropfont X_U$ denote the dual graph of a geometric fiber of $X$ over $U$ and let $\tropfont V_U$ be its universal cover.  Pulling back along the projection $\tropfont X_s \to \tropfont X_U$ we obtain an \'etale cover $\tropfont V_s$ of $\tropfont X_s$, which corresponds to an \'etale cover of $X_s$.  By construction, this cover extends to an \'etale cover $\rho : V \to X$ of all of $X$.

We also use $\rho$ to denote the restriction of $\rho$ to the preimage of $Z$.  The pullback $\rho^\ast \overnorm K_Z$ is trivial.  Indeed, it suffices to trivialize $\rho^\ast \overnorm K_s$, and $\overnorm K_s$ has trivial monodromy  around all loops in $\tropfont V_s$, by its construction and the assumption of bounded monodromy in $\overnorm L$.  Then $\rho^\ast \overnorm K_Z$ extends trivially to an $\overnorm N_X^{\rm gp}$-torsor $\overnorm K'$ on $V$ and the action of deck transformations extends as well.  By descent, we obtain an $\overnorm N_X$-torsor $\overnorm K$ on $X$ extending $\overnorm K_Z$.

We may now define $\overnorm L = \overnorm K \mathop\times_{i_\ast \overnorm K_Z} i_\ast \overnorm L_Z$ where $i$ is the inclusion of $X_Z$ in $X$.  This is a torsor under $\overnorm N \mathop\times_{i_\ast \overnorm N_{X_Z}} i_\ast \overnorm M_{X_Z}$, which is isomorphic to $\overnorm M_X$ by the canonical map.
\end{proof}

Suppose now that $S$ is a strictly henselian valuation ring with special point $\xi$ and generic point $\eta$.  We have a commutative diagram:
\begin{equation*} \xymatrix{
\Gamma(X_\xi, \mathrm B M_{X_\xi}^{\rm gp}) \ar[d] & \Gamma(X, \mathrm B M_X^{\rm gp}) \ar[d] \ar[r] \ar[l] & \Gamma(X_\eta, \mathrm B M_{X_\eta}^{\rm gp}) \ar[d] \\
\Gamma(X_\xi, \mathrm B \overnorm M_{X_\xi}^{\rm gp}) \ar[d] & \Gamma(X, \mathrm B M_X^{\rm gp}) \ar[d] \ar[r] \ar[l] & \Gamma(X_\eta, \mathrm B \overnorm M_{X_\eta}^{\rm gp}) \ar[d] \\
\mathrm B\NS(X_\xi) & \Gamma(S, \mathrm B\NS(X/S)) \ar[r] \ar[l] & \mathrm B\NS(X_\eta)
} \end{equation*}
Upon passage to the bounded monodromy subgroups and composing, we obtain~\eqref{eqn:120}:
\setcounter{equation}{\value{theorem}}
\begin{equation} \label{eqn:120} \vcenter{\xymatrix{
\bLogPic(X/S)(\xi) \ar[d] & \bLogPic(X/S)(S) \ar[d] \ar[r] \ar[l] & \bLogPic(X/S)(\eta) \ar[d] \\
\Gamma(X_\xi, \mathrm B\overnorm M_{X_\xi}^{\rm gp})^\dagger  \ar[d] & \Gamma(X, \mathrm B\overnorm M_X^{\rm gp})^\dagger \ar[l]_-\sim \ar[r] \ar[d] & \Gamma(X_\eta, \mathrm B\overnorm M_{X_\eta})^\dagger \ar[d] \\
\mathrm B \NS(X_\xi) & (\mathrm B \NS(X/S))(S) \ar[r] \ar[l]_-\sim & \mathrm B \NS(X_\eta)
}} \end{equation}
The isomorphism in the second row is Proposition~\ref{prop:bdd-mono-gen} and we get the isomorphism $(\mathrm B \NS(X/S))(S) \simeq \mathrm B \NS(X_\xi)$ from the knowledge that $\NS(X/S)$ is an \'etale sheaf over $S$.  

The vertical compositions in diagram~\eqref{eqn:120} are canonically trivialized, as was discussed earlier.  Proposition~\ref{prop:log-lin} implies that $\bTroPic(\tropfont X_\xi)$ is the kernel of $\Gamma(\tropfont X_\xi, \mathrm B\overnorm M_{X_\xi}^{\rm gp})^\dagger \to \mathrm B\NS(X_\xi)$ (and similarly over $\eta$) so we obtain a commutative diagram~\eqref{eqn:121} :
\begin{equation} \label{eqn:121} \vcenter{\xymatrix{ 
\bLogPic(X/S)(\xi) \ar[d] & \bLogPic(X/S)(S) \ar[r] \ar[l] \ar[dl] & \bLogPic(X/S)(\eta) \ar[d] \\
\bTroPic(\tropfont X_\xi) \ar[rr] & & \bTroPic(X_\eta)
}} \end{equation}
We leave it to the reader to verify that the construction in the proof of Proposition~\ref{prop:bdd-mono-gen} is the same as the one used in the proof of Proposition~\ref{prop:tropic-func} so that the map $\bTroPic(\tropfont X_\xi) \to \bTroPic(X_\eta)$ displayed above is indeed the same as the one guaranteed by Proposition~\ref{prop:tropic-func}.  The commutativity of the inner trapezoid gives the compatibility of the tropicalization map with generization.

\setcounter{theorem}{\value{equation}}
\numberwithin{equation}{theorem}
\begin{theorem} \label{thm:trop-pic}
Let $X$ be a proper, vertical logarithmic curve over $S$ and let $\tropfont X$ be its tropicalization.  Then there are exact sequences (in the \'etale topology):
\begin{gather*}
	0 \to \bPic^{[0]}(X/S) \to \bLogPic(X/S) \to \bTroPic(\tropfont X/S) \to 0 \\
	0 \to \Pic^{[0]}(X/S) \to \LogPic(X/S) \to \TroPic(\tropfont X/S) \to 0
\end{gather*}
\end{theorem}
\begin{proof}
	The second exact sequence is obtained from the first by dividing, term by term, by the exact sequence~\eqref{eqn:92}:
	\begin{equation} \label{eqn:92}
		0 \to \mathrm B\Gm \to \mathrm B\logGm \to \mathrm B\ologGm \to 0
	\end{equation}

	We have exact sequences~\eqref{eqn:122}:
\begin{gather} \label{eqn:122} 
0 \to \mathcal O_X^\ast \to M_X^{\rm gp} \to \overnorm M_X^{\rm gp} \to 0 \qquad \text{(on $X$)} \\
0 \to \L \to \PL \to \mathsf V \to 0 \qquad \text{(on $\tropfont X$)} \notag
\end{gather}
Rotating these sequences, pushing forward to $S$, and restricting to bounded monodromy, we get a commutative diagram of exact sequences (with $\rho_\ast \mathrm B \PL$ denoting the stack on $S$ of $\PL$-torsors on $\tropfont X$):
\begin{equation*} \xymatrix{
0 \ar[r] & \bPic(X/S) \ar[r] \ar[d] & \bLogPic(X/S) \ar[r] \ar[d] & \pi_\ast (\mathrm B \overnorm M_X^{\rm gp})^\dagger \ar@{=}[d] \ar[r] & 0 \\
0 \ar[r] & \NS(X/S) \ar[r] & \bTroPic(\tropfont X/S) \ar[r] & \rho_\ast ( \mathrm B \PL )^\dagger \ar[r] & 0
} \end{equation*}
The kernel of $\bLogPic(X/S) \to \bTroPic(X/S)$ therefore coincides with the kernel of the map $\bPic(X/S) \to \NS(X/S)$, which is $\bPic^{[0]}(X/S)$.  Likewise, $\bPic(X/S)$ surjects onto $\NS(X/S)$ so $\bLogPic(X/S) \to \bTroPic(X/S)$ is surjective as well.
\end{proof}

\numberwithin{equation}{subsection}
\subsection{Logarithmic abelian variety structure}
\label{sec:log-ab-var}

In this section, we explain how the logarithmic Jacobian carries the structure of a log abelian variety, in the sense of \cite{kajiwara_kato_nakayama_2015}. For the convenience of the reader, we recall briefly the necessary definitions. For details and proofs, we refer the reader to \cite{kajiwara_kato_nakayama_2015}.  We try to keep the notations of \cite{kajiwara_kato_nakayama_2015} as much as possible, but some changes will be necessary in order to avoid conflicts with notation already introduced here. We fix a base logarithmic scheme $S$, and form the site $\rm{fs}/S$, whose objects are fine and saturated log schemes over $S$ and whose coverings are strict \'etale surjections.  

Let $G$ be a semiabelian group scheme, that is, an extension 

\begin{equation}
1 \rightarrow T \rightarrow G \rightarrow A \rightarrow 1
\end{equation} 

\noindent of an abelian variety $A$ by a torus $T = \Spec \ZZ[H] = \underline\Hom(H, \Gm)$. Here $H$ is a sheaf of lattices over $\rm{fs}/S$.   Just as $\logGm$ extends $\Gm$, there is a sheaf $T^{\rm{log}} = \logGm \otimes_{\Gm} T$ extending $T$ that can be defined on $\rm{fs}/S$ by the folowing formula: 
\begin{equation}
T^{\rm{log}}(S') = \Hom(H,M_{S'}^{\rm{gp}})
\end{equation}

Equivalently, $T^{\rm{log}} = \underline{\Hom}(H,\logGm)$, where we regard $H$ as a sheaf on $\rm{fs}/S$, and $\underline{\Hom}$ denotes the sheaf of homomorphisms.  There is an evident inclusion $T \rightarrow T^{\log}$ induced from $\Gm \to \logGm$, and pushing out $T \rightarrow G$ along this map we obtain an exact sequence 
\begin{equation}
1 \rightarrow T^{\rm{log}} \rightarrow G^{\rm{log}} \rightarrow A \rightarrow 1
\end{equation} 
where $G^{\rm{log}} = T^{\rm{log}} \oplus_T G$. 

\setcounter{theorem}{\value{equation}}
\begin{definition}[{\cite[Definition~2.2]{kajiwara_kato_nakayama_2008b}}]
A \emph{log 1-motif} is a map $K \rightarrow G^{\rm{log}}$, where $G$ is a semiabelian group scheme and $K$ is sheaf of locally free abelian groups of finite rank on $\rm{fs}/S$. 
\end{definition}

The map $K \rightarrow G^{\rm{log}}$ naturally defines a subsheaf $G_{(K)}^{\rm{log}} \subset G^{\rm{log}}$ as follows. The composed map from $K$ to the quotient $G^{\rm{log}}/G \cong T^{\log}/T = \overnorm{T}^{\rm{log}}$ determines a pairing $\langle\,,\,\rangle:H \times K \rightarrow \ologGm$, and a subsheaf $\overnorm{T}^{\rm{log}}_{(K)}$, determined by the formula 

\setcounter{equation}{\value{theorem}}
\begin{equation}
\overnorm{T}^{\rm{log}}_{(K)}(S') = \Big\{ \phi \in \overnorm{T}^{\rm{log}}(S') \: \Big| \: \parbox{6.6cm}{$\forall \textup{ geometric points } s \in S', x \in H_{s}, \newline \exists y,y' \in K \textup{ s.t } \langle x,y \rangle \le \phi(x) \le \langle x,y' \rangle$} \Big\}
\end{equation}

We thus obtain $G^{\rm{log}}_{(K)}$ by simply pulling back $\overnorm{T}^{\rm{log}}_{(K)}$ under the map $G \rightarrow \overnorm{T}^{\rm{log}}$. A log 1-motif defines an \emph{abelian variety with constant degeneration}, by assigning to $K \rightarrow G^{\rm{log}}$ the quotient sheaf $G^{\rm{log}}_{(K)}/K$. 

\setcounter{theorem}{\value{equation}}
\numberwithin{equation}{theorem}
\begin{definition}[{\cite[Definition~4.1]{kajiwara_kato_nakayama_2008}}]
\label{definition:logabelianvariety}
A \emph{log abelian variety} is a sheaf $\mathcal{A}$ on $\rm{fs}/S$ such that all of the following properties hold: 
\begin{enumerate}[label=(\arabic{*})]
\item For each geometric point $s \in S$, the pullback of $\mathcal{A}$ to $\rm{fs}/s$ is a log abelian variety with constant degeneration. 
\item \'Etale locally on $S$, there is an exact sequence~\eqref{eqn:70} for some semiabelian variety $G$ over $S$, some bilinear form $H \times K \to \Gamma(S, \overnorm M_S^{\rm gp})$, and $\overnorm T^{\log} = \underline\Hom(H, \ologGm)$:
\begin{equation} \label{eqn:70}
0 \rightarrow G \rightarrow \mathcal{A} \rightarrow \overnorm{T}^{\rm{log}}_{(K)}/K \rightarrow 0
\end{equation} 
\item Let $\overnorm{K}$ denote the image of $K$ in $\underline{\Hom}(H,\ologGm)$ and $\overnorm{H}$ the image of $H$ in $\underline{\Hom}(K,\ologGm)$. For each geometric point $s \in S$, there exists a map $\phi: \overnorm{K}_s \rightarrow \overnorm{H}_s$ with finite cokernel such that $\langle \phi(y), z \rangle = \langle y , \phi(z) \rangle$ for all $y,z \in \overnorm{K}_s$, and $\langle \phi(y),y\rangle \in \overnorm{M}_{S,s}$.  
\item The diagonal $\mathcal{A} \rightarrow \mathcal{A} \times \mathcal{A}$ is representable by finite morphisms. 
\end{enumerate}
\end{definition}
We are now ready to indicate how the logarithmic Jacobian fits into this context. 

\begin{theorem} \label{thm:log-ab-var}
Let $X$ be a proper, vertical logarithmic curve over $S$.  Then $\LogPic^0(X/S)$ is a logarithmic abelian variety in the sense of Kajiwara, Kato, and Nakayama~\cite{kajiwara_kato_nakayama_2008}.
\end{theorem}
\begin{proof}
We verify the conditions of Definition~\ref{definition:logabelianvariety}.

Given a family of logarithmic curves $X \rightarrow S$, with dual graph $\tropfont{X}$, we obtain a sheaf of lattices $H_1(\tropfont{X})$. We set $H = K = H_1(\tropfont{X})$ for the lattices appearing in the definition above, and take $\langle \,,\, \rangle$ to be the intersection pairing.  We let $G = \rm{Pic}^{[0]}(X)$ denote the multidegree $0$ part of $\Pic(X/S)$. 

The third condition in the definition is immediate in our context: The two lattices $H$ and $K$ are $H_1(\tropfont{X})$, and we may take $\phi = \mathrm{id}$. For any $y \in H_1(\tropfont X_s)$, the pairing $\langle y, y \rangle$ is a sum of elements of $\overnorm M_{S,s}$ by Definition~\ref{def:intersection}, and therefore is in $\overnorm M_{S,s}$.  

The last condition is exactly Theorem~\ref{thm:diag}. 

The first and second condition follow from Corollary~\ref{cor:quotient} and the exact sequence of \ref{thm:trop-ex-seq} respectively, once we observe: 
\begin{lemma}
\label{lemma:boundedmonodromyisYpart}
For $K = H_1(\tropfont{X})$, the subsheaf $\overnorm{T}_{(K)}^{\rm{log}}$ coincides with the subsheaf of elements with bounded monodromy $(\overnorm{T}^{\rm{log}})^{\dagger}$ in $\overnorm{T}^{\rm{log}}$.
\end{lemma} 
\begin{proof}
Since both the bounded monodromy condition and the condition defining $\overnorm{T}_{(Y)}^{\rm{log}}$ are defined pointwise, we may check that the two groups are the same on a logarithmic scheme $s$ whose underlying scheme is the spectrum of an algebraically closed field. If $\phi: H_1(\tropfont X) \rightarrow \overnorm M^{\rm{gp}}$ has bounded monodromy then, by definition, there are integers $m$ and $n$ such that $m\langle x,x \rangle \le \phi(x) \le n \langle x, x \rangle $.  Thus $\phi \in \overnorm{T}^{\rm{log}}_{(Y)}$ as it verifies the definition with $y = mx,y'=nx$. 

For the converse, suppose that $\phi : H_1(\tropfont X) \to \overnorm M_s^{\rm gp}$ and, for every $x \in H_1(\tropfont X)$, there are $y,y' \in H_1(\tropfont X)$ such that $\langle x, y \rangle \leq \phi(x) \leq \langle x, y' \rangle$.  For any $y \in H_1(\tropfont X)$, we have $\langle x, y \rangle \leq n \langle x, x \rangle$ for some positive integer $n$.  Indeed, we may take $n$ to be the maximum of the coefficients of $y$ as a linear combination of edges of $\tropfont X$.  We likewise have $\langle x, y' \rangle \geq -m \langle x, x \rangle$ for some positive integer $m$, and therefore $-m \langle x, x \rangle \leq \phi(x) \leq n \langle x, x \rangle$, as required.
\end{proof}

\end{proof}

\subsection{Prorepresentability}
\label{sec:logpic-prorep}

The logarithmic Picard group and logarithmic Jacobian cannot be represented by schemes, or even by algbraic stacks, with logarithmic structures.  This follows from the nonrepresentability of the logarithmic multiplicative group, which was proved in Proposition~\ref{prop:nonrep}.  We have already seen in Section~\ref{sec:cover} that both are nearly representable in the sense that they have logarithmically smooth covers by logarithmic schemes.  In this section we will consider another near-representability property.

Let $\tropfont X$ be the tropicalization of a logarithmic curve $X$ over $S$.  Theorem~\ref{thm:trop-ex-seq} shows that $\LogPic^0(X/S)$ is a torsor under the algebraic group $\Pic^{[0]}(X/S)$ over $\TroJac(\tropfont X/S)$ and Theorem~\ref{thm:trop-pic} shows that $\LogPic(X/S)$ is a $\Pic^{[0]}(X/S)$-torsor over $\TroPic(\tropfont X/S)$.  Therefore the nonrepresentability of $\LogPic(X/S)$ can be attributed to the nonrepresentability of $\TroPic(\tropfont X/S)$.  However, we saw in Section~\ref{sec:trojac-subdiv} that $\TroPic(\tropfont X/S)$ is prorepresentable.  We might therefore reasonably expect $\TroPic(X/S)$ to be similarly prorepresentable.

We saw in Proposition~\ref{prop:prorep} that $\Hom(H_1(\tropfont X), \ologGm)^\dagger$ is, locally in $S$, pro-representable by a collection of submonoids of $\overnorm M_S^{\rm gp} + H_1(\tropfont X)$.  Each of these submonoids represents a functor on logarithmic schemes that can be represented by an algebraic stack with a logarithmic structure (see \cite[Section~6]{cavalieri2017moduli} for further details).  Therefore we can think of $\Hom(H_1(\tropfont X), \ologGm)$ as ind-representable on logarithmic schemes by algebraic stacks with logarithmic structure.  Since $\TroJac(\tropfont X/S)$ is a quotient of $\Hom(H_1(\tropfont X), \ologGm)^\dagger$ by $H_1(\tropfont X)$, we conclude that $\TroJac(\tropfont X/S)$ is, locally in $S$, the quotient of an ind-algebraic stack with logarithmic structure by $H_1(\tropfont X)$.  The same applies to $\TroPic^d(\tropfont X/S)$ for all $d$, since it is a torsor under $\TroPic^0(\tropfont X/S) = \TroJac(\tropfont X/S)$.

\begin{proposition} \label{prop:tropic-indrep}
	$\LogPic(X/S)$ is, locally in $S$, the quotient of an ind-algebraic stack with a logarithmic structure by the action of $H_1(\tropfont X)$.
\end{proposition}
\begin{proof}
	$\LogPic(X/S)$ is a torsor over $\TroPic(\tropfont X/S)$ under the algebraic group $\Pic^{[0]}(X/S)$.  
\end{proof}

For a moduli problem $F$ on logarithmic schemes, one defines a minimal logarithmic structure on an $S$-point of $F$ in such a way that when $F$ is representable, minimality corresponds to strictness of the morphism $S \to F$.  We introduce a similar notion that corresponds to strictness \emph{at the level of associated groups}.

\begin{definition} \label{def:pseudominimal}
	Let $S$ be a logarithmic scheme and let $F$ be a covariant functor valued in sets on logarithmic structures over $M_S$ such that $F(M_S)$ has one element.  We say that a logarithmic structure $N$ over $M_S$ and an object $\xi \in F(N)$ is \emph{pseudominimal} if, for every $\eta \in F(P)$, there is a unique morphism $u : N^{\rm gp} \to P^{\rm gp}$ and $\xi' \in F(u^{-1} P \cap N)$ that is sent to $\xi$ under $u^{-1} P \cap N \to N$ and is sent to $\eta$ under $u^{-1} P \cap N \to P$.

	If $F$ is a presheaf on logarithmic schemes then we say $\xi \in F(T)$ is \emph{pseudominimal} if $\xi$ is pseudominimal when $F$ is regarded as a functor on logarithmic structures over $M_T$.
\end{definition}

Note that if $\xi_1 \in F(N_1)$ and $\xi_2 \in F(N_2)$ are both pseudominimal then there is a canonical isomorphism $N_1^{\rm gp} \simeq N_2^{\rm gp}$.

\begin{proposition}
With notation as in Definition~\ref{def:pseudominimal}, if pseudominimal elements exist then the collection of pseudominimal objects of $F$ pro-represents $F$.
\end{proposition}
\begin{proof}
Let $G(P) = \varinjlim_{\text{$(N, \xi)$}} \Hom(N, P)$ with the colimit taken over all pseudominimal $(N,\xi)$.  There is a canonical morphism $G \to F$ that we want to show is an isomorphism.  It is surjective by the existence of pseudominimal objects.  Now suppose that $\xi_1 \in F(N_1)$ and $\xi_2 \in F(N_2)$ are pseudominimal objects projecting to the same $\eta \in F(P)$.  By definition of pseudominimality, there is a unique morphism $u : N_1^{\rm gp} \to N_2^{\rm gp}$ and an object $\xi' \in F(u^{-1} N_2 \cap N_1)$ projecting to both $\xi_1$ and $\xi_2$ along the canonical maps to $N_1$ and $N_2$.  It is immediate that $\xi'$ is pseudominimal, which implies that the pseudominimal objects are cofiltered.  Moreover, diagram~\eqref{eqn:125} commutes
\begin{equation} \label{eqn:125} \vcenter{ \xymatrix{
u^{-1} N_2 \cap N_1 \ar[r] \ar[d] & N_1 \ar[d] \\
N_2 \ar[r] & P
}}\end{equation}
so the homomorphisms $N_1 \to P$ and $N_2 \to P$ represent the same element of $G(P)$.  This demonstrates the injectivity of $G \to F$ and completes the proof.
\end{proof}

\begin{proposition} \label{prop:logpic-pseudo}
	A $T$-point of $\LogPic^0(X/S)$ over $f : T \to S$ is pseudominimal if and only if the canonical map $f^\ast \overnorm M^{\rm gp}_S + H_1(\tropfont X) \to \overnorm M_T^{\rm gp}$ is a bijection.
\end{proposition}
\begin{proof}
	Since $\LogPic^0(X/S) \to \TroJac(X/S)$ is strict, a $T$-point of $\LogPic^0(X/S)$ is pseudominimal if and only if the induced $T$-point of $\TroJac(X/S)$ is pseudominimal.  The proposition therefore follows from Proposition~\ref{prop:prorep}.
\end{proof}

\begin{example}
The group $\logGm$ is not prorepresentable, because it lacks pseudominimal objects.  For simplicity we will work with $\ologGm$ instead.  Indeed, consider a logarithmic structure over a point with characteristic monoid $\overnorm M = \mathbf N e_1 + \mathbf N e_2$.  Then $\ologGm(S) = \mathbf Z e_1 + \mathbf Z e_2$.  Consider the element $e_1 - e_2$.  Any pseudominimal object $(\overnorm N, \xi)$ must have $\overnorm N^{\rm gp} = \mathbf Z$.  But if $u : \mathbf Z \to \overnorm M$ is a homomorphism taking $1$ to $e_1 - e_2$ then $u^{-1} \overnorm M = \{ 0 \}$, and then there is no element of $\ologGm(0) = 0$ inducing $e_1 - e_2$ via a homomorphism to $\overnorm M$.

On the other hand, we have already seen that the bounded monodromy subfunctors $\Hom(H, \ologGm)^\dagger \subset \Hom(H, \ologGm)$ associated to a positive definite quadratic form on $H$ do admit pseudominimal objects.  For a concrete example, we work over a base monoid containing a nonzero element $\delta$ and consider $H = \mathbf Z$ with the quadratic form $\ell$ that has $\ell(1) = \delta$.  Let $\overnorm N$ be a submonoid of $\overnorm M \times H$ that contains $(n \delta, 1)$ and $(n \delta, -1)$ for some positive integer $n$, and let $\xi \in \ologGm^\dagger(H) = \overnorm M^{\rm gp} \times \mathbf Z$ be the element $(0,1)$.  We will check that $(\overnorm N, \xi)$ is pseudominimal relative to any element of $\ologGm^\dagger(\overnorm M)$.  Indeed, suppose that $\eta \in \ologGm^\dagger(\overnorm P) \subset \overnorm P^{\rm gp}$.  Then there is a positive integer $m$ such that $-m \delta \leq \eta \leq m \delta$, so $\overnorm P$ contains both $\eta + m \delta$ and $m \delta - \eta$; for simplicity we choose $m \geq n$.  Therefore the preimage of $\overnorm P$ under the natural homomorphism $u : \overnorm N^{\rm gp} \to \overnorm P^{\rm gp}$ contains $(m \delta, 1)$ and $(m \delta, -1)$.  Hence $(0,1)$ lies in the associated group of $u^{-1} \overnorm P \cap \overnorm N$, so we can find $\xi' = (0,1) \in \ologGm(u^{-1} \overnorm P \cap \overnorm N)$ lifting both $\xi \in \ologGm(\overnorm N)$ and $\eta \in \ologGm(\overnorm P)$.
\end{example}

\subsection{Schematic models}
\label{sec:prop-mod}

We show that the combinatorics of the tropical Picard group can be used to construct toroidal compactifications of $\LogPic^d(X/S)$.  This section is inspired directly by Kajiwara, Kato, and Nakayama~\cite{kajiwara1993logarithmic,kajiwara_kato_nakayama_2015} and is, for the most part, only a tropical reinterpretation of their results.

Suppose that $X$ is a logarithmic curve over a logarithmic scheme $S$ with tropicalization $\tropfont X$.  For simplicitly, we assume that $S$ is atomic, or at least that it has a morphism to $\sigma$ for some rational polyhedral cone $\sigma$, dual to $\overnorm M$, and that $\tropfont X$ is pulled back from a tropical curve $\tropfont Y$ over $\sigma$ (we abuse notation here and do not distinguish notationally between $\sigma$ and its Artin cone).  Then $\TroPic(\tropfont X/S)$ is pulled back from $\TroPic(\tropfont Y)$.  A subdivision $\tropfont Z$ of $\TroPic(\tropfont Y)$ induces a subdivision of $\TroPic(\tropfont X/S)$ and a subdivision $\LogPic(X/S)_{\tropfont Z}$ of $\LogPic(X/S)$ by pullback.  Since subdivisions are proper and $\LogPic(X/S)$ is proper, the subdivision, $\LogPic(X/S)_{\tropfont Z}$, is proper as well.

Suppose now that $\tropfont Z$ is actually representable by a cone space in the sense of~\cite{cavalieri2017moduli}.  Then $\tropfont Z \mathop\times_{\sigma} S$ is representable by an algebraic stack over $S$ with a logarithmic structure.  By Theorem~\ref{thm:trop-pic}, $\LogPic(X/S)$ is a torsor over $\TroPic(\tropfont X/S)$ under the group scheme $\Pic^{[0]}(X/S)$.  Therefore $\LogPic(X/S)_{\tropfont Z}$ is also a torsor over $\tropfont Z_S$ under the same group scheme.  This implies that $\LogPic(X/S)_{\tropfont Z}$ is representable by an algebraic stack with a logarithmic structure.

\begin{lemma} \label{lem:aut}
	Let $L$ be a logarithmic line bundle on a proper, vertical logarithmic curve $X$ over $S$.  Assume that the logarithmic structure of $S$ is pseudominimal.  Then the automorphism group of $L$, fixing the underlying schemes of $X$ and $S$ and the minimal logarithmic structure of $X$, is $\Gamma(S, M_S^{\rm gp})$.
\end{lemma}
\begin{proof}
	Let $N_S$ and $N_X$ be the minimal logarithmic structures on $X$ and on $S$, respectively, associated with the curve $X$ over $S$.  Then $\overnorm M_S^{\rm gp}$ is isomorphic to the direct sum $\overnorm N_S^{\rm gp} + H_1(\tropfont X)$, where $\tropfont X$ is the tropicalization of $X$ over $S$.  We choose the splitting as follows: $L$ induces $\bar\lambda \in H^1(\tropfont X, \PL) = \Hom(H_1(\tropfont X), \overnorm M_S^{\rm gp}) / \partial \mathsf E(\tropfont X)$.  Locally in $S$, we choose a lift of $\bar\lambda$ to $\lambda \in \Hom(H_1(\tropfont X), \overnorm M_S^{\rm gp})$.  Then identify $H_1(\tropfont X)$ with its image in $\overnorm M_S^{\rm gp}$ under $\lambda$.

	Now consider an automorphism of $L$ as in the statement of the lemma.  This comprises an automorphism of $\phi$ of $M_S$ and an isomorphism $L \simeq \phi_\ast L$.  Since $L$ and $\phi_\ast L$ are isomorphic, the homomorphisms $\lambda$ and $\phi \lambda$ must coincide modulo $\partial \mathsf E(\tropfont X)$.  That is, $\phi \lambda = \lambda + \sum a_i \partial(e_i)$.  In particular, for each $\gamma \in H_1(\tropfont X)$, we must therefore have
	\begin{equation}
		\phi \lambda(\gamma) = \gamma + \sum_{e_i \in \gamma} a_i \ell(e_i)
	\end{equation}
	where the sum is taken over the constituent edges of $\gamma$.  If any $a_i < 0$ then $\phi^k(\gamma)$ cannot lie in $\overnorm M_S$ for $k \gg 0$.  Similarly, if any $a_i > 0$ then $\phi^k(\gamma)$ will not lie in $\overnorm M_S$ for $k \ll 0$.  Therefore $a_i = 0$ for all $i$ and $\phi$ acts as the identity on $\overnorm M_S^{\rm gp}$.

	The automorphism $\phi$ of $M_S$ is therefore determined uniquely by a homomorphism $\overnorm M_S^{\rm gp} \to \mathcal O_S^\ast$.  This homomorphism must vanish on $\overnorm N_S^{\rm gp} \subset \overnorm M_S^{\rm gp}$, so it descends to a homomorphism $H_1(\tropfont X) \to \mathcal O_S^\ast$.  Writing $A$ for the automorphism group in the statment of the lemma, this gives us an exact sequence:
	\begin{equation} \label{eqn:94}
		0 \to M_X^{\rm gp} \to A \to \Hom(H_1(\tropfont X), \mathcal O_S^\ast)
	\end{equation}
	The $M_X^{\rm gp}$ on the left represents the automorphisms of $L$ when the logarithmic structure of $S$ is held fixed.  We wish to show that only the zero homomorphism $H_1(\tropfont X) \to \mathcal O_S^\ast$ lifts to $A$.

	Consider the sheaf $\tilde A$ on $X$ whose sections over an \'etale $U \to X$ consist of an automorphism $\phi$ of $\pi^\ast M_S \big|_U$ that fixes the minimal logarithmic structure of $X$ and the characteristic monoid of $\pi^\ast M_S \big|_U$ and an isomorphism between $L \big|_U$ and $\phi_\ast L \big|_U$.  Since logarithmic line bundles are locally trivial, an isomorphism between $L \big|_U$ and $\phi_\ast L \big|_U$ always exists locally in $X$ and there is therefore an exact sequence of sheaves on $X$:
	\begin{equation*}
		0 \to M_X^{\rm gp} \to \tilde A \to \Hom(H_1(\tropfont X), \mathcal O_X^\ast) \to 0
	\end{equation*}
	Pushing forward to $S$, we get an extension~\eqref{eqn:66} of~\eqref{eqn:94}:
	\begin{equation} \label{eqn:66}
		0 \to \pi_\ast M_X^{\rm gp} \to A \to \Hom(H_1(\tropfont X), \mathcal O_S^\ast) \to \mathrm R^1 \pi_\ast M_X^{\rm gp}
	\end{equation}
	The map $\Hom(H_1(\tropfont X), \mathcal O_S^\ast) \to \mathrm R^1 \pi_\ast M_X^{\rm gp}$ sends a homomorphism $\phi$ to the multidegree~$0$ line bundle on $X$ obtained by gluing using $\phi$ around the loops of $\tropfont X$.  It is, in other words, the inclusion of the torus part of $\Pic^{[0]}(X)$ in $\LogPic(X)$, and in particular is injective.  It follows that $\pi_\ast M_X^{\rm gp} \to A$ is bijective.  By Lemma~\ref{lem:global-sections}, $\pi_\ast M_X^{\rm gp} = M_S^{\rm gp}$ and the lemma is proved.
\end{proof}

\begin{corollary}
	Let $\LogPic(X/S)_{\tropfont Z}$ be a subdivision of $\LogPic(X/S)$ that is representable by an algebraic stack with a logarithmic structure.  Then $\LogPic(X/S)_{\tropfont Z}$ is representable by an algebraic \emph{space} with a logarithmic structure.
\end{corollary}
\begin{proof}
	Since objects of $\LogPic(X/S)_{\tropfont Z}$ are pseudominimal, Lemma~\ref{lem:aut} shows that objects of $\LogPic(X/S)_{\tropfont Z}$ have no nontrivial automorphisms.  Therefore $\LogPic(X/S)_{\tropfont Z}$ is a sheaf, and hence an algebraic space.
\end{proof}

\subsection{Unintegrable torsors}
\label{sec:unint}

We will show that a $\logGm$-torsor on a logarithmic curve that deforms to all infinitesimal orders does not necessarily integrate to a $\logGm$-torsor over a complete noetherian local ring.  Such objects are excluded from the logarithmic Picard group by the bounded monodromy condition of Definition~\ref{def:bdd-mono}, and this section is meant to explain the necessity of that condition.

In this section, we can take cohomology either in the Zariski topology or the \'etale topology.

Let $P$ be a $\logGm$-torsor on a logarithmic scheme $X$.  By the projection $\logGm \to \ologGm$, this induces a $\ologGm$-torsor $\overnorm P$ over $X$.  We note that there is an exact sequence:
\begin{equation*}
H^1(X, M_X^{\rm gp}) \to H^1(X, \overnorm M_X^{\rm gp}) \to H^2(X, \mathcal O_X^\ast)
\end{equation*}
As $H^2(X, \mathcal O_X^\ast)$ vanishes for a curve over an algebraically closed field (or, more generally, over an artinian local ring with algebraically closed residue field), every $\ologGm$-torsor on such a curve lifts to a $\logGm$-torsor.  To prove the existence of an unintegrable $\logGm$-torsor, it will therefore suffice to give an example of an unintegrable $\ologGm$-torsor on a family of logarithmic curves over a complete noetherian local ring with algebraically closed residue field.

Let $S = \Spec \mathbf C[[t]]$ and let $X$ be a family of curves with smooth total space such that the general fiber is smooth and connected, but the special fiber has two irreducible components, joined to each other at two ordinary double points, but is otherwise smooth.  This is essentially the simplest example where \'etale cohomology with non-torsion coefficients does not commute with base change~\cite[\S2]{sga4-XII}.  In this example, cohomology in the Zariski topology also fails to commute with base change.

Let $M_S$ be the divisorial logarithmic structure on $S$ and let $M'_S$ be an extension of $M_S$ with $\overnorm M'_S = \overnorm M_S^{\rm gp} \times \mathbf Z$.  One may simply take $\overnorm M'_S = \overnorm M_S \times \mathbf N$, but it will be convenient later to have a valuative example; for this one can take a logarithmic structure with characteristic monoid $\mathbf N$ on the generic point of $S$ and give $S$ the logarithmic structure obtained by pushforward.  It is convenient to write $S' = (S, M'_S)$, so that $M'_S = M_{S'}$.  If $M_X$ is the divisorial logarithmic structure on $X$ then let $X' = (X, M'_X)$ be the pullback of $(X, M_X) \to (S, M_S)$ along $S' \to (S, M_S)$.  We construct a $\overnorm M_{X'}^{\rm gp}$-torsor on the special fiber $X'_0$ that lifts to all finite orders (this is automatic, by infinitesimal invariance of the \'etale site) but not to~$X'$.

We compute $H^1(X, \overnorm M_{X'}^{\rm gp})$ by means of the following exact sequence:
\begin{equation*}
H^0(X, \overnorm M_{X'/S'}^{\rm gp}) \to H^1(X, \pi^{-1} \overnorm M_{S'}^{\rm gp}) \to H^1(X, \overnorm M_{X'}^{\rm gp}) \to H^1(X, \overnorm M_{X'/S'}^{\rm gp})
\end{equation*}
As $\overnorm M_{X'/S'}^{\rm gp}$ is concentrated in dimension~$0$ on $X$, the last term in the sequence vanishes.  The group $H^1(X, \pi^{-1} \overnorm M_{S'}^{\rm gp})$ vanishes because $X$ is normal (see~\cite[\S2]{sga4-XII}).  Hence $H^1(X, \overnorm M_{X'}^{\rm gp}) = 0$.

On the other hand, in the exact sequence
\begin{equation*}
H^0(X_0, \overnorm M_{X'/S'}^{\rm gp}) \xrightarrow{\partial} H^1(X_0, \pi^{-1} \overnorm M_{S'}^{\rm gp}) \to H^1(X_0, \overnorm M_{X'}^{\rm gp}) \to H^1(X_0, \overnorm M_{X'/S'}^{\rm gp})
\end{equation*}
we still have $H^1(X_0, \overnorm M_{X'/S'}^{\rm gp}) = 0$, for the same reason, but 
\begin{equation*}
H^1(X_0, \pi^{-1} \overnorm M_{S'}^{\rm gp}) = H^1(X_0, \mathbf Z^2) \simeq \mathbf Z^2
\end{equation*}
since the fundamental group of $X_0$ is $\mathbf Z$ in the Zariski topology.  (In the \'etale topology, it is the non-torsion part of the fundamental group that is $\mathbf Z$.)

The sheaf $\overnorm M_{X'/S'}^{\rm gp}$ is a skyscraper $\mathbf Z$, concentrated at the nodes of $X_0$.  Therefore $H^0(X_0, \overnorm M_{X'/S'}^{\rm gp}) = \mathbf Z^2$.  The map $\partial$ is the intersection pairing and one can verify directly that its rank is~$1$.  Alternatively, one may observe that it is induced by pushout from the intersection pairing on $X$, which certainly has rank at most~$1$ because $H^1(X_0, \pi^{-1} \overnorm M_S) \simeq \mathbf Z$.  In any case, there is a nonzero element in $H^1(X_0, \overnorm M_{X'}^{\rm gp})$ (and one can verify that this group is free of rank~$1$).

This gives a formal collection of elements of $H^1(X_n, M_{X'}^{\rm gp})$, where $X_n$ is the reduction of $X$ modulo $t^{n+1}$, for every $n \geq 0$, whose image in $H^1(X_n, \overnorm M_{X'}^{\rm gp})$ is nonzero.  However, $H^1(X, \overnorm M_{X'}^{\rm gp}) = 0$, so this formal collection cannot be integrated.

\begin{proposition}
Let $X'$ and $S'$ be as above and let $Z$ be either the category fibered in groupoids on $\LogSch/S'$ whose value is the groupoid of $\logGm$-torsors on $X'_T$, or the the sheaf of isomorphism classes of such.  Then $Z$ has no logarithmically smooth cover by a logarithmic scheme.
\end{proposition}
\begin{proof}
Suppose that $U$ is a logarithmic scheme and $U \to Z$ is a logarithmically smooth cover.  Choose $S'$ as in the discussion preceding the statement of the proposition, with algebraically closed residue field.  Since the logarithmic structure of $S'_0$ is valuative, the map $S'_0 \to Z$ lifts to $U$.  Then the formal family of points $S'_n \to Z$ constructed above lifts to $S'_n \to U$ by the infinitesimal criterion for logarithmic smoothness.  Since $U$ is a logarithmic scheme, this family can be integrated to a map $S' \to U$, and therefore the maps $S'_n \to Z$ can be integrated to $S' \to Z$.  We have just seen no such integration exists.
\end{proof}

\numberwithin{equation}{subsection}
\section{Examples}
\label{sec:ex}

We calculate some examples of $\rm{LogPic}(X/S)$, over a base $S$ whose underlying scheme is the spectrum of an algebraically closed field $k$.  We use the quotient presentation of Corollary~\ref{cor:quotient}, which requires an explicit understanding of $H^1(X^\nu, \logGm)$ and the map $H_1(\tropfont X) \to H^1(X^\nu, \logGm)$.

\subsection{The Tate curve}
\label{sec:tate}

Let $Y \rightarrow \Spec k[[t]]$ be a family of curves whose generic fiber $Y_\eta$ is a smooth curve of genus~$1$ and whose special fiber $X$ consists of $n$ rational curves arranged in a circle.  We give $\Spec k[[t]]$ its divisorial logarithmic structure and we take $S$ to be the closed point of $\Spec k[[t]]$, with the logarithmic structure induced by restriction.

Let $\tropfont X$ be the tropicalization of $X$.  This is a graph with $n$ vertices in a circle, and we have $H_0(\tropfont X) = \mathbf Z$ and $H_1(\tropfont X) = \mathbf Z$.  The intersection pairing $\mathbf Z \times \mathbf Z \to \overnorm M_S^{\rm gp}$ sends $(a,b)$ to $ab \delta$ where $\delta$ is the sum of the lengths of the edges of $\tropfont X$.  Corollary~\ref{cor:tropic} then gives exact sequences~\eqref{eqn:75} and~\eqref{eqn:76}:
\begin{gather}
0 \to \mathbf Z \xrightarrow{\delta} \ologGm^\dagger \to \TroJac(\tropfont X/S) \to 0 \label{eqn:75} \\
0 \to \TroJac(\tropfont X/S) \to \TroPic(\tropfont X/S) \to \mathbf Z \to 0 \label{eqn:76}
\end{gather}
That is $\TroJac(\tropfont X/S) = \ologGm^\dagger / \mathbf Z\delta$.  In particular, if $\overnorm M_T = \mathbf R_{\geq 0}$ then the $T$-points of $\TroJac(\tropfont X/S)$ may be identified with $\mathbf R / \mathbf Z\delta$.  By Theorem~\ref{thm:trop-pic}, $\LogPic^0(X/S)$ is an extension of ${\ologGm}^\dagger / \mathbf Z \delta$ by $\Pic^{[0]}(X/S) \simeq \Gm$.

In order to understand this extension more explicitly, we will use the quotient presentation of Corollary~\ref{cor:quotient}.  Recall from Equation~\eqref{eqn:30} that we may identify $H^1(X, \pi^\ast \logGm)^{[0]}$ with $\Hom(H_1(\tropfont X), \logGm)$.  Therefore Corollary~\ref{cor:quotient} gives us the exact sequence~\eqref{eqn:77}:
\begin{equation} \label{eqn:77}
0 \to H_1(\tropfont X) \to \Hom(H_1(\tropfont X), \logGm)^\dagger \to \LogPic^0(X/S) \to 0
\end{equation}
The pairing $H_1(\tropfont X) \times H_1(\tropfont X) \to \logGm$ lifts the intersection pairing on $\tropfont X$, valued in $\ologGm$.  Substituting $H_1(\tropfont X) = \mathbf Z$, we obtain $\LogPic^0(X/S) = \logGm^\dagger / \mathbf Z \tilde\delta$ where $\logGm^\dagger$ denotes the subfunctor of $\logGm$ that is bounded by $\delta$, and $\tilde\delta$ is a lift of $\delta$ to $M_S$.

The tropicalization sequence from Theorem~\ref{thm:trop-pic} now becomes~\eqref{eqn:78}:
\begin{equation} \label{eqn:78}
0 \to \Gm \to \logGm^\dagger / \mathbf Z \tilde\delta \to \ologGm^\dagger / \mathbf Z \delta \to 0
\end{equation}

The element $\tilde\delta \in \logGm$ can be understood as a `logarithmic period', in the following sense.  The map $d \log : M_X^{\rm gp} \to \Omega_{X/S}^{\log}$ factors through $M_{X/S}^{\rm gp}$ and therefore gives us a logarithmic differential $\phi$ on $X$.  We wish to compute $\int_\gamma \phi$ where $\gamma$ is a basis for $H_1(\tropfont X)$, without attempting to introduce any general theory of integration.

Let $\tilde X$ be the `universal cover' of $X$, whose tropicalization $\tilde{\tropfont X}$ has vertices indexed by the integers, with consecutive vertices connected by an edge.  We can recognize $\tilde X$ as a subdivision of $\logGm^\dagger$ and we have $X = \tilde X / H_1(\tropfont X) = \tilde X / \mathbf Z \gamma$.

Locally in $X$, there is no obstruction to lifting $\phi$ to $M_{\tilde X}^{\rm gp}$, so there is a global section $\Phi$ of $M_{\tilde X}^{\rm gp}$ lifting $\phi$.  Then $\Phi(\gamma . x) - \Phi(x)$ is a function of $x \in \tilde X$ valued in $\pi^\ast M_S^{\rm gp}$.  It is therefore constant and represents the coboundary of $\gamma$ in $H^1(X, \pi^\ast M_S^{\rm gp}) = M_S^{\rm gp}$.

\subsection{A curve of genus~$2$}
\label{sec:ex-gen2}

\begin{figure}

\begin{tikzpicture}[
	arcnode/.style={
		decoration={
			markings,
			mark=at position .5 with {\arrow[scale=2]{>}};
		},
		postaction={decorate}
	}
]

\node [circle,fill=black,inner sep=2pt,label=left:$v_1$] (v1) at (-2,0) {} ;
\node [circle,fill=black,inner sep=2pt,label=right:$v_2$] (v2) at (2,0) {} ;

\draw [arcnode] (v1) arc (180:0:2) coordinate [midway,label=above:$e_1$];
\draw [arcnode] (v1) -- (v2) coordinate [midway,label=above:$e_2$];
\draw [arcnode] (v1) arc (180:360:2) coordinate [midway,label=above:$e_3$];

\end{tikzpicture}

\label{fig:genus-2}
\caption{A tropical curve of genus 2.}
\end{figure}

Let $X$ consist of $2$ rational components joined along $3$ nodes.  The tropicalization $\tropfont X$ has $2$ vertices, $v_1$ and $v_2$ and $3$ edges, $e_1$, $e_2$, and $e_3$, which we choose to orient from $v_1$ to $v_2$, as shown in Figure~\ref{fig:genus-2}.  We write $\delta_i$ for the length of $e_i$ in $\overnorm M_S$.  The differences $e_1 - e_2$ and $e_2 - e_3$ form a basis for $H_1(\tropfont X)$.  In this basis, the matrix of the intersection pairing is~\eqref{eqn:79}:
\begin{equation} \label{eqn:79}
A = 
\begin{pmatrix}
\delta_1 + \delta_2 & -\delta_2 \\
-\delta_2 & \delta_2 + \delta_3
\end{pmatrix}
\end{equation}
The presentation $\TroJac(X/S) = \Hom(H_1(\tropfont X),\ologGm)^\dagger / H_1(\tropfont X)$ becomes~\eqref{eqn:80}:
\begin{equation} \label{eqn:80}
\TroJac(X/S) = (\ologGm \times \ologGm)^\dagger / A \mathbf Z^2 
\end{equation}
In particular, the real points are $\mathbf R^2 / A \mathbf Z^2 \simeq S^1 \times S^1$.

The commutative diagram in~\eqref{eqn:34} gives a morphism~\eqref{eqn:115}:
\begin{equation} \label{eqn:115}
H^0(\tropfont X, \mathsf V) \to \TroPic(\tropfont X) \subset H^1(\tropfont X, \L)
\end{equation}
In concrete terms, this sends an integer linear combination of vertices $D$ on $\tropfont X$ to the torsor of piecewise linear functions on $\tropfont X$ that are linear along the edges of $\tropfont X$ and whose failure of linearity at each vertex $v$ of $\tropfont X$ is $D(v)$.  We denote this torsor $\L(D)$.

The exact sequence in the first row of~\eqref{eqn:34} shows that lifts of $\L(D)$ to $H^1(\tropfont X, \overnorm M_S^{\rm gp}) = \Hom(H_1(\tropfont X), \overnorm M_S^{\rm gp})$ correspond to the trivializations of the induced $\mathsf H$-torsor $\mathsf H(D)$.  This torsor is the sheaf of assignments of integers to the vertices of $\tropfont X$ such that the sum of outgoing slopes at each vertex $v$ is $D(v)$.

\begin{figure}
\begin{tikzpicture}

\def\clipleft{-3.8}
\def\clipright{9.4}
\def\clipbottom{-4.3}
\def\cliptop{8.8}

	\clip (\clipleft,\clipbottom) rectangle (\clipright,\cliptop);

\def\da{5}
\def\db{3}
\def\dc{5}

\def\typeall#1#2#3#4#5#6{ 
	\pgfmathsetmacro{\shiftx}{.5*\da*#1 + .5*\da*#2 - .5*\db*#3 - .5*\db*#4}
	\pgfmathsetmacro{\shifty}{.5*\db*#3 + .5*\db*#4 - .5*\dc*#5 - .5*\dc*#6}

	\pgfmathparse{(\shiftx >= \clipleft-2) && (\shiftx <= \clipright+2) &&
				  (\shifty >= \clipbottom-2) && (\shifty <= \cliptop+2)}
	\ifnum\pgfmathresult=1

	\begin{scope}[
		arcnode/.style={
			decoration={
				markings,
				mark=at position .5 with {\arrow{>}};
			},
			postaction={decorate}
		},
		shift={(\shiftx,\shifty)},
		scale=.4,
		every node/.style={scale=.6}
	]

		\coordinate (v1) at (-1,0) {} ;
		\coordinate (v2) at (1,0) {} ;
		\coordinate (v3) at (0,1) {} ;
		\coordinate (v4) at (0,0) {} ;
		\coordinate (v5) at (0,-1) {} ;

		\ifnum#1=#2 
			\draw [arcnode] (v1) arc (180:0:1) coordinate [midway,label=above:#1];
		\else
			\draw [arcnode] (v1) arc (180:90:1) coordinate [midway,label=above left:#1];
			\draw [arcnode] (v3) arc (90:0:1) coordinate [midway,label=above right:#2];
			\node [circle,fill,inner sep=1.5pt,label=above:1] at (v3) {};
		\fi

		\ifnum#3=#4
			\draw [arcnode] (v1) -- (v2) coordinate [midway,label=above:#3];
		\else
			\draw [arcnode] (v1) -- (v4) coordinate [midway,label=below:#3];
			\draw [arcnode] (v4) -- (v2) coordinate [midway,label=below:#4];
			\node [circle,fill,inner sep=1.5pt,label=above:1] at (v4) {};
		\fi

		\ifnum#5=#6
			\draw [arcnode] (v1) arc (180:360:1) coordinate [midway,label=below:#5];
		\else
			\draw [arcnode] (v1) arc (180:270:1) coordinate [midway,auto,label=below left:#5];
			\draw [arcnode] (v5) arc (270:360:1) coordinate [midway,auto,label=below right:#6];
			\node [circle,fill,inner sep=1.5pt,label=below:1] at (v5) {};
		\fi

		\pgfmathtruncatemacro{\leftsum}{#1+#3+#5}
		\pgfmathparse{\leftsum > 0}
		\ifnum\pgfmathresult=1
			\node [circle,fill,inner sep=1.5pt,label=left:\leftsum] at (v1) {};
		\fi

		\pgfmathtruncatemacro{\rightsum}{-#2-#4-#6}
		\pgfmathparse{\rightsum < 0}
		\ifnum\pgfmathresult=1
			\node [circle,fill,inner sep=1.5pt,label=right:\rightsum] at (v2) {};
		\fi
	\end{scope}

	\coordinate (p1) at (#2*\da - #4*\db, #4*\db - #6*\dc);
 	\coordinate (p2) at (#1*\da - #3*\db, #3*\db - #5*\dc);

	\ifnum11=
		\ifnum#1=#2 1\fi \ifnum#3=#4 1\fi \ifnum#5=#6 1\fi
		
		\coordinate (start) at ($(p1)!.15!(p2)$);
		\coordinate (end) at ($(p1)!.35!(p2)$);
		\draw [color=gray] (start) -- (end);
		\coordinate (start) at ($(p1)!.65!(p2)$);
		\coordinate (end) at ($(p1)!.85!(p2)$);
		\draw [color=gray] (start) -- (end);
	\fi

	\fi
}

\foreach \al in {-1,0,1,2} {
	\foreach \bl in {-1,0,1} {
		\foreach \s in {0,1} {
			\foreach \t in {0,1} {
				\foreach \u in {0,1} {
					\pgfmathtruncatemacro{\v}{\s+\t+\u}
					\ifnum\v<3
						\pgfmathtruncatemacro{\maxltotal}{2-\v}
						\foreach \ltotal in {0, ..., \maxltotal} {
							\pgfmathtruncatemacro{\cl}{\ltotal-\al-\bl}
							\pgfmathtruncatemacro{\ar}{\al+\s}
							\pgfmathtruncatemacro{\br}{\bl+\t}
							\pgfmathtruncatemacro{\cr}{\cl+\u} 
							\typeall \al \ar \bl \br \cl \cr
						}
					\fi
				}
			}
		}
	}
}

\draw [color=red,thick,fill=pink,opacity=.4] (0,0) -- (\da+\db,-\db) -- (\da,\dc) -- (-\db,\db+\dc) -- cycle;

\end{tikzpicture}

\label{fig:fund-dom}
\caption{A fundamental domain for the quotient $\Hom(H_1(\tropfont X), \mathbf R) / \partial H_1(\tropfont X)$ and the subdivision, under an isomorphism to $\TroPic^2(\tropfont X)$, into regions parameterizing balanced tropical divisors on quasistable models of $\tropfont X$.}
\end{figure}

The same reasoning applies equally well to any subdivision $\tropfont Y$ of $\tropfont X$.  Since $\TroPic(\tropfont Y) = \TroPic(\tropfont X)$ and $\Hom(H_1(\tropfont Y), \overnorm M_S^{\rm gp}) = \Hom(H_1(\tropfont X), \overnorm M_S^{\rm gp})$, giving $D \in H^0(\tropfont Y, \mathsf V)$ and a trivialization of $\mathsf H(D)$ will also produce points in $\TroPic(\tropfont X)$ and $\Hom(H_1(\tropfont X), \overnorm M_S^{\rm gp})$.  Figure~\ref{fig:fund-dom} shows a piece of $\Hom(H_1(\tropfont X), \mathbf R)$ with horizontal coordinate $e_1 - e_2$ and vertical coordinate $e_2 - e_3$.  For $D \in H^0(\tropfont X, \mathsf V)$ and trivialization of $\mathsf H(D)$ chosen according to the following rules, we have plotted a picture of those data at the corresponding position in $\Hom(H_1(\tropfont X), \mathbf R)$:  
\begin{enumerate}
\item $D$ is supported on a quasistable model $\tropfont Y$ of $\tropfont X$, meaning that each edge of $\tropfont X$ is subdivided at most once;
\item if $v \in \tropfont Y$ is a point of subdivision of $\tropfont X$ then $D(v) = 1$;
\item we have $0 \leq D(v_1) \leq 2$ and $-2 \leq D(v_2) \leq 0$.
\end{enumerate}
In the picture, each vertex $v$ is labelled by $D(v)$ unless $D(v) = 0$ and each edge is labelled by the slope it has been assigned in a choice of trivialization of $\mathsf H(D)$.  The shaded parallelogram is the fundamental domain~\eqref{eqn:116} for the quotient by $\partial H_1(\tropfont X)$.
\begin{equation} \label{eqn:116}
\big\{ x \partial(e_1 - e_2) + y \partial(e_2 - e_3) \: \big| \:  0 \leq x \leq 1 \text{ and } 0 \leq y \leq 1 \big\}
\end{equation}

This subdivision is suggested by Caporaso's compactification of $\Pic^2(X)$.  We originally computed it with the help of Margarida Melo, Martin Ulirsch, and Filippo Viviani.  The same example also appears in \cite[Figure~1]{abks} and \cite[Figure~4]{abreu-pacini}.

\subsection{Nonmaximal degeneracy}
\label{sec:nonmax}

Let us finally look at an example which is not maximally degenerate. Suppose $X$ is the union of two curves $Y_1$ and $Y_2$, glued along two points $p_1,p_2$, with $p_i$ in the first copy glued to $p_i$ in the second copy. The dual graph $\tropfont X$ of $X$ is again topologically a circle, with two vertices, $v_1$ and $v_2$, and two edges, $e_1$ and $e_2$, with lengths $\delta_1$ and $\delta_2$.  As in Section~\ref{sec:tate}, we find that $\TroJac(\tropfont X/S) = \ologGm^\dagger / \mathbf Z (\delta_1 + \delta_2)$ and $\LogPic^0(X/S)$ is an extension of this torus by the algebraic Jacobian.

To compute $\LogPic^0(X/S)$, we use the quotient presentation from Corollary~\ref{cor:quotient}.  Equation~\eqref{eqn:30} presents $H^1(X, \pi^\ast \logGm)$ as an extension of $H^1(X^\nu, \Gm) = \Pic(Y_1) \times \Pic(Y_2)$ by $\Hom(H_1(\tropfont X), \logGm)$:
\begin{equation} \label{eqn:81}
0 \to \logGm^\dagger \to H^1(X, \pi^\ast \logGm)^\dagger \to \Pic^0(Y_1) \times \Pic^0(Y_2) \to 0
\end{equation}
Then Corollary~\ref{cor:quotient} says that $\LogPic^0(X/S)$ is the quotient of $H^1(X, \pi^\ast \logGm)^\dagger$ by $H_1(\tropfont X)$.  

In general, the composition~\eqref{eqn:83} is nonzero:
\begin{equation} \label{eqn:83}
H_1(\tropfont X) \to H^1(X, \pi^\ast \logGm) \to H^1(X^\nu, \Gm) = \Pic^0(Y_1) \times \Pic^0(Y_2)
\end{equation} 
Indeed, recall that the map $H_1(\tropfont X) \to H^1(X, \pi^\ast \logGm)$ is induced from the composition~\eqref{eqn:82},
\begin{equation} \label{eqn:82}
H_1(\tropfont X) \subset H^0(X, \overnorm M_{X/S}^{\rm gp}) \to H^1(X, \pi^\ast \overnorm M_S^{\rm gp})
\end{equation}
which was itself induced from the short exact sequence~\eqref{eqn:15}.  Identifying $H^0(X, \overnorm M_{X/S}^{\rm gp}) = \mathbf Z^E$, where $E$ is the set of edges of $\tropfont X$, the basis element $e$ corresponding to the node $p$ is sent to $(\mathcal O_{Y_1}(p), \mathcal O_{Y_2}(-p))$.  Therefore the basis $e_1 - e_2$ of $H_1(\tropfont X)$ is sent to $(\mathcal O_{Y_1}(p_1 - p_2), \mathcal O_{Y_2}(-p_1 + p_2))$.

If $Y_1$ or $Y_2$ has positive genus, the map $H_1(\tropfont X) \to \Pic^{[0]}(X)$ is therefore nonzero, and will even be injective if $\mathcal O_{Y_i}(p_1 - p_2)$ is not a torsion point of the Jacobians of both curves.  This shows that the surjection ${H^1(X, \pi^\ast \logGm)^{[0]}}^\dagger \to \Pic^{[0]}(X^\nu/S)$ does not factor through $\LogPic^0(X/S)$, even though its restriction to $\Pic^{[0]}(X/S) \subset H^1(X, \pi^\ast \logGm)$ does factor through its image in $\LogPic^0(X/S)$.  Indeed, the map $\Pic^{[0]}(X/S) \to \LogPic^0(X/S)$ is injective by Theorem~\ref{thm:trop-pic}.

\bibliographystyle{amsalpha}
\bibliography{logpic}

\providecommand{\bysame}{\leavevmode\hbox to3em{\hrulefill}\thinspace}
\providecommand{\MR}{\relax\ifhmode\unskip\space\fi MR }
\providecommand{\MRhref}[2]{%
  \href{http://www.ams.org/mathscinet-getitem?mr=#1}{#2}
}
\providecommand{\href}[2]{#2}
\begin{thebibliography}{CCUW20}

\bibitem[ABKS14]{abks}
Yang An, Matthew Baker, Greg Kuperberg, and Farbod Shokrieh, \emph{Canonical
  representatives for divisor classes on tropical curves and the matrix-tree
  theorem}, Forum Math. Sigma \textbf{2} (2014), e24, 25. \MR{3264262}

\bibitem[AC13]{AMINI20131}
Omid Amini and Lucia Caporaso, \emph{Riemann–roch theory for weighted graphs
  and tropical curves}, Advances in Mathematics \textbf{240} (2013), 1 -- 23.

\bibitem[AC14]{AC}
Dan Abramovich and Qile Chen, \emph{Stable logarithmic maps to
  {D}eligne-{F}altings pairs {II}}, Asian J. Math. \textbf{18} (2014), no.~3,
  465--488. \MR{3257836}

\bibitem[ACV03]{abramovich2003twisted}
Dan Abramovich, Alessio Corti, and Angelo Vistoli, \emph{Twisted bundles and
  admissible covers}, Communications in Algebra \textbf{31} (2003), no.~8,
  3547--3618.

\bibitem[AK79]{altman1980compactifying2}
Allen~B. Altman and Steven~L. Kleiman, \emph{Compactifying the {P}icard scheme.
  {II}}, Amer. J. Math. \textbf{101} (1979), no.~1, 10--41. \MR{527824}

\bibitem[AK80]{altman1980compactifying1}
\bysame, \emph{Compactifying the {P}icard scheme}, Adv. in Math. \textbf{35}
  (1980), no.~1, 50--112. \MR{555258}

\bibitem[AP18]{abreu-pacini}
Alex Abreu and Marco Pacini, \emph{The universal tropical {Jacobian} and the
  skeleton of the {Esteves'} universal {Jacobian}}, June 2018,
  \href{https://arxiv.org/abs/1806.05527}{arXiv:1806.05527}.

\bibitem[Art73]{sga4-XII}
Michael Artin, \emph{Th\'eor\`eme de changement de base pour un morphisme
  propre}, Th\'eorie des topos et cohomologie \'etale des sch\'emas. {T}ome 3,
  Lecture Notes in Mathematics, Vol. 305, Springer-Verlag, Berlin-New York,
  1973, S\'eminaire de G\'eom\'etrie Alg\'ebrique du Bois-Marie 1963--1964 (SGA
  4), Dirig\'e par M. Artin, A. Grothendieck et J. L. Verdier. Avec la
  collaboration de P. Deligne et B. Saint-Donat, pp.~79--131. \MR{0354654}

\bibitem[Bar20]{Barrott}
Lawrence~Jack Barrott, \emph{Logarithmic chow theory}, 2020.

\bibitem[BN07]{baker_norine_2007}
Matthew Baker and Serguei Norine, \emph{Riemann-{R}och and {A}bel-{J}acobi
  theory on a finite graph}, Adv. Math. \textbf{215} (2007), no.~2, 766--788.
  \MR{2355607}

\bibitem[BV12]{MR2964607}
Niels Borne and Angelo Vistoli, \emph{Parabolic sheaves on logarithmic
  schemes}, Adv. Math. \textbf{231} (2012), no.~3-4, 1327--1363. \MR{2964607}

\bibitem[Cap94]{caporaso1994a}
Lucia Caporaso, \emph{A compactification of the universal {P}icard variety over
  the moduli space of stable curves}, J. Amer. Math. Soc. \textbf{7} (1994),
  no.~3, 589--660. \MR{1254134}

\bibitem[Cap08a]{caporaso2008compactified}
\bysame, \emph{Compactified {J}acobians, {A}bel maps and theta divisors},
  Curves and abelian varieties, Contemp. Math., vol. 465, Amer. Math. Soc.,
  Providence, RI, 2008, pp.~1--23. \MR{2457733}

\bibitem[Cap08b]{caporaso2008neron}
\bysame, \emph{N\'eron models and compactified {P}icard schemes over the moduli
  stack of stable curves}, Amer. J. Math. \textbf{130} (2008), no.~1, 1--47.
  \MR{2382140}

\bibitem[CCUW20]{cavalieri2017moduli}
Renzo Cavalieri, Melody Chan, Martin Ulirsch, and Jonathan Wise, \emph{A moduli
  stack of tropical curves}, Forum Math. Sigma \textbf{8} (2020), Paper No.
  e23, 93. \MR{4091085}

\bibitem[Che14]{Chen}
Qile Chen, \emph{Stable logarithmic maps to {D}eligne-{F}altings pairs {I}},
  Ann. of Math. (2) \textbf{180} (2014), no.~2, 455--521. \MR{3224717}

\bibitem[Chi15]{chiodo2015n}
Alessandro Chiodo, \emph{{N\'eron} models of {Pic} via {Pic}}, arXiv preprint
  arXiv:1509.06483 (2015).

\bibitem[D'S79]{dsouza1979compactification}
Cyril D'Souza, \emph{Compactification of generalised {J}acobians}, Proc. Indian
  Acad. Sci. Sect. A Math. Sci. \textbf{88} (1979), no.~5, 419--457.
  \MR{569548}

\bibitem[Est01]{MR1828599}
Eduardo Esteves, \emph{Compactifying the relative {J}acobian over families of
  reduced curves}, Trans. Amer. Math. Soc. \textbf{353} (2001), no.~8,
  3045--3095. \MR{1828599}

\bibitem[FRTU16]{foster2016logarithmic}
Tyler Foster, Dhruv Ranganathan, Mattia Talpo, and Martin Ulirsch,
  \emph{Logarithmic picard groups, chip firing, and the combinatorial rank},
  arXiv preprint arXiv:1611.10233 (2016).

\bibitem[Ful93]{Fulton}
William Fulton, \emph{Introduction to toric varieties}, Annals of Mathematics
  Studies, vol. 131, Princeton University Press, Princeton, NJ, 1993, The
  William H. Roever Lectures in Geometry. \MR{1234037}

\bibitem[GD67]{EGA}
A.~Grothendieck and J.~Dieudonn\'e, \emph{\'el\'ements de g\'eom\'etrie
  alg\'ebrique}, Inst. Hautes \'Etudes Sci. Publ. Math.
  (1960,1961,1961,1963,1964,1965,1966,1967), no.~4,8,11,17,20,24,28,32,
  228,222,167,91,259,231,255,361.

\bibitem[GK08]{gathmann_kerber_2008}
Andreas Gathmann and Michael Kerber, \emph{A {R}iemann-{R}och theorem in
  tropical geometry}, Math. Z. \textbf{259} (2008), no.~1, 217--230.
  \MR{2377750}

\bibitem[Gri17]{grillet2017semigroups}
Pierre~A Grillet, \emph{Semigroups: an introduction to the structure theory},
  Routledge, 2017.

\bibitem[Gro66]{ega4-3}
A.~Grothendieck, \emph{\'{E}l\'{e}ments de g\'{e}om\'{e}trie alg\'{e}brique.
  {IV}. \'{E}tude locale des sch\'{e}mas et des morphismes de sch\'{e}mas.
  {III}}, Inst. Hautes \'{E}tudes Sci. Publ. Math. (1966), no.~28, 255.
  \MR{217086}

\bibitem[Gro72]{sga7-IX}
Alexandre Grothendieck, \emph{{Mod\`eles} de {N\'eron} et monodromie}, Groupes
  de monodromie en g\'{e}om\'{e}trie alg\'{e}brique. {I}, Lecture Notes in
  Mathematics, Vol. 288, Springer-Verlag, Berlin-New York, 1972, S\'{e}minaire
  de G\'{e}om\'{e}trie Alg\'{e}brique du Bois-Marie 1967--1969 (SGA 7 I),
  Dirig\'{e} par A. Grothendieck. Avec la collaboration de M. Raynaud et D. S.
  Rim, pp.~313--523. \MR{0354656}

\bibitem[Gro95]{FGA4}
Alexander Grothendieck, \emph{Techniques de construction et th\'eor\`emes
  d'existence en g\'eom\'etrie alg\'ebrique. {IV}. {L}es sch\'emas de
  {H}ilbert}, S\'eminaire {B}ourbaki, {V}ol.\ 6, Soc. Math. France, Paris,
  1995, pp.~Exp.\ No.\ 221, 249--276. \MR{1611822}

\bibitem[GS13]{GS}
Mark Gross and Bernd Siebert, \emph{Logarithmic gromov-witten invariants},
  Journal of the American Mathematical Society \textbf{26} (2013), no.~2,
  451--510.

\bibitem[GV72]{sga4-VI}
A.~Grothendieck and J.-L. Verdier, \emph{Conditions de finitude. topos et sites
  fibr\'es. application aux questions de passage \`a la limite}, Lecture Notes
  in Mathematics, Vol. 270, Springer-Verlag, Berlin-New York, 1972,
  S\'{e}minaire de G\'{e}om\'{e}trie Alg\'{e}brique du Bois-Marie 1963--1964
  (SGA 4), Dirig\'{e} par M. Artin, A. Grothendieck et J. L. Verdier. Avec la
  collaboration de N. Bourbaki, P. Deligne et B. Saint-Donat. \MR{0354653}

\bibitem[GW]{GW}
William~D. Gillam and Jonathan Wise, \emph{Tropicalization of logarithmic
  schemes}, In preparation.

\bibitem[{Hah}07]{Hahn}
H.~{Hahn}, \emph{{\"Uber die nichtarchimedischen Gr\"o{\ss}ensysteme.}}, {Wien.
  Ber.} \textbf{116} (1907), 601--655 (German).

\bibitem[Her19]{Herr}
Leo Herr, \emph{The log product formula}, 2019.

\bibitem[H{\"o}l01]{Holder}
O.~H{\"o}lder, \emph{{Die Axiome der Quantit\"at und die Lehre vom Ma{\ss}}}
  (German).

\bibitem[Ill94]{illusie1994logarithmic}
Luc Illusie, \emph{Logarithmic spaces (according to {K. Kato})}, Barsotti
  symposium in algebraic geometry, eds. V. Cristante and W. Messing,
  Perspectives in Mathematics, vol.~15, 1994, pp.~183--204.

\bibitem[Ish78]{ishida1978compactifications}
Masa-Nori Ishida, \emph{Compactifications of a family of generalized {J}acobian
  varieties}, Proceedings of the {I}nternational {S}ymposium on {A}lgebraic
  {G}eometry ({K}yoto {U}niv., {K}yoto, 1977), Kinokuniya Book Store, Tokyo,
  1978, pp.~503--524. \MR{578869}

\bibitem[Jar00]{jarvis2000compactification}
Tyler~J. Jarvis, \emph{Compactification of the universal {P}icard over the
  moduli of stable curves}, Math. Z. \textbf{235} (2000), no.~1, 123--149.
  \MR{1785075}

\bibitem[Kaj93]{kajiwara1993logarithmic}
Takeshi Kajiwara, \emph{Logarithmic compactifications of the generalized
  {J}acobian variety}, J. Fac. Sci. Univ. Tokyo Sect. IA Math. \textbf{40}
  (1993), no.~2, 473--502. \MR{1255052}

\bibitem[Kat]{Kato-LogDeg}
Kazuya Kato, \emph{Logarithmic degeneration and {D}ieudonn\'e theory},
  Unpublished manuscript.

\bibitem[Kat89]{kato1989logarithmic}
\bysame, \emph{Logarithmic structures of {F}ontaine-{I}llusie}, Algebraic
  analysis, geometry, and number theory ({B}altimore, {MD}, 1988), Johns
  Hopkins Univ. Press, Baltimore, MD, 1989, pp.~191--224. \MR{1463703}

\bibitem[Kat19]{kato2019logarithmic}
Kazuya Kato, \emph{Logarithmic structures of fontaine-illusie. ii}, 2019.

\bibitem[Kim10]{Kim}
Bumsig Kim, \emph{Logarithmic stable maps}, New developments in algebraic
  geometry, integrable systems and mirror symmetry ({RIMS}, {K}yoto, 2008),
  Adv. Stud. Pure Math., vol.~59, Math. Soc. Japan, Tokyo, 2010, pp.~167--200.
  \MR{2683209}

\bibitem[KKN08a]{kajiwara_kato_nakayama_2008}
Takeshi Kajiwara, Kazuya Kato, and Chikara Nakayama, \emph{Analytic log picard
  varieties}, Nagoya Mathematical Journal \textbf{191} (2008), 149–180.

\bibitem[KKN08b]{kajiwara_kato_nakayama_2008b}
\bysame, \emph{Logarithmic abelian varieties}, Nagoya Math. J. \textbf{189}
  (2008), 63--138. \MR{2396584}

\bibitem[KKN08c]{kajiwara_kato_nakayama_2008a}
\bysame, \emph{Logarithmic abelian varieties. {I}. {C}omplex analytic theory},
  J. Math. Sci. Univ. Tokyo \textbf{15} (2008), no.~1, 69--193. \MR{2422590}

\bibitem[KKN13]{kajiwara_kato_nakayama_2013}
\bysame, \emph{Logarithmic abelian varieties, {III}: logarithmic elliptic
  curves and modular curves}, Nagoya Math. J. \textbf{210} (2013), 59--81.
  \MR{3079275}

\bibitem[KKN15]{kajiwara_kato_nakayama_2015}
\bysame, \emph{Logarithmic abelian varieties, part iv: Proper models}, Nagoya
  Mathematical Journal \textbf{219} (2015), 9–63.

\bibitem[Li01]{Li1}
Jun Li, \emph{Stable morphisms to singular schemes and relative stable
  morphisms}, J. Differential Geom. \textbf{57} (2001), no.~3, 509--578.
  \MR{1882667}

\bibitem[Li02]{Li2}
\bysame, \emph{A degeneration formula of {GW}-invariants}, J. Differential
  Geom. \textbf{60} (2002), no.~2, 199--293. \MR{1938113}

\bibitem[Mel11]{melo2011compactified}
Margarida Melo, \emph{Compactified picard stacks over the moduli stack of
  stable curves with marked points}, Advances in Mathematics \textbf{226}
  (2011), no.~1, 727 -- 763.

\bibitem[MR20]{MR}
Davesh Maulik and Dhruv Ranganathan, \emph{Logarithmic donaldson-thomas
  theory}, 2020.

\bibitem[MW17]{aj}
Steffen Marcus and Jonathan Wise, \emph{Logarithmic compactification of the
  {A}bel--{J}acobi section}, August 2017,
  \href{https://arxiv.org/abs/1708.04471}{\texttt{arXiv:1708.04471}}.

\bibitem[MZ08]{mikhalkin2008tropical}
Grigory Mikhalkin and Ilia Zharkov, \emph{Tropical curves, their {J}acobians
  and theta functions}, Curves and abelian varieties, Contemp. Math., vol. 465,
  Amer. Math. Soc., Providence, RI, 2008, pp.~203--230. \MR{2457739}

\bibitem[Nak17a]{nakayama2017logarithmic}
Chikara Nakayama, \emph{Logarithmic {\'e}tale cohomology, ii}, Advances in
  Mathematics \textbf{314} (2017), 663--725.

\bibitem[Nak17b]{Nakayama}
\bysame, \emph{Logarithmic \'{e}tale cohomology, {II}}, Adv. Math. \textbf{314}
  (2017), 663--725. \MR{3658728}

\bibitem[Ogu18]{Ogus}
Arthur Ogus, \emph{Lectures on logarithmic algebraic geometry}, Cambridge
  Studies in Advanced Mathematics, vol. 178, Cambridge University Press,
  Cambridge, 2018. \MR{3838359}

\bibitem[Ols04]{olsson2004semistable}
Martin~C Olsson, \emph{Semistable degenerations and period spaces for polarized
  {K3} surfaces}, Duke Mathematical Journal \textbf{125} (2004), no.~1,
  121--203.

\bibitem[OS79]{oda1979compactifications}
Tadao Oda and C.~S. Seshadri, \emph{Compactifications of the generalized
  {J}acobian variety}, Trans. Amer. Math. Soc. \textbf{253} (1979), 1--90.
  \MR{536936}

\bibitem[Pan96]{MR1308406}
Rahul Pandharipande, \emph{A compactification over {$\overline {M}_g$} of the
  universal moduli space of slope-semistable vector bundles}, J. Amer. Math.
  Soc. \textbf{9} (1996), no.~2, 425--471. \MR{1308406}

\bibitem[Ran20]{Dhruv}
Dhruv Ranganathan, \emph{Logarithmic {G}romov-{W}itten theory with expansions},
  2020.

\bibitem[Ray70]{sga3-XVII}
M.~Raynaud, \emph{Groupes alg\'ebriques unipotents. extensions entre groupes
  unipotents et groupes de type multiplicatif}, Sch\'{e}mas en groupes. {III}:
  {S}tructure des sch\'{e}mas en groupes r\'{e}ductifs, S\'{e}minaire de
  G\'{e}om\'{e}trie Alg\'{e}brique du Bois Marie 1962/64 (SGA 3). Dirig\'{e}
  par M. Demazure et A. Grothendieck. Lecture Notes in Mathematics, Vol. 153,
  Springer-Verlag, Berlin-New York, 1970, pp.~viii+529. \MR{0274460}

\bibitem[{Sta}18]{stacks-project}
The {Stacks Project Authors}, \emph{\textit{Stacks Project}},
  \url{https://stacks.math.columbia.edu}, 2018.

\end{thebibliography}

\end{document}